\pdfoutput=1\relax
\documentclass{amsart}

\usepackage{verbatim}
\usepackage[textsize=scriptsize]{todonotes}
\usepackage{tikz-cd}
\usepackage{etoolbox}
\usepackage{etex}
\usepackage{wasysym}
\usepackage[T1]{fontenc}
\usepackage{chemarr}
\usepackage{amssymb}
\usepackage[leqno]{amsmath}
\usepackage{comment}
\usepackage{mathtools}
\usepackage{rotating}
\usepackage{wrapfig}
\usepackage{outlines}
\usepackage{graphicx}
\usepackage{scalerel}
\usepackage{bbm}
\usepackage{multicol}
\usepackage{float} 
\usepackage{amsthm}
\usepackage[all,arc]{xy}
\usepackage{stackrel}

\usepackage{spectralsequences}

\let\oldwidetilde\widetilde
\protected\def\widetilde{\oldwidetilde}

\newtheorem{theorem}{Theorem}

\usepackage{hyperref}
\usepackage{cleveref}

\hypersetup{
   colorlinks,
   linkcolor={red},
   citecolor={green!30!black},
   urlcolor={blue}
}

\usepackage{tikz}
\usetikzlibrary{matrix,arrows,decorations}
\usepackage{tikz-cd}

\usepackage{adjustbox}

\let\oldtocsection=\tocsection
 
\let\oldtocsubsection=\tocsubsection
 
\let\oldtocsubsubsection=\tocsubsubsection
 
\renewcommand{\tocsection}[2]{\hspace{0em}\oldtocsection{#1}{#2}}
\renewcommand{\tocsubsection}[2]{\hspace{1em}\oldtocsubsection{#1}{#2}}
\renewcommand{\tocsubsubsection}[2]{\hspace{2em}\oldtocsubsubsection{#1}{#2}}

\makeatletter
\newcommand{\leqnomode}{\tagsleft@true}
\newcommand{\reqnomode}{\tagsleft@false}
\makeatother

\theoremstyle{definition}

\newtheorem{nul}{}[section]
\newtheorem{dfn}[nul]{Definition}

\newtheorem{rmk}[nul]{Remark}

\newtheorem{cnstr}[nul]{Construction}
\newtheorem{cnv}[nul]{Convention}
\newtheorem{ntn}[nul]{Notation}
\newtheorem{exm}[nul]{Example}

\newtheorem{rec}[nul]{Recollection}

\newtheorem{qst}[nul]{Question}

\newtheorem{warn}[nul]{Warning}

\newtheorem*{warn*}{Warning}
\newtheorem*{dfn*}{Definition}
\newtheorem*{axm*}{Axiom}
\newtheorem*{ntn*}{Notation}
\newtheorem*{exm*}{Example}
\newtheorem*{exr*}{Exercise}
\newtheorem*{int*}{Intuition}
\newtheorem*{qst*}{Question}
\newtheorem*{rmk*}{Remark}

\theoremstyle{plain}

\newtheorem{thm}[nul]{Theorem}
\newtheorem{prop}[nul]{Proposition}

\newtheorem{lem}[nul]{Lemma}

\newtheorem{cnj}[nul]{Conjecture}
\newtheorem{cor}[nul]{Corollary}

\newtheorem*{thm*}{Theorem}
\newtheorem*{prop*}{Proposition}
\newtheorem*{cor*}{Corollary}
\newtheorem*{lem*}{Lemma}
\newtheorem*{cnj*}{Conjecture}

\DeclareMathOperator*{\colim}{colim}

\DeclareMathOperator{\cof}{cof}

\DeclareMathOperator{\Hom}{Hom}
\DeclareMathOperator{\Map}{Map}
\DeclareMathOperator{\End}{End}

\DeclareMathOperator{\Fun}{Fun}

\DeclareMathOperator{\Spec}{\mathrm{Spec}}
\DeclareMathOperator{\Spf}{\mathrm{Spf}}
\DeclareMathOperator{\gl1}{gl_1}
\DeclareMathOperator{\sl1}{sl^E_1}

\DeclareMathOperator{\pic}{pic}

\DeclareMathOperator{\Mod}{Mod}

\DeclareMathOperator{\Def}{Def}

\def\ct{\mathrm{ct}}

\def\Pr{\mathrm{Pr}}
\def\PrL{\mathrm{Pr}^{\mathrm{L}}}
\def\st{\mathrm{st}}
\def\rig{\mathrm{rig}}

\def\tr{\mathrm{tr}}
\def\height{\mathrm{height}}

\def\len{\mathrm{len}}
\def\supp{\mathrm{supp}}
\newcommand{\ev}[0]{{\mathrm{ev}}}

\def\E{\mathbb{E}}
\def\F{\mathbb{F}}
\def\G{\mathbb{G}}
\def\H{\mathbb{H}}
\def\N{\mathbb{N}}

\def\Q{\mathbb{Q}}

\def\Ss{\mathbb{S}}
\def\T{\mathbb{T}}
\def\W{\mathbb{W}}
\def\Z{\mathbb{Z}}
\def\cC{\mathcal{C}}
\def\CC{\mathcal{C}}
\def\DD{\mathcal{D}}
\def\AA{\mathcal{A}}

\def\EE{\mathcal{E}}
\def\D{\mathcal{D}}
\def\mE{\mathcal{E}}

\def\cO{\mathcal{O}}

\def\hR{\widehat{\mathcal{R}}}

\DeclareMathOperator{\CAlg}{CAlg}
\def\CAlgw{\mathrm{CAlg}^{\wedge}}
\def\Modw{\mathrm{Mod}^{\wedge}}

\def\one{\mathbbm{1}}

\def\Fp{\mathbb{F}_p}
\def\Fpbar{\overline{\mathbb{F}}_p}

\def\ku{\mathrm{ku}}

\def\KU{\mathrm{KU}}

\def\MU{\mathrm{MU}} 
\def\MUP{\mathrm{MUP}}

\def\Mfg{\mathcal{M}_{\mathrm{fg}}}

\def\m{\mathfrak{m}}
\def\modm{/\!\!/\m}
\def\mm{/\!\!/}
\def\ll{[\![}
\def\rr{]\!]}

\def\Alg{\mathrm{Alg}}
\def\CAlg{\mathrm{CAlg}}
\def\CRing{\mathrm{CRing}}

\def\CAlgh{\mathrm{CRing}}
\def\Modh{\mathrm{Mod}^{\heartsuit}}

\def\Sp{\mathrm{Sp}}
\def\LnfSp{L_n^f\mathrm{Sp}}

\def\Set{\mathrm{Set}}

\def\BSL{\mathrm{BSL}}
\def\Top{\mathrm{Top}}
\def\CHaus{\mathrm{CHaus}}

\def\op{\mathrm{op}}
\def\Gr{\mathrm{Gr}}
\def\gr{\mathrm{gr}}
\def\Fil{\mathrm{Fil}}

\def\cons{\mathrm{cons}}

\newcommand{\pushout}{\arrow[ul, phantom, "\ulcorner", very near start]}

\newcommand{\pointL}[2]{\one_{#1}[#2]}
\newcommand{\pointR}[1]{\Omega^{\infty}_{#1}}
\newcommand{\ldbl}{(\!(}
\newcommand{\QQ}{\mathbb{Q}}
\newcommand{\ZZ}{\mathbb{Z}}
\newcommand{\NN}{\mathbb{N}}
\newcommand{\val}{\nu}
\newcommand{\arc}{\mathrm{arc}}
\newcommand{\HRing}[2]{#1\ll t^{#2} \rr}

\newcommand{\rdbl}{)\!)}
\newcommand{\mate}[1]{{#1}^{\tiny\gemini}}
\newcommand{\wt}{\widetilde}

\newcommand{\Prig}{\mathrm{Pr^{rig}}}
\newcommand{\dualz}{\diamondsuit}
\newcommand{\dual}{\vee}
\newcommand{\qind}{\quad\in\quad}%
\global\long\def\oto#1{\xrightarrow{#1}}%
\definecolor{DefColor}{rgb}{0.6,0.15,0.25}
\newcommand{\Setc}{\mathrm{Set}}
\newcommand{\mdef}[1]{\textcolor{DefColor}{#1}} 
\newcommand{\tdef}[1]{\textit{\mdef{#1}}}
\newcommand{\deff}[1]{\tdef{#1}}
\newcommand{\idem}{\mathrm{idem}}
\newcommand{\Spd}{\mathrm{Spd}}
\def\id{\mathrm{id}}
\def\Id{\mathrm{Id}}
\newcommand{\NS}{\mathrm{NS}}
\newcommand{\NSP}{\mathrm{NS}^{\Pi}}
\newcommand{\CSpec}{\mathrm{Spec}^{\mathrm{cons}}_{T(n)}}
\newcommand\xqed[1]{%
  \leavevmode\unskip\penalty9999 \hbox{}\nobreak\hfill
  \quad\hbox{#1}}
\newcommand\tqed{\xqed{$\triangleleft$}}

\newcommand{\Npm}{t^{\pm 1 /p^{\infty}}}

\newcommand{\Np}{t^{1/p^{\infty}}}

\def\Perf{\mathrm{Perf}}
\newcommand{\noloc}{\;\mathord{:}\,}

\usepackage{amssymb,graphicx}

\usepackage[margin=1.5in]{geometry}

\newtoggle{draft}
\togglefalse{draft}

\iftoggle{draft} {

\newcommand{\NB}[1]{\todo[color=gray!40]{#1}}
\newcommand{\TODO}[1]{\todo[color=red]{#1}}
}{ 
\newcommand{\NB}[1]{}
\newcommand{\TODO}[1]{}
\renewcommand{\todo}[1]{}
\renewcommand{\todo}[1]{}

}

\title{The Chromatic Nullstellensatz}
\date{\today}

\author{Robert Burklund}
\address{Department of Mathematical Sciences, University of Copenhagen, Denmark}
\email{rb@math.ku.dk}

\author{Tomer M. Schlank}
\address{Einstein Institute of Mathematics, The Hebrew University of Jerusalem}
\email{tomer.schlank@gmail.com}

\author{Allen Yuan}
\address{Department of Mathematics, Columbia University, New York, NY, USA}
\email{allenyua@gmail.com}

\begin{document}
\maketitle

\begin{abstract}
We show that Lubin--Tate theories  attached to algebraically closed fields  are characterized among  $T(n)$-local $\mathbb{E}_{\infty}$-rings as those that satisfy an analogue of Hilbert's Nullstellensatz. Furthermore, we show that for  every $T(n)$-local $\mathbb{E}_{\infty}$-ring $R$, the collection of $\E_\infty$-ring maps from $R$ to such Lubin--Tate theories jointly detect nilpotence. In particular, we deduce that every non-zero $T(n)$-local $\mathbb{E}_{\infty}$-ring $R$ admits an $\E_\infty$-ring map to such a Lubin--Tate theory. 
As consequences, we construct $\E_{\infty}$ complex orientations of algebraically closed Lubin--Tate theories, compute the strict Picard spectra of such Lubin--Tate theories, and prove redshift for the algebraic $\mathrm{K}$-theory of arbitrary $\E_{\infty}$-rings.\vspace{0.1cm}  
\end{abstract}

\setcounter{tocdepth}{1}

\begin{figure}[H]
  \centering{}
  \setlength{\fboxsep}{-5pt}
  \setlength{\fboxsep}{5pt}
  \frame{\includegraphics[scale=0.05]{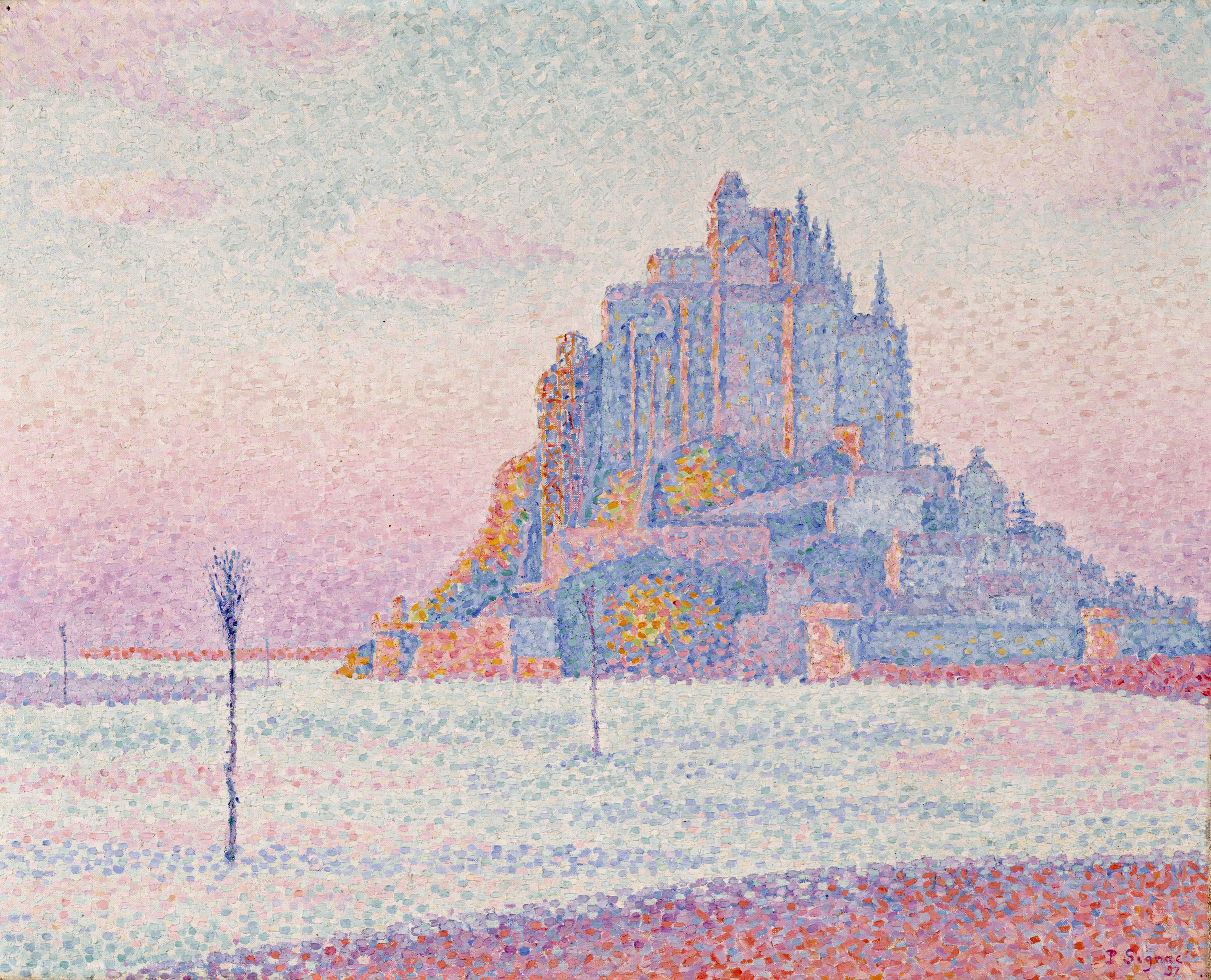}}
  \caption{\footnotesize Mont Saint-Michel, Setting Sun, Paul Signac,\\
    {[}Dallas Museum of Art, The Eugene and Margaret McDermott Art Fund, Inc., bequest of Mrs. Eugene McDermott in honor of Bill Booziotis{]}}
\end{figure}

\tableofcontents
\vbadness 5000



\section{Introduction}
\label{sec:introduction}

Stable homotopy theory has greatly benefited from insights offered by three fundamental perspectives. First, spectra should be considered as $\infty$-categorical analogues of abelian groups. Second, it is fruitful to generalize notions from algebra and algebraic geometry to the world of spectra. Third, these generalizations should avoid element-based formulae and be given in terms of categorical properties. In this way, for example, the Zariski spectrum of a ring is replaced by the notion of the Balmer spectrum, which presents chromatic homotopy theory as the analog for spectra of the primary decomposition for abelian groups. Some aspects of homotopy theory thus became akin to a game of Taboo, where classical notions from algebra are redefined without using the words, \emph{element}, \emph{equation} or \emph{subset}. 
\\

This paper aims to study such a redefinition for algebraically closed fields. The idea is that algebraically closed fields are precisely those commutative rings that satisfy a form of Hilbert's Nullstellensatz.

\begin{dfn}
  Let $\CC$ be a presentable $\infty$-category.
  We say that $\CC$ is \deff{Nullstellensatzian} if every compact and non-terminal object in $\CC$ admits a map to the initial object of $\CC$. Similarly, we say that an object $A \in \CC$ is  \deff{Nullstellensatzian} if $A$ is non-terminal and $\CC_{A/-}$ is Nullstellensatzian.
  \tqed
\end{dfn}

Hilbert's Nullstellensatz is essentially the statement that an object in the category of commutative rings satisfies the Nullstellensatz if and only if it is an algebraically closed field. 

Our first result is the classification of Nullstellensatzian $\E_\infty$-algebras in the monochromatic world. Through the connection between being Nullstellensatzian and being algebraically closed (in the sense of Galois theory) work of Baker and Richter \cite{BakerRichter} on Galois extensions suggests that the natural candidates for Nullstellensatzian $T(n)$-local $\E_\infty$-algebras are the Lubin-Tate theories attached to algebraically closed fields. Indeed, we show that these are exactly the Nullstellensatzian $T(n)$-local $\E_\infty$-algebras.


\begin{theorem}[Chromatic Nullstellensatz, \Cref{thm:alpha-null}]\label{thm:intro_Nullstellensatz}
A $T(n)$-local $\mathbb{E}_{\infty}$-algebra $R$ is Nullstellensatzian in  $\mathrm{CAlg}(\mathrm{Sp}_{T(n)})$ if and only if $R \cong E(L)$, where 
$E(L)$ is the Lubin-Tate spectrum attached to an algebraically closed field $L$.\footnote{ Here and throughout the paper, by ``$E(L)$ for an algebraically closed field $L$'' we mean for height $n>0$ that $\mathrm{char}\,L=p$ and $E(L)$ is as in \cite{GoerssHopkins, ECII},  and for height $n=0$ that $\mathrm{char}\,L=0$  and $E(L) \coloneqq L[u^{\pm 1}]$ where the generator $u$ is placed in degree $2$, that is, $E(L) \cong L\otimes KU$.}
\end{theorem}

\subsection{The constructible spectrum}\hfill

Given an arbitrary $T(n)$-local $\mathbb{E}_{\infty}$-algebra $R$, \Cref{thm:intro_Nullstellensatz} supports the idea of considering 
$\mathbb{E}_{\infty}$ maps $R\to E(L)$ out to Lubin--Tate theories of algebraically closed fields as ``geometric points
of $\mathrm{Spec}(R)$''. 
In classical algebra, the geometric points of $A$ are usually organized into a topological space---the Zariski spectrum of $A$. The utility of the spectrum comes from the fact that often algebraic questions over a base ring $A$ can be studied locally or even point-wise over $\mathrm{Spec}(A)$. 
Our understanding of the algebraically closed fields in $\CAlg(\Sp_{T(n)})$ allows us to make an analogous construction in the chromatic setting.



\begin{dfn}
Let $R \in \CAlg(\Sp_{T(n)})$. A \deff{geometric point} of $R$ is an equivalence class of maps 
$f\colon R \to E(L)$ for $L$ algebraically closed, under the equivalence relation identifying two maps $f_{1}\colon R \to E(L_{1})$ and $f_2 \colon R \to E(L_2)$ whenever 
$E(L_1)\otimes_{R} E(L_2) \neq 0$.\footnote{Note that for a discrete ring $A$, taking equivalence classes of maps $A \to L$ to algebraically closed fields under the analogous equivalence relation gives rise to the set of points of $\Spec(A).$ }
\tqed
\end{dfn}

The set\footnote{Although it is not immediate, the collection of geometric points of $R$ turns out to be a set.} of geometric points can be endowed with the so-called \textbf{constructible topology}, which gives it the structure of a compact Hausdorff topological space.  

\begin{theorem}[\Cref{thm:con-tn}]
There is a functor 
\deff{\[
\CSpec\colon \CAlg(\Sp_{T(n)})^{\op} \to \CHaus 
\]}
which sends $R \in \CAlg(\Sp_{T(n)}) $ to the set of geometric points endowed with the topology in which a subset $U \subset \CSpec(R)$ is closed if and only if it is the image a map 
$\CSpec(S)\to \CSpec(R)$ induced by some map of algebras  $R \to S$.
\end{theorem}


In the classical case, a not-often-mentioned property is that is that one can check whether an element $a \in A$ is nilpotent by checking whether it is nilpotent at every geometric point of $A$.
In fact, it is this property which guarantees that the Zariski spectrum of $A$ has enough points.
The analogous result in the $T(n)$-local setting is the following theorem:

\begin{dfn}
Let $R\in \CAlg(\Sp_{T(n)})$ and let 
$g\colon M \to N$ be a map of compact $T(n)$-local $R$-modules. 
We say that $g$ is \deff{nilpotent} if there exists some $k \gg 0$ such that 
$g^{\otimes k} \colon M^{\otimes k} \to N^{\otimes k}$
is null. 
A map $f \colon R \to S$ in $\CAlg(\Sp_{T(n)})$   \deff{detects nilpotence} if a map $M \to N$ of compact $T(n)$-local $R$-modules is nilpotent if and only if the induced map $M \otimes_R S \to N\otimes_R S$ is nilpotent in $\Mod_{S}(\Sp_{T(n)})^{\omega}$.
\tqed
\end{dfn}

\begin{theorem}[\Cref{prop:cons-surj-dn}]\label{thm:intro_points}
Let $f\colon R \to S$ be a map in $\CAlg(\Sp_{T(n)})$. Then $f$ detects nilpotence  if and only if the induced map $\CSpec(S) \to \CSpec(R) $ is surjective. 
\end{theorem}

\Cref{thm:intro_points} tells us that, in this theory, we have ``enough points''.  
In particular, since the map $R \to 0$ detects nilpotence only if $R=0$, we deduce that any nonzero $R$ has at least one geometric point. In other words:

\begin{theorem}[\Cref{cor:mod_algclosed}]\label{cor:intro_map_to_E}
Let $R$ be a non-zero $T(n)$-local $\E_\infty$-algebra.
Then there exists some algebraically closed field $L$ and a map of $\E_\infty$-algebras
$R \to E(L)$.
\end{theorem}

We hope that the invariant $\CSpec(-)$ will find additional applications in the future and will allow more ideas from algebraic geometry to be transported to the chromatic world.  In this paper, we initiate the study of $\CSpec(-)$ by computing $\CSpec(R)$ for some choice examples of $R \in \CAlg(\Sp_{T(n)})$.
\begin{prop}[\Cref{prop:spec-LT}, \Cref{prop:spec_affine_E}, \Cref{prop:spec_affine_T}]
Assume that $n\geq 1$, let $A$ be a perfect $\F_p$-algebra, and let $E(A)$ denote any Lubin--Tate theory associated to $A$ \cite{ECII}.  Then we have the following homeomorphisms in $\CHaus$, where $\Spec^{\mathrm{cons}}_{\mathrm{Zar}}(-)$ denotes the classical constructible spectrum:
\begin{enumerate}
    \item 
    \[\CSpec(E(A))\cong \Spec^{\mathrm{cons}}_{\mathrm{Zar}}(A)\]
    \item 
    \[\CSpec(E(A)[t])\cong \Spec^{\mathrm{cons}}_{\mathrm{Zar}}(A[t])\]
    \item 
    \[\CSpec(\one_{T(n)}[t]) \cong \Spec^{\mathrm{cons}}_{\mathrm{Zar}}(\F_p[t]))\]
\end{enumerate}
\end{prop}

\subsection{Applications}\hfill

\Cref{cor:intro_map_to_E} has wide-ranging applications in the study of $T(n)$-local $\E_\infty$-algebras.  As an immediate consequence we obtain an alternative proof for Hahn's celebrated result on the chromatic support of $\mathbb{E}_\infty$-algebras.

\begin{thm}[Hahn \cite{Hahnsupport}, \Cref{thm:Hahn-later}]\label{thm:hahn}
Let $R \in \mathrm{CAlg}(\mathrm{Sp})$. Then for every $n \geq 0$, we have that  $R\otimes T(n)=0$ implies $R\otimes T(n+1)=0$.
\end{thm}
Indeed, since algebras over the zero algebra are zero, \Cref{cor:intro_map_to_E} allows us to reduce the statement to the case of $E(L)$, where it is a straightforward computation.

In view of \Cref{thm:hahn}, it is natural to define the \deff{height} of a nonzero $\mathbb{E}_{\infty}$-algebra $R\in \mathrm{CAlg}(\mathrm{Sp})$ by
\[
\mdef{\mathrm{height}(R)} := \max\{n \geq -1  | T(n)\otimes R \neq 0 \}.
\footnote{Here we set $T(-1) = \mathbb{S}$.}\]

Based on computations at small heights, Ausoni and Rognes suggested a far-reaching conjectural organizing principle for the interaction between algebraic K-theory and chromatic height. This phenomena, known as \emph{redshift}, can be summarized by the slogan  ``algebraic K-theory raises the chromatic height by one.''
\Cref{cor:intro_map_to_E} allows us to prove this conjecture for arbitrary $\mathbb{E}_{\infty}$-algebras. Note that if $R$ is an $\mathbb{E}_{\infty}$-algebra, then $K(R)$ also admits the natural structure of an $\mathbb{E}_{\infty}$-algebra.  We get:

\begin{theorem}[Redshift for $\mathbb{E}_{\infty}$-algebras, \Cref{thm:redshift'}]\label{thm:redshift}
Let $0 \neq R \in \mathrm{CAlg}(\mathrm{Sp})$ be such that $\height(R) \geq 0$.  Then 
\[\mathrm{height}(K(R)) =  \mathrm{height}(R)+1. \]
\end{theorem}

The inequality  $\mathrm{height}(K(R)) \leq  \mathrm{height}(R)+1$ has been recently proved in the groundbreaking papers \cite{LMMT, CMNN}, so we are reduced to proving that the height always increases. Once again, \Cref{cor:intro_map_to_E} allows us to reduce the claim to the case of $E(L)$, where it was proven by the third author in \cite{Yuanred}.

Many of the best studied $\E_{\infty}$-algebras, including cobordism rings, occur as Thom spectra. 
As a consequence of \Cref{cor:intro_map_to_E} we find that $\E_\infty$-algebra maps from Thom spectra to algebraically closed Lubin--Tate theories, known as orientations, are particularly well-behaved.



\begin{theorem}[Universal orientability, \Cref{cor:main-E-picard}]\label{thm:Universal_orientability}
 Let $L$ be an algebraically closed field and let 
 \[ f \colon X \to \pic({\Mod_{E(L)}(\Sp_{T(n)})})\] 
 be a map in $\Sp_{\geq 0}$ with $T(n)$-local Thom spectrum $Mf$.  
 Then the following are equivalent:
\begin{enumerate}
    \item $Mf \neq 0$.
    \item There is map $Mf \to E(L)$ in ${\CAlg_{E(L)}(\Sp_{T(n)})}$.
    \item $f$ is null-homotopic.
    \item $Mf \cong E(L)[X] \in {\CAlg_{E(L)}(\Sp_{T(n)})}$.
\end{enumerate}
\end{theorem}

\begin{rmk}
\Cref{thm:Universal_orientability} can be considered as a higher (and categorified) version of orthogonality of characters. 
\tqed
\end{rmk}

For a  map  $g\colon X \to \pic(\Sp)$ in $\Sp_{\geq 0}$, it follows from \Cref{thm:Universal_orientability} that there exists an $\mathbb{E}_{\infty}$-algebra map $Mg \to E(\overline{\F}_p) $ if and only if $K(n)\otimes Mg\neq 0$ (\Cref{cor:Thom_sp}).

\begin{theorem}
\label{{cor:Universal_orientability}}
Taking $g$ to be the complex $J$-homomorphism $\ku \to \pic(\Sp)$, we obtain an equivalence of spaces 
\[\Map_{\CAlg(\Sp)}(\MUP,E(\overline{\F}_p)) \cong \Map_{\Sp_{\geq 0}}(\ku ,\gl1(E(\overline{\F}_p))).\]
In particular, there exists an $\E_{\infty}$-algebra map 
\[\MU \to E(\overline{\F}_p). \]
\end{theorem}

\begin{rmk}
The question of whether such $\E_{\infty}$ complex orientations of Lubin--Tate theories exist has a long history.  In \cite{Ando}, Ando gave a ``norm-coherence'' condition based on power operations for when a Lubin--Tate theory admits an $\mathbb{H}_{\infty}$ map from $\MU$. Building on work by Ando in the case of the Honda formal group, Zhu \cite{Zhu} checked this condition for all Lubin--Tate theories.  A general obstruction theory for constructing $\E_{\infty}$ complex orientations was described by Hopkins--Lawson \cite{HopkinsLawson}, which recovered previous  results of Walker and M\"ollers \cite{Walker, Mol} at height $1$.  The more general case of $\MUP$-orientations was demonstrated in a height $1$ example by \cite{HahnYuan}, and then proven for all Lubin--Tate theories of height $n\leq 2$ by Balderrama \cite{Balderrama}. 
\tqed
\end{rmk}

\Cref{thm:Universal_orientability} can, in turn, be be used to study the $\E_{\infty}$-algebra $E(L)$.  
In particular, we can compute its strict Picard spectrum.

\begin{theorem}[\Cref{thm:strict-pic}]\label{thm:intro_strict}
Let $E(L)$ be a Lubin--Tate spectrum attached to an algebraically closed field $L$.  Then there is an equivalence of connective spectra
\[
\mathrm{Hom}_{\mathrm{Sp_{\geq 0}}}(\mathbb{Z}, \pic({\Mod_{E(L)}(\Sp_{T(n)})})) \cong \Sigma^{n+2}\Z_p \oplus \Sigma L^{\times}.
\]
\end{theorem}

Desuspending both sides of the equality above,  we get an equivalence\footnote{This result was announced by Hopkins and Lurie and a sketch of their proof (which differs significantly from ours) was presented in the Thursday seminar in Spring 2014 (c.f. \cite[Theorem 1.6]{TNO}). As far as we are aware, there is currently no full written account of their approach.}
\[
\mathbb{G}_m(E(L)) \coloneqq \Hom_{\Sp_{\geq 0}}(\Z, \gl1(E(L))) \simeq \Sigma^{n+1}\Z_p \oplus L^{\times}.
\]
\Cref{thm:intro_strict} can be used to deduce other results about $E(L)$.  For example, in the upcoming paper \cite{fourier} of T. Barthel, S. Carmeli, L. Yanovski and the second author, it is shown that one can deduce the following implication:


\begin{prop}\label{prop:discrep_intro}
 Let $L$ be an algebraically closed field.
 Then, there is a fiber sequence
 \[ \tau_{\geq 0}\Sigma^{n} I_{\QQ_p/\ZZ_p} \to \gl1 E(L) \to \tau_{\geq 0} L_n^f\gl1 E(L) \]
 in $\Sp_{\geq 0}$.\footnote{In the case $n=0$, one can take an arbitrary prime and set $L_0^f = L_{\mathbb{S}[1/p]}$, or alternatively, take $L_0^f = L_{\QQ}$ and replace $I_{\QQ_p/\ZZ_p}$ with   $I_{\QQ/\ZZ}$.}
 In other words, the discrepancy spectrum of $E(L)$ is equivalent to the $n$-fold suspension of the Brown--Comenetz dual of the sphere.
 \end{prop}
 
 In \cite{HL}, it is further shown that from \Cref{thm:intro_strict}, one can deduce the following:
 
 \begin{prop}[{\cite[Corollary 5.4.10]{HL}}]
 The $E(k)$-cochains functor 
 \[(\mathcal{S}_p^{\leq n})^{\op} \to \CAlgw_{E(L)}\]
 \[X \mapsto  E(L)^X\]
 is fully faithful on $n$-truncated $p$-finite spaces. 
 \end{prop}

\subsection{The proof}\hfill

We now sketch the proof of our main theorems.  For this, it will be convenient to use the notational conventions of the body of this paper, so we suggest that the reader familiarizes themselves with the conventions listed at the end of this introduction at this point.  

In this paper we make considerable use of a key fact, due to Lurie, that Lubin--Tate spectra can be defined not only for perfect fields, but more generally for any perfect $\F_p$-algebra $A$.  We show that in fact the assignment $A \mapsto E(A)$ enjoys many nice properties.

\begin{thm}[\Cref{thm:E_functor}]\label{thm:intro_lubin_tate}
There is an adjunction  
\[
E(-) \colon \Perf_{k} \rightleftarrows \CAlgw_{E(L)} \noloc (\pi_0(-)/\m)^{\flat}
\]
where $\m$ denotes the Landweber ideal $(p, u_1, \cdots, u_{n-1})$ and $(-)^{\flat}$ denotes the inverse limit along Frobenius.  Furthermore,
\begin{enumerate}
    \item $E(-)$ is fully-faithful and preserves arbitrary products.
    \item $R \in \CAlgw_{E(k)}$ belongs to the essential image of $E(-)$ if and only if $R\modm$ has vanishing odd homotopy groups and $\pi_0(R)/\m = \pi_0(R\modm)$ is perfect.
\end{enumerate}
\end{thm}

The fully faithfulness of $E(-)$ implies that $\Perf_k$ is a colocalization of $\CAlgw_{E(k)}$ and in particular that every $R \in \CAlgw_{E(k)}$ has an $E$-colocalization map $E((\pi_0(R)/\m)^\flat) \to E$ universal among maps to $R$ from Lubin--Tate theories. 
$E(-)$ allows us to translate questions about Lubin--Tate theories in $\CAlg(\Sp_{T(n)})$ into questions in the $1$-category $\Perf_k$. 

\begin{rmk}
We find it noteworthy that  $\Perf_k$ admits so many different naturally defined fully faithful embeddings into so many different $\infty$-categories. We suggest that this should be considered as a manifestation of a general principle that ``derived perfect algebras are just perfect algebras.''
\tqed
\end{rmk}

The main technical result in this paper from which we deduce most of our theorems is the following:

\begin{thm}[\Cref{thm:modmain}]\label{thm:into_mapsout1}
  Given a $T(n)$-local commutative algebra $R$, there exists a perfect algebra $A$ of Krull dimension $0$  and a map of commutative algebras $R \to E(A)$ such that the base-change functor 
  \[
\Modw_R \to \Modw_{E(A)}
\]
detects nilpotence.
\end{thm}

The first five sections of the paper are dedicated to the proof of   \Cref{thm:into_mapsout1}. 
By the Devinatz--Hopkins--Smith nilpotence theorem \cite{DHS}, for every $T(n)$-local ring $R$, the map  $R\to R \otimes E(k)$ detects nilpotence. Thus, when proving  \Cref{thm:into_mapsout1}, we are free to work in $\CAlgw_{E(k)}$.

Conceptually our  basic strategy is now to inductively  map our commutative $E(k)$-algebra  $R$ to another commutative $E(k)$-algebra $R'$ in such a way that the map $ R\to R'$ detects nilpotence and such that  $R'$ is ``closer'' to satisfying the conditions of \Cref{thm:intro_lubin_tate}(2) together with the condition that $\pi_0(R)/\m$ is of Krull dimension $0$. 

\begin{exm} \label{exm:motivation}
Let us consider a $T(n)$-local commutative $E(k)$-algebra $R$ which has a non-zero class $\alpha \in \pi_1(R)$. As Lubin--Tate theories are even, we know that $R$ is not yet a Lubin--Tate theory and that we would like to kill $\alpha$. For this, we consider the pushout
\[ \begin{tikzcd}
E(k)\{z^1\} \ar[r, "z^1 \mapsto 0"] \ar[d, "z^1 \mapsto \alpha"]\ar[d] & E(k) \ar[d] \\
R \ar[r] & R'
\end{tikzcd} \]
in $\CAlgw_{E(k)}$. The map $R \to R'$ has the property that $\alpha$ is sent to zero, so as long as this map detects nilpotence we can proceed. For this we prove in \Cref{sec:nilpotence2} that nilpotence detecting maps are closed under co-base change, thereby reducing to the universal example of killing an odd degree class: the map 
\[ E(k)\{z^1\} \xrightarrow{z^1 \mapsto 0} E(k). \] 
\tqed
\end{exm}

Formalizing our strategy, we construct 
three maps  $f,g,h \in \CAlgw_{E(k)}$ with the following properties:
\begin{enumerate}
    \item All three maps $f,g,h$ detects nilpotence.
    \item $R \in \CAlgw_{E(k)}$ satisfies the right lifting property with respect to $f, g$ and $h$ if and only if  there exists a perfect $k$-algebra $A$ of Krull dimension $0$  and an equivalence  $R \cong E(A)$.
\end{enumerate}
Using a small object argument we show that every $R$ has a nilpotence detecting maps out to an object which has the right lifting property with respect to $f$, $g$ and $h$ and condition (2) then implies that this suffices to prove \Cref{thm:into_mapsout1}.


The three maps $f$, $g$ and $h$ which we use are
\begin{itemize}
    \item $f \colon E(k)[t] \to E(k)[t^{\pm 1}]\times E(k)$ 
    is the map associated to inverting $t$ and sending $t$ to zero on the two components.
    \item $g\colon E(k)\{z^0\} \to E(A )$ 
    is more difficult to construct, but is characterized by the fact that the induced map on $\pi_0(-)/\m$,
    $\pi_0(E(k)\{z^0\})/\m \to A$, is a direct limit perfection.
    \item $h\colon E(k)\{z^1\} \to E(k)$ is the map sending $z^1$ to zero which we have already encountered in \Cref{exm:motivation}.
\end{itemize}

The map $f$ is conservative, and therefore it is easily seen to detect nilpotence.  
The map $h$ is easy to construct but our proof that it detects nilpotence requires a strong form of the fact that the square of an odd degree class is null\footnote{At the prime $2$ this is a special feature of the fact that we are working $T(n)$-locally.} implied by the May nilpotence conjecture \cite{MNNmaynilp}.

The  map $g$ is also conservative and thus easily seen to detect nilpotence.  However, its construction requires strong control on maps \textbf{to} Lubin--Tate theories.  We obtain this control by utilizing the theory of power operations over $E(k)$.  Rezk studied the power operations on a $T(n)$-local commutative $E(k)$-algebra $R$ extensively in \cite{RezkCong}, and showed that these power operations endow $\pi_0(R)$ with the structure of an algebra over a certain monad $\T$ on $\Modh_{\pi_0(E)}$.  Composing with $\pi_0$, we thus get a functor 
 
 \[E^{\T}(-) :=\pi_0\circ E(-) \colon \Perf_{k} \to \Alg_{\T}. \]
 
 On the other hand, we have an obvious forgetful functor
 \[
U_{\T}\colon \Alg_{\T} \to \CAlgh_{\pi_0E(k)},
\]
which we may compose with the functor 
\[
(-\otimes_{\pi_0E(k)}k)^{\sharp} \colon \CRing_{\pi_0E(k)} \to \Perf_k
\]
to obtain a functor
\[
\overline{U}_{\T}\colon \Alg_{\T} \to \Perf_{k}.
\]
 
 The key step in constructing the map $g$ is an adjunction between the functors  $E^{\T}(-)$ and $\overline{U}_{\T}$, which we believe is of independent interest.   
 
 \begin{theorem}[Cofreeness of $\pi_0E(k)$, \Cref{thm:cofree}]\label{thm:intro_cofree}
The functor $E^{\T}( -)\colon \Perf_k \to \Alg_{\T}$ is fully faithful and right adjoint to $\overline{U}_{\T}$.
 \end{theorem}
 
 \begin{rmk}
 Note that while \Cref{thm:intro_lubin_tate} gives an embedding of  $\Perf_{k}$ via a \textbf{left} adjoint in  $\CAlgw_{E(k)}$, \Cref{thm:intro_cofree} gives an embedding of  $\Perf_{k}$ in $\Alg_{\T}$ via a \textbf{right} adjoint. We thus gain some control on maps both to and from $E(A)$ for perfect $A$, which we utilize to prove \Cref{thm:into_mapsout1}.
 \tqed
 \end{rmk}
 
 \begin{rmk}
 It is worth spelling out the content of \Cref{thm:intro_lubin_tate} at low heights explicitly.  The case $n=0$ is trivial, as the theorem reduces to the self adjunction of the identity functor on $\CRing_{\mathbb{Q}}$.  In the case $n=1$, the category $\Alg_{\T}$ is the category of $\delta$-$W(k)$-rings, and we recover the fully faithful embedding of perfect $k$-algebras in $\delta$-$W(k)$-rings, which is a classical theorem of Joyal on $p$-typical Witt vectors \cite{JoyalWitt}. As far as we are aware, the case $n=2$ is already new.  
\end{rmk}
 
 \begin{rmk}
 After writing the proof, it came to our attention that Charles Rezk had previously announced a proof of \Cref{thm:intro_cofree}. As we understand it, Rezk's proof differs significantly from  ours.
 \tqed
 \end{rmk}
 
 \begin{rmk}\label{rmk:intro-induct}
 For most of the paper, the height $n$ is fixed. There is, however, one exception. At one point in \Cref{sec:cofree} our proof uses an induction on $n$.  Specifically, we use the fact that for every perfect field $k$ and $n>0$, there exists a perfect field $k'$ and a map in $\CAlg(\Sp)$
 \[E_n(k) \to E_{n-1}(k').\]
 We deduce this fact from the height $n-1$ case of \Cref{cor:intro_map_to_E}, together with the fact that $L_{T(n-1)}E_n(k) \neq 0$.  For this reason, our proof of \Cref{cor:intro_map_to_E} is in fact inductive on the height $n$. 
 \end{rmk}
 
 \subsection{Organization of the paper} \hfill
 
 In \S 2, we prove \Cref{thm:intro_lubin_tate} and study the functors appearing in its statement.  In particular, we observe that the functor $(\pi_0(\bullet)/\m)^{\flat}$ can be adapted to quite a general setup, and we study it in that generality. 
 In \S 3, we study the power operations on Lubin--Tate theories and prove \Cref{thm:intro_cofree}. 
 In \S 4, we prove general statements about detecting nilpotence, which we employ in the proof of \Cref{thm:into_mapsout1}.
 In \S 5, we assemble the results of the previous sections to prove \Cref{thm:into_mapsout1}.
 In \S 6, we use \Cref{thm:into_mapsout1} to prove \Cref{thm:intro_Nullstellensatz}. 
 In \S 7, we use \Cref{thm:into_mapsout1} to define $\CSpec$ and prove \Cref{thm:intro_points}. The construction of $\CSpec$ can be generalized to the setting of an $\infty$-category $\cC$ satisfying relatively mild conditions, and we develop this general theory in Appendix A.  The results of Appendix A are used in \S 7 for the specific case $\cC = \CAlg(\Sp_{T(n)})$. 
  The final two sections are devoted to applications of the general theory. In \S 8, we prove \Cref{thm:Universal_orientability} and its applications, and in \S 9, we give applications to chromatic support and redshift. 
 
\subsection*{Acknowledgments}\hfill
 
We would like to thank 
Tobias Barthel,
Clark Barwick, 
Omer Ben-Neria, 
Shachar Carmeli,
Jeremy Hahn, 
Mike Hopkins, 
Moshe Kamensky, 
Jacob Lurie, 
Charles Rezk, 
Andy Senger, 
Nathaniel Stapleton, 
Lior Yanovski and 
the entire  Seminarak group for useful discussions. 
We would like to thank 
Shaul Barkan,
Shay Ben Moshe, 
Shai Keidar, 
Shaul Ragimov and 
Asaf Yekutieli 
for comments on previous drafts. 
The second author is supported by ISF1588/18 and BSF 2018389 and the third author is  supported in part by NSF grant DMS-2002029.    
 
 \subsection*{Conventions}
 \begin{enumerate}
     \item For the remainder of the paper we fix a height $n\in \mathbb{N}$ and a prime $p$. 
     \item If $n>0$, we fix a perfect field $k$ of characteristic $p$ and a height $n$ formal group  $\mathbb{G}_0$ over it.
     \item If $n=0$, we fix a characteristic $0$ field $k$.
     \item We say category for $\infty$-category.
      \item We say commutative algebra for $\mathbb{E}_{\infty}$-algebra, and write $\CAlg(-)$ for it. By contrast, $\CRing$ denotes the category of discrete commutative rings. 
    \item For $R\in \CAlg(\Sp_{T(n)})$ we write $\CAlg^{\wedge}_{R}$ for $\CAlg(\Sp_{T(n)})_{R/-}.$
     \item For a commutative algebra $R$, we denote by $R[x_1,\dots,x_d]$ the flat polynomial algebra  $R[\mathbb{N}^d]$.
     \item\label{ffff} For a commutative algebra $R$, we denote by $R\{z^{i}\}$ the free commutative algebra under $R$ with a class in $\pi_i$. When $R$ is $T(n)$-local, we occasionally abuse notation and  use   $R\{z^{i}\}$ for the $T(n)$-localized version of the free algebra.  
     \item  For a commutative algebra $R$ and $\alpha \in \pi_*(R)$, we denote by $R \mm^{\infty} \alpha$ the commutative $R$-algebra that corepresents a null homotopy of $\alpha$. We use this notation to differ from $R/\alpha$, which is just the cofiber of $\alpha$. As in (\ref{ffff}), we occasionally implicitly localize.  
     \item We denote $\Sp_p$ for the category of $p$-complete spectra.
     \item We denote $\mathrm{Perf}_k$ for the category of perfect $k$-algebras and $\mathrm{Perf}$ for $\mathrm{Perf}_{\mathbb{F}_p}$.
     \item $\mathrm{Perf}$ is presentable and the inclusion $\Perf \subset \CRing$ admits a left adjoint which we denote by $(-)^{\sharp}$ (colimit perfection) and a right adjoint which we denote by $(-)^{\flat}$ (limit perfection).  
     \item For a category $\CC$ and $a,b \in \CC$, we denote 
     \[[a,b]_{\CC} := \pi_0(\mathrm{Map}_{\CC}(a,b)).\]
     When the category $\CC$ is clear from context, we omit the subscript.
   \item We write $W(A)$ for the $p$-typical Witt vectors of a perfect $\F_p$-algebra $A$,
     reserving the notation $\W(A)$ for the $p$-complete spherical Witt vectors of \Cref{sec:tilting}.
   \end{enumerate}

\section{Lubin--Tate theories and tilting}
\label{sec:tilting}
In this section, which is preparatory for all those which follow, we introduce the first fundamental idea in the proof of \Cref{thm:into_mapsout1}: the $E$-colocalization map. 

\begin{thm} \label{thm:E-and-tilt}
  Let $k$ be a perfect field of characteristic $p$ and
  let $\Perf_k$ denote the category of perfect $k$-algebras.
  Then, there is an adjunction
  \[ E(-) \colon \Perf_k \rightleftarrows \CAlg_{E(k)}^{\wedge} \noloc (-)^{\flat} \]
  which identifies $\Perf_k$ as a colocalization of $\CAlg_{E(k)}^{\wedge}$.
\end{thm}

The main input to this theorem is Lurie's spherical Witt vector construction, which is the sphere analogue of the desired functor $E(-)$.  We review this construction in \Cref{subsec:thickenings} and discuss properties of it and its right adjoint, tilting, in Sections \ref{subsec:W-properties} and \ref{subsec:tilting}.  Then, we give some preliminaries on Lubin--Tate theories in \Cref{subsec:LT-theory} and use the previous constructions to construct the functor $E(-)$ and prove \Cref{thm:E-and-tilt} in \Cref{subsec:E}.


\subsection{The spherical Witt functor}
\label{subsec:thickenings}\ 

Let $A$ be a perfect $\F_p$-algebra.  Using deformation theory, Lurie \cite{ECII} constructed its ring of  \emph{spherical Witt vectors}, which is a $p$-complete flat $\Ss_p$-algebra $\W(A)$ such that 
\[
\pi_*\W(A) \cong \left( (\pi_* \Ss_p) \otimes W(A) \right)_p.
\]
In this subsection, which draws heavily from \cite[Section 5.2]{ECII}, we recall this construction and extend it to a Witt vector--tilting adjunction.

\begin{prop} \label{prop:witt-and-tilt}
  There is an adjunction
  \[ \W(-) \colon \Perf \rightleftarrows \CAlg(\Sp_p) \noloc (-)^{\flat} \]
  which identifies $\Perf$ as a colocalization of $\CAlg(\Sp_p)$.
  The right adjoint $(-)^\flat$ is computed by the inverse limit along Frobenius on $\pi_0(-)/p$.\footnote{This agrees with the use of $(-)^\flat$ for tilting in discrete commutative algebra and so we will often make use of the formula $R^\flat \cong (\pi_0R/p)^\flat$.}
  
  The essential image of the (fully faithful) functor $\W$ consists of those $R\in \CAlg(\Sp_p)$ such that $R$ is connective and $\F_p\otimes R$ is a discrete perfect ring.  In this situation, we have $R\simeq \W(\F_p\otimes R)$.  
\end{prop}

We prove \Cref{prop:witt-and-tilt} at the end of \Cref{subsec:thickenings}.

\begin{dfn}[{\cite[Definition 5.2.1]{ECII}}]\label{dfn:thickening}
  Given a discrete, commutative $\F_p$-algebra $A$ we say that a map $\sigma \colon R \to A$ in $\CAlg(\Sp_p)_{\geq 0}$ \deff{exhibits $R$ as an $\Ss_p$-thickening of $A$} if
  \begin{enumerate}
  \item[(a)] $\sigma$ induces an isomorphism $\pi_0(R)/p \to A$.
  \item[(b)] For any $S \in \CAlg(\Sp_p)$, the canonical map
    \[ \Map_{\CAlg}(R, S) \to \Map_{\CAlg}(A, \pi_0(S)/p) \]
    is an equivalence. \tqed
  \end{enumerate}
\end{dfn}

\begin{prop}[{\cite[Example 5.2.7]{ECII}}] \label{prop:thickenings-exist}
  Given a perfect $\F_p$-algebra $A$, there exists an $\Ss_p$-thickening
  $R \to A$ in the sense of \Cref{dfn:thickening}
  such that the natural map
  $ \F_p \otimes R \to A $
  is an equivalence.
\end{prop}

\begin{rmk} \label{rmk:thickening-flat}
  The equivalence $A \simeq \F_p \otimes R$ implies that the $\F_p$-homology of $R$ is discrete.
  In turn this implies that $R$ is flat over $\Ss_p$.
  \tqed
\end{rmk}

\begin{cnstr} \label{cnstr:Witt}
  Using \Cref{dfn:thickening}(b) to compute the spaces of maps between the $\Ss_p$-thickenings produced by \Cref{prop:thickenings-exist}, we find that the full subcategory of $\Ss_p$-thickenings of perfect $\F_p$-algebras is equivalent to the category of perfect $\F_p$-algebras (with equivalence given by $\pi_0(-)/p$). Using the inverse of this equivalence, we obtain a construction of the spherical Witt vectors of a perfect $\F_p$-algebra as a fully faithful functor, which we denote \deff{
  \[ \W(-) \colon \Perf \to \CAlg(\Sp_p). \]}
  \tqed
\end{cnstr}



\begin{proof}[Proof (of \Cref{prop:witt-and-tilt}).]
  From the construction of $\W(-)$ via $\Ss_p$-thickenings, we can read off that
  there is a natural equivalence between $(\pi_0(\W(A))/p)^{\flat}\cong \pi_0(\W(A))/p $ and the perfect $\F_p$-algebra $A$.
  We use this natural equivalence $\eta \colon \Id \xrightarrow{\cong} (\W(-))^\flat$
  as the unit of the Witt--tilt adjunction.
  In order to conclude that $\eta$ exhibits $(-)^\flat$ as right adjoint to $\W(-)$, it now suffices to argue that the induced map
  \begin{align*}
    \Map_{\CAlg(\Sp_p)}&(\W(A), S) \to \Map_{\CAlg(\Sp_p)}(\pi_0\W(A)/p, \pi_0(S)/p) \\
                &\to \Map_{\Perf}((\pi_0\W(A)/p)^\flat, (\pi_0(S)/p)^\flat) \xrightarrow{\eta} \Map_{\Perf}(A, (\pi_0(S)/p)^\flat)
  \end{align*}
  is an equivalence. The first map is an equivalence by condition (b) of the definition of $\Ss_p$-thickening. The second map is an equivalence because the source is discrete and perfect. The final map is an equivalence because $\eta$ is an equivalence.
  The claim that $\Perf$ is a colocalization is equivalent to the unit of the adjunction, $\eta$, being an equivalence (alternately, $\W$ is fully faithful by construction).
  
  Finally, the statement about the essential image follows from \cite[Proposition 5.2.9]{ECII} and property (b) in \Cref{dfn:thickening}, which guarantees the uniqueness of thickenings.
  \qedhere
  
\end{proof}

\subsection{Properties of $\W$}
\label{subsec:W-properties}\ 

There are two properties which make the spherical Witt vector functor particularly easy to work with:
\begin{itemize}
\item The spherical Witt vectors $\W(A)$ are $p$-adically formally \'etale over $\Ss_p$.
\item It is relatively easy to describe $\W(A)$ in terms of square-zero extensions.
\end{itemize}
We give precise meaning to these statements in the first two lemmas of \Cref{subsub:defW} and
spend the remainder of the subsection extracting easy consequences.  Among these consequences are a spherical analogue of the multiplicative lift\footnote{This construction is sometimes known as the Teichm\"uller lift.}, which we construct in \Cref{subsubsec:multlift}.

\subsubsection{Deformation theory and $\W$}
\label{subsub:defW}

\begin{lem}\label{lem:formally-etale}
  The spherical Witt vectors $\W(A)$ associated to a perfect $\F_p$-algebra $A$ are
  $p$-adically formally \'etale in the sense that
  the $p$-complete cotangent complex $(L_{\W(A)/\Ss_p})_p$ vanishes. 
\end{lem}

\begin{proof}
  Since $A = \W(A)\otimes \F_p$ by \Cref{prop:thickenings-exist} and $A$ is perfect by assumption, we have equivalences
  \[ \F_p \otimes L_{\W(A)/\Ss_p} \cong A \otimes_{\W(A)} L_{\W(A)/\Ss_p} \cong L_{A/\F_p} = 0. \]  
  Since $\W(A)$ is connective, $L_{\W(A)/\Ss_p}$ is connective as well \cite[Proposition 7.4.3.9(1)]{HA}.
  $\F_p$-homology is conservative on $p$-complete bounded below spectra,
  therefore $(L_{\W(A)/\Ss_p})_p = 0$ as desired.
\end{proof}

\begin{rmk} \label{rmk:using-formally-etale}
  If the unit map $\Ss_p \to R$ is $p$-adically formally \'etale,
  then deformation theory tells us that for any square zero extension $B \to A$, the induced map
  \[ \Map_{\CAlg(\Sp_p)}(R, B) \to \Map_{\CAlg(\Sp_p)}(R, A) \]
  is an equivalence \cite[Remark 7.4.1.8]{HA}. 
  \tqed
\end{rmk}

The classical Witt vectors $W(A)$ can be expressed as the inverse limit of its $p$-adic tower, which is an $\omega$-indexed tower of square zero extensions beginning at $A$.  There is an analogous statement for spherical Witt vectors:

\begin{lem} \label{lem:functorial-sqz-tower}
 The commutative algebra $\W(A)$ can be functorially expressed as the inverse limit of a $2\omega$-indexed tower starting at $A$, where each map is a square zero extension by a suspension of $A$.
\end{lem}

\begin{proof}
  The Postnikov tower of $\Ss_p$ is a tower of square-zero extensions building $\Ss_p$ from $\Z_p$.
  Serre finiteness lets us refine this to an $\omega$-indexed tower of extensions by copies of $\Sigma^j\F_p$.  
  The $p$-adic tower of $\Z_p$ is a $\omega$-indexed tower of square-zero extensions building $\Z_p$ from $\F_p$ using copies of $\F_p$.
  Tensoring this pair of towers with $\W(A)$ and observing that $\F_p \otimes \W(A) \simeq A$ allows us to conclude.
\end{proof}

Using \Cref{lem:functorial-sqz-tower} we can provide a simple criterion for a limit to be preserved by $\W$.

\begin{lem}\label{lem:W-limits}
  The spherical Witt vectors functor preserves those limits of perfect $\F_p$-algebras whose limit, taken in $\Sp_p$, is connective.
  Examples include:
  \begin{enumerate}
  \item $\W(-)$ commutes with arbitrary products.
  \item $\W(-)$ commutes with $\N$-indexed inverse limits whose $\lim^1$ vanishes.
  \item $\W(-)$ satisfies arc descent.
  \end{enumerate}
\end{lem}

\begin{proof}  
  Suppose we are given a limit diagram $F^{\triangleleft}\colon D^{\triangleleft} \to \Perf$ that we wish to show is preserved by $\W$. 
  It suffices to show that $F^{\triangleleft} \circ \W$ is a limit diagram on underlying spectra.
  Using \Cref{lem:functorial-sqz-tower} it therefore suffices to show that $F^{\triangleleft} \circ \Sigma^j$ is a limit diagram. The latter is a reformulation of the hypothesis on $F^{\triangleleft}$.
   


  For example (1) we note that arbitrary products of connective $\F_p$-modules are connective.
  The vanishing condition on $\lim^{1}$ in (2) is a restatement of connectivity.
  Example (3) follows from arc descent for the category of perfect complexes on perfect, qcqs schemes \cite[Theorem 5.16]{arctopology}.
\end{proof}

We also give another description of the essential image of $\W$.

\begin{lem} \label{lem:W-image}
A commutative algebra $R\in \CAlg(\Sp_p)_{\geq 0}$ is in the essential image of $\W$ if and only if $R$ is  $p$-adically formally \'etale and  $\pi_0(R)/p$ is perfect.

\end{lem}

\begin{proof}
 We have already seen that any $R$ in the essential image satisfies these conditions.  In light of \Cref{prop:witt-and-tilt}, we only need to check that if $R$ is $p$-adically formally \'etale, then it satisfies condition (b) of \Cref{dfn:thickening}.   Similarly to the proof of \Cref{lem:functorial-sqz-tower}, we may exhibit any $S\in \CAlg(\Sp_p)_{\geq 0}$ as a limit of a tower of square zero extensions of $\pi_0(S)/p$ (i.e., first along its $p$-adic tower on $\pi_0$, and then its Postnikov tower); then, since $R$ is $p$-adically formally \'etale, we obtain the desired equivalence
  \[ \Map_{\CAlg_p}(R, S) \xrightarrow{\cong} \Map_{\CAlg_p}(R, \pi_0(S)/p). \]  
\end{proof}

\subsubsection{The multiplicative lift}\hfill
\label{subsubsec:multlift}

Using our understanding of the essential image of $\W$, we can now give some examples.


\begin{exm} \label{exm:line}
Let $M$ be a (discrete) commutative monoid on which $p$ acts invertibly.  Then since $\F_p \otimes \Ss[M]_p  \cong \F_p[M]$ is discrete and perfect, \Cref{prop:witt-and-tilt} implies that there is an equivalence
\[
\W(\F_p[M]) \simeq \Ss[M]_p.
\]
In particular, applying this to the case $M = \N[1/p]$, the additive monoid of non-negative elements in $\mathbb{Z}[1/p]$, we have an equivalence
  \[ \W(\F_p[\Np]) \simeq \Ss[\Np]_p \]
  between the spherical Witt vectors on the perfection of a polynomial algebra and
  the $p$-complete spherical group ring on $\N[1/p]$.
  \tqed
\end{exm}

\begin{exm} \label{exm:free-algebra-perfd}
  Consider the $p$-completion of the commutative algebra
  \[ \colim  \left( \Ss\{x\} \xrightarrow{x \mapsto x^p} \Ss\{x\} \xrightarrow{x \mapsto x^p} \cdots \right) \]
  which might be described as the perfection of a free algebra.

  Since $\Ss\{x\}$ is a free algebra on a single class in degree zero,
  its cotangent complex is a free module on a class in degree 0 (which we might call $dx$).
  After tensoring down to $\F_p$, the induced maps on cotangent complexes are each zero as $dx \mapsto d(x^p) = px^{p-1} dx = 0$.
  This means that the perfection of a free algebra is $p$-adically formally \'etale.
  After examining what happens on $ \pi_0(-)/p$, we can use \Cref{lem:W-image} to conclude that this algebra is equivalent to $\W(\F_p[\Np])$ as well.  
  \tqed
\end{exm}

\begin{cnstr}[Multiplicative lifts for $\W$]\label{cnstr:teichmuller-lift}
For a discrete ring $A$, let $A^*$ denote the commutative monoid of nonzero elements of $A$ under multiplication (in contrast to $A^{\times}$, the group of units).  If $A$ is a perfect $\F_p$-algebra, then $p$ acts invertibly on $A^*$ and so the natural map
\[
\F_p[A^*] \to A
\]
is a map between perfect $\F_p$-algebras.  Applying the functor $\W$ and using the equivalence of  \Cref{exm:line}, we obtain a composite map
\[
\tau\colon \Ss[A^*]_p \xrightarrow{\simeq} \W(\F_p[A^*]) \to \W(A)
\]
which we refer to as the \deff{(spherical) multiplicative lift}.  
\tqed
\end{cnstr}

In the spherical setting, the multiplicative lift has a novel feature,
which is that it provides a simple way to construct many strict elements (cf. \Cref{rmk:strict}) in $\pi_0(R)$.  Namely, for any $z \in (R^\flat)^*$, we have the composite
\[ \Ss[t] \xrightarrow{t\mapsto z} \Ss[ (R^{\flat}) * ]_p \xrightarrow{\tau} \W(R^\flat) \to R, \]
and we will write $[z]$ for the corresponding strict element in $\pi_0(R)$.

\subsection{Deformation theory and tilting}
\label{subsec:tilting}\ 

We now turn to studying the tilt functor $(-)^\flat$ which is right adjoint to $\W(-)$.
Since this functor lands in a $1$-category, where objects are relatively well described by their underlying set, we will mostly focus on understanding $(-)^\flat$ at this level.

\begin{lem} \label{lem:tilt-corepresented}
  The (underlying set of the) tilt functor $(-)^{\flat}$ is corepresented by the commutative algebra
  $\Ss[\Np]_p$.
\end{lem}

\begin{proof}
  This follows from the equivalence $\Ss[\Np]_p \simeq \W(\F_p[\Np])$ of \Cref{exm:line}
  together the fact that $\Fp[\Np]$ corepresents the underlying set functor on perfect $\F_p$-algebras.
\end{proof}

Using the corepresentability of $(-)^\flat$, we can establish various properties of this functor in short order by leaning on our understanding of the spherical Witt vectors.

\begin{lem} \label{lem:calg-sqz}
  Given a square zero extension $R \to S$ in $\CAlg(\Sp_p)$ the induced map
  \[ R^\flat \to S^\flat \]
  is an isomorphism.  
\end{lem}

\begin{proof}
  It suffices to prove this at the level of the underlying set.
  Since $(-)^\flat$ is corepresented by $\W(\F_p[\Np])$, which is formally \'etale over $\Ss_p$, this follows from \Cref{rmk:using-formally-etale}.
\end{proof}

\begin{lem} \label{lem:calg-flat-pi0}
  $(-)^{\flat}$ sends the maps in the span 
  \[\pi_0(R) \leftarrow  \tau_{\geq 0}R  \to R \]
  to equivalences for every $R \in \CAlg(\Sp_p)$.
\end{lem}

\begin{proof}
  The statement for $\tau_{\geq 0}R  \to R$ follows from the fact that $\Ss[\Np]_p$ is connective.
  The Postnikov tower is a tower of square-zero extensions, therefore \Cref{lem:calg-sqz} implies that
  $(\tau_{\geq 0}R)^\flat \to (\pi_0(R))^\flat$ is an equivalence as well.
\end{proof}

\begin{lem} \label{lem:tilt-of-quotient}
  Let $R \in \CAlg(\Sp_p)$ be a commutative algebra and $x \in \pi_0R$ a class such that $\pi_0R$ is (derived) $x$-complete.
  Then, there are equivalences
  \[ R^\flat \cong (\pi_0R)^\flat \cong ((\pi_0R)/x)^{\flat}. \]
  In particular, this holds if $R$ is $x$-complete.
\end{lem}

\begin{proof}
  The first equivalence is copied from \Cref{lem:calg-flat-pi0}.  For the second one, note that since $\pi_0R$ is discrete and $x$-complete, we can write it as the limit of its $x$-adic tower as a tower of square-zero extensions of commutative algebras.  \Cref{lem:calg-sqz} and the fact that tilting commutes with limits now give the desired equivalence.
  
 For the ``in particular,'' note that expressing $x$-completeness as the vanishing of the inverse limit along multiplication by $x$ and using the Milnor sequence for homotopy groups of an inverse limit, one sees that $R$ is $x$-complete exactly when its homotopy groups are derived $x$-complete.
\end{proof}

\begin{rmk}
  Note that in \Cref{lem:tilt-of-quotient} the distinction between
  $(\pi_0R)/x$ and $\pi_0((\pi_0R)/x)$ doesn't matter because
  \Cref{lem:calg-flat-pi0} lets us replace commutative algebras by their $\pi_0$ when taking tilts.
  \tqed
\end{rmk}

\begin{lem} \label{lem:tilt-compact}
  The tilting functor $(-)^{\flat}\colon \CAlg(\Sp_p) \to \Perf$ commutes with $\omega_1$-filtered colimits.
  In particular, this means that, for any uncountable regular ordinal $\kappa$, the spherical Witt vectors functor sends $\kappa$-compact objects to $\kappa$-compact objects.
\end{lem}

\begin{proof}
First, observe that $\pi_0 \colon \CAlg(\Sp_p) \to \Set$ commutes with $\omega_1$-filtered colimits.  This is because while $\pi_0$ commutes with filtered colimits in $\CAlg(\Sp)$, colimits in $p$-complete algebras are computed by colimit followed by $p$-completion.  Since $p$-completion is an $\omega$-indexed limit of functors which commute with filtered colimits, it commutes with $\omega_1$-filtered colimits. Here, we are using that $\omega$-indexed limits commute with $\omega_1$-filtered colimits in sets.

To finish,  we use the description of the tilt as
  $\lim_{x \mapsto x^p} \pi_0(-)/p$ from \Cref{prop:witt-and-tilt}
 to see that the tilt is formed from an $\omega$-indexed limit of terms which
  individually commute with $\omega_1$-filtered colimits, and therefore again it commutes with $\omega_1$-filtered colimits.  
\end{proof}

\subsection{Lubin--Tate theories}
\label{subsec:LT-theory}\

Quillen's work on complex cobordism allows us to attach to each complex orientable homotopy commutative ring spectrum $R$ a formal group $\G^Q$ over the graded ring $\pi_{*}R$. This association further identifies choices of complex orientation of $R$ and choices of coordinate on $\G^Q$.

\begin{dfn}
  If $R$ is complex orientable, then
  the stratification of the moduli of formal groups by height
  provides a sequence of ideals
  \[ \m_0 \subset \m_1 \subset \cdots \subseteq \pi_*R \]
  in $\pi_*R$ known as the \deff{Landweber ideals}.
  We let \deff{$\overline{v}_k$}$ \in (\pi_*R)/\m_{k-1}$ denote the
  $k^{\mathrm{th}}$ Hasse invariant of $\G^Q$, which cuts out the locus where $\G^Q$ has height at least $k+1$. Note that $\overline{v}_k$ is of degree $2(p^k-1)$.
  \tqed
\end{dfn}

\begin{cnv}
In this paper, we take the convention that once a ring spectrum $R$ is known to be complex orientable, we make a choice of lifts $v_0, v_1 ,\cdots \in \pi_*(R)$ of the Hasse invariants.  The particular choice will not be important to us, except that if a ring is constructed as an algebra over another complex oriented ring, we take the image of the previously chosen $v_i$'s.  
\tqed
\end{cnv}

Given this choice, we obtain a preferred choice of generators for the Landweber ideals.

\begin{dfn}
  We say that a homotopy commutative ring spectrum $R$
  is \deff{complex periodic} if $R$ is complex orientable,
  $\pi_2(R)$ is projective of rank 1 as a $\pi_0R$-module,
  and the multiplication map
  \[ \pi_2(R) \otimes_{\pi_0(R)} \pi_n(R) \to \pi_{n+2}(R) \]
  is an isomorphism for every $n$.
  \tqed
\end{dfn}

For complex periodic ring spectra, we can reformulate the Quillen formal group as a formal group on $\pi_0R$ whose dualizing line $\omega_{\G^Q}$ is given by $\pi_{2}(R)$ \cite[Example 4.2.19]{ECII}. In the situation where $R$ is additionally even, we can extract the homotopy ring of $R$ from $\pi_0R$ together with $\G^Q$ via the formula
\[ \pi_*(R) \cong \begin{cases} \omega_{\G^Q}^{\otimes k} & *=2k \\ 0 & \text{otherwise} \end{cases}. \]

If $\pi_*R$ has a unit in degree 2, then we will typically write $u$ for a choice of such a unit.  This choice provides a trivialization of $\omega_{\G^Q}$, with which we can push the Hasse invariants of $\G^Q$ into degree zero where we write $u_i$ instead of $v_i$.
Concretely, we have $v_i= u_iu^{p^i-1} $.

If $R$ is an even\footnote{Note that evenness guarantees complex orientability.} commutative algebra,
then although we cannot lift the quotients $\pi_*(R) / \m_k$ to the level of commutative algebras,
we can construct a sequence of $\E_1$-$R$-algebras $R \modm_k$ 
which act as quotients by the Landweber ideals.

\begin{cnstr} \label{cnstr:modm}
  Given an even commutative algebra $R$ and an element $x\in \pi_*(R)$, Hahn and Wilson \cite{HahnWilsoneven}, building on work of Angeltveit \cite{Ang}, show that the $R$-module $R/x$ admits the structure of an $\E_1$-$R$-algebra.  Hence, given a choice of elements $v_i$ as above, we may obtain  an $\E_1$-$R$-algebra $R\modm_j$ whose underlying $R$-module is given by the tensor product
  \[ (R / v_0 ) \otimes_R \cdots \otimes_R (R /v_{j}). \]
  
  More generally, given a commutative $R$-algebra $S$
  we write \deff{$S\modm_j$} for the $\E_1$-$A$-algebra $R\modm_j \otimes_{R} S$ obtained by base-change.
  Note that the tensor product description implies that if $S$ is $T(n)$-local,
  then $S\modm_j$ is $T(n)$-local as well.    
  \tqed
\end{cnstr}

\begin{cnv}
  In this paper we will almost exclusively work with objects of height $n$.  Consequently, for the sake of brevity, we will use $\m$ for the Landweber ideal $\m_{n-1}$ which cuts out the locus of height $\geq n$.
  \tqed
\end{cnv}


The importance of the universal deformation of a formal group of height $n$ to homotopy theory was first recognized by Morava, who used Landweber’s exact functor theorem to attach to each perfect field $k$ and formal group $\G_0$ of height $n$ over $k$ a homotopy commutative ring spectrum $E(k;\,\G_0)$. Building on this, Goerss, Hopkins and Miller developed a collection of obstruction theoretic techniques for analyzing the space of coherent ring structures on a homotopy commutative ring, and proved the following theorem which provides us with our basic objects of interest.

\begin{thm}[Goerss--Hopkins--Miller \cite{GoerssHopkins, RezkGHM}] \label{thm:GHM}
  Given a perfect field $k$ of characteristic $p$ and a formal group $\G_0$ over $k$ of height $n < \infty$ there exists an even periodic $K(n)$-local commutative algebra $E(k, \G_0)$ such that $\pi_0E(k, \G_0)$ can be identified with the Lubin--Tate ring parametrizing deformations of $\G_0$.
\end{thm}



Generalizing the Goerss--Hopkins--Miller theorem,
Lurie has constructed a functor
\[ E(-;-) \colon \int_{\Perf} \Mfg^{=n} \to \CAlg(\Sp_{T(n)}) \]
which takes in a perfect $\F_p$-algebra $A$ together with a formal group $\mathbb{H}_0$ of height exactly $n$ over $A$ and produces a $T(n)$-local commutative algebra $E(A;\H_0)$
(see \cite[Section 5]{ECII}).

\begin{thm}[{\cite[Theorems 5.0.2, 5.1.5 and 5.4.1]{ECII}}] \label{thm:GHML}\ \\
  The functor $E(-;-)$ has the following properties:
  \begin{enumerate}
  \item $E(-;-)$ is fully faithful.
  \item $E(A;\, \H_0)$ is $K(n)$-local and even periodic.
  \item 
    Any sequence
    $v_0,\dots,v_{n-1} \in \pi_*E(A;\,  \H_0)$
    lifting the Hasse invariants is regular.
  \item The essential image of $E(-;-)$ consists of those
    commutative algebras $R$ satisfying $(2)$ and $(3)$ for which    
    $\pi_0(R)/\m$ is perfect.
  \item There is a natural identification of
    $(\pi_0E(A;\,  \H_0) /\m, (\G^Q)_{\pi_0E(A;\, \H_0) /\m})$ with $(A, \H_0)$
    which, when $A$ is a perfect field, exhibits $\G^Q$ as a universal deformation of $\H_0$.
  \end{enumerate}
\end{thm}


\begin{rmk}
  When applied to Lubin--Tate theories, \Cref{cnstr:modm} attaches to each Lubin--Tate theory
  an associated Morava $K$-theory.
  Since the Landweber ideal in $\pi_*E(A;\,  \H_0)$
  is generated by a regular sequence, we can read off the homotopy groups of this $K$-theory
  \[ \pi_{*}(E(A; \, \H_0)\modm) \cong \begin{cases} \omega_{\H_0}^{\otimes k} & *=2k \\ 0 & \text{otherwise} \end{cases}. \]
  \tqed
\end{rmk}

As we are in the business of producing and manipulating Lubin--Tate theories,
it is useful for us to be able to recognize these objects from as great a distance as possible.   To prove our first Lubin--Tate recognition theorem, we first need the following lemma, which compresses the conditions of \Cref{thm:GHML}(4) into a more manageable form.

\begin{lem} \label{lem:strong-even}
  Given a $T(n)$-local commutative algebra $R$,
  if $V_n \otimes R$ is even for a single type $n$ generalized Moore spectrum $V_n$,
  then
  \begin{enumerate}
  \item $R \otimes V$ is even for every generalized Moore spectrum $V$.
  \item $R$ is $K(n)$-local, even and complex orientable.
  \item The sequence of classes $v_0,\dots,v_{n-1}$ is regular in $\pi_*R$ and
    the $\E_1$-$R$-algebra $R\modm$ of \Cref{cnstr:modm} will have
    $ \pi_*(R\modm) \cong \pi_*(R)/\m $.
  \end{enumerate}
\end{lem}

\begin{proof}
  We begin by showing that $R$ is even.
  Write the generalized Moore spectrum $V_n$ as
  \[ V_n = \Ss/(p^{i_0},v_1^{i_1}, \dots v_{n-1}^{i_{n-1}}) \]
  so that we have a tower of generalized Moore spectra $V_m$ such that
  $ V_m = V_{m-1}/v_{m-1}^{i_{m-1}} $.
  We will show that $V_m \otimes R$ is even by downward induction on $m$ starting with the case $m=n$, which is our assumption on $R$. For the inductive step, we consider the $v_m^{i_m}$-Bockstein tower
  \begin{center}
    \begin{tikzcd}
      &
      \Sigma^{2|v_m^{i_m}|} V_m/v_m^{i_m} \otimes R \ar[d] &
      \Sigma^{|v_m^{i_m}|} V_m/v_m^{i_m} \otimes R \ar[d] & \\
      \cdots \ar[r] &
      V_m/v_m^{3i_m} \otimes R \ar[r] &
      V_m/v_m^{2i_m} \otimes R \ar[r] &
      V_m/v_m^{i_m} \otimes R
    \end{tikzcd}
  \end{center}
  which has the feature that the $T(n)$-local inverse limit of this tower is $V_m \otimes R$.
  Since $V_{m+1} = V_m/v_m^{i_m}$ is even, we can induct up this tower showing that every term is even (an extension of even objects is even).  From evenness we can also read off that the maps in the tower are surjective on homotopy groups and therefore the inverse limit is even as well.
  The case $V_0=\Ss$ is the conclusion that $R$ is even.  
  
  Since evenness implies complex orientability and complex orientability implies that $T(n)$-local implies $K(n)$-local \cite[Corollary 1.10]{Hovey}, we have shown (2).
  Using that $R$ is complex orientable, we can now pick classes $v_i$ in $\pi_*(R)$.
  Examining the long exact sequence on homotopy groups associated to the cofiber sequence
  \[ \Sigma^{|v_m^{i_m}|} V_m \otimes R \xrightarrow{v_m^{i_m}} V_m \otimes R \to V_{m+1} \otimes R \]
  we can read off from the fact that all three terms have even homotopy
  that $v_m^{i_m}$ acts injectively on the homotopy of $V_m \otimes R$.
  In particular, inducting upwards on $m$, we find that
  \[ \pi_*(V_m \otimes R) \cong \pi_*(R)/(p^{i_0}, \dots, v_{m-1}^{i_{m-1}}) \]
  and that $ p^{i_0}, v_1^{i_1}, \dots v_{n-1}^{i_{n-1}} $ is a regular sequence in $\pi_*(R)$.  
  This implies that $ p, v_1, \dots v_{n-1} $ is a regular sequence as well (see \cite[\href{https://stacks.math.columbia.edu/tag/07DV}{Lemma 07DV}]{stacks-project}) which is (3).  This in turn implies conclusion (1) by the long exact sequence on homotopy.  
\end{proof}

Note that over an algebraically closed field, all formal groups of height $n$ are isomorphic.  Therefore, there is no danger in dropping the formal group from our notation and we do so.  An immediate corollary of \Cref{lem:strong-even} and \Cref{thm:GHML}(4) is:

\begin{cor}\label{cor:alg-closed-E}
  Suppose that $R\in \CAlg(\Sp_{T(n)})$ such that $R\otimes V$ is even for a single type $n$ generalized Moore spectrum $V$.  Then \Cref{lem:strong-even} implies that $R$ is even and we have the $\E_1$-$R$-algebra $R\modm$ of \Cref{cnstr:modm}.  
  
  In this situation, if $\pi_*(R\modm)$ is an even periodic algebraically closed field, then there is an equivalence $R \cong E(\pi_0(R\modm))$.
\end{cor}


  
The situation simplifies significantly when working under a fixed $E(k;\,\G_0)$, because the compatibility of the Quillen formal group with ring maps fixes the choice of formal group on Lubin--Tate theories.  

\begin{cnv}
For the remainder of the paper, fix a perfect field $k$ and a formal group $\G_0$ of height $n$ over $k$.  This determines a Lubin--Tate theory $E(k; \,\G_0)$, which we denote simply by $E(k)$.  We additionally fix a choice of unit $u\in \pi_2 E(k)$, which (given the previous choices of $v_i$'s) determines elements $u_0=p, u_1, \cdots ,u_{n-1}\in \pi_0E(k)$ and an isomorphism
\[
\pi_0E(k) \cong W(k)[\![ u_1,\dots,u_{n-1} ]\!][u^{\pm1}]. 
\]
\tqed
\end{cnv}

\begin{dfn}\label{dfn:LTfunctor}
Define the functor 
\begin{align*}
\mdef{E(-)}\colon \Perf_k &\to \CAlgw_{E(k)}\\   
(k\to A) &\mapsto E(A;\, (\G_0)_A).
\end{align*}
Note that given our conventions, we have for $A\in \Perf_k$ a canonical isomorphism
\[
\pi_0E(A) \cong W(A)[\![ u_1,\dots,u_{n-1} ]\!][u^{\pm1}] 
\]
by \cite[Corollary 5.4.2]{ECII}.  
\tqed
\end{dfn}

In this setting, the conditions of \Cref{thm:GHML}(4) are even easier to check.

\begin{cor}\label{cor:E-alg-image}
An algebra $R\in \CAlgw_{E(k)}$ is in the essential image of 
\[
E(-)\colon \Perf_k \to \CAlgw_{E(k)}
\]
if and only if $R\modm$ is even and $\pi_0 (R)/\m $ is perfect\footnote{In fact, as we will see in the proof, the first assumption implies that $\pi_0(R)/\m \cong \pi_0(R\modm)$.}. 
\end{cor}
\begin{proof}
The only if direction is clear, so assume that $R\modm$ is even and $\pi_0(R)/\m$ is perfect.  But \cite[Lemma 2.1.25]{HL} (and its proof) show that $R$ is even and any choice of $v_0, \cdots v_{n-1}$ is a regular sequence\footnote{Alternately, we may run an argument analogous to the proof of \Cref{lem:strong-even}.}.  Hence, the result follows from \Cref{thm:GHML}(4).  
\end{proof}

In the next section, we will see that, in contrast to the functor $E(-;-)$ of \Cref{thm:GHML}, this relative version of the Lubin--Tate functor has an additional feature of central importance to us: it admits a right adjoint.  

\subsection{The Lubin--Tate functor} \label{subsec:E}\

While our discussion of spherical Witt vectors takes values in $\CAlg(\Sp_p)$, the next construction shows that we can make a similar construction much more generally.

\begin{cnstr}\label{cnstr:C-witt-tilt}
  Let $\CC$ be a $p$-complete stable presentably symmetric monoidal category.
  The symmetric monoidal unit map $\Sp_p \to \CC$ in $\CAlg(\PrL)$ provides us with an adjunction
  \[ \iota^* \colon \CAlg(\Sp_p) \rightleftarrows \CAlg(\CC) \noloc \iota_* \]
  which we can compose with the Witt-tilt adjunction to obtain an adjunction
  \[ \Perf \rightleftarrows \CAlg(\CC). \]
  Setting $B = (\iota_*{\one}_{\CC})^{\flat}$, we can refine this to an adjunction
  \[ \W_{\CC}(-) \colon \Perf_{B} \rightleftarrows \CAlg(\CC) \noloc (-)_{\CC}^{\flat} \]
  where $\W_{\CC}(-)$ sends a perfect $B$-algebra $A$ to the pushout
  ${\one}_\CC \otimes_{\iota^*\W(B)} \iota^*\W(A)$ and $(-)_{\CC}^\flat$ sends $R$ to
  $B = (\iota_*{\one}_\CC)^\flat \to (\iota_*R)^\flat$.
  \tqed
\end{cnstr}

We can extend many of the properties of $(-)^{\flat}$ to $(-)_{\CC}^{\flat}$ (and for this reason we will often drop $\CC$ from the notation when $\CC$ is clear from context).

\begin{lem}\label{lem:tilt-pi0}\hfill
\begin{enumerate}
    \item   The functor $(-)^{\flat}$ is invariant under square-zero extensions in $\CAlg(\CC)$ (in the sense of \cite[Definition 7.4.1.6]{HA}).
  \item We have natural isomorphisms
  $(R)_{\CC}^{\flat} \cong \pi_0(R)^\flat$.  

\end{enumerate}
\end{lem}
\begin{proof}
Statement (1) follows because $\iota_*$ preserves square-zero extensions, and (2) follows because
\[
R^{\flat}_{\CC} \simeq (\iota_* R)^{\flat} \simeq (\pi_0 \iota_* R)^{\flat} \simeq (\pi_0 R)^{\flat}.
\]
\end{proof}


At this point we are finally ready to move back to the $T(n)$-local world, prove \Cref{thm:E-and-tilt}, and end the section.  Specializing the above discussion to the case $\CC = \Modw_{E(k)}$, we have:

\begin{lem}\label{lem:Etilt} \hfill
\begin{enumerate}
    \item For any $R\in \CAlgw_{E(k)}$, we have 
    \[
    R^{\flat} \cong (\pi_0R)^{\flat} \cong (\pi_0(R)/m)^{\flat}.
    \]
    In particular, we have $E(k)^{\flat} \cong k^{\flat} \cong k$ and we may regard $(-)^{\flat}$ as a functor $(-)^{\flat}\colon \CAlgw_{E(k)} \to \Perf_k$.  
    \item The composite of the Lubin--Tate functor of \Cref{dfn:LTfunctor} with tilting
  \[
  \Perf_k \xrightarrow{E(-)} \CAlgw_{E(k)} \xrightarrow{(-)^{\flat}} \Perf_k
  \]
  is naturally isomorphic to the identity functor.
\end{enumerate}
\end{lem}
\begin{proof}
The first isomorphism of (1) follows from combining \Cref{lem:tilt-pi0}(2) and \Cref{lem:calg-flat-pi0}.  For the second, note that for $E(k)$-algebras, we can rewrite the $T(n)$-local condition as being $\m$-complete \cite[Proposition 7.10]{HovStrick}.    Using \Cref{lem:tilt-of-quotient} repeatedly for each of the generators $p, \dots, u_{n-1}$ of this ideal, we obtain the description of the tilt as $(\pi_0(-)/\m)^{\flat}$.
  Then (2) is an immediate consequence of (1) because 
  \[
  E(A)^{\flat} \cong A^{\flat} \cong A
  \]
  and these are functors valued in 1-categories.  
\end{proof}



As a consequence of (1) and the previous discussion, we obtain an adjunction 
\[ \W_{E(k)}(-) \colon \Perf_{k} \rightleftarrows \CAlg_{E(k)}^{\wedge} \noloc (-)^{\flat}. \]
In fact, $\W_{E(k)}(-)$ agrees with Lurie's Lubin--Tate theory functor of \Cref{dfn:LTfunctor}.



  

\begin{lem} \label{lem:agrees-with-jacob}
  The functor $\W_{E(k)}(-):\Perf_k \to \CAlgw_{E(k)}$ is equivalent to the functor $E(-)$ of \Cref{dfn:LTfunctor}.
\end{lem}

\begin{proof}
  

  The natural isomorphism of \Cref{lem:Etilt}(2) is adjoint to a natural transformation
  \[
  \W_{E(k)}(-) \to E(-).
  \]
  It suffices to show that this is an equivalence modulo $\m$, but $\W_{E(k)}(A) \simeq \W(A)\otimes_{\W(k)}E(k)$ so this is clear by the description of $\pi_*E(A)$ of \Cref{dfn:LTfunctor}.  
  \qedhere
  
  
 
    
\end{proof}

In view of this lemma we will adopt the notation $E(-)$ for $\W_{E(k)}(-)$.

\begin{thm}\label{thm:E_functor}
The adjunction 
  \[ E(-) \colon \Perf_k \rightleftarrows \CAlg_{E(k)}^{\wedge} \noloc (-)^{\flat} \]
  of \Cref{cnstr:C-witt-tilt} has the following properties:
  \begin{enumerate}
  \item $E(-)$ is a colocalization, and in particular fully faithful.
  \item The tilt can be computed via any of
    \[(-)^{\flat},\quad \pi_0(-)^{\flat}\quad\text{or}\quad (\pi_0(-)/\m)^{\flat}.\]
  \item There are natural isomorphisms
    \[ \pi_*E(A) \cong W(A)[\![ u_1,\dots,u_{n-1} ]\!][u^{\pm1}]. \]
    In particular, $\pi_*E(A)$ is even, $u_0,\dots,u_{n-1}$ is a regular sequence and
    \[ \pi_*(E(A)\modm) \cong A[u^{\pm1}]. \]
  \item The essential image of $E(-)$ consists of those $R \in \CAlg_{E(k)}^{\wedge} $ for which\\
    (i) $R\modm$ has vanishing odd homotopy groups\
    (ii) $\pi_0(R)/\m$ (which is isomorphic to $\pi_0(R\modm)$ given (i)) is perfect.
  \item $E(-)$ preserves those limits of perfect $k$-algebras whose limit, taken in $\Sp_p$, is connective.
  \item Tilting commutes with $\omega_1$-filtered colimits.
    In particular, $E(-)$ preserves $\kappa$-compactness
    for every uncountable regular $\kappa$.
  \end{enumerate}
\end{thm}


\begin{proof}
 Statement (1) follows because \Cref{lem:Etilt} shows that the unit map is an equivalence, and (2) is just \Cref{lem:Etilt}(1).  Moreover, (3) follows from \cite[Corollary 5.4.2]{ECII} and (4) is just \Cref{cor:E-alg-image}. 

  The argument for (5) is similar to the argument in \Cref{lem:W-limits}.
  Suppose we are given a limit diagram $F^{\triangleleft}\colon D^{\triangleleft} \to \Perf_k$ that we wish to show is preserved by $E(-)$.
  It suffices to show that $F^{\triangleleft} \circ E(-)$ is a limit diagram on underlying spectra.
  Since $E(-) \simeq L_{K(n)}(\tau_{\geq 0}E(-))$ it suffices to prove the claim on connective covers.
  Filtering by the Postnikov tower it suffices to argue that $F^{\triangleleft} \circ \pi_sE(-)$ is a limit diagram. Since $E(-)$ is even periodic it suffices to handle the $s=0$ case.
  Since $\pi_0E(-)$ is $\m$-adically complete we can write this functor as the limit of the tower of quotients by powers of $\m$. The associated graded of this tower can be identified with
  $ \m^j/\m^{j+1} \otimes_k \pi_0E(F^{\triangleleft}(-))/\m$.
  Finally we observe that $\m^j/\m^{j+1}$ is a dualizable $k$-module and $\pi_0E(F^{\triangleleft}(-))/\m$ is equivalent to $F^{\triangleleft}(-)$ which means it sufficed for $F^{\triangleleft}(-)$ to be a limit diagram on underlying spectra, which was our assumption.


  
  (6) follows from \Cref{lem:tilt-compact} together with the fact that $T(n)$-localization
  (in its guise as $\m$-adic completion)
  commutes with $\omega_1$-filtered colimits.
  \qedhere
  
\end{proof}

\section{Power operations on Lubin--Tate theories}
\label{sec:cofree}
Our goal in this section is to describe how power operations act on the homotopy groups of the
$T(n)$-local commutative $E(k)$-algebra $E(A)$ for $A \in \Perf_k$.

\begin{cnv}
For ease of notation, with this section we will sometimes write $E$ for $E(k)$ and $E_0$ for $\pi_0E(k)$.
\tqed
\end{cnv}

\subsubsection{Algebraic power operations and the monad $\T$}\hfill

Let $R\in \CAlg(\Sp)$ be a commutative ring spectrum.  Then, the commutative ring $\pi_0 R$ comes equipped with additional algebraic structure from power operations.  This algebraic structure was studied systematically by Rezk in the case that $R$ is a $T(n)$-local commutative $E$-algebra.  More specifically, he constructed a monad $\T$ on the category of (discrete) $E_0$-modules whose category of algebras $\Alg_{\T}$ serves as the natural target for the functor $\pi_0(-)$ on commutative $E$-algebras; that is, there is a natural lift
\begin{center}
  \begin{tikzcd}
    & \Alg_{\T}\arrow[d,"U_{\T}"]\\
    \CAlgw_{E} \arrow[ur,dashed, "\pi_0"]\arrow[r,"\pi_0"]& \CAlgh_{E_0}.
  \end{tikzcd}
\end{center}

The free $\T$-algebra $\T(E_0)$ is closely related to the homotopy groups of the free $T(n)$-local commutative $E$-algebra on one generator, and as such, the monad $\T :\Modh_{E_0} \to \Modh_{E_0}$ can be thought of as an algebraic approximation to the free $T(n)$-local commutative $E$-algebra monad.  

\begin{exm}\label{exm:ht1}
Consider the case $n=1$ and $E = E(\F_p, \G_m) = \KU_p$.  Then, $\Alg_{\T}$ can be identified with the category $\CAlgh_{\delta}$ of $\delta$-rings\footnote{Also known as a $\theta$-rings or $p$-typical $\lambda$-rings.}.  This means that if $R$ is a $T(1)$-local commutative $\KU_p$-algebra, then $\pi_0R$ has a (non-additive) operation $\delta\colon \pi_0(R) \to \pi_0(R)$ which acts as a $p$-derivation, so that the operation
  \[ \psi(x) = x^p + p\delta(x), \]
  is a ring endomorphism lifting Frobenius.  Here, $\psi$ is the $p$-th Adams operation. 
  \tqed 
\end{exm}

For formal reasons, the forgetful functor $U_{\T}\colon \Alg_{\T} \to \CAlgh_{E_0}$ admits both a left adjoint $F_{\T}$ and a right adjoint $W_{\T}$
\begin{center}
  \begin{tikzcd}[sep=huge]
    \Alg_{\T}
    \arrow[rr,"U_{\T}"] & &
    \CAlgh_{E_0}.
    \arrow[ll, bend right=25, swap, "F_{\T}"]
    \arrow[ll, bend left=25, swap, "W_{\T}"] 
  \end{tikzcd}
\end{center}
In fact, $\Alg_{\T}$ is both monadic and comonadic over $\CAlgh_{E_0}$ \cite[4.23]{RezkCong}. 
The notation $W_{\T}$ is justified by the following theorem.

\begin{thm}[Joyal \cite{JoyalWitt}]
For any ring $B$, the $p$-typical Witt vectors $W(B)$ naturally acquires the structure of a $\delta$-ring.  The resulting functor $W\colon \CAlgh \to \CAlgh_{\delta}$ is right adjoint to the forgetful functor $U\colon \CAlgh_{\delta} \to \CAlgh$.    
\end{thm}

That is, in the case where $\Alg_{\T} = \CAlgh_{\delta}$ at height $1$, we have $W_{\T}(A)= W(A) = \pi_0E(A)$.  In particular, by composing with the adjunction 
\[
(-/p)^{\sharp} \colon \CAlgh  \rightleftarrows \Perf_{\F_p}  \noloc \mathrm{incl},
\]
we obtain an adjunction
\[
(U(-)/p)^{\sharp} \colon \CAlgh_{\delta}  \rightleftarrows \Perf_{\F_p}  \noloc \pi_0E(-).
\]

The main theorem of this section is a generalization of this adjunction to the case of arbitrary height, which allows us to describe the $\T$-algebra structure on $\pi_0(E)$ in terms of the cofree $\T$-algebra functor $W_{\T}(-)$.

\begin{thm}[Cofreeness of Lubin--Tate theory]\label{thm:cofree}
  There is an adjunction
  \[
  (U_{\T}(-)/\m)^{\sharp} \colon \Alg_{\T} \rightleftarrows \Perf_k  \noloc \pi_0E(-)  
  \]
  where the right adjoint  $\pi_0E(-)$ is fully faithful.
\end{thm}

\begin{rmk}
We have seen (cf. \Cref{thm:E_functor}) that the construction $E(-)$ is naturally a left adjoint.  The task of constructing maps \emph{into} Lubin--Tate theories, to which much of this paper is devoted, would be considerably easier one could realize $E(-)$ as a \emph{right adjoint}.  This theorem asserts that at least this is the case in algebra, at the level of $\pi_0$ equipped with its power operations.
\tqed
\end{rmk}

We remark that this result has also been independently obtained by Rezk.  We now outline the ideas which go into our proof of this result.  

\subsubsection{The map $\overline{\ev}$}\hfill

  Given a stable presentably symmetric monoidal category $\CC$, the functor
  \begin{align*}
       \pi_0 \colon \CAlg(\CC) &\to \mathrm{Set} \\
        R &\mapsto \pi_0 \Map_{\CC}({\one}, R)
  \end{align*} 
  is represented by the free commutative algebra on a copy of the unit, ${\one}\{t\}$.  Accordingly, the operations on $\pi_0R$ are given by elements of $\pi_0 {\one}\{t\} = \pi_0 \bigoplus {\one}^{\otimes s}_{h\Sigma_s}\cong \pi_0 \bigoplus {\one}_{h\Sigma_s}$.  

\begin{dfn}
  To each class $P \in \pi_0({\one}_{h\Sigma_r})$, we can associate a
  \deff{power operation of weight} $r$ by sending a class $x \in \pi_0R$ to the composite 
  \begin{equation*}\label{intro:powop} \one \xrightarrow{P} {\one}_{h\Sigma_r} \hookrightarrow \oplus_s {\one}_{h\Sigma_s} \cong {\one}\{t\} \xrightarrow{t \mapsto x} R. \end{equation*}
  Dually, by omitting the composition with $P$ and $R$-linearizing, we obtain a map of $R$-modules $R[B\Sigma_r] \to R$, which induces a map
  \[ \mathcal{P}_r \colon \pi_0R \longrightarrow [B\Sigma_r, R] \]
  which we call the \deff{total power operation of weight} $r$.\footnote{This mirrors the duality between cohomology and homology. As in that case, although it is easier to think in terms of actions, coactions have the moral high ground.}
  \tqed
\end{dfn}

\begin{rmk}\label{rmk:strict}
The power operations on a class $x\in \pi_0 R$ measure the failure of $x$ to strictly commute with itself.  In the case that $x$ is a \deff{strict element}, i.e., $x$ arises as the image of $t$ under a commutative algebra map ${\one}[t] \to R$, the induced map ${\one}\{t \} \xrightarrow{t\mapsto x} R$ factors through the projection \[ {\one}\{ t \} \xrightarrow{t\mapsto t} {\one}[t], \]
and so there are no interesting power operations on $x$ -- that is, the total power operation on $x$ is given simply by the formula $\mathcal{P}_r(x) = x^r$.  
\tqed
\end{rmk}

Although the functor $W_{\T}$ initially seems mysterious, it turns out to be closely related to the total power operation.

\begin{rmk}\label{rmk:WTintuition}
Concretely, an element of $W_{\T}(B)$ can be thought of as a $B$-valued functional on the the $E_0$-module of power operations.  Then, the unit of the $(U_{\T}, W_{\T})$-adjunction 
\[ \ev_B \colon B \longrightarrow W_{\T} U_{\T}(B) \]
has a relatively simple interpretation:  it sends $b\in B$ to the functional which evaluates each power operation on $b$.  As such, $\ev_B$ is a universal refinement of the total power operation on $B$, and we think of $W_\T(A)$ as the universal target for $A$-valued power operations; we make these ideas precise in \Cref{sub:background}.
\tqed
\end{rmk}

It is through this unit map $\ev$ that we approach the proof of \Cref{thm:cofree}.

\begin{cnstr}
Let $A$ be a perfect $k$-algebra.  Then by construction, $\pi_0E(A)$ acquires the structure of a $\T$-algebra, and we saw in \Cref{thm:E_functor}(3) that there is a natural equivalence $\pi_0E(A) / \m \cong A$ for $A \in \Perf_k$.  Therefore, we can construct a natural transformation $\overline{\ev}_A$ as the following composite:
  \begin{center}
    \begin{tikzcd}
      & W_{\T}U_{\T}\pi_0E(A) \ar[dr,"W_{\T}(-/\m)"] & \\    
      \pi_0E(A) \ar[ur, "\ev_A"] \ar[rr, "\overline{\ev}_A"] & & W_{\T}(A).
    \end{tikzcd}
  \end{center}
  \tqed
\end{cnstr}

  In the spirit of \Cref{rmk:WTintuition}, the natural transformation $\overline{\ev}$ can be interpreted as recording the value of the total mod $\m$ power operation on $\pi_0E(A)$.  Our strategy for proving \Cref{thm:cofree} is to show that this map $\overline{\ev}_A$ is an isomorphism for all $A \in \Perf_k$.  The theorem follows easily from this fact (and will be deduced at the end of \Cref{sub:pder}), so we spend the rest of this section showing that $\overline{\ev}$ is an isomorphism.
  
\subsubsection{The Witt filtration}\hfill

An important organizing principle in our proof that $\overline{\ev}$ is an isomorphism is the \emph{Witt filtration}.  This is a filtration of $W_{\T}(A)$ by ideals induced by the weight grading on $\T$, so that $W^{\leq r}_\T(A)$ is the universal target for $A$-valued power operations of weight at most $p^r$.  One consequence of having this filtration is that we are able to control the dependence of $W_\T(A)$ on $A$:

\begin{prop} \label{prop:completed-tensor-formula}
  There is an equivalence
  \[ \left( W_{\T}(k) \otimes_{W(k)} W(A) \right)_W^{\wedge} \xrightarrow{\cong} W_{\T}(A) \]
  natural in the choice of a perfect $k$-algebra $A$,
  where the completion on the left\footnote{That is, the inverse limit of the corresponding quotients.} is with respect to the Witt filtration on $W_{\T}(k)$.
\end{prop}

As a corollary of \Cref{prop:completed-tensor-formula}, we show that $\overline{\ev}_A$ is an isomorphism in general as soon as it is an isomorphism in the case $A=k$. For this reason, we are reduced to the case of understanding power operations on $\pi_0E = \pi_0E(k)$ for the remainder of the proof.


\subsubsection{Additive operations on Lubin-Tate theory}\ 

The theory of power operations for algebras over Lubin-Tate theory was pioneered by Ando, Hopkins and Strickland in the papers \cite{Ando, Strickland2, Strickland1, AHS}.  The basis of their understanding was a careful understanding of the \deff{additive} power operations -- that is, those power operations $P$ such that $P(x_1+x_2) = P(x_1) + P(x_2)$.  
The \deff{total additive power operation} (of weight $p^r$) is the ring map
\[ \tau_r \colon E_0 \longrightarrow E^0(B\Sigma_{p^r})/I_{\mathrm{tr}} \]
obtained by composing $\mathcal{P}_{p^r}$ with the quotient by the ideal generated by transfers from subgroups of the form $\Sigma_{p^r -i}\times \Sigma_i \subset \Sigma_{p^r}$ (which enforces additivity).\footnote{Note that we have restricted to $p$-power weights, as the transfer ideal contains the unit otherwise.}  The key insight of the Ando-Hopkins-Strickland theory, which we review in \Cref{sub:evsurj}, is that the map $\tau_r$ and its concomitant structures can be interpreted in terms of deformations of isogenies of formal groups.

\begin{exm}\label{exm:ht1-add}
In the height 1 situation of \Cref{exm:ht1}, the additive power operations on $\pi_0R$ for $R \in \CAlgw_{\KU_p}$ are generated under addition and composition by the $p$-th Adams operation
\[
\psi(x) = x^p + p\delta(x).
\]
Hence, the action of additive power operations gives $\pi_0R$ the structure of a module over $\Z_p[\psi]$.
\tqed
\end{exm}

Two of the three core inputs in the proof of \Cref{thm:cofree} come purely from the understanding additive operations on $E_0$.  We have:

\begin{prop} \label{prop:detecting-classes}
  Given $x \in E_0$ which is non-zero mod $p$,
  there exists an additive power operation $Q$ such that the reduction mod $\m$ of $Q(x)$ is non-zero.
\end{prop}

\begin{prop} \label{prop:power-op-surj}
  For any $r\geq 0$, the mod $\m$ total additive power operation
  \[ \overline{\tau}_r \colon E_0 \to E^0(B\Sigma_{p^r})/(I_{\mathrm{tr}}, \m) \]
  is surjective.
\end{prop}

The key step in the proof of \Cref{prop:detecting-classes}, which is given in \Cref{sub:detection},
is an induction on height using transchromatic maps between Lubin--Tate theories of different heights.
\begin{rmk}\label{rmk:induct}
The transchromatic maps that we use are maps in $\CAlg(\Sp)$ from a Lubin--Tate theory of height $n$ to a Lubin--Tate theory of height $n-1$.  These are constructed using the height $n-1$ case of \Cref{cor:mod_algclosed}, and it is for this reason that our paper is inductive on the height (cf. \Cref{rmk:intro-induct}).
\tqed
\end{rmk}

The proof of \Cref{prop:power-op-surj}, given in \Cref{sub:evsurj} requires us to delve deeper into the algebro-geometric perspective on additive power operations and is essentially a corollary of the fact that isogenies of height $n$ formal groups are rigid.

In the context of proving \Cref{thm:cofree}, \Cref{prop:detecting-classes} is almost sufficient to conclude that $\overline{\ev}_k$ is injective and
\Cref{prop:power-op-surj} falls just short of proving $\overline{\ev}_k$ is surjective.
In both cases, the missing piece is that we have yet to take non-additive operations into account.  

\begin{rmk}
Surprisingly, although the algebra of additive operations grows more complicated as height increases, it is the case at any height that passing from additive operations to general operations only requires adding compositions with a \emph{single} non-additive operation: 
the (additive) $p$-derivation $\theta$ first defined in unpublished work of Rezk (cf. \Cref{prop:pder}).
\tqed
\end{rmk}

Our final input to the proof of \Cref{thm:cofree} is an understanding of how this non-additive operation $\theta$ interacts with the $p$-adic filtration on $E_0$.  Using the Witt filtration as a book-keeping device, we combine this with \Cref{prop:detecting-classes} and \Cref{prop:power-op-surj} to make a rank counting argument which shows that $\overline{\ev}$ is an isomorphism.  This is done in \Cref{sub:pder}.



\subsection{$W_{\T}$ and the total power operation}
\label{sub:background}\

In \cite{RezkCong}, Rezk constructed a monad $\T$ on the category of $E_0$-modules which is an algebraic approximation to the free  $K(n)$-local commutative $E$-algebra monad.  It is naturally graded by a weight $j$, and defined so that for a finite free $E$-module $M$, we have
\[
\T(\pi_0 M) = \bigoplus_j \T_j(\pi_0 M) = \bigoplus_j  \pi_0 (M^{\otimes_E j}_{h\Sigma_j}).
\]
In particular, $\pi_0$ of the free $K(n)$-local commutative $E$-algebra on one generator is a completion of $\T(E_0)$.  Accordingly, the elements of $\T(E_0)$ can be thought of as (degree zero) power operations which act on the $\pi_0$ of any $K(n)$-local commutative $E$-algebra.

One can also take a dual perspective on power operations: for $R\in \CAlgw_E$,  rather than thinking of power operations as elements of $\T(E_0)$ acting on $\pi_0(R)$, one can consider them as \emph{co-acting} on $\pi_0(R)$ via the various total power operation maps: more precisely, since the $E^0(B\Sigma_i)$ are finite free $E_0$-modules by \cite{Strickland1}, there are composites
\[
\pi_0R \to R^0(B\Sigma_i) \cong \pi_0R \otimes_{E_0} E^0(B\Sigma_i)
\]
for $i\geq 0$.  As remarked in the introduction, this dual perspective corresponds to the fact that the forgetful functor $U_{\T}\colon \Alg_{\T} \to \CAlgh_{E_0}$ is also \emph{comonadic}, with cofree functor $W_{\T}$.  In this section, we give a more explicit description of $W_{\T}$ and relate it to the total power operation.

\begin{rmk}\label{rmk:Tstruct}
Since the free algebra $\T(E_0)$ corepresents the functor $U_{\T} \colon \Alg_{\T} \to \CAlgh_{E_0}$ to $E_0$-algebras (and not just sets), it admits the structure of a \emph{co-$E_0$-algebra} in $\Alg_{\T}$, and in particular it admits the following structures:
\reqnomode
\begin{align}
\Delta^+ &\colon \T_l (E_0) \to \bigoplus_{i+j = l} \T_i(E_0) \otimes \T_j(E_0)  & &  \tag{coaddition}\\
\Delta^{\times} &\colon \T_l(E_0) \to \T_l(E_0) \otimes \T_l(E_0)  & & \tag{comultiplication}\\
\varepsilon &\colon E_0  \to \Hom_{\Alg_{\T}}(\T(E_0), \T(E_0)) & & \tag{co-$E_0$-unit}
\end{align}
where we have indicated the interaction with the grading.  Moreover, the monad structure on $\T$ induces maps
\[
\circ\colon \T_l(E_0) \otimes \T_{m}(E_0) \to \T_{lm}(E_0)
\]
corresponding to composition of power operations.
\tqed
\end{rmk}

\begin{prop}\label{prop:Wstruct}\hfill
\begin{enumerate}
    \item For any $A\in \CAlgh_{E_0}$, there is an isomorphism of $E_0$-algebras 
    \[ W_{\T}(A) = \Hom_{\CAlgh_{E_0}}(\T(E_0), A),\] where the $E_0$-algebra structure on the right is induced by the maps in \Cref{rmk:Tstruct}.  
    \item Via the identification of (1), the unit map on $B \in \Alg_{\T}$
    \[
    \ev_B \colon B \to W_{\T} U_{\T} (B) \cong \Hom_{\CAlgh_{E_0}}(\T(E_0),U_{\T}(B))
    \] sends $b\in B$ to the function $\T(E_0) \to B$ which sends a power operation  $\gamma \in \T(E_0)$ to $\gamma(b)$.    
\end{enumerate}
\end{prop}
\begin{proof}
Recall that for a (discrete) commutative ring $B$, a $B$-plethory is a (discrete) commutative $B$-algebra with a comonad structure on the covariant functor that it represents \cite{BorgerWieland}. As discussed in \cite[\S 4.22]{RezkCong}, $\T(E_0)$ is an example of an \emph{$E_0$-plethory}.  Statement (1) is \cite[\S 1.10]{BorgerWieland}, and (2) is straightforward. \end{proof}

The identification in part (1) above allows us to think of elements $v\in W_{\T}(A)$ in terms of their evaluations at power operations:

\begin{ntn}\label{ntn:t-and-wt}
Any power operation $\lambda \in \T(E_0)$, determines, via evaluation, a functional
\[
\lambda^* \colon W_{\T}(A) \cong \Hom_{\CAlgh_{E_0}}(\T(E_0), A) \to A.
\]
Note that $\lambda^*$ is neither additive nor multiplicative in general. 

On the other hand, $W_{\T}(A)$ is also a $\T$-algebra by construction, and therefore $\T(E_0)$ acts on it.  For $v\in W_{\T}(A)$, we denote the result of the action by an operation $\gamma \in \T(E_0)$ simply by $\gamma v \in W_{\T}(A)$.  These notations interact via the formula
\[
\lambda^* (\gamma v) = (\lambda \circ \gamma)^* v.
\]
\tqed
\end{ntn}

For $B\in \Alg_\T$, the evaluation map $B \to W_{\T}U_{\T}(B)$ can be thought of as a universal lift of the total power operation on $B$ to a $\T$-algebra map; we make this precise as follows:

\begin{rmk}[$W_{\T}$ and the total power operation] \label{rmk:witt-total-pow}
If $R$ is a $K(n)$-local commutative $E$-algebra, we note that the composite
\[
\iota\colon W_{\T}(\pi_0R) \cong \Hom_{\CAlgh_{E_0}}(\T(E_0), \pi_0R) \to \Hom_{\Modh_{E_0}}(\T(E_0), \pi_0R) \cong \prod_i R^0(B\Sigma_i) 
\]
defines an embedding of $W_{\T}(\pi_0R)$ into a product of $R$-cohomologies of symmetric groups which is multiplicative but not additive.  By \Cref{prop:Wstruct}, the composite
\[
\pi_0R\xrightarrow{\ev_{\pi_0R}} W_{\T}(\pi_0R)\xrightarrow{\iota} \prod_j R^0(B\Sigma_j) 
\]
can be identified with the product of the total power operation maps (which is also multiplicative but not additive). \tqed
\end{rmk}

One can also further quotient by transfers to relate $W_{\T}$ to the total \emph{additive} power operation:

\begin{prop}\label{prop:add-tot-pow}
  Consider the composite  
  \[
\iota^+ \colon W_{\T}(\pi_0R) \to  \prod_j R^0(B\Sigma_j)  \to  \prod_j R^0(B\Sigma_j)/I_{\tr}
\]
  of the embedding $\iota$ of \Cref{rmk:witt-total-pow} with quotient by the transfer ideals.  Then $\iota^+$ is a ring homomorphism, and the composite 
  \[
  \iota^+ \circ \ev_{\pi_0R}\colon \pi_0R \to \prod_j R^0(B\Sigma_j)/I_{\tr} 
  \]
  is the product over $j$ of the total order $j$ additive power operations.
\end{prop}
\begin{proof}
The second assertion is clear from \Cref{rmk:witt-total-pow}, so the content of the statement is that $\iota^+$ is a ring homomorphism.  In fact, since we have observed above that $\iota$ is already multiplicative, we just need to show that $\iota^+$ is additive.  

To see this, note that the coaddition on $\T(E_0)$ is the map induced on $E$-homology by unique commutative algebra map
\[
\Ss\{ t \} \to \Ss \{ x,y \}
\]
which sends $t\mapsto x+y$ (recall that $\{ -\}$ denotes the free commutative algebra).  It follows that for $\lambda \in \T_{p^r}(E_0)$, $\Delta^+ \lambda$ is given by $\lambda \otimes 1 + 1\otimes \lambda$ plus terms in the image of 
\[
\tr\colon E_0(B\Sigma_{p^r}) \to E_0(B\Sigma_i \times B\Sigma_j) \simeq \T_i(E_0) \otimes \T_j(E_0)
\]
where $i,j>0$ and $i+j= p^r$.  In other words, the equation $\Delta^+ \lambda= \lambda \otimes 1 + 1\otimes \lambda$ holds modulo transfers.  Dualizing, this means that the map
\[
\iota\colon \Hom_{\CAlgh_{E_0}}(\T(E_0), \pi_0R)  \to  \Hom_{\Modh_{E_0}}(\T(E_0), \pi_0R) 
\]
is additive modulo transfers, as required.  
\end{proof}

\subsection{The Witt filtration}
\label{sub:wittfilt}\

Recall that the monad $\T$ admits a grading $\T = \bigoplus_j \T_j$ by weight.  Due to the presence of non-additive operations, this grading on $\T$ does not quite induce a grading on the dual construction $W_{\T}$, but it does induce a filtration
\[
W_{\T}(A) \to \cdots \to W_{\T}^{\leq 2}(A) \to W_{\T}^{\leq 1}(A) \to W_{\T}^{\leq 0}(A) \cong A
\]
which we call the \deff{Witt filtration} (cf. \Cref{dfn:wittfilt}).  This filtration will serve two primary functions:
\begin{enumerate}
    \item The Witt filtration provides an organizational principle for the calculation of $W_{\T}(k)$ in \Cref{sub:pder}, where we use the filtration to control how the additive $p$-derivation $\theta$ and the various additive total power operations interact.  
    \item We will show that for a perfect $k$-algebra $A$, the construction $W_{\T}(A)$ is, on associated graded for the Witt filtration, obtained from $W_{\T}(k)$ by base change along $W(k) \to W(A)$ (\Cref{prop:change-of-ring}).  This allows us to reduce the proof of \Cref{thm:cofree} to the case of a perfect field $k$ (\Cref{cor:only-need-k}). 
\end{enumerate}

The key statements needed for (1) will be proved after defining the Witt filtration (\Cref{prop:witt-op-shift} and \Cref{prop:add-witt-factor}).  Then, we will accomplish (2) in \Cref{subsub:witt-gr} by understanding the associated graded of the Witt filtration on $W_{\T}(A)$ as a module over $W(A)$ via the map $\overline{\ev}$ (and the inclusion $W(A)\to \pi_0E(A)$, cf. \Cref{thm:E_functor}(3)).

\begin{prop}\label{prop:wittideal}
Let $A\in \CAlgh_{E_0}$.  
Then, the subset of \deff{$W_{\T}^{\geq r}(A)$} of $W_{\T}(A)$ defined by
\[
W_{\T}^{\geq r}(A) := \{ v\in W_{\T}(A)| \lambda^* v = 0 \text{ for any } \lambda \in \T(E_0) \text{ of positive weight less than }p^r \}
\]
is an ideal.    
\end{prop}  
\begin{proof}
It follows from the interaction of $\Delta^+$ and $\Delta^{\times}$ with the grading on $\T(E_0)$ (cf. \Cref{rmk:Tstruct}) that if $\lambda \in \T(E_0)$ has positive weight less than $p^r$, then
\begin{align*}
    \Delta^+ \lambda &\in \bigoplus_{i+ j < p^r} \T_i(E_0) \otimes \T_j(E_0)\\
    \Delta^{\times} \lambda &\in \bigoplus_{i,j < p^r} \T_i(E_0) \otimes \T_j(E_0).
\end{align*}
Therefore, for $v,v' \in W_{\T}^{\geq r}(A)$ and $w\in W_{\T}(A)$, the element $v\otimes v'$ (resp. $w\otimes v$) evaluates to zero on $\Delta^+ \lambda$ (resp. $\Delta^{\times} \lambda$).  Hence, we have $\lambda^*(v+v') =0$ and $\lambda^* (wv)=0$.  Thus, $v+v'$ and $wv$ are also in $W_{\T}^{\geq r}(A)$ and so $W_{\T}^{\geq r}(A)\subset W_{\T}(A)$ is an ideal.  
\end{proof}

We may now define:

\begin{dfn}[Witt filtration]\label{dfn:wittfilt}
Define the ring 
{\[
\mdef{W_{\T}^{\leq r}(A)} := W_{\T}(A) / W_{\T}^{\geq r+1}(A).
\]}
We refer to the resulting filtration
\[
W_{\T}(A) \to \cdots \to W_{\T}^{\leq 2}(A) \to W_{\T}^{\leq 1}(A) \to W_{\T}^{\leq 0}(A) \cong A
\]
as the \deff{Witt filtration} on $W_{\T}(A)$.  
\tqed
\end{dfn}

The Witt filtration interacts predictably with the action of $\T$ on $W_{\T}(A)$:

\begin{prop}\label{prop:witt-op-shift}
Let $\gamma \in \T_p(E_0)$ have weight $p$.  Then, for $r\geq 0$,  if $v\in W_{\T}^{\geq r+1}(A)$ then $\gamma v \in W_{\T}^{\geq r}(A)$.  Consequently, if $v \in W_{\T}(A)$ is such that $\gamma v$ is detected in Witt filtration $r-1$, then $v$ is detected in Witt filtration $r$.

\end{prop}
\begin{proof}
  This follows immediately from \Cref{ntn:t-and-wt} and the definition of the Witt filtration, because if $\lambda$ has weight less than $p^{r-1}$ and $\gamma$ has weight $p$, then $\lambda \circ \gamma$ has weight less than $p^r$.  
\end{proof}

We will also need to understand the relationship of the Witt filtration with the total \emph{additive} power operation:

\begin{prop}\label{prop:add-witt-factor}
The total order $r$ additive power operation
\[
\tau_r\colon E_0 \to E^0(B\Sigma_{p^r})/I_{\tr}
\]
factors, as a ring map, through the composite
\[
E_0 \xrightarrow{\ev_{E_0}} W_{\T}(E_0) \to W_{\T}^{\leq r}(E_0)
\]
of evaluation with projection to the $r$-th stage of the Witt filtration.  Similarly, the total mod $\m$ order $r$ additive power operation $\overline{\tau}_r$ factors through the composite
\[
E_0 \xrightarrow{\overline{\ev}_k} W_{\T}(k) \to W_{\T}^{\leq r}(k).
\]
\end{prop}
\begin{proof}
It follows from \Cref{prop:add-tot-pow} that the additive total power operation of order $r$ factors as a composite
\[
E_0 \xrightarrow{\ev_{E_0}} W_{\T}(E_0) \xrightarrow{\iota^+} \prod_j E^0(B\Sigma_j)/I_{\tr} \xrightarrow{\mathrm{proj}_{p^r}} E^0(B\Sigma_{p^r})/I_{\tr},
\]
where $\iota^+$ is a ring map and $\mathrm{proj}_{p^r}$ projects onto the $p^r$th factor.  It suffices to see that the composite $\mathrm{proj}_{p^r}\circ \iota^+$ factors through $W_{\T}^{\leq r}(E_0)$, which amounts to showing that it vanishes on the ideal $W_{\T}^{\geq r+1}(E_0)$.  

But the map $\mathrm{proj}_{p^r}\circ \iota^+$ factors through 
\[
W_{\T}(E_0) \xrightarrow{\iota} \prod_{j} E^0(B\Sigma_j) \to \prod_{j\leq p^r} E^0(B\Sigma_j),
\]
(cf. \Cref{rmk:witt-total-pow}), which vanishes on $W_{\T}^{\geq r+1}(E_0)$ essentially by definition.  

The mod $\m$ version of the statement follows by considering the diagram
\[
\begin{tikzcd}
W_{\T}(E_0) \arrow[r,"\iota"]\arrow[d] & \prod_j E^0(B\Sigma_j) \arrow[r]\arrow[d]& \prod_j E^0(B\Sigma_j)/I_{\tr}\arrow[d] \\
W_{\T}(k)\arrow[r,"\overline{\iota}"] & \prod_j E^0(B\Sigma_j)\otimes_{E_0} k \arrow[r] & \prod_j E^0(B\Sigma_j)/I_{\tr} \otimes_{E_0} k
\end{tikzcd}
\]
obtained by naturality, and applying a straightforward analogue of the above argument.  
\end{proof}

\subsubsection{Witt components and the associated graded}\label{subsub:witt-gr}\hfill

With the change-of-rings theorem (\Cref{prop:completed-tensor-formula}) as a goal, we turn our attention to understanding the associated graded pieces
\[
W_{\T}^{=r}(A) := W_{\T}^{\geq r}(A) / W_{\T}^{\geq r+1}(A)
\]
of the Witt filtration.  The primary input to our understanding is the following theorem of Strickland:

\begin{thm}[\cite{Strickland1}]\label{thm:strickland}
For $n,l\geq 0$, let 
\[
\overline{d}(l) \coloneqq \genfrac[]{0pt}{0}{n+l-1}{n-1}_p = \prod_{j=1}^{n-1} \frac{p^{l+j}-1}{p^j-1}
\]
denote the Gaussian binomial coefficient, which counts the number of subgroups $H\subset (\Q_p/\Z_p)^n$ of cardinality $|H|=p^l$.  

Then $\T(E_0)$ is a polynomial $E_0$-algebra on generators in $p$-power degrees, and for $r\geq 0$, the $E_0$-module $(Q\T(E_0))_{p^r}$ of indecomposables of weight $p^r$ is free of rank $\overline{d}(r)$.
\end{thm}

Moving forward, we fix a choice of polynomial generators for $\T(E_0)$ as an $E_0$-algebra, which determines a bijection of \emph{sets}
\[
\omega \colon W_{\T}(A) \xrightarrow{\cong} \Hom_{\CAlgh_{E_0}}(\T(E_0),A) \xrightarrow{\cong} \prod_{r\geq 0} A^{\overline{d}(r)}
\]
which is natural in $A\in \CAlgh_{E_0}$.  

\begin{rmk}
The map $\omega$ is analogous to fixing a choice of \emph{Witt components} of a Witt vector.  More precisely, in the case where $E$ is a height $1$ Lubin-Tate theory corresponding to the multiplicative formal group, the construction $W_{\T}$ can be identified with the classical $p$-typical Witt vectors: then, we have that $\overline{d}(r) = 1$ for $r\geq 0$ and there exists a choice of polynomial generators of $\T(E_0)$ such that the components of $\omega$ are the usual Witt components. 
\tqed
\end{rmk}

We can identify the ideals in the Witt filtration $W^{>r}_{\T}(A)$ as those where the ``Witt components of weight at most $p^r$'' vanish:

\[
\begin{tikzcd}
W_{\T}^{\geq r}(A)  \arrow[d,"\cong","\omega^{\geq r}"'] \arrow[r]& W_{\T}(A)\arrow[d,"\cong", "\omega"'] \\
\prod_{i \geq r} A^{\overline{d}(i)}\arrow[r]  & \prod_{i\geq 0} A^{\overline{d}(i)}. 
\end{tikzcd}
\]

While this identification is not quite additive, it \emph{is} additive at the level of associated graded, which allows us to identify $W_{\T}^{=r}(A)$ \emph{as an abelian group}:


\begin{prop}\label{prop:witt-gr-ab}
  For $A\in \CAlgh_{E_0}$, there is a natural isomorphism of abelian groups
    \[
  \omega^{=r}\colon W_{\T}^{=r}(A) \xrightarrow{\cong}  A^{\overline{d}(r)}
  \]
  which fits into a commutative diagram of \emph{sets}
  \[
  \begin{tikzcd}
  W_{\T}^{\geq r}(A) \arrow[d,"\cong","\omega^{\geq r}"'] \arrow[r] &  W_{\T}^{=r}(A) \arrow[d,dashed, "\cong","\omega^{=r}"']\\
  \prod_{i \geq r}A^{\overline{d}(i)} \arrow[r]& A^{\overline{d}(r)}.  
  \end{tikzcd}
  \]
\end{prop}
\begin{proof}
Consider the (non-dotted) diagram of sets
\[
  \begin{tikzcd}
  W_{\T}^{\geq r+1}(A) \arrow[d,"\cong","\omega^{\geq r+1}"'] \arrow[r] & W_{\T}^{\geq r}(A) \arrow[d,"\cong","\omega^{\geq r}"'] \arrow[r] &  W_{\T}^{=r}(A)\arrow[d,dashed]\\
  \prod_{i\geq r+1}A^{\overline{d}(i)} \arrow[r] & \prod_{i \geq r}A^{\overline{d}(i)} \arrow[r]& A^{\overline{d}(r)}
  \end{tikzcd}
\]
where the top and bottom rows are short exact sequences of abelian groups.  Recall from the proof of \Cref{prop:add-tot-pow} that the coaddition map
\[
\Delta^+ \colon \T_l(E_0) \to \bigoplus_{i+j = l} \T_i(E_0) \otimes \T_j(E_0)
\]
sends $\lambda \in \T_l(E_0)$ to $\lambda \otimes 1 + 1\otimes \lambda$ plus elements of the form $\lambda' \otimes \lambda''$ where $\lambda',\lambda''$ both have weight less than $l$.  This means that the composite  $W_{\T}^{\geq r}(A) \to A^{\overline{d}(r)}$ is a map of abelian groups.  Moreover, note that the kernel of this composite is exactly $W_{\T}^{\geq r+1}(A)$, and therefore we obtain an isomorphism of groups filling in the dotted arrow in the diagram, which we define to be $\omega^{=r}$.
\end{proof}

However, in order to prove \Cref{prop:completed-tensor-formula}, we will need finer control over this associated graded as $A$ changes.  Let us now specialize to the case of interest, where $A$ is a perfect $k$-algebra regarded as an $E_0$-algebra via the composite $E_0 \to k \to A$.  Then, the map
\[
\overline{\ev}_A \colon \pi_0E(A) \to W_{\T}(A) 
\]
equips $W_{\T}(A)$ with an $\pi_0E(A)$-algebra structure.  Under the isomorphism
\[
\pi_0E(A) \cong W(A)\ll u_1, \cdots ,u_{n-1} \rr 
\]
of \Cref{thm:E_functor}(3), this in particular equips $W_{\T}(A)$ with a $W(A)$-algebra structure by restriction of scalars.

Since $p=0$ in $A$ and $W_{\T}^{=r}(A)\cong A^{\overline{d}(r)}$ as abelian groups, this $W(A)$-module structure descends to an $A$-module structure on the associated graded.  It turns out that the resulting $A$-module structure is \emph{not} quite the obvious pointwise one, but rather an $r$-fold Frobenius twist thereof.  To see this, we have the following lemma about the action of the multiplicative lifts $[a]\in W(A)$ on $W_{\T}(A)$:

\begin{lem}\label{lem:grAmod}
For a perfect $k$-algebra $A$, regard $W_{\T}(A)$ as a $W(A)$-module via $\overline{\ev}_A$ as above and let $[-]\colon A\to W(A)$ denote the multiplicative lift.  Then for $a\in A$, $v \in W_{\T}(A)$, and $\lambda \in \T(E_0)_{p^r}$, we have
\[
\lambda^* ([a] \cdot v) = a^{p^r}\lambda^* v.
\]
\end{lem}
\begin{proof}
Recall from \Cref{rmk:witt-total-pow} that there is a multiplicative embedding
\begin{equation}\label{eqn:multembA}
W_{\T}(A) \hookrightarrow \Hom_{\Modh_{E_0}}(\T(E_0),A)  \cong \prod_i E^0(B\Sigma_i) \otimes_{E_0} A
\end{equation}
such that the composite 
\[
\pi_0E(A) \xrightarrow{\overline{\ev}_A} W_{\T}(A) \to \prod_i E^0(B\Sigma_i) \otimes_{E_0} A
\]
is the product of the total mod $\m$ power operations on $E(A)$.  Since the map $\lambda^*$ factors through the projection of (\ref{eqn:multembA}) to the $p^r$th component, it suffices to show that $[a]$ acts by $a^{p^r}$ along the total mod $\m$ power operation map
\[
\pi_0E(A) \to E^0(B\Sigma_{p^r})\otimes_{E_0} A.
\]
But $[a]\in \pi_0E(A)$ is a strict element by \Cref{cnstr:teichmuller-lift}, so the weight $p^r$ total power operation on it is just given by raising to the $p^r$th power (\Cref{rmk:strict}).

\end{proof}

We may now identify the associated graded of the Witt filtration. 

\begin{prop}\label{prop:grAmod}
Let $A$ be a perfect $k$-algebra and regard $W_{\T}^{=r}(A)$ as an $A$-module as above.  Then:
\begin{enumerate}
    \item Under the identification $\omega^{=r}\colon W_{\T}^{=r}(A) \cong A^{\overline{d}(r)}$, the action of $a\in A$ through $\overline{\ev}_A$ is given by multiplication by $a^{p^r}$ on each component.  
    \item Consequently, $W_{\T}^{=r}(A)$ is a free $A$-module of rank $\overline{d}(r)$ (on the standard generators of $A^{\overline{d}(r)}$).
\end{enumerate}
\end{prop}
\begin{proof}
  Part (1) follows from \Cref{lem:grAmod} by noting that since $p=0$ in $W_{\T}^{=r}(A)$, the action of $x\in W(A)$ on the associated graded depends only on the reduction $\overline{x}$ of $x$ modulo $p$ (and so we may replace $x$ by $[\overline{x}]$).  The second part follows because $A$ is perfect.  
\end{proof}

We now use our knowledge of the associated graded to prove \Cref{prop:completed-tensor-formula}.
Note that by naturality of $\overline{\ev}$, a perfect $k$-algebra $A$ yields a natural map
\[
W_{\T}(k)\otimes_{W(k)} W(A) \to W_{\T}(A).
\]
We have:



\begin{prop} \label{prop:change-of-ring}
  Given a perfect $k$-algebra $A$, the induced map
  \[ W_{\T}(k) \otimes_{W(k)} W(A) \to W_{\T}(A) \]
  sends the Witt filtration on the source to the Witt filtration on the target and becomes an isomorphism upon completing the source with respect to the Witt filtration.
\end{prop}

\begin{proof}
The map respects the Witt filtration because $W_{\T}(k)\to W_{\T}(A)$ respects filtration essentially by definition, and the Witt filtration on $W_{\T}(A)$ is a filtration by $W(A)$-modules.   Thus, there is an induced map
\[
W_{\T}^{=r}(k) \otimes_{W(k)} W(A) \to W_{\T}^{=r}(A)
\]
on associated graded.  But this map is an isomorphism by \Cref{prop:grAmod} and the naturality of $\omega^{=r}$, so the conclusion follows from the fact that the right-hand side is complete with respect to the Witt filtration.
\end{proof}

\begin{cor} \label{cor:only-need-k}
In order to prove that the evaluation map
  \[
  \overline{\ev}_A\colon \pi_0E(A) \longrightarrow W_{\T}(A)
  \]
  is an isomorphism for any perfect $k$-algebra $A$, it suffices to show that the evaluation map
  \[ \overline{\ev}_k \colon \pi_0E(k) \longrightarrow W_{\T}(k) \]
  is an isomorphism.
\end{cor}

\begin{proof}
  Using the description of the homotopy groups of $E(A)$ from \Cref{thm:E_functor}(3) and
  the description of $W_{\T}(A)$ from \Cref{prop:change-of-ring} we can identify
  the evaluation map at a perfect $k$-algebra $A$ with the map
  \[ \left( E(k)_0 \otimes_{W(k)} W(A) \right)_{\m}^{\wedge} \longrightarrow \left( W_{\T}(k) \otimes_{W(k)} W(A) \right)_{W}^{\wedge}, \]
  where $(-)^{\wedge}_W$ indicates completion along the Witt filtration. 
  By hypothesis, this map is an isomorphism before completion, so it suffices to show that the topologies induced by $\m$ and $W_{\T}$ coincide.  
  
  To see this, note that both topologies are induced from $\pi_0E(k)$, so it suffices to show that the $\m$-adic topology and the $W_{\T}$-adic topology coincide on $\pi_0E(k)$.  But this follows because $\m$ is finitely generated and the Witt filtration is exhaustive (again by hypothesis).
\end{proof}

\subsection{Change of height}
\label{sub:detection}\ 


In this subsection we prove \Cref{prop:detecting-classes} which is a reformulation of the following proposition.





\begin{prop}\label{prop:Gammainjrefined}
  Given an $x \in E_0$ which is nonzero mod $p$,
  there exists an $r \gg 0$ such that 
  $ \tau_r (x) \not\equiv 0 \mod{\m} $
  where $\tau_r$ is the total additive power operation of weight $p^r$.
\end{prop}

Our approach will be inductive, relying on transchromatic maps from $E(k)$ to a Lubin-Tate theory of height $n-1$ (cf. \Cref{rmk:induct}).   

\begin{lem}\label{lem:lowerht}
  There exists a map of commutative algebras
  \[ \nu \colon E(k) \to E_{n-1}(K) \]
  from $E$ to some Lubin-Tate theory $E_{n-1}(K)$ of height $n-1$
  and perfect residue field $K$
  such that the induced map on $\pi_0$ is injective mod $p$.
\end{lem}

\begin{proof}
  The composite of the sequence of commutative algebra maps 
  \[
    E(k) \to E(k)[u_{n-1}^{-1}] \to L_{K(n-1)}(E(k)[u_{n-1}^{-1}])
  \]
  can be identified on $\pi_0$ with the map of rings
  \begin{equation}\label{eqn:lowerhtcomp}
    W(k)[\![u_1, \dots ,u_{n-1}]\!] \to (W(k)[\![u_1, \dots , u_{n-1}]\!][u_{n-1}^{-1}])^{\wedge}_{\m_{n-2}}
  \end{equation}
  (cf. \cite[Proposition 7.10]{HovStrick}).
  In particular, this latter ring is $T(n-1)$-locally nontrivial, and so by \Cref{cor:mod_algclosed} (cf. \Cref{rmk:induct}), there exists a Lubin-Tate theory $E_{n-1}(K)$ of height $n-1$ with perfect residue field $K$ and a map of commutative algebras
  \[
    \nu\colon L_{K(n-1)}(E(k)[u_{n-1}^{-1}]) \to E_{n-1}(K).
  \]
  Since the map (\ref{eqn:lowerhtcomp}) is visibly injective, it suffices to see that $\nu$ is injective on $\pi_0$ modulo $p$.  
  Choose Lubin-Tate parameters $p, w_1, \dots ,w_{n-2} \in \pi_0 E_{n-1}(K)$ so that 
  \[
    \pi_0 E_{n-1}(K) \cong W(K)[\![w_1, \dots ,w_{n-2}]\!].
  \]
  Then by the invariant prime ideal theorem, 
  \[
    \nu(u_i) \equiv w_i \pmod{(p, w_1, \dots, w_{i-1})}
  \]
  for $1\leq i \leq n-2$.
  Since the sequence $p, u_1, \dots , u_{n-2}$ (resp. $p, w_1, \dots ,w_{n-2}$)
  is regular in $\pi_0 L_{K(n-1)}(E(k)[u_{n-1}^{-1}])$ (resp. $\pi_0 E_{n-1}(K)$)
  and the respective rings are complete with respect to these elements,
  it suffices to check that $\nu$ is injective on $\pi_0$
  after passing to the quotient by the ideal $(p,u_1,\dots u_{n-2})$ in the source
  and the ideal $(p,w_1,\dots ,w_{n-2})$ in the target.
  But 
  \[
    \pi_0 L_{K(n-1)}(E(k)[u_{n-1}^{-1}]) / (p, u_1, \dots ,u_{n-2}) \cong k\ldbl u_{n-1} \rdbl 
  \]
  is a field, so this is automatic.  
\end{proof}

The proof of \Cref{prop:Gammainjrefined} builds on the following lemma of Hahn.

\begin{lem}[Hahn] \label{lem:Hahn}
  Given an $x \in \pi_0E(k)$ which is nonzero mod $\m_{n-2}$,
  there exists a $j \gg 0$ such that 
  $ \tau_j(x) \not\equiv 0 \mod{\m} $.
\end{lem}


\begin{proof}
  This is a reformulation of \cite[Chapter 3, Lemma 5.5]{HahnThesis} which
  (when iterated) asserts that
  there exists some additive power operation $Q$ of weight $p^j$ 
  that fixes the ideal $\m_{n-2} \subset E_0$ and for which
  $ Q(x) \not\equiv 0 \mod{\m} $.
  Passing back to the total power operation
  it follows that $\tau_{j}(x) \not\equiv 0 \mod{\m}$ as desired.
\end{proof}


\begin{proof}[Proof of \Cref{prop:Gammainjrefined}]
  We proceed by induction on the height of $\G_0$.
  The case $n=1$ is trivial as $\tau_r$ on $E_0(k)/p$ is just the $r$-fold Frobenius on $k$.
  For the inductive step we use \Cref{lem:lowerht} to provide us with
  a height $n-1$ Lubin-Tate theory $E_{n-1}$ with perfect residue field $K$ and
  a commutative algebra map
  \[
    \nu\colon E(k) \to E_{n-1}
  \]
  which is injective on $\pi_0$ mod $p$.
  We now consider the diagram of rings:
  \[
    \begin{tikzcd}
      \pi_0E(k)/p \arrow[r,"\tau_r"] \arrow[d,"\nu_*"] &
      E(k)^0(B\Sigma_{p^r})/(I_{\tr}, p) \arrow[r] \arrow[d] &
      E(k)^0(B\Sigma_{p^r})/(I_{\tr}, \m_{n-2}) \arrow[d] \\
      \pi_0E_{n-1}/p \arrow[r,"\tau_r^{(n-1)}"] &
      E_{n-1}^0(B\Sigma_{p^r})/(I_{\tr}, p) \arrow[r] &
      E_{n-1}^0(B\Sigma_{p^r})/(I_{\tr}, \m_{n-2}).
    \end{tikzcd}
  \]
  Here, we have written $\tau_r^{(n-1)}$ for
  the version of $\tau_r$ corresponding to the height $n-1$ Lubin-Tate theory $E_{n-1}$.
  Note first that the diagram commutes:
  the left square commutes because $E(k) \to E_{n-1}$ is a map of commutative algebras
  and thus is compatible with the total power operation,
  and the right square commutes because the map of ring spectra $E(k) \to E_{n-1}$
  sends $\m_{n-2}$ to $\m_{n-2}$ by the invariant prime ideal theorem,
  and sends $I_{\tr}$ to $I_{\tr}$ because it is a map of spectra.

  Note further that the bottom composite is the height $n-1$ variant of the map $\overline{\tau}_r$ under consideration. It follows that if $x \in \pi_0E(k)/p$ is a nonzero element, then $\nu_*(x)$ is nonzero by construction of $\nu$, and by induction, we may choose $r$ large enough such that 
  \[
    \tau_r^{(n-1)} \nu_*(x) \not\equiv 0 \pmod{\m_{n-2}}.
  \]
  By the commutativity of the above diagram, it follows that 
  $ \tau_r(x) \not\equiv 0 \pmod{\m_{n-2}}$. 
  Therefore, there exists some additive power operation $P$ of weight $p^r$ such that
  $ P(x) \not\equiv 0 \pmod{\m_{n-2}} $.
  Now we are in a situation where we can apply \Cref{lem:Hahn} to find an additive operation $Q$ of some weight $p^j$ such that
  \[ Q(P(x)) \not\equiv 0 \pmod{\m}. \] 
  It follows that $\tau_{r+j}(x) \not\equiv 0 \pmod{\m}$, as desired.  
\end{proof}

\subsection{Rigidity of isogenies}
\label{sub:evsurj}\ 

Our goal in this subsection is to prove \Cref{prop:power-op-surj} which we reproduce below.

\begin{prop} \label{prop:Gammasurj}
  For every $r\geq 0$, the mod $\m$ total additive power operation
  \[ \overline{\tau}_r \colon \pi_0E(k) \to E(k)^0(B\Sigma_{p^r})/(I_{\tr}, \m) \]
  is surjective.
\end{prop}

The proof of \Cref{prop:Gammasurj} will involve interpreting $\tau_r$ in terms of the deformation theory of formal groups.

\begin{ntn}
Let $\hR$ denote the category of complete Noetherian local rings with characteristic $p$ residue field and local ring homomorphisms. For $B\in \hR$, we will let $\m_B$ denote the maximal ideal and $\pi :B \to B/\m_B$ denote the reduction modulo $\m_B$.
\tqed
\end{ntn}

We will consider moduli problems which are defined on $\hR$.  For $B\in \hR$, we let $\Spf(B)$ to denote the functor
\[
\Hom_{\hR}(B, -)\colon \hR \to \Set
\]
and we will use the term \emph{scheme} to mean a functor of the form $\Spf(B)$ for $B\in \hR$.  Accordingly, for $X=\Spf(B)$, we will sometimes write ${\cO}_X$ for $B$.   

\begin{dfn}
For $B\in \hR$, a \deff{deformation of $\G_0$ over $B$} is a triple $(\G, i, \alpha)$ where $\G$ is a formal group over $B$, $i:k\to B/\m_B$ is a map of rings, and $\alpha\colon \pi^* \G \simeq i^* \G_0$ is an isomorphism of formal groups.
\tqed
\end{dfn}

The deformations of $\G_0$ over $B$ can be organized into a groupoid $\Def(\G_0)(B)$ which turns out to be discrete \cite{lubin1966formal}.  The resulting functor
\[
\Def(\G_0)\colon \hR \to \Set
\]
is represented by $E_0$ -- that is, $\Def(\G_0) \cong \Spf(E_0).$

Similarly, given formal groups $\G_0$ and $\G_0'$ over $k$ and an isogeny $q_0:\G_0 \to \G_0'$, a \deff{deformation of $q_0$ over $B$} is the data $(q, i, \G, \G', \alpha, \alpha')$ where $(\G, i, \alpha)\in \Def(\G_0)(B)$ and $(\G',i,\alpha')\in \Def(\G_0')(B)$ are deformations, and $q:\G \to \G'$ is an isogeny fitting into the following diagram:
\[
\begin{tikzcd}
\pi^*\G \arrow[r,"\pi^* q"]\arrow[d,"\alpha"', "\simeq"] & \pi^* \G' \arrow[d,"\alpha'", "\simeq"'] \\
i^* \G_0 \arrow[r, "i^*q_0"] & i^* \G_0'.\\
\end{tikzcd}
\]  There is a scheme $\Def(q_0)$ which classifies deformations of $q_0$ \cite[\S 13]{Strickland2}.  By construction, $\Def(q_0)$ admits natural maps $\sigma \colon \Def(q_0) \to \Def(\G_0)$ and $\tau :\Def(q_0) \to \Def(\G_0')$ which remember the source and target deformations, respectively.  

\begin{rmk}
Since $k$ is a perfect $\F_p$-algebra, the relative cotangent complex $L_{k/\F_p}$ vanishes and a standard deformation theory argument shows that for $B\in \hR$, any map $k\to B/\m_B$ lifts uniquely to a map $W(k)\to B$.  Hence, the above schemes $\Def(\G_0)$ and $\Def(q_0)$ admit natural maps to $\Spf (W(k))$, and there is a commutative square
\[
\begin{tikzcd}
\Def(q_0) \arrow[r,"\sigma"]\arrow[d,"\tau "] & \Def(\G_0)\arrow[d] \\
\Def(\G_0') \arrow[r] & \Spf (W(k))
\end{tikzcd}
\]
which induces a map of schemes
\[
(\sigma,\tau)\colon \Def(q_0) \to \Def(\G_0) \times_{\Spf (W(k))} \Def(\G_0').
\]
\tqed
\end{rmk}

The key ingredient in the proof of \Cref{prop:Gammasurj} will be the following statement about $(\sigma,\tau)$:
\begin{prop}\label{prop:Gammasurjgeometric}
Let $q_0 \colon \G_0 \to \G_0'$ be an isogeny of formal groups over $k$.  Then the map of schemes
\[
(\sigma,\tau)\colon \Def(q_0) \to \Def(\G_0) \times_{\Spf (W(k))} \Def(\G_0')
\]
is a closed immersion; that is, the corresponding map on functions
\[
(\sigma,\tau)\colon {\cO}_{\Def(\G_0)} \hat{\otimes}_{W(k)} {\cO}_{\Def(\G_0')} \to {\cO}_{\Def(q_0)}
\]
is surjective.  
\end{prop}

Before we prove this, we need a preliminary lemma about deformations of maps of formal groups:

\begin{lem}\label{lem:defzero}
Suppose that $B\in \hR$ such that $p=0$ in $B$.  Let $\G$ and $\G'$ denote formal groups over $B$ of height $n\geq 1$ and let $\psi_0 \colon \pi^* \G \to \pi^* \G'$ be a map of formal groups over the special fiber.  Then there is at most one map $\psi\colon \G \to \G'$ such that $\pi^* \psi = \psi_0$.  
\end{lem}

\begin{proof}
  Suppose there were two maps $\psi_1$ and $\psi_2$ extending $\psi_0$.
  Their difference $\psi_1 -_{\G'} \psi_2$ is a map extending the zero map.
  Therefore, it suffices to prove the lemma in the case where $\psi_0 = 0$.

  Suppose $\psi \colon \G \to \G'$ is a map extending zero,
  we will show that $\psi=0$.
  Since $B \in \hR$, every line bundle over $B$ is trivial and thus the formal groups $\G$ and $\G'$ are coordinatizable and we choose coordinates $x$ and $y$ so that ${\cO}_{\G} \cong B[\![x]\!]$ and ${\cO}_{\G'} \cong B[\![y]\!]$.
  The map $\psi$ is now determined by the image of $y$ which we view as a power series
  $ g(x) \in B[\![x]\!] $.
  The condition that $\psi$ extends zero tells us that $g$ reduces to zero modulo $\m$.
  
  If $g$ is nonzero,
  then there is a maximal $j$ such that all coefficients of $g$ are in $\m^j$ and we have 
  \[
    g(s)+g(t) \equiv g(s)+_{\G'} g(t) \equiv g(s+_{\G} t) \pmod{\m^{j+1}}
  \]
  from which it follows that 
  \[
    0 = p\cdot g(x) \equiv g([p]_{\G}(x)) \pmod{\m^{j+1}}.
  \]
  On the other hand,
  there is an $i$ such that 
  $g(x) \equiv ax^i \pmod{\m^{j+1}, x^{i+1}}$ for some non-zero $a \in \m^j$
  and using the condition that $\G$ has height $n$
  we can compute that
  \[ g([p]_{\G}(x)) \equiv a(u_n  x^{p^n})^i \pmod{\m^{j+1}, x^{i p^n + 1}} \]
  for some unit $u_n$.
  This is a contradiction and implies $\psi=0$.  
\end{proof}

\begin{proof}[Proof of \Cref{prop:Gammasurjgeometric}]
Since $(\sigma,\tau)$ is a map of complete local rings, it suffices to show surjectivity modulo the maximal ideal of the source.  Geometrically, this corresponds to showing the induced map of schemes
\begin{equation}\label{eqn:defisog}
\Spf(k)\underset{\Def(\G_0)}{\times}^{\sigma} \Def(q_0) {}^\tau\underset{\Def(\G_0')}{\times} \Spf (k)\to  \Spf(k)
\end{equation}
is a closed immersion.  Here, the map $\Spf(k) \to \Def(\G_0)$ classifies sending a $k$-algebra $f:k\to B$ to the \emph{trivial deformation} of $\G_0$ over $B$: that is, the deformation
\[ (f^*\G_0, \pi \circ f, \id_{(\pi \circ f)^* \G_0}) \in \Def(\G_0)(B) \]
(and similarly for $\G_0'$).  Thus, the source of (\ref{eqn:defisog}) classifies deformations of $q_0$ where the source and target are \emph{trivial} deformations: that is, a $B$-point is the data of a ring homomorphism $f\colon k\to B$, an isogeny $q\colon f^*\G_0 \to f^*\G_0'$, and an identification of $\pi^* q$ with $q_0$.  But in this situation, we may apply \Cref{lem:defzero} with $\psi_0 = q_0$.  Since $f^* q_0$ provides an example of a deformation of $q_0$, we conclude that there is \emph{exactly} one $q\colon f^*\G_0 \to f^*\G_0'$ deforming $q_0$.  It follows that the map of schemes (\ref{eqn:defisog}) is an isomorphism, completing the proof.
\end{proof}

We now apply the above results on the deformation theory of isogenies to prove \Cref{prop:Gammasurj}.  By work of Ando, Hopkins, and Strickland \cite{Ando, Strickland1, Strickland2, AHS}, there is a close relationship between power operations in Lubin-Tate theory and deformations of iterates of the Frobenius isogeny, which we summarize here:

\begin{prop}[Proposition 12.12, \cite{AHS}]\label{prop:citeAHS}
Let $E$ be the Lubin-Tate theory associated to a formal group $\G_0$ over a perfect field $k$, and let $\varphi^r \colon \G_0 \to \G_0^{(r)}$ denote the $r$-fold relative Frobenius isogeny.  Then:
\begin{enumerate}
    \item There is an isomorphism of schemes 
    \[
    \Def(\varphi^r) \cong \Spf (E^0(B\Sigma_{p^r})/I_{\tr}).
    \]
    \item Under the above isomorphism and the identification $\Def(\G_0) \cong \Spf E_0$, the map \[ \sigma_r\colon \Spf E^0(B\Sigma_{p^r})/I_{\tr} \to \Spf E_0 \] induced by the canonical $E_0$-algebra structure corresponds to the map of schemes  $\sigma\colon \Def(\varphi^r) \to \Def(\G_0)$, and the map \[ \tau_r :\Spf E^0(B\Sigma_{p^r})/I_{\tr} \to \Spf E_0 \] induced by the total additive power operation corresponds to the composite 
    \[
    \Def(\varphi^r) \xrightarrow{\tau} \Def(\G_0^{(r)}) \cong \Def(\G_0).
    \]
    Here, the isomorphism $\Def(\G_0^{(r)}) \cong \Def(\G_0)$ is given on $B$-points by
    \[ 
    (\G, i,\alpha) \in \Def(\G_0^{(r)}) \quad \mapsto \quad (\G, i\circ \varphi^r, \alpha) \in \Def(\G_0),
    \]which is an isomorphism because $k$ is a perfect field.  
    \item The ring maps
    \[ \circ_{i,j} \colon E^0(B\Sigma_{p^{i+j}})/I_{\tr} \longrightarrow (E^0(B\Sigma_{p^{i}})/I_{\tr}) {{}_{\tau_{i}}\otimes_{\sigma_{j}}} (E^0(B\Sigma_{p^{j}})/I_{\tr}) \]
    induced by composition of power operations correspond to composition of isogenies under the identifications above.  
\end{enumerate}
\end{prop}

\begin{rmk}
The formal schemes $\Spf (E^0(B\Sigma_{p^r}))$, together with the maps $\sigma_r$, $\tau_r$ and $\circ_{i,j}$ above, fit together into a (graded) category object in formal schemes, i.e. a lax formal stack.  Moreover, the structure of additive power operations makes $\pi_0$ of any commutative $E$-algebra $R$ into a sheaf of algebras on this lax formal stack.  The work of Ando-Hopkins-Strickland can be summarized concisely as identifying this lax formal stack with
the lax formal stack determined by the schemes $\mathrm{Def}(\varphi^r)$ and their natural structures described above.
\tqed
\end{rmk}

We are now ready to prove \Cref{prop:Gammasurj}.
\begin{proof}[Proof of \Cref{prop:Gammasurj}]
  From \Cref{prop:citeAHS}, we know that $\overline{\tau}_r$ is the map on functions corresponding to the map of schemes
  \begin{align*}
    \Spf(k)\underset{\Def(\G_0)}{\times}^{\sigma}\Def(\varphi^r) &\to \Spf(k)\underset{\Spf(W(k))}{\times}\Def(\varphi^r) \\
    &\xrightarrow{\tau} \Spf(k)\underset{\Spf(W(k))}{\times}\Def(\G_0^{(r)}) \\
    &\cong \Spf(k)\underset{\Spf(W(k))}{\times}\Def(\G_0) \\
    &\to \Def(\G_0).
  \end{align*}
 The final map is clearly surjective on functions (as it is reduction modulo $p$), so it suffices to show that the composite of the first three maps is surjective on functions.  But, up to the isomorphism $\Def(\G_0^{(r)}) \cong \Def(\G_0)$, this is a base change of the map
  \[
    \Def(\varphi^r) \xrightarrow{(\sigma,\tau)} \Def(\G_0) \underset{\Spf (W(k))}{\times} \Def(\G_0^{(r)}),
  \]
  which is surjective on functions by \Cref{prop:Gammasurjgeometric}.  
\end{proof}

\subsection{Using the $p$-derivation}
\label{sub:pder}\ 

At the point we are essentially ready to complete the proof of \Cref{thm:cofree}.
However, we need one more ingredient in order to assemble the components from the previous subsections. This ingredient is the non-additive power operation $\theta$ of weight $p$, which acts as an \emph{additive $p$-derivation}:

\begin{prop}[Rezk]\label{prop:pder}
  There is an operation $\theta \in \T(E_0)_{p}$ such that for any $\T$-algebra $B$ and $x,y\in B$, we have the relation
  \[
  \theta(x+y) = \theta(x) + \theta(y) + \frac{1}{p}(x^p + y^p - (x+y)^p).
  \]
  Here, the division by $p$ is in the formal sense: the coefficients of the parenthesized polynomial are divisible by $p$, and one divides them by $p$ before evaluating. 
\end{prop}
\begin{proof}
Such an operation appears in unpublished work of Rezk, and independent constructions may be found in \cite{Stapleton}, \cite[Lemma 2.2]{MNNmaynilp}, and \cite[Theorem 4.3.2]{TeleAmbi}.
\end{proof}

The above relation implies that $\theta$ decreases $p$-adic valuation in the following sense:

\begin{lem}\label{lem:thetacong}
  Given a $\T$-algebra $B$ and $x \in B$, we have
  \[ \theta^j (p^k x + p^{k+1}(\cdots) ) \equiv p^{k-j} x^{p^j} \pmod{p^{k-j+1}}. \]
\end{lem}

\begin{proof}  
  As $\theta$ is an additive $p$-derivation, it satisfies the congruences
  \begin{align*}
    \theta(x + y) &\equiv \theta(x) + \theta(y) \pmod{(xy)} \\
    \theta(p^k x) &\equiv p^{k-1} x^p \pmod{p^k}
  \end{align*}
  from which we may conclude.
\end{proof}

Using $\theta$ and these congruences, we can now assemble our results about the various mod $\m$ additive total power operations
\[
\overline{\tau}_r\colon E_0 \to E^0(B\Sigma_{p^r})/(I_{\tr},\m)
\]
from \Cref{sub:detection} and \Cref{sub:evsurj} to prove the main theorem.  First, we have:

\begin{prop}\label{lem:multpWitt}\hfill
\begin{enumerate}
    \item For $0\leq j \leq r$, suppose that $x \in E_0$ is an element such that $\overline{\tau}_j(x) \neq 0$.  Then for any $i\geq 0$ and element $y\in E_0$, the element $p^ix + p^{i+1}y$ is detected in $W^{\leq i+j}(k)$.  That is, its image under the composite 
\[
E_0 \xrightarrow{\overline{\ev}_k} W_{\T}(k) \to W^{\leq i+j}_{\T}(k)
\]
is nonzero.  
\item For any $y \in E_0$, there is
  an additive power operation $Q$ and $i \in \N$
  such that the mod $\m$ reduction of $\theta^iQ(y)$ is nonzero.

\end{enumerate}



\end{prop}

\begin{proof}
  Since $x$ is detected under the map
  \[ \overline{\tau}_j \colon E_0 \to E^0(B\Sigma_{p^j})/(I_{\tr}, \m), \]
  there is an additive operation $Q\in \T(E_0)$ of weight $p^j$
  such that $\overline{Qx}\in k$ is nonzero.
  Then, we may consider the operation $\theta^i Q$.  By \Cref{lem:thetacong}, we have that
  \begin{align*}
    \theta^{i}Q \left( p^i x + p^{i+1}y \right) = \theta^{i} \left( p^iQ(x) + p^{i+1}Q(y) \right) \equiv Q(x)^{p^i} \not\equiv 0 \pmod{p}.
  \end{align*}
  Therefore, since $\theta^i Q$ has weight $p^{i+j}$, $x$ is detected in $W_{\T}^{\leq r}(k)$ as desired.

For the second claim, write $y= p^i z$ where $z\not\equiv 0\pmod p$.  Then, by \Cref{prop:Gammainjrefined}, there is an additive power operation $Q$ of weight $p^j$ such that $Qz \not\equiv 0 \pmod{\m}$; the calculation above then shows that $\theta^i Q(y) \equiv Q(z)^{p^i} \pmod{p}$, which is nonzero modulo $\m$ since $E_0/\m = k$ has no nilpotents.  
\end{proof}


\Cref{prop:Gammasurj} asserts that $\overline{\tau}_j$ is surjective, which ensures that there is a good supply of such elements $x_j$ which are detected by $\overline{\tau}_j$; in conjunction with the above proposition, this will give a lower bound on the size of the image of $E_0$ in $W_{\T}^{\leq r}(k)$.  To explain this bound, we use the following technical lemma:

\begin{lem}\label{lem:dvrcombo}
Let $\cO$ be a discrete valuation ring with uniformizer $\pi$ and residue field $\kappa = \cO/\pi$.  Suppose that $V_0, V_1, \cdots , V_r$ are finite dimensional $\kappa$-vector spaces and $\mE$ is a finitely generated $\cO$-module equipped with surjections of $\kappa$-vector spaces $f_i\colon \mE/\pi \mE \to V_i$ for $0\leq i \leq r$ satisfying the following condition:
\begin{align*}
&\text{Let } i \geq 0 \text{ and } x\in \mE\text{ such that }f_i(x)\neq 0.   \\
&\text{Then for any }y\in\mE, \text{ we have } \pi^{i} x + \pi^{i+1}y \neq 0.
\end{align*}
Then we have the following bound on the $\cO$-module length of $\mE$:
\[
\len_{\cO}(\mE) \geq \sum_{i=0}^{r} \dim_{\kappa}V_i.
\]
\end{lem}

\begin{proof}
Let $\gr_i(\pi^\bullet \mE)$ denote the $i^{\mathrm{th}}$ piece of the associated graded of the $\pi$-adic filtration on $\mE$.
Now choose sections $g_i \colon V_i \to \gr_0(\pi^\bullet \mE)$ of the surjective maps $f_i \colon \gr_0(\pi^\bullet \mE) \to V_i$ of $\kappa$-vector spaces. The condition on $f_i$ implies that the map 
\[ \pi^{i} g_i \colon V_i \to \gr_i(\pi^\bullet \mE) \]
is injective, therefore we have
\[ \len_{\cO}(\mE) = \sum_{i \geq 0} \dim_{\kappa}( \gr_i(\pi^\bullet \mE)) \geq \sum_{i = 0}^r \dim_{\kappa}( V_{i} ) \]
as desired.
\end{proof}



The final ingredient we need is the following theorem of Strickland, which computes the size of target of $\overline{\tau}_j$:

\begin{thm}[Strickland, \cite{Strickland1}]\label{thm:strickgamma}
  For $r\geq 0$, the $E_0$-module $E^0(B\Sigma_{p^r})/I_{\tr}$ is free of rank $\overline{d}(r)$\footnote{See \Cref{thm:strickland} for the notation $\overline{d}(-)$.}.  
\end{thm}

Putting these together:

\begin{ntn}
  Let $E_{\T}^{\leq r}(k)$ denote the image of the map
  $ E_0 \xrightarrow{\overline{\ev}_k}  W_{\T}(k) \to W_{\T}^{\leq r}(k) $.
  \tqed
\end{ntn}

\begin{cor}\label{cor:lenbound}
We have the inequality
 \[ \len_{W(k)} \left( E_{\T}^{\leq r}(k) \right) \geq \sum_{j=0}^r \overline{d}(j). \]  
\end{cor}
\begin{proof}
By \Cref{prop:Gammasurj}, each of the maps $\overline{\tau}_j$ is surjective.  While they are not quite $W(k)$-linear, the field $k$ is perfect so we may replace $\overline{\tau}_j$ with a $W(k)$-linear surjection $\overline{\tau}'_j$ by Frobenius twisting the target.  Then, each map $\overline{\tau}'_j$ for $j\leq r$ factors through $W_{\T}^{\leq r}(k)$ by \Cref{prop:add-witt-factor}.  Moreover, by \Cref{lem:multpWitt}(1), the condition of \Cref{lem:dvrcombo} is satisfied with  $\cO=W(k)$, $\mE = E_{\T}^{\leq r}(k)$, and the maps $f_i = \overline{\tau}'_{r-i}$, so we find that
\[
\len_{W(k)} \left( E_{\T}^{\leq r}(k) \right) \geq \sum_{j=0}^r \dim_{k} E^0(B\Sigma_{p^j})/(I_{\tr}, \m).  
\]
Finally, using \Cref{thm:strickgamma}, we obtain the desired bound.  
\end{proof}

We are now ready to prove the main result of this section:

\begin{thm} \label{thm:ev-iso}
  The evaluation map
  \[ \overline{\ev}_A \colon \pi_0E(A) \longrightarrow W_{\T}(A) \]
  is an isomorphism at every perfect $k$-algebra $A$.
\end{thm}

\begin{proof}
By \Cref{cor:only-need-k}, it suffices to consider the case $A=k$.
  Interpreting $\overline{\ev}_k$ as the total mod $\m$ power operation, we see that the second claim of \Cref{lem:multpWitt} implies injectivity.


  To finish the proof we must show $\overline{\ev}_k$ is surjective.
  Since the $\m$-adic topology on $E(k)$ is the finest topology induced by a collection of maps out to Artinian $W(k)$-algebras (each with the discrete topology), it suffices to argue that 
  $ \pi_0E(k) \to W_{\T}^{\leq r}(k) $
  is surjective for each $r$.  This is equivalent to showing that 
  $E_{\T}^{\leq r}(k) \to W_{\T}^{\leq r}(k)$
  is an isomorphism, and to do that we only need to prove that these objects have equal length as $W(k)$-modules.  
  We computed the length of $W_{\T}^{=r}(k)$ to be $\overline{d}(r)$ in \Cref{prop:grAmod}
  and gave a lower bound on the length of $E_{\T}^{\leq r}(k)$ in \Cref{cor:lenbound}.
  Taken together, we have 
  \[ \sum_{i=0}^r \overline{d}(i)
    \leq \len_{W(k)} \left( E_{\T}^{\leq r}(k) \right) 
    \leq \len_{W(k)} \left( W_{\T}^{\leq r}(k) \right)
    = \sum_{i=0}^r \len_{W(k)} \left( W_{\T}^{= i}(k) \right)
    = \sum_{i=0}^r \overline{d}(i), \]
  from which we may conclude.
\end{proof}

Finally, we deduce the form of the main theorem stated in the introduction.

\begin{proof}[Proof of \Cref{thm:cofree}]
Consider the composite of the following adjunctions (where the left adjoints are on top):

\[
\begin{tikzcd}
\Alg_{\T} \arrow[r,shift left=.9, "U_{\T}"] & \CAlgh_{E_0} \arrow[l, shift left=.9,"W_{\T}"]\arrow[r, shift left=.9, "-\otimes_{E_0}k"] &  \CAlgh_{k} \arrow[l, shift left=.9, "(E_0\to k)_*"] \arrow[r, shift left= .9, "(-)^{\sharp}"]& \Perf_k \arrow[l, shift left= .9, "(-)^{\natural}"].\\
\end{tikzcd}
\]

By \Cref{thm:ev-iso}, we conclude that the right adjoint is naturally equivalent to the functor $\pi_0E(-)\colon \Perf_k \to \Alg_{\T}$ in the statement of the theorem.  Given this, the fully faithfulness follows from the fact that the counit map
\[
(\pi_0 E(A)/\m)^{\sharp} \xrightarrow{\cong}  A
\]
is an isomorphism.

\end{proof}

We record here an immediate consequence of \Cref{thm:cofree}.

\begin{cor}\label{cor:cofree_inj}
Suppose that $f\colon \pi_0E(A) \to R$ is a map of $\T$-algebras such that the reduction of $f$ modulo $\m$ is injective.  Then $f$ itself is injective.
\end{cor}
\begin{proof}
Consider the square
\[
\begin{tikzcd}
\pi_0E(A) \arrow[r]\arrow[d] &  W_{\T}(\pi_0E(A)) \arrow[r]\arrow[d]& W_{\T}(\pi_0E(A)/\m)  \arrow[d] \\
R\arrow[r] & W_{\T}(R)\arrow[r]  & W_{\T}(R/\m) 
\end{tikzcd}
\]
where the left square comes from the units of the $(U_{\T},W_{\T})$-adjunction, and the right square comes from reducing modulo $\m$.  The top composite is an isomorphism by \Cref{thm:cofree} and the right vertical map is injective by hypothesis (since $W_{\T}$ preserves injectivity, say by the description in  \Cref{prop:Wstruct}), so the left vertical map is injective as well.
\end{proof}

\section{Detecting nilpotence}
\label{sec:nilpotence2}

In this section we develop the theory of nilpotence detecting objects in preparation for \Cref{sec:mapout}, where we prove the key results of the paper.

\subsection{Nilpotence detecting objects in locally rigid $\infty$-categories}

\subsubsection{Locally rigid $\infty$-categories}

\begin{dfn}
Let 
\[
\mdef{\Pr^{\rig}}\subset \CAlg(\Pr^{\st})
\]
be the (non-full) subcategory whose objects are compactly generated symmetric monoidal stable $\infty$-categories with the property that every compact object is dualizable, and whose morphisms are functors that preserve compact objects.  
\tqed
\end{dfn}

\begin{warn}
Note that for $\cC$ to belong to $\Prig$, we do not assume that the unit ${\one}_{\cC}$ is compact. In fact, our main examples do not satisfy this extra assumption.  
\tqed
\end{warn}

\begin{lem}\label{Prig_Conserv}
For $\cC \in \Prig$, we have 
\begin{enumerate}
    \item If $c\in \cC^{\omega}$ then $c^{\dual} \in \cC^{\omega} $.
    \item If $a \in \CC$ is such that $a\otimes c =0$ for all  $c\in \cC^{\omega}$,  then $a =0$.
\end{enumerate}
\begin{proof}
For (1) note that since $c^{\dual} \in \cC^{\dualz}$, we have that $c^{\dual}\otimes c \otimes c^{\dual} \in \cC^{\omega}$. Now since $c^{\dual}$ is retract of $c^{\dual}\otimes c \otimes c^{\dual}$ we have that  $c^{\dual} \in \cC^{\omega} $. For (2) let $a \in \CC$ such that for all  $c\in \cC^\omega$ we have $a\otimes c =0$. By (1), we have that $\Hom(c,a) = a\otimes c^{\dual} =0$ for all $c\in \cC^{\omega}$. Since $\cC$ is compactly generated, it follows that $a =0$. 
\end{proof}

\end{lem}

\begin{exm}
  For every $0 \leq n < \infty$, we have that $\Sp_{T(n)}$ is generated by $L_{T(n)}V(n)$ for any non-zero type $n$-complex   (See for e.g,   \cite{heuts2021lie}) So in particular 
  \[
  \Sp_{T(n)} \in \Prig. 
  \] 
  \tqed
\end{exm}

\begin{lem}\label{lem:ModB_is_PRig}
Let $\cC \in \Prig$ and $B\in \CAlg(\cC)$. Consider the $\cC$-linear $\infty$-category $\Mod_B(\cC)$. We have that 
\begin{enumerate}
    \item If $X\in \cC^{\omega}$ and $Y\in \Mod_B(\cC)^{\dualz}$, then the tensor product $X\otimes Y$ (using the $\cC$-linear structure on $\Mod_B(\cC)$) is a compact $B$-module in $\cC$.  We denote the full subcategory of such objects by $\cC^{\omega}\otimes \Mod_B(\cC)^{\dualz} \subset \Mod_B(\cC)^{\omega}$.  
    \item There is an equivalence of categories \[\Mod_B(\cC)^{\omega} = \left(\cC^{\omega}\otimes \Mod_B(\cC)^{\dualz} \right)^{\idem}\]
    where $(-)^{\idem}$ denotes idempotent completion.  
    \item The category $\Mod_{B}(\cC)$ is in $\Prig.$
    \item The functor $\cC \xrightarrow{-\otimes B} \Mod_{B}(\cC)$ is a morphism in $\Prig.$
\end{enumerate}
\end{lem}
\begin{proof}
First, we note that the functor $-\otimes B$ preserves compact objects -- this follows from the fact that its right adjoint preserves colimits.
Thus, (4) follows from (3).   For (3), note first that by \cite[Corollary 4.2.3.7]{HA}, $\Mod_B(\cC)$ is presentable.
Since compact objects are preserved under tensoring with dualizable objects and retracts, we deduce statement (1) and that there is a natural inclusion
\[
\Mod_B(\cC)^{\omega} \supset \mathrm{Thick}\left(\cC^{\omega}\otimes \Mod_B(\cC)^{\dualz} \right).
\]
of the smallest thick subcategory containing $\cC^{\omega}\otimes \Mod_B(\cC)^{\dualz}.$  We claim that this inclusion is actually an equivalence.  To see this, we simply note that if $\{ X_{\alpha}\}_{\alpha \in I}$ are compact generators of $\cC$ (which is compactly generated by hypothesis), then $\{ X_{\alpha} \otimes B\}_{\alpha \in I}$ will be compact generators of $\Mod_B(\cC)$.  Since any object in $\cC^{\omega}\otimes \Mod_B(\cC)^{\dualz}$ is dualizable and dualizable objects are closed under taking thick subcategories, we conclude that $\Mod_B(\cC) \in \Prig$, which is (3).  

Finally, we need (2), which amounts to showing that every object
$M \in \Mod_{B}(\cC)^{\omega}$ is a retract of an object in 
$\cC^{\omega} \otimes \Mod_{B}(\cC)^{\dualz}$. 
Since $M$ is compact, it is enough to show that it can written down as a filtered colimit of objects in $\cC^{\omega} \otimes \Mod_{B}(\cC)^{\dualz}$.  By (3), it is in fact enough to show that any $M \in \Mod_{B}(\cC)^{\dualz}$ can be written down as a filtered colimit of objects in $\cC^{\omega} \otimes \Mod_{B}(\cC)^{\dualz}$.  
Since this property is clearly preserved by tensoring with an object in
$\Mod_{B}(\cC)^{\dualz}$, we are reduced to the case $M=B$.
Indeed, since $\cC\in \Prig$, ${\one}_{\cC}$ is a filtered colimit of objects from $\cC^{\omega}$. By tensoring this colimit with $B$, we get (2).

\end{proof}

\subsubsection{Nilpotence, compactness, and dualizability}

\begin{dfn}
Let $\cC \in \Prig$ and let $f \colon a \to b$ be a map in $\cC$, let $e \in \cC$ be an object. We say that  \deff{$f$ is nilpotent at $e$} if there exists some $m\in \mathbb{N}$ such that 
\[f^{\otimes m} \otimes e \colon a^{\otimes m}\otimes e \to b^{\otimes m} \otimes e \] is null. Moreover 
\begin{itemize}

    \item We say that $f$ is \deff{nilpotent} if it is nilpotent at ${\one}_{\cC}$.
    \item  We say that $f$ is \deff{locally nilpotent} if it is nilpotent at $c$ for every compact object $c \in \cC^{\omega}$. \tqed
\end{itemize}
\end{dfn}

\begin{exm} \label{exm:vi-top-nil}
  As a consequence of the fact that compact $T(n)$-local objects are type $n$,
  the classes $v_i$, for $i<n$, are locally nilpotent on every $T(n)$-local $E(k)$-algebra.
  \tqed
\end{exm}

\begin{lem}\label{lem:nilpotent_vs_lnilpotent}
Let $\cC \in \Prig$ and let $f \colon c \to a$ be a map in $\cC$ with $c \in \cC^{\omega}$.  Then $f$ is nilpotent if and only if it is locally nilpotent.
\end{lem}
\begin{proof}
One direction is clear. For the other, let $f\colon c \to a$ be a locally nilpotent map. In particular $f$ is nilpotent at $c^{\dual}$, so by taking the mate we get that the map $c \otimes c^{\dual} \otimes c^{\otimes N} \to a^{\otimes N}$ is null for some $N$. Precomposing with the map induced by ${\one}_{\cC} \to c^{\dual}\otimes c$, the result then follows from the zig-zag identities.  
\end{proof}

\begin{lem}\label{lem:niloptent_mate}
Let $\cC \in \Prig$ and $a,b,e\in \cC$, $d\in \cC^{\dualz}$,  and $f\colon d\otimes a \to b$ be a map.
Then:
\begin{enumerate}
    \item If $f$ is nilpotent at $e$, then  any retract of $f$ is nilpotent at $e$.
    \item $f$ is nilpotent at $e$ if and only if the mate $\mate{f}:a\to b\otimes d^{\dual}$ is nilpotent at $e$.
\end{enumerate}
\end{lem}
\begin{proof}
(1) follows from the fact that a retract of a null map is null.
(2) follows from the fact that the mate of $f^{\otimes N}$ with respect to the dualizable object $d^{\otimes N}$ is
$(\mate{f})^{\otimes N}$, and the mate of a null map is null.
\end{proof}

\begin{lem}\label{lem:peel_dualz}
Let $\cC \in \Prig$, $d\in \cC$ be an object, and  $f\colon a\to b \in \cC$ be a map.
If $f$ is nilpotent at $d$, then $f\otimes d \colon a\otimes d \to b\otimes d$ is nilpotent. If $d\in \cC^{\dualz}$, the converse also holds.
\end{lem}
\begin{proof}
  For $m \geq 1$, denote by $\cC_m$ the full subcategory of $\cC$ spanned by objects $x$ such that $x\otimes f^{\otimes m}$ is null.
  It is clear that $\cC_m$ is closed under retracts and tensoring with objects in $\cC$.
  Now if $f$ is nilpotent at $d$, then $d\in \cC_m$ and thus also $d^{\otimes m}\in  \cC_m$, so $f\otimes d$ is nilpotent. On the other hand, if $f\otimes d$ is nilpotent, then $d^{\otimes m} \in \cC_m$. Now if $d \in \cC^{\dualz}$, since $d$ is a retract of $d^{\otimes 2}\otimes d^{\dual}$, we deduce  that  $d \in \cC_m$. 
  \end{proof}


\subsubsection{Nilpotence detecting weak rings}

\begin{dfn}
  Let $\cC \in \Prig$, $T \in \cC$ be an  object  in $\cC$, and  $c\in \cC^{\omega}$.  We say that $T$ \deff{detects nilpotence at $c$} if for every map  $f\colon  c \to a$ in $\cC$, $f$ is nilpotent at  $T$  if and only if it is nilpotent.  We say that $T$ \deff{detects nilpotence} if it detects nilpotence at every $c\in \cC^{\omega}$.
  \tqed
\end{dfn}

We will be most interested in the case where the object $T$ is a \emph{weak ring} in the following sense:

\begin{dfn}\label{dfn:Weak_Ring}
    Let $\cC \in \Prig$.
	An object $(u_T\colon {\one}_{\cC} \to T ) \in \cC_{{\one}_{\cC}/}$ is called a \deff{weak ring}
	in $\cC$ if there exists a ``multiplication'' map $\mu{\colon} T\otimes T\to T$ such that the
	composition 
	\[T\oto{u_T\otimes T}T\otimes T\oto{\mu_T}T\] 
	is homotopic to the identity. That is, if $u_T\otimes T$ admits a retract\footnote{We differ slightly from most references (like for e.g \cite[Definition 5.1.4]{TeleAmbi} or \cite[Definition 4.8]{HovStrick}) by not taking $\mu$ to be part of the data of $T$.}.
	\tqed
\end{dfn}

\begin{lem}\label{lem:peel_weak_ring}
Let $\cC \in \Prig$,  $T\in \cC$ a  weak ring,  and $f\colon a \to b$ a map in $\cC$. Then the following are equivalent:
\begin{enumerate}
    \item $f \colon a \to b$ is nilpotent at $T$.
    \item $g := f\otimes T \colon   a\otimes T \to b \otimes T$ is nilpotent.
    \item $h := f\otimes u_{T} \colon  a\to b\otimes T$ is nilpotent. 
\end{enumerate}
\end{lem}
\begin{proof}
We have (1) implies (2) by \Cref{lem:peel_dualz}, so we only need (2) implies (1).  

Assume that $g= f\otimes T$ is nilpotent.  Then $T^{\otimes m} \otimes f^{\otimes m}$ is null.   
  Since $T$ is a weak ring, then $T$ is a retract of $T^{\otimes 2}$, so we deduce  that  $T \otimes f^{\otimes m}$ is null. 
  
  We now show that (2) and (3) are equivalent. (2) implies (3) since $h = g\circ ( a\otimes u_T)$.  For the other direction, note that if $h = f\otimes u_{T}$ is nilpotent, then so is $f\otimes u_{T}\otimes T \colon a\otimes T \to b\otimes T^{\otimes 2}$. By post-composing with $b\otimes \mu_T$, we conclude that $g$ is nilpotent.
\end{proof}

\begin{lem}\label{lem:DN_tensor} 
  Let $\cC \in \Prig$ and let $T_1,T_2$ be two weak rings in $\cC$. If $T_1$ and $T_2$  detect nilpotence,  then their tensor product $T_1 \otimes T_2$ detects nilpotence as well.
\end{lem} 
\begin{proof}
Note that $T_1\otimes T_2$ is itself a weak ring with $u_{T_1\otimes T_2} = u_{T_1}\otimes u_{T_2}$.
Let $g \colon c \to a$ be a map with a compact source. 
If $g$ is nilpotent at $T_1\otimes T_2$,
then by \Cref{lem:peel_weak_ring}, $g\otimes u_{T_1} \otimes u_{T_2}$ is nilpotent. Thus, by \Cref{lem:peel_weak_ring}, $g\otimes u_{T_1} $ is nilpotent at $T_2$. By the assumption that $T_2$ detects nilpotence, it follows that $g\otimes u_{T_1} $ is nilpotent.  Using \Cref{lem:peel_weak_ring} again, we conclude that
$g$ is nilpotent at $T_1$ and thus nilpotent since $T_1$ was assumed to detect nilpotence.
\end{proof}

\begin{lem} \label{lem:topnil-equiv}
  Let $\cC \in \Prig$ and $T$ be an object of $\cC$. Consider the following variations on the condition that $T$ detects nilpotence.
  \begin{enumerate}
  \item Given a map $h\colon {\one} \to a$, if $T \otimes h$ is null, then $h$ is locally nilpotent.
   \item Given a map $g\colon d\to a$ with $d\in \cC^{\dualz}$, if $T \otimes g$ is null, then $g$ is locally nilpotent.
  \item $T$ detects nilpotence.
  \item Given a map $f \in \cC^{\omega}$, if $T \otimes f$ is null, then $f$ is nilpotent.
 \end{enumerate}
 Then we have the implications 
 \[(1) \Rightarrow (2)  \Rightarrow (3)  \Rightarrow (4).  \]
 Moreover, if $T$ is a weak ring, (4) implies (1) and thus (1)-(4) are equivalent.   
\end{lem}
\begin{proof}  
  It is clear that (3) implies (4).

  To show that (2) implies (3), let $g\colon c \to a$ be a map in $\cC$ with $c\in \cC^{\omega}$.
  Assume that $g^{\otimes N}\otimes T$ is null.  Since $c \in \cC^{\omega} \subset \cC^{\dualz}$,  (2) implies that $g$ is locally nilpotent and thus by \Cref{lem:nilpotent_vs_lnilpotent}, $g$ is nilpotent.
  
  To show that (1) implies (2), let 
  $g\colon d \to a$ be a map with $d\in \cC^{\dualz}$ and assume that $T\otimes g$ is null. Let $\mate{g}\colon {\one} \to d^{\dual}\otimes a$ the mate of $g$. We get that 
  $\mate{g}\otimes T$ is null, and thus by assumption $\mate{g}$ is locally nilpotent. Hence, by \Cref{lem:niloptent_mate}(2), so is $g$.

  The heart of this lemma lies in showing that (4) implies (1). Let $h\colon {\one} \to a$ be a map with $T\otimes h$ null, and let $c \in \cC^{\omega}$.
  We need to show that $h$ is nilpotent at $c$.   Write $a$ as a  colimit of compact objects $\{a_\alpha\}_{\alpha \in A}$ for some filtered poset $A$.
  Since $c$ is compact, the map $c \oto{c\otimes h} c\otimes a$  factors as 
  \[c \oto{h_{\beta}} c\otimes a_{\beta} \to c \otimes a.\]
  
  Consider the composite
  \[
  c \oto{c\otimes u_T} c\otimes T \oto{ h_{\beta} \otimes T} c\otimes a_{\beta} \otimes T  \to c\otimes a \otimes T . \]
  Since $T\otimes h$ is null, the composition $c \to c\otimes a \otimes T $ is null. Since $c$ is compact, the nullhomotopy of the map $c\to c\otimes a_{\beta} \otimes T$ occurs at some stage $\gamma \geq \beta$; i.e., the composite
  \[c \oto{c\otimes u_T} c\otimes T \oto{h_{\gamma}\otimes T} c\otimes a_{\gamma}\otimes T \]
   is null.
   
   The composite above is $h_{\gamma}\otimes u_T$ so we deduce by \Cref{lem:peel_weak_ring} that $h_{\gamma}$ is nilpotent at $T$.
   Let $N$ be such that $h_{\gamma}^{\otimes N}\otimes T$ is null. 
   Since the source and target of $h_{\gamma}^{\otimes N}$ are both compact, by our assumption, $h_{\gamma}^{\otimes N}$ is nilpotent and thus $h_{\gamma}$ is also nilpotent. Since $c\otimes h$ factors through $h_{\gamma}$,  we get that $c\otimes h$ is nilpotent. Thus by \Cref{lem:peel_dualz}, $h$ is nilpotent at $c$.
  
\end{proof}

In particular, by condition (1) of \Cref{lem:topnil-equiv} with $a= \Sigma^*{\one}_{\cC}$ a shift of the unit, we have the following:

\begin{cor}\label{cor:finding-top-nil}
For $\cC \in \Prig$, if $T$ is a weak ring which detects nilpotence and 
  $x \in \pi_*{\one}_{\cC}$ maps to zero in $\pi_*T$, then $x$ is locally nilpotent.
\end{cor}

Many of our examples of nilpotence detecting weak rings come from objects satisfying the following stronger condition:

\begin{dfn}
  Let $\cC \in \Prig$ and  $T$ be an  object  in $\cC$.  We say that $T$ is \deff{conservative} if the functor $ T\otimes -$ is conservative.
  \tqed
 \end{dfn}
 
\begin{lem}\label{lem:DN_conservative}
Let $\cC \in \Prig$ and  $T$ be a conservative  object  in $\cC$.  Then $T$ detects nilpotence.
\end{lem}

\begin{proof}
 We use the claim (1) implies (3) from \Cref{lem:topnil-equiv} and we adapt the argument from \cite{NilpII}. 
 Let $f{\colon {\one}} \to a$ be a map in $\cC$, consider the filtered  diagram
  \[ {\one \to}a \to a^{\otimes 2} \to \cdots , \]
  and denote the colimit by $a^{\otimes \infty}$. 
  If $T\otimes f$ is null, then all the maps in the diagram
    \[ T \to T\otimes a \to T\otimes a^{\otimes 2} \to \cdots  \]
    are null. Thus, $T\otimes a^{\otimes \infty} = 0$ and by the conservativity of $ T\otimes -$ , it follows that $a^{\otimes \infty } = 0$.
    Now let $c\in \cC^{\omega}$ be a compact object.
    Since $c \otimes a^{\otimes \infty } =0$, there exists some $N$ for which the map 
    $c \to c\otimes a^{\otimes N}$ is null.
\end{proof}

\begin{cor}\label{Conserative_smash}
Let $\cC \in \Prig$, and let $d \in \cC^{\dualz}$. There exists an idempotent algebra ${\one}_\cC^{d = 0} \in \CAlg(\cC)$ such that $M \in \cC$ is a module over ${\one}_\cC^{d=0}$ if and only if $M\otimes d = 0$. Further ${\one}_\cC^{d=0}\oplus d$ is a conservative object in $\cC$ (and in particular detects nilpotence).
\end{cor}
\begin{proof}
The existence of ${\one}_\cC^{d = 0}$ follows from \cite[Corollary 7.8]{Ragimov}. 
Now let $M \in \cC$.  If $M\otimes ({\one}_\cC^{d=0} \oplus d)=0$, then both $M\otimes d =0$ and
$M\otimes {\one}_\cC^{d=0} =0$. Since $M\otimes d =0$, $M$ is a ${\one}_\cC^{d=0}$-module so $M$ is a retract of (in fact equivalent to) $M\otimes {\one}_\cC^{d=0} =0$ and thus zero.
\end{proof}

\begin{lem}\label{DN_Gm}
Let $\cC \in \Prig$, and let $e \in \mathrm{Pic}(\cC)$ and  let $e \xrightarrow{x} {\one}_{\cC}  $.  Then ${\one}_{\cC}[x^{-1}] \times ({\one}_{\cC}/x)$ is conservative (and in particular detects nilpotence).
\end{lem}
\begin{proof}
Since ${\one}_{\cC}/x$ is dualizable, by \Cref{Conserative_smash}, it is enough to show that ${\one}_{\cC}^{{\one}_{\cC}/x=0} = {\one}_{\cC}[x^{-1}]$, which is the statement of \cite[Appendix C, Proposition C.5]{bunke2018beilinson}.
\end{proof}

\begin{lem} \label{lem:mod-top-nil}
Let $\cC \in \Prig$, and let $d \xrightarrow{x} e \in \CC$ be a locally nilpotent map and assume that $e$ is conservative.  Then
$\cof(x)$, the cofiber of $x$, is conservative.  In particular, if $x\in \pi_* {\one}_{\cC}$ is locally nilpotent, then ${\one}_{\cC}/x$ is conservative (and in particular, detects nilpotence).
\end{lem}

\begin{proof}
Let $M$ be such that $M \otimes  \cof(x) =0$.  
Now let $c\in \cC^{\omega}$.  
Then, $M\otimes  c \otimes x^{\otimes m} $ is an isomorphism but $M \otimes c \otimes x^{\otimes m}$ is null for $m \gg 0$.  Thus, $M\otimes c \otimes e^{\otimes m} =0$ and since $e$ is conservative, we also have $M\otimes c = 0$.  This holds for all $c\in \cC^{\omega}$, so $M=0$. 
\end{proof}

\begin{exm} \label{exm:K-tnil}
  If $R$ is a commutative $K(n)$-local $E(k)$-algebra, then
  since each $v_i \in \pi_*R$ is locally nilpotent for $i <n$ (cf. \Cref{exm:vi-top-nil}),
  writing $R\modm$ as $R/v_0 \otimes_R \cdots \otimes_R R/v_{n-1}$, we conclude from \Cref{lem:mod-top-nil} and \Cref{lem:DN_tensor} that $ R\modm $ detects nilpotence in $\Modw_R$. 
  \tqed
\end{exm}


We will now show that the class of nilpotence detecting weak rings has pleasant closure properties.  First, we will see that it is not only closed under tensor products (cf. \Cref{lem:DN_tensor}), but also closed under base change.  Moreover, in contrast to conservativity, nilpotence detecting for weak rings is closed under filtered colimits.  Before we prove these, we will first need the following lemma:

\begin{lem}\label{lem:DN_source_closure}
Let $\cC \in \Prig$, $T \in \cC$ and let $\cC_T \subset \cC^{\omega}$ be the full subcategory of objects $c \in \cC^{\omega}$ such that $T$ detects nilpotence at $c$. Then:
\begin{enumerate}
    \item $\cC_T \otimes \cC^{\dualz} \subset \cC_T$.
    \item $\cC_T$ is closed under retracts.
\end{enumerate}
\end{lem}
\begin{proof}
(1) follows from \Cref{lem:niloptent_mate}(2).  For (2), assume that $T$ detects nilpotence at $c = a\oplus b$ and that $f:a\to e$ is a map such that $T\otimes f$ is null.  Then the map
\[
f'\colon c \xrightarrow{\mathrm{proj}_1} a \to e
\]
has the feature that $T\otimes f'$ is null and $c\in \cC_T$.  The conclusion then follows from  \Cref{lem:niloptent_mate}(1).
\end{proof}

\begin{prop} \label{lem:tnil-po}
  Let $\cC\in \mathrm{Pr}^{\mathrm{rig}}$, let $T$ be a weak ring in $\cC$, and $B \in \CAlg(\cC)$. 
   If $T$ detects nilpotence in $\cC$, then 
  $T\otimes B $ detects nilpotence in $\Mod_B(\cC)$.
\end{prop}
\begin{proof}
Let $\mathcal{D} \subset \Mod_{B}(\cC)^{\omega}$ be the full subcategory on objects  $c$ 
such that $B\otimes T$ detects nilpotence for maps with source $c$.
Our goal is to show that $\mathcal{D} = \Mod_{B}(\cC)^{\omega}$. By \Cref{lem:ModB_is_PRig}(2) and and \Cref{lem:DN_source_closure}(2), it suffices to show that $\cC^{\omega} \otimes \Mod_B(\cC)^{\dualz} \subset \mathcal{D}$, and so by \Cref{lem:DN_source_closure}(1), it is enough to show that 
$\cC^{\omega} \otimes B \subset \mathcal{D}$.

Indeed, let $c\in \cC^{\omega}$ and 
$f\colon c \otimes B \to a$ a map in $\Mod_{B}(\cC)$ and assume that 
\[f\otimes_B (B\otimes T)\colon c \otimes B \otimes T  \to a\otimes T\] is null.
Note that the adjunction 
\[ -\otimes B \colon \cC \rightleftarrows \Mod_{B}(\cC) \noloc U\]
exhibits $f$ as a mate of a map $\mate{f} \colon c \to U(a)$ in $\cC$.  
 We deduce that the mate of $f\otimes_B (B\otimes T)$ is also null.  But the mate of $f\otimes_B (B\otimes T)$ is $\mate{f} \otimes T$, so we deduce that $\mate{f}$ is nilpotent by the nilpotence detection of $T$; that is, there is some $N$ such that 
\[(\mate{f})^{\otimes N} \colon c^{\otimes N} \to U(a)^{\otimes N}\]
is null.

To finish the proof, we claim that the map 
\[
f^{\otimes_B N} \colon (c\otimes B)^{\otimes_B N} \to a^{\otimes_B N}
\]
is null.  For this, we show that its mate is null.  But its mate
\[
\mate{(f^{\otimes_B N})}\colon c^{\otimes N} \to U(a^{\otimes_B N})
\]
factors through $(\mate{f})^{\otimes N} $ and is therefore null.
\end{proof}

\begin{lem} \label{lem:topnil-fil}
  Let $\cC \in \Prig$ and let  
  \[T_\bullet\colon A \to \cC_{{\one}_\cC/}\] be a filtered diagram  such that $T_{\alpha}$ is a nilpotence detecting  weak ring for every $\alpha\in A$ . Let $u_{T_{\infty}} \in \cC_{{\one}_{\cC}/}$ denote the unit map of the colimit
  \[u_{T_{\infty} }\colon {\one}_{\cC} \to T_{\infty}:=\colim_{\alpha \in A} T_{\alpha}.\] Then $T_{\infty}$ detects nilpotence.
\end{lem}

\begin{proof}
Let $f\colon c \to a$ be map in $\cC$ with $c\in \cC^{\omega}$.
Assume that $f$ is nilpotent at $T_{\infty}$. 
Then there exists some $N$ such that 
$f^{\otimes N} \otimes T_{\infty} \colon c^{\otimes N} \otimes T_{\infty} \to a^{\otimes N} \otimes T_{\infty}$ is null.
By pre-composing with $c^{\otimes N}\otimes u_{T^{\infty}}$, we deduce that $f^{\otimes N}\otimes u_{T_{\infty}}$ is null.
 Since $c^{\otimes N}$ is compact, we get that there is $\alpha \in A$ such that $f^{\otimes N}\otimes u_{T_{\alpha}}$ is null. By \Cref{lem:peel_weak_ring}, it follows that $f^{\otimes N}$ is nilpotent at $T_{\alpha}$. Finally, by the assumption that $T_{\alpha}$ detects nilpotence, $f^{\otimes N}$ is nilpotent, and thus $f$ is as well.
\end{proof}




\subsubsection{Phantom maps and detecting nilpotence}\hfill

Another source of nilpotence detecting maps that will be important for us arise from phantom maps.

\begin{dfn}
Let $\cC \in \Prig$.  We say that a map $f\colon a\to b$ in $\cC$ is \deff{phantom} if for every compact object $c\in \cC^{\omega}$ and map $g:c\to a$, the composite $f\circ g$ is null.  For $m\in \NN$, we say a map $f:a\to b$ is \deff{$\otimes^m$-phantom} if $f^{\otimes m}$ is phantom.  
\tqed
\end{dfn}

\begin{lem}\label{lem:phantom_colim}
Let $I$ be a filtered category and $f_{(-)}\colon I\to \cC^{\Delta^1}$ be a functor such that $f_i$ is $\otimes^m$-phantom.  Then $f = \colim_{i\in I} f_i$ is $\otimes^m$-phantom.  
\end{lem}
\begin{proof}
Since $f^{\otimes m} = \colim_{i\in I} f_i^{\otimes m}$, it suffices to show this for $m=1$.  Now let $f_i \colon a_i\to b_i$ and suppose $c\in \cC^{\omega}$.  Then, by compactness of $c$, any map $g:c\to \colim_I a_i$ factors through some $g'\colon c\to a_j$.  Now, this means $f$ factors through $f_j\circ g'$ which is null because $f_j$ was assumed to be phantom.  
\end{proof}

\begin{lem}\label{lem:DN_phantom}
Let $\cC \in \Prig$ and let $T$ be a weak ring.  Suppose that the fiber $F\to \one$ of $u_T$, the unit of $T$, is $\otimes^m$-phantom for some $m$.  Then $T$ detects nilpotence. 
\end{lem}

\begin{proof}
Let $g\colon c \to c'$ be a map in $\cC^{\omega}$ such that $g\otimes T\colon c\otimes T \to c'\otimes T$ is null and consider the square
\[
\begin{tikzcd}
c\arrow[d] \arrow[r,"g"] & c'\arrow[d]\\
c\otimes T \arrow[r,"g\otimes T"] & c'\otimes T.
\end{tikzcd}
\]
By \Cref{lem:topnil-equiv}(4), we would like to show that $g$ is nilpotent. Since $g\otimes T$ is null, the map $g:c \to c'$ factors as
\[
c \to c'\otimes F \to c'.
\]
Taking $m^{\mathrm{th}}$ powers, we obtain a factorization of $g^{\otimes m}$ as 
\[
c^{\otimes m} \to c'^{\otimes m} \otimes F^{\otimes m} \to c'^{\otimes m}.
\]
We claim that this map is null.  Since compact objects are dualizable, it is enough to show that the mate (with respect to $c'^{\otimes m}$) is null.  But this is a map
\[
(c'^{\dual})^{\otimes m} \otimes c^{\otimes m} \to F^{\otimes m} \to \one_{\cC},
\]
which must be null because the source is compact and $F \to \one_{\cC}$ is $\otimes^m$-phantom.
\end{proof}

\subsection{Nilpotence detection for commutative algebras }

\begin{dfn}
Let $\cC\in \Prig$. We say that a map $A \to B \in \CAlg(\cC)$ \deff{detects nilpotence} if the object $B \in \Mod_A(\cC)$ detects nilpotence.
\tqed
\end{dfn}

\begin{lem}\label{lem:DN_composition}
Let $\cC \in \Prig$ and let $A \xrightarrow {f} B \xrightarrow{g} C \in \CAlg(\cC)$.
Then:
\begin{enumerate}
    \item If $f$ and $g$ detect nilpotence, so does $g\circ f$.
    \item If $g\circ f$ detects nilpotence, then so does $f$.
    
\end{enumerate}
\end{lem}
\begin{proof}
For (1) assume that $f$ and $g$ detect nilpotence. We wish to show that $C$ detects nilpotence in $\Mod_{A}(\cC)$.
Indeed let $c \xrightarrow{h} c' \in \Mod_{A}(\cC)^{\omega}$ and assume that $C\otimes_A h$ is null.  By \Cref{lem:topnil-equiv}, it is enough to show that $h$ is nilpotent. 
By \Cref{lem:ModB_is_PRig} we have  $h\otimes_A B \in \Mod_{B}(\cC)^{\omega}$
and since $ (h\otimes_A B) \otimes_B C  = h \otimes_A C = 0$
and  $g$ detects nilpotence, we get by  \Cref{lem:topnil-equiv}  that $(h\otimes_A B)$ is nilpotent; in other words, we have $h^{\otimes N} \otimes_A B  = (h\otimes_A B)^{\otimes_B N} = 0$ for some $N$.  Applying \Cref{lem:topnil-equiv} again for the nilpotence detecting map $f$, we get that $h^{\otimes N}$ is nilpotent and thus so is $h$.

For (2) assume that $g\circ f$ detects nilpotence. We wish to show that $B$ detects nilpotence in $\Mod_{A}(\cC)$.
Indeed, let $c \xrightarrow{h} c' \in \Mod_{A}(\cC)^{\omega}$ and assume that $B\otimes_A h$ is null.  By \Cref{lem:topnil-equiv}, it is enough to show that $h$ is nilpotent.
Indeed we have 
$(h\otimes_A C) = (h\otimes_A B) \otimes_B C  = 0$.
Thus, by the assumption that $g\circ f$ detects nilpotence and \Cref{lem:topnil-equiv}, we conclude that $h$ is nilpotent.
\end{proof}

\begin{lem}\label{lem:fil_arrow}
 Let $\cC \in \Prig$ and let $I$ be a filtered $\infty$-category and 
 \[F\colon I \to \CAlg(\cC)^{\Delta^1}\] be a functor. If $F(i)$ detects nilpotence for every $i \in I$, then
 $\colim_{I}F$ detects nilpotence.
\end{lem}
\begin{proof}
Write $R \coloneqq \colim_{I}F(0)$ and define
 $\widetilde{F}\colon I \to \CAlg(\cC)^{\Delta^1}$ to be 
\[\widetilde{F}(i)\colon R \to R \coprod_{F(i)(0)} F(i)(1).  \]
Note that $\colim\widetilde{F} = \colim F \in  \CAlg(\cC)^{\Delta^1}$ and that by \Cref{lem:tnil-po}, $\widetilde{F}(i)$ detects nilpotence for each $i\in I$.  We are therefore reduced to case of $\widetilde{F}$, which follows from \Cref{lem:topnil-fil}.
\end{proof}

\begin{lem}\label{lem:DN_nilconservative}
Let $\cC\in \Prig$.  If  a map $A \to B \in \CAlg(\cC)$ detects nilpotence, then the functor 
\[- \otimes_A B \colon \Mod_A(\cC) \to \Mod_B(\cC)\]
is nil-conservative in the sense of \cite{TeleAmbi}.
\end{lem}
\begin{proof}
By \Cref{lem:ModB_is_PRig}, we can replace $\cC$ with $\Mod_A(\cC)$ and  assume without loss of generality that  $A = {\one}_\CC$.
Now let $R \in \Alg(\cC)$ and assume that $R\otimes B =0$.
We deduce from \Cref{lem:topnil-equiv}(1) that the unit map
${\one}_{\cC} \to R$ is locally nilpotent. We wish to show that $R=0$.  Assume that $R \neq 0$; then, since $\cC \in \Prig$, there is some compact object $c\in \cC^{\omega}$ such that $c \otimes R  \neq 0$ (by \Cref{Prig_Conserv}(2)). 
Since ${\one}_{\cC} \to R$ is locally nilpotent, there is some $N$ such that the map 
\[c \to c \otimes R^{\otimes N}\]
is null. Tensoring with $R$ and composing with the product in $R$, we conclude that the composition 

\[c\otimes R \to c \otimes R^{\otimes N+1} \to c\otimes R\]
is null. But since this is the identity map of $c \otimes R$, it follows that
$c \otimes R  = 0$.
\end{proof}

 Recall from \cite{DAGX} the definition of a weakly saturated class of morphisms:
 
\begin{dfn}\label{dfn:saturated}
 Let $\cC$ be an $\infty$-category and let $S$ be a collection of morphisms in $\cC$. We say that $S$ is \deff{weakly saturated} if:
 \begin{enumerate}
     \item $S$ is closed under cobase-change.  That is, for every pushout diagram in $\cC$
     \[\xymatrix{
     A\ar[r]\ar[d]^{f} & B\ar[d]^{f'} \\
     C\ar[r] & D
     }\]
     with $f \in S$, we also have that $f'\in S$.
     \item $S$ is closed under retracts in $\cC^{\Delta^{1}}$.
     \item $S$ is closed under transfinite composition.  That is, for every ordinal $\alpha$ and any functor $F\colon \alpha \to \cC$ such that
     \begin{enumerate}
         \item for each non-zero limit ordinal $\beta < \alpha$, the diagram $F|_{\beta +1}$ is a colimit diagram and 
         \item for any ordinal $\beta$ such that $\beta + 1 < \alpha$ the map 
         \[F(\beta) \to F(\beta +1)\] is in $S$,
     \end{enumerate}
     then the map $F(0) \to F(\beta)$ is in $S$ for every $\beta < \alpha$.
     \tqed
\end{enumerate}
\end{dfn}

\begin{dfn}
Let $\cC$ be an $\infty$-category. Let $f \colon C\to D$ be a morphism in $\cC$ and let $X$ be an object in $\cC$. We will say that \deff{$f$ has the right lifting property with respect to $X$} if for every  map $C\to X$, we have a lift:
     \[\xymatrix{
     C\ar[r]\ar[d]^{f} & X.\\
     D\ar@{-->}[ru] & \empty
     }\]
If $f$ has the right lifting property with respect to $X$, we write \deff{$f \perp X$}.
 \tqed
\end{dfn}

\begin{prop}[The small object argument, {\cite[Proposition 1.4.7]{DAGX}}]\label{prop:small_object}
 Let $\cC$ be a presentable $\infty$-category, let $S$ be a weakly saturated class of morphisms in $\cC$, and let $S_0 \subset S$ be a set of morphisms in $S$.  Then for every $A\in \cC$, there is a morphism $A\to B$ in $\cC$ such that 
 \begin{enumerate}
     \item $A \to B$ is in $S$.
     \item For every $f\in S_0$ we have $f \perp B$ .
 \end{enumerate}
\end{prop}

\begin{thm}\label{thm:nil_saturated}
Let $\cC \in \Prig$.  Then the collection of nilpotence detecting maps is weakly saturated in  $\CAlg(\cC)$.
\end{thm}
\begin{proof}
Condition (1) follows directly from \Cref{lem:tnil-po}.

For condition (2):
Let 
\[
\xymatrix{
A\ar[r]\ar[d]  & C\ar[r]\ar[d]   & A\ar[d] \\
B\ar[r]  & D\ar[r]  & B
}
\]
be a retract diagram in $\CAlg(\cC)^{\Delta^1}$ with $C \to D$ detecting nilpotence. We need to show that 
$A\to B$ detects nilpotence. Indeed, 
let $f\colon M \to N$ be a map of compact $A$-modules and assume that $f\otimes_A B$ is null.  Then, $f\otimes_A D$ is null as well. Thus, by assumption, $f \otimes_A C$ is nilpotent. Taking the base chage along the map $C \to A$, we see that $f  = (f \otimes_A C)\otimes_C A$ is nilpotent.

Finally for condition (3):
Let $F\colon \alpha \to \CAlg(\cC)$ be a diagram as in \Cref{dfn:saturated}. 
We shall prove by transfinite induction that 
$F(0) \to F(\beta)$  detects nilpotence for all $\beta < \alpha$. 
Indeed, the base case $F(0) \to F(0)$ is clear. 
If $\beta = \beta'+1$, then we have a factorization of $F(0) \to F(\beta)$ as $F(0) \to F(\beta')\to F(\beta)$, where the first map detects nilpotence  by the inductive hypothesis and the second by assumption. 
We conclude that $F(0) \to F(\beta)$ detects nilpotence by \Cref{lem:DN_composition}.  
Now assume that $\beta$ is a limit ordinal; then, the claim follows from \Cref{lem:topnil-fil}.


\end{proof}

Combining this with \Cref{prop:small_object}, we obtain the following corollary, which will be our main tool to produce nilpotence detecting maps to rings with prescribed properties:

\begin{cor}\label{cor:Small_object_DN}
Let \[\{A_{s} \to B_s\}_{s\in S}\] be a set of maps in $\CAlg(\cC)$ such that for every $s\in S$ the map 
$A_s \to B_s$ detects nilpotence, and let $R\in \CAlg(\cC)$. Then there exists a map
\[R\to R' \in \CAlg(\cC) \] such that:
\begin{enumerate}
    \item $R\to R'$ detects nilpotence.
    \item For every $s\in S$, we have $(A_s \to B_s)\perp R'$.
\end{enumerate}

\end{cor}

\subsection{Strict Elements} \hfill

Let $\cC \in \CAlg(\Pr)$ be a pointed presentable $\infty$-category.
Since $\mathcal{S}_*$ is the universal pointed presentable $\infty$-category, we get a unique adjunction
\[ \pointL{\cC}{-}\colon \mathcal{S}_* \rightleftarrows \cC \noloc \pointR{\cC}\]
with symmetric monoidal left adjoint.  This gives rise to an adjunction
\[ \pointL{\cC}{-}\colon \CAlg(\mathcal{S}_*) \rightleftarrows \CAlg(\cC) \noloc \pointR{\cC}\]
which we may restrict to a functor
\[ \pointL{\cC}{-} \colon \CAlg(\Setc_{*}) \to \CAlg(\cC).\]

For a commutative monoid $M$ in sets denote by $M_+$ the corresponding commutative  monoid in $\Setc_*$, obtained by adding a disjoint basepoint $+$.  Note that in the ring $\pointL{\cC}{M_+}$, the identity  $e\in M$ corresponds to $1$ and the disjoint basepoint $+$ corresponds to $0$.  
Accordingly, we have that 
\[
 {\one}_{\cC}= \pointL{\cC}{\{e \}_+}, \quad {\one}_{\cC}[t]:= {\one}_{\cC}[{\NN}_+] \quad \text{and} \quad {\one}_{\cC}[t^{\pm 1}]:= {\one}_{\cC}[{\ZZ}_+] .
\]
\begin{dfn}\label{dfn:A1_Gm}
The map $\NN_+ \to \ZZ_+$  by $t\mapsto t$ and the map  $\NN_+ \to \{e\}_+$ by $t \mapsto +$ together induce  a  map
\deff{\[
{\one}_{\cC}[t] \xrightarrow{t\mapsto (t,0)} {\one}_{\cC}[t^{\pm 1}]\times {\one}_{\cC}.
\]}
\tqed
\end{dfn}

\begin{lem}\label{lem:strict_nilp}
Let $\cC \in \Pr^{\rig}$.  Then  the map 
\[{\one}_{\cC}[t] \to {\one}_{\cC}[t^{\pm 1}]\times {\one}_{\cC}\] defined in \Cref{dfn:A1_Gm} 
detects nilpotence.
\end{lem}
\begin{proof}
This is just \Cref{DN_Gm} for 
${\one}_{\cC}[t]  \xrightarrow{t\mapsto t} {\one}_{\cC}[t]  $ in the category $\Mod_{{\one}_{\cC}[t]}(\cC)$. 
\end{proof}

Denote ${\one}_{\cC}[\Np]:= {\one}_{\cC}[{\NN}[{1/p}]_+]$ and ${\one}_{\cC}[\Npm]:= {\one}_{\cC}[{\ZZ}[{1/p}]_+]$.
Similarly to above, we have a map 
\[f\colon {\one}_{\cC}[{\Np}] \to {\one}_{\cC}[{\Npm}]\times {\one}_{\cC}.\]

\begin{thm}\label{thm:perf_nilp}
Let $\cC \in \Pr^{\rig}$.  Then the map 
\[f\colon {\one}_{\cC}[{\Np}] \to {\one}_{\cC}[{\Npm}]\times {\one}_{\cC}\]
 defined above detects nilpotence.
\end{thm}
\begin{proof}
Consider the diagram 
\[\xymatrix{
\one_{\cC}[t] \ar[r]\ar[d] & \one_{\cC}[t^{1/p}] \ar[r]\ar[d] & \one_{\cC}[t^{1/p^2}] \ar[r]\ar[d] & \cdots\ar[d] \\
\one_{\cC}[t^{\pm 1}] \times \one_{\cC} \ar[r] & \one_{\cC}[t^{\pm 1/p}] \times \one_{\cC} \ar[r] & \one_{\cC}[t^{\pm 1/p^2}] \times \one_{\cC} \ar[r] & \cdots .
}\]

We are interested in the map which is the horizontal colimit of all the vertical maps in the diagram. Each vertical map in this diagram  detects nilpotence  by \Cref{lem:strict_nilp}, so we are done by  \Cref{lem:fil_arrow}.
\end{proof}

\subsection{Perfect algebras of Krull dimension 0}\hfill

In this section, we give some preliminaries on perfect $\F_p$-algebras of Krull dimension $0$.  Most of the facts will be deduced from facts about reduced rings of Krull dimension zero, which have been studied extensively in the literature as ``absolutely flat'' (cf. \Cref{prop:TFAE_Red_Krull0}(\ref{c6})) or ``von Neumann regular'' commutative rings.

\begin{prop}\label{prop:TFAE_Red_Krull0}
Let $B$ be a commutative  ring.  Then the following are equivalent
  \begin{enumerate}
      \item $B$ is  reduced and of Krull dimension $0$. \label{c1}
      \item Every principal ideal in $B$ is generated by an idempotent. \label{c2}
      \item Every finitely generated ideal in $B$ is generated by an idempotent.\label{c3}
      \item For every $x\in B$, there exists $y\in B$ such that $x^2y=x$. \label{c4}
      \item We have 
      \[(\mathbb{Z}[t] \xrightarrow{(t \mapsto t)\times (t \mapsto 0)} \mathbb{Z}[t^{\pm 1}] \times \mathbb{Z}) \perp B.\] \label{c5}
      \item Every $B$-module is flat.  \label{c6} 
  \end{enumerate}
\end{prop}
\begin{proof}

The equivalence of (\ref{c1}) and (\ref{c6}) is \cite[Tag 092A, Lemma 15.104.5]{stacks-project}, the equivalence of (\ref{c4}) and (\ref{c6}) is \cite[Theorem 1]{auslander1957regular}, and the equivalence of (\ref{c2}), (\ref{c3}) and (\ref{c4}) is \cite[Theorem 1.1]{Goodearl}.  



It remains to show that (\ref{c5}) is equivalent to these:  we start by showing that (\ref{c4}) implies (\ref{c5}).  Let $x\in B$  be the image of $t$ under a map 
$\ZZ[t] \to B$, take $y$ such that $x^2y = x$, and denote $e = xy$. Now $e$ is an idempotent
in $B$, and thus it suffices to show that the map $\ZZ[t] \to B \to B[e^{-1}]$ factors through 
$\ZZ[t^{\pm 1}]$ and that the map $\ZZ[t] \to B \to B[(1-e)^{-1}]$ factors through $\ZZ$.
Indeed, in $B[e^{-1}]$, $e^{-1}y$ is an inverse to $x$ and in $ B[(1-e)^{-1}]$, we have 
\[
x = x(1-e)(1-e)^{-1} = (x-x^2y)(1-e)^{-1} = 0(1-e)^{-1} = 0.
\]

Finally, we show that (\ref{c5}) implies (\ref{c4}). Let $x \in B$.  By taking the corresponding map 
$\mathbb{Z}[t]\to B$ by $t\mapsto x$, we get an extension to a map $\mathbb{Z}[t^{\pm 1}]\times \mathbb{Z} \to B$.  We take $y \in B$ to be the image of $(t^{-1},0)$.
\end{proof}

\begin{cor}\label{cor:perfvnrprops}
  Let $A$ be a perfect $\mathbb{F}_p$-algebra.  Then the following are equivalent:
  \begin{enumerate}
      \item $A$ is of Krull dimension $0$.\label{c'1}
      \item Every principal ideal in $A$ is generated by an idempotent.     \label{c'2}
      \item Every finitely generated ideal in $A$ is generated by an idempotent.\label{c'3}

      \item For every $x\in A$, there exists a $y\in A$ such that $x^2y=x$.\label{c'4}
      \item \[(\mathbb{F}_p[\Np] \xrightarrow{(t \mapsto t)\times (t \mapsto 0)} \mathbb{F}_p[\Npm] \times \mathbb{F}_p) \perp A.\] \label{c'5}
      \item Every $A$-module is flat.  \label{c'6}
  \end{enumerate}
\end{cor}

\begin{proof}
These are just the 6 conditions above specialized to perfect algebras. For (\ref{c'1}), note that perfect algebras are always reduced and for (\ref{c'5}), note that the inclusion of perfect $\mathbb{F}_p$-algebras into rings admits a left adjoint sending the map $(\mathbb{Z}[t] \xrightarrow{(t \mapsto t)\times (t \mapsto 0)} \mathbb{Z}[t^{\pm 1}] \times \mathbb{Z})$ to the map $(\mathbb{F}_p[{\Np}] \xrightarrow{(t \mapsto t)\times (t \mapsto 0)} \mathbb{F}_p[{\Npm}] \times \mathbb{F}_p)$.
\end{proof}

\begin{cor}\label{cor:flatmap_inj}
  Suppose $B$ is an $\F_p$-algebra such that $B^{\flat}$ is of Krull dimension $0$.  Then the natural map $B^{\flat} \to B$ is injective.
\end{cor}
\begin{proof}
Suppose $x \in \ker (B^{\flat} \to B)$.  Then by \Cref{cor:perfvnrprops}(\ref{c'2}), the principal ideal $(x)\subset B^{\flat}$ is generated by an idempotent $e \in B^{\flat}$.  But by definition of $x$, each of the $p$th root representatives of $e$ must be nilpotent; since nilpotent idempotents are zero, this means that $e=0$ and therefore $x=0$.  
\end{proof}

We now pass to the situation of an arbitrary stable $p$-complete presentably symmetric monoidal $\infty$-category $\cC$. For the following corollary, recall the situation of \Cref{cnstr:C-witt-tilt}, and in particular the functors $\W_{\CC}$ and $(-)^{\flat}_{\CC}$. 


\begin{cor}\label{cor:vnrlift}
  Let $\cC$ be a stable $p$-complete presentably symmetric monoidal $\infty$-category.
  Let 
  \[f\colon {\one}_{\cC}[{\Np}] \to {\one}_{\cC}[{\Npm}]\times {\one}_{\cC}\]
  be the map as in the statement of \Cref{thm:perf_nilp}.  Then an algebra  $R \in \CAlg(\cC)$ satisfies $f\perp R$ if and only if $R^{\flat}_{\cC} \in \Perf_{\mathbb{F}_p}$ is of Krull dimension $0$.
\end{cor}
\begin{proof}
By \Cref{exm:line}, the map $\Ss[{\Np}]_p \to \Ss[{\Npm}]_p\times \Ss_p$ in $\CAlg(\Sp_p)$ is obtained by applying $\W$ to the map $\mathbb{F}_p[\Np] \to \mathbb{F}_p[\Npm]\times \mathbb{F}_p$.  It follows by the formula for $\W_{\CC}$ after \Cref{cnstr:C-witt-tilt} that ${\one}_{\cC}[{\Np}] \to {\one}_{\cC}[{\Npm}]\times {\one}_{\cC}$ is obtained by applying $\mathbb{W}_{\cC}$ to the map $\mathbb{F}_p[\Np] \to \mathbb{F}_p[\Npm]\times \mathbb{F}_p$, so the conclusion follows by \Cref{cor:perfvnrprops}(\ref{c'5}) and the fact that $(-)^{\flat}_{\CC}$ is right adjoint to $\W_{\CC}$.
\end{proof}

\begin{lem}\label{Von Neuman flat}
Let $B$ be a reduced ring of Krull dimension $0$ and let $\{B\to F_i\}_{i\in I}$ be the set of quotient maps by maximal ideals of $B$.  Then 
    the collection of functors  
    \[\Modh_B \to \Modh_{F_i}\] is jointly conservative.
\end{lem}
\begin{proof}
Let $M$ be a non-zero module and let $m\in M$ a non-zero element.  Denote by $J \subset B$ the annihilator of $m$. Since $m$ is non-zero, $J$ is a proper ideal and is thus contained in some  maximal ideal  $J \subset P_i$ with $B/P_i = F_i$. We shall show that $F_i\otimes_B M \neq 0$. Indeed, we have an injection 
$B/J \to  M$ given by $m$. Since $F_i$ is flat over $B$ by \Cref{prop:TFAE_Red_Krull0}(\ref{c6}),  we have an injection 
$ F_i = F_i\otimes_B B/J \to F_i\otimes_B M$.
\end{proof}

\begin{lem}\label{lem:flat}
Let $A \in \Perf_{\mathbb{F}_p}$ be of Krull dimension $0$ and let $A \to F$ be a map to a perfect domain. Let $\cC\in \Prig$ be $p$-complete, and let $c\in \Mod_{\W_{\cC}(A)}(\cC)^{\omega}$ and $d\in \Mod_{\W_{\cC}(A)}(\cC)$.  Then the natural map
\[
W(F)\otimes_{W(A)}\left[c,d\right]_{\Mod_{\W_{\cC}(A)}(\cC)} \to \left[c\otimes_{\W_{\cC}(A)}\W_{\cC}(F),d\otimes_{\W_{\cC}(A)}\W_{\cC}(F)\right]_{\Mod_{\W_{\cC}(F)}(\cC)}
\] is a bijection. 
\end{lem}
\begin{proof}
Let $\mathcal{F} \subset \Perf_A$ be the collection of perfect $A$-algebras that satisfy  the statement of the theorem.  We wish to show that $\mathcal{F}$ contains all perfect domains.
We claim that
\begin{enumerate}
    \item If $B \in \mathcal{F}$ and $e\in B$ is an idempotent, then $B[e^{-1}] \in \mathcal{F}$.
    \item If $I$ is filtered and $G\colon I \to \Perf_{A}$ such that  $G(i) \in \mathcal{F}$ for all $i\in I$, then we also have $\colim_{i \in I} G(i) \in \mathcal{F}$.
    \item If $B\in \mathcal{F}$ and $B \to B'$ is a map which exhibits $B'$ as a free $B$-module, then $B' \in \mathcal{F}$.
\end{enumerate}
Claim (1) is clear since a retract of a bijection is a bijection.  Claim (2) follows from the compactness of $c$ and the fact that $W(-)$ and $\W_{\cC}$ commute with filtered colimits.  Finally, claim (3) is clear for finitely generated free modules $B'$, and the general case follows from the finitely generated case and the compactness of $c$.  

We now claim that if $A \to B$ is surjective, then $B\in \mathcal{F}$. Indeed,
write $B = A/J$.  Now, we can write $J$ as the filtered colimit of its finitely generated subideals $J_i \subset J$.  Then we have that $B$ is the filtered colimit of the $A/J_i$, which by \Cref{prop:TFAE_Red_Krull0}(\ref{c3}) is a filtered colimit of quotients of $A$ by idempotents.   Each of these quotients belongs to $\mathcal{F}$ by $(1)$, and thus $B \in \mathcal{F}$ by the above claim (2).

Now given a map $A \to F$ to a domain, we can decompose the map as $A \to F' \to F$ with $F' \to F$ the inclusion of the image, so $A \to F'$ is surjective.  Since $F'$ is a domain, it is the quotient of $A$ by a prime ideal, which is necessarily maximal since $A$ has Krull dimension 0.  Thus, $F'$ is a field and so $F$ is free over $F'$, and the lemma follows from claim (3).
\end{proof}


\begin{thm}\label{thm:DN_Krull0}
  Let $\cC\in \Prig$ and assume that $\cC$ is $p$-complete.  Let $A \to B$ be a  
  map in $\Perf_{\mathbb{F}_p}$ and assume that $A$ has Krull dimension $0$.  Then the following are equivalent:
  \begin{enumerate}
      \item The object   $\W_{\cC}(B) \in \Mod_{\W_{\cC}(A)} $ is  conservative.
      \item The map  $\W_{\cC}(A) \to \W_{\cC}(B)$ detects nilpotence.
      \item The base change functor \[
      \Mod_{\W_{\cC}(A)} \to \Mod_{\W_{\cC}(B)}
     \] is nil-conservative. 
      \item The map $\Spec(B) \to \Spec(A)$ is  surjective.
  \end{enumerate}
\end{thm}
\begin{proof}
(1) implies (2) by \Cref{lem:DN_conservative}.
(2) implies (3)  by \Cref{lem:DN_nilconservative}.
We now show that (3) implies (4):
Indeed, let $x\in \Spec(A)$ which is not in the image of  the map $\Spec(B) \to \Spec(A)$. Let $A \to F$ be the map to the residue field of $x$. By assumption  
$F\otimes_A B= 0$.  Since the functor  
\[\W_{\cC}\colon \Perf_{\mathbb{F}_p} \to \CAlg(\cC)\] preserves pushouts, we get that $\W_{\cC}(F)\otimes_{\W_{\cC}(A)} \W_{\cC}(B)= 0$; but since 
$\W_{\cC}(F)$ is a ring, this contradicts the nil-conservativity.

Finally we show that  (4) implies (1):  
Indeed, assume that the map $\Spec(B) \to \Spec(A)$ is  surjective and let $ M\in \Mod_{\W_{\cC}(A)}$ such that $M \otimes_{\W_{\cC}(A)} \W_{\cC}(B) = 0$. 
Since $\Mod_{\W_{\cC}(A)}$ is compactly generated by \Cref{lem:ModB_is_PRig}, it is enough to show that $[c,M]_{\Mod_{\W_{\cC}(A)}} = 0 $ for all $c\in \Mod_{\W_{\cC}(A)}^{\omega}$.
Since $\Mod_{\W_{\cC}(A)}$ is $p$-complete and $c$ is compact, $p$ acts nilpotently on $c$ and thus on  $[c,M]_{\Mod_{\W_{\cC}(A)}}$. So it is enough to show that the $A$-module $[c,M]_{\Mod_{\W_{\cC}(A)}}\otimes_\ZZ \mathbb{F}_p  = [c,M]_{\Mod_{\W_{\cC}(A)}}\otimes_{W(A)} A $ is zero.
Indeed, by \Cref{Von Neuman flat}, if this is not the case, there is some perfect residue field $F$ of $A$ such that  
\[
0 \neq ([c,M]_{\Mod_{\W_{\cC}(A)}})\otimes_{W(A)} A \otimes_A F =  ([c,M]_{\Mod_{\W_{\cC}(A)}})\otimes_{W(A)} F.
\]
So we get by \Cref{lem:flat} that 
\[0 \neq ([c,M]_{\Mod_{\W_{\cC}(A)}})\otimes_{W(A)} W(F) = \left[c\otimes_{\W_{\cC}(A)}\W_{\cC}(F),M\otimes_{\W_{\cC}(A)}\W_{\cC}(F)\right]_{\Mod_{\W_{\cC}(F)}}.\]
Now since $\Spec(B) \to \Spec(A)$ is surjective, 
we can fit $F$ into a diagram 
\[\xymatrix{
A\ar[r]\ar[d] & F\ar[d] \\
B\ar[r] & F'
}\]
where $F'$ is a field.
Since $F'$ is free over $F$, we get that 
\[\left[c\otimes_{\W_{\cC}(A)}\W_{\cC}(F'),M\otimes_{\W_{\cC}(A)}\W_{\cC}(F')\right]_{\Mod_{\W_\cC(F')}}\neq 0 \]
which contradicts the assumption that $M \otimes_{\W_{\cC}(A)} \W_{\cC}(B) = 0$, since $A \to F'$ factors through $B$.
\end{proof}

\begin{lem}\label{lem:map_to_Krull0}
 Let $\cC\in \Prig$ and assume that $\cC$ is $p$-complete. 
 Let $A \in \Perf_{\mathbb{F}_p}$ be a perfect $\F_p$-algebra.
 Then there exists a map $f\colon A \to A'$ in  $\Perf_{\mathbb{F}_p}$ such that $A'$ is of Krull dimension $0$ and the induced map 
 \[\W_{\cC}(f)\colon \W_{\cC}(A) \to \W_{\cC}(A')\] detects nilpotence.
\end{lem}
\begin{proof}
Let $S \subset \Perf_{\mathbb{F}_p}^{\Delta^1}$ the collection of maps $A \to B$ such that $\W_{\cC}(A) \to \W_{\cC}(B)$ detects nilpotence. We claim that $S$ is weakly saturated. Indeed, the functor $\W_{\cC}$ preserves colimits by \Cref{prop:witt-and-tilt}, so the claim follows from 
\Cref{thm:nil_saturated}. 
Now consider the singleton set 
\[S_0 = \{\mathbb{F}_p[\Np] \xrightarrow{(t \mapsto t)\times (t \mapsto 0)} \mathbb{F}_p[\Npm] \times \mathbb{F}_p\}.\]
By \Cref{thm:perf_nilp}, we have  $S_0 \subset S$.
Thus, by \Cref{prop:small_object}, there is a map $f \colon A\to A'$ such that 
\begin{enumerate}
    \item $\W_{\cC}(f)\colon \W_{\cC}(A) \to \W_{\cC}(A')$ detects nilpotence.
    \item $(\mathbb{F}_p[\Np] \xrightarrow{(t \mapsto t)\times (t \mapsto 0)} \mathbb{F}_p[\Npm] \times \mathbb{F}_p) \perp A'$.
\end{enumerate}
So we are done by \Cref{cor:perfvnrprops}.
\end{proof}

\begin{thm}\label{DN_perf}
  Let $\cC\in \Prig$ and assume that $\cC$ is $p$-complete.  Let $A \to B$ be a  
  map in $\Perf_{\mathbb{F}_p}$.  Then the following are equivalent:
  \begin{enumerate}
      \item The map $\W_{\cC}(A) \to \W_{\cC}(B)$ detects nilpotence.
      \item The functor \[
      \Mod_{\W_{\cC}(A)} \to \Mod_{\W_{\cC}(B)}
     \] is nil-conservative. 
      \item The map $\Spec(B) \to \Spec(A)$ is  surjective.
  \end{enumerate}
\end{thm}
\begin{proof}
We have (1) implies (2) by \Cref{lem:DN_nilconservative} and the proof that (2) implies (3) is exactly as in the proof of \Cref{thm:DN_Krull0}.

We now show that (3) implies (1). By \Cref{lem:map_to_Krull0} we have a map $A \to A'$ with $A'$ of Krull dimension $0$ such that the map $\W_{\cC}(A) \to \W_{\cC}(A')$ detects nilpotence.
Note that since (1) implies (3), we have that $\Spec(A') \to \Spec(A)$ is surjective. 
Since the map $\Spec(B) \to \Spec(A)$ is  surjective, so is the map 
$\Spec(B\otimes _A A') \to \Spec(A')$. Thus, by \Cref{thm:DN_Krull0} the map 
$\W_{\cC}(A') \to \W_{\cC}(B\otimes_A A')$ detects nilpotence. So by \Cref{lem:DN_composition}(1), we get that the map 
$\W_{\cC}(A) \to \W_{\cC}(B\otimes _A A')$ detects nilpotence.  But since it factors through the map 
$\W_{\cC}(A) \to \W_{\cC}(B)$, we are done by \Cref{lem:DN_composition}(2).

\end{proof}

\subsubsection{The case of a compact unit}

\begin{dfn}\label{dfn:flat}
Let $\cC \in \Prig$  and let $R \in \CAlg(\cC)$. We say that $R$ is \deff{flat} if for all $c\in \cC^{\omega}$ and $a\in \cC$ the map

\[\pi_0(R)\otimes_{\pi_0(\one_{\cC})}[c,a]_{\cC}  \to [c,a\otimes R]_{\cC}\]
is an isomorphism. 
\tqed
\end{dfn}

\begin{lem}\label{lem:Quotient_flat}
Let $\cC\in \Prig$ such that $\one_{\cC} \in \cC^{\omega}$ and denote $A:=\pi_0(\one_{\cC})$.  Assume that $A$ is a reduced algebra of Krull dimension $0$, and let $J \subset A$ be an ideal. Then there exists a commutative algebra $\one/J \in \CAlg(C)$ such that    
\begin{enumerate}
    \item  $\pi_0(\one/J) \cong A/J$ 
    \item $\one/J$ is flat.  
\end{enumerate}
\end{lem}
\begin{proof}

We first prove the statement for  every finitely generated ideal $J \subset A$. Indeed, 
by \Cref{prop:TFAE_Red_Krull0}(\ref{c3}), in this case $J$ is generated by an idempotent $e\in A$ and we can take 
\[
\one/J = \one_{\cC}[(1-e)^{-1}] := \colim\big( \one_{\cC} \xrightarrow{1-e} \one_{\cC} \xrightarrow{1-e} \one_{\cC} \to ... \big).
\]
Claim (1) is now clear by compactness of the unit.  Claim (2) follows from the fact that $\one_{\cC}[(1-e)^{-1}]$ is a retract of $\one_{\cC}$ and a retract of a bijection is a bijection.
Now given an arbitrary ideal $J\subset A$, we can write $J$ as the filtered colimit of its finitely generated subideals $J_i \subset J$.  We take $\one/J  = \colim \one/J_i$. Now (1) follows again from the compactness of the unit and (2) follows from the compactness of $c$ in \Cref{dfn:flat}.
\end{proof}

\begin{thm}\label{DN_Krull0QQ}
  Let $\cC\in \Prig$ and assume that $\one_{\cC} \in \cC^{\omega}$ and $\pi_0(\one_{\cC})$ is reduced of Krull dimension $0$.
  Then for any flat $R\in \CAlg(\cC)$, the following are equivalent:
  \begin{enumerate}
      \item The object   $R\in \cC $ is  conservative.
      \item The object   $R \in \cC $ detects nilpotence.
      \item The functor \[-\otimes R\colon \cC \to \Mod_{R}(\cC)\]
      is nil-conservative. 
      \item The map $\Spec(\pi_0(R)) \to \Spec(\pi_0(\one_{\cC}))$ is  surjective.
  \end{enumerate}
\end{thm}

\begin{proof}
(1) implies (2) by \Cref{lem:DN_conservative}.
(2) implies (3)  by \Cref{lem:DN_nilconservative}.
We now show that (3) implies (4). 
Indeed let $x\in \Spec(\pi_0(\one_{\cC}))$ which is not in the image
of the map $\Spec(\pi_0(R)) \to \Spec(\pi_0(\one_{\cC}))$. Let $\pi_0(\one_{\cC}) \to F$ be the map to the residue field of $x$ and denote the kernel of this map by $J$. Now by \Cref{lem:Quotient_flat}, there is a flat commutative algebra $\one/J \in \CAlg(\cC)$ with 
$\pi_0(\one/J) = F$. Hence, 
\[\pi_0(\one/J \otimes R) \cong \pi_0(\one/J)\otimes_{\pi_0(\one_{\cC})} \pi_0(R) \cong F\otimes_{\pi_0(\one_{\cC})} \pi_0(R) \cong 0, \]
where the first isomorphism is by the flatness of $\one/J$ and the compactness of $\one_{\cC}$, and the third isomorphism is because $x$ is not in the image of $\Spec(\pi_0(R)) \to \Spec(\pi_0(\one_{\cC}))$.
But since $\one/J \otimes R$ is a ring, we deduce  that $\one/J \otimes R =0$  which contradicts the nil-conservativity of $- \otimes R$.

Finally, we show that (4) implies (1). Indeed, assume that the map $\Spec(\pi_0(R)) \to \Spec(\pi_0(\one_{\cC}))$ is  surjective and let $ b\in \cC$ such that $b \otimes R = 0$. 
 Since $\cC$ is compactly generated, it is enough to show that $[c,b]_{\cC} = 0 $ for all $c\in \cC^{\omega}$.
 Now as the map $\pi_0(\one_\cC) \to \pi_0(R)$ is both flat and surjective after applying $\Spec$, it is faithfully flat.  So it is enough to show that 
 $\pi_0(R)\otimes_{\pi_0(\one_{\cC})}[c,b] =0$. But $R$ is flat so
 $\pi_0(R)\otimes_{\pi_0(\one_{\cC})}[c,b]_{\cC}  = [c,b\otimes R]_{\cC} =0 $.
 \end{proof}

\section{Constructing Lubin--Tate covers}
\label{sec:mapout}

In this section we use the ideas developed in the previous sections to prove our first main theorem.

\begin{thm}\label{thm:modmain}
Let $R \in \CAlg(\Sp_{T(n)})$.  Then there is a perfect algebra $A$ of Krull dimension $0$ and a nilpotence detecting map
$R \to E(A)$.\footnote{Recall that in the case of height $n=0$, ``perfect'' means that $A$ is reduced, and $E(A) := A[u^{\pm 1}]$ where the generator $u$ is in degree $2$.}
\end{thm}

In particular, as a corollary, we have:

\begin{cor}\label{cor:mod_algclosed}
Let $0 \neq R \in \CAlg(\Sp_{T(n)})$ be a non-zero $T(n)$-local commutative algebra.  Then there exists an algebraically closed field $L$ and a map 
\[R \to E(L)\]
in $\CAlg(\Sp_{T(n)})$.
\end{cor}
\begin{proof}
Let $i\colon R\to E(A)$ be the nilpotence detecting map given by \Cref{thm:modmain}.
Since $R \neq 0$ and $i$ is nilpotence detecting, we have that $A\neq 0$, so there exist a ring map $A\to L$ for some algebraically closed field $L$. Composing $i$ with the map $E(A) \to E(L)$, we conclude. 
\end{proof}

One can immediately reduce to the case where $R$ is an algebra over $E(k)$ for some perfect $k$, which gives a canonical $E$-colocalization map (cf. \Cref{thm:E_functor})
\[
E(R^{\flat}) \to R.
\]
Informally speaking, the proof proceeds by transfinitely modifying the ring $R$ until the $E$-colocalization map is an equivalence.  The modifications each produce a nilpotence detecting map $R\to R'$ (in particular, $R'\neq 0$) and come in three flavors: getting rid of odd homotopy elements (\Cref{prop:oddnilp}), ensuring that the $E$-colocalization map is surjective on $\pi_0$ (\Cref{prop:insepnilp}), and ensuring $R^{\flat}$ is of Krull dimension zero\footnote{In this situation, it turns out this automatically guarantees the injectivity of the $E$-colocalization on $\pi_0$.} (\Cref{thm:perf_nilp} and \Cref{cor:vnrlift}).  Using the small object argument (\Cref{cor:Small_object_DN}), we may transfinitely iterate these processes to produce a ring with all these properties.  We then show that the resulting ring is necessarily of the desired form $E(A)$ (cf. \Cref{prop:vnrE}).

In order to illustrate this argument, we first prove \Cref{thm:modmain} in the simpler case of characteristic $0$ in \Cref{sub:ht0mod}, where the argument takes the same general form.  We then give a more careful outline of the strategy for positive heights in \Cref{sub:genmodoutline}, leaving the important technical results to the subsequent sections.

\subsection{The case of height 0}\label{sub:ht0mod}\hfill

For a $\QQ$-algebra $A$, we have $\pi_*(E(A)) = A[{u}^{\pm 1}]$ for $|{u}| = 2$.
Our goal is to prove:
\begin{thm}\label{thm:Main_charzero}
  Let $R\in \CAlg(\Sp_{\QQ})$.  Then there exists a reduced $\QQ$-algebra $A$ of Krull dimension $0$ and a nilpotence detecting map $R \to E(A)$.
\end{thm}

First, note that the map $\QQ \to E(\QQ)$ admits a retract in spectra and thus the functor  $-\otimes E(\QQ)$ is conservative. Hence, by \Cref{lem:DN_conservative}, the map $\QQ \to E(\QQ)$ detects nilpotence.  It therefore suffices to show:

\begin{thm}\label{thm:Main_charzero_red}
  Let $k$ be a field of characteristic $0$ and $R\in \CAlg_{E(k)}(\Sp_{\Q})$.  Then there exists a reduced $k$-algebra $A$ of Krull dimension $0$ and a nilpotence detecting map $R \to E(A)$.
\end{thm}

We will prove \Cref{thm:Main_charzero_red} by using \Cref{cor:Small_object_DN} for a set $S$ comprised of two maps.

\begin{dfn}
Define the following two maps in $\CAlg_{E(k)}(\Sp_{\Q})$:
\begin{enumerate}
    \item $\mdef{f_0} \colon E(k)[t] \to E(k)[t^{\pm 1}] \times E(k)$ as in \Cref{lem:strict_nilp}.
    \item $\mdef{h_0} \colon E(k)\{ z^1\} \xrightarrow{z^1 \mapsto 0} E(k)$ for $z^1$ a generator of degree $1$. \tqed
\end{enumerate}
\end{dfn}

\begin{prop}\label{prop:DN_Set_charzero}
The maps $f_0$ and $h_0$ detect nilpotence.  
\end{prop}
\begin{proof}
$f_0$ detects nilpotence by  \Cref{lem:strict_nilp}.  By \cite[Proposition 3.8]{mathew2017residue} $h_0$ is conservative and thus by \Cref{lem:DN_conservative} also detects nilpotence. 
\end{proof}

\begin{prop}\label{prop:ratl-E-char} 
Let $S_0 = \{f_0,h_0\} \subset \CAlg_{E(k)}(\Sp_{\Q})^{\Delta^{1}}$ and $R \in\CAlg_{E(k)}(\Sp_{\Q})$. Then $S_0 \perp R$ if and only if $\pi_0(R)$ is reduced of Krull dimension $0$ and $\pi_*(R) = \pi_0(R)[{u}^{\pm 1}]$.
\end{prop}
\begin{proof}
It is clear that $h_0 \perp R$ if and only if $\pi_{1}(R) = 0$ if and only if $\pi_*(R) \cong \pi_0(R)[{u}^{\pm 1}]$. It thus remains to show that $f_0\perp R$ if and only $\pi_0(R)$ is reduced of Krull dimension $0$. Indeed  $f_0 \perp R$ if and only if the map 

\[
\pi_0(\Map_{\CAlg_{E(k)}(\Sp_{\Q})}(E(k)[t^{\pm 1}] \times E(k),R)) \to \pi_0(\Map_{\CAlg_{E(k)}(\Sp_{\Q})}(E(k)[t],R)) 
\]
is surjective.
Since $k$ is rational, $E(k)[t]$ is a free commutative algebra and so this map of sets is isomorphic to the map of sets:
\[
 \Map_{{\CRing}}(\mathbb{Z}[x^{\pm 1}]\times \mathbb{Z},\pi_0(R))\to  \Map_{{\CRing}}(\mathbb{Z}[x],\pi_0(R)).
\]
The conclusion then follows from \Cref{prop:TFAE_Red_Krull0}.
\end{proof}

Combining the previous two results, we may now prove \Cref{thm:Main_charzero_red}.

\begin{proof}[Proof of \Cref{thm:Main_charzero_red}]
Applying \Cref{cor:Small_object_DN} with $S_0$ as in \Cref{prop:ratl-E-char}, we conclude that there is a nilpotence detecting map $R \to R'$ such that $S_0 \perp R'$.  By \Cref{prop:ratl-E-char}, $\pi_0(R')$ is reduced of Krull dimension 0 and $\pi_*(R') \cong \pi_0(R')[{u}^{\pm 1}]$.  

Unfortunately, it is not clear that this implies $R' \simeq E(\pi_0(R'))$.  In order to get around this, we will map $R'$ to the product of its residue fields.  Let $\{F_x\}_{x\in \Spec(\pi_0(R')) }$ be the set of residue fields of $\pi_0(R')$. By \Cref{lem:Quotient_flat}, we have for each $x\in \Spec(\pi_0(R'))$ a flat $R'$-algebra $R'_x$ such that $\pi_*(R'_x) \cong \pi_0(R'_x)[{u}^{\pm 1}]$ and such that the map $\pi_0(R') \to \pi_0(R'_x)$ is isomorphic to the map $\pi_0(R') \to F_x$.
We claim that there is an equivalence   $R'_x \simeq E(F_x)$  in $\CAlg_{E(k)}(\Sp_{\Q})$.
Indeed, as $E(F_x) = E(k)\otimes_k F_x$, it is enough to construct a map $F_x \to R'_x$ in $\CAlg_{k}(\Sp_{\Q})$ that induces an isomorphism on $\pi_0$. Now since $F_x$ is connective, this is equivalent  to constructing such a map to $\widetilde{R'_x}:=\tau_{\geq 0}R'_x$.
We have that $F_x \cong \tau_{\leq 0}(\widetilde{R'_x})$, so our goal is to lift this map inductively from $\tau_{\leq n}\widetilde{R'_x}$ to $\tau_{\leq n+1}\widetilde{R'_x}$.  Since the map $\tau_{\leq n+1} \widetilde{R'_x} \to \tau_{\leq n}\widetilde{R'_x}$ is a square-zero extension, this is always possible as $F_x$ is formally smooth over $k$.

Collecting the resulting maps $R' \to E(F_x)$ for all $x\in \Spec(\pi_0(R'))$,
we obtain a map 
\[
f\colon R' \to \smashoperator{\prod_{x \in \Spec(\pi_0(R'))}} E(F_x)\cong E\left(\prod_{x }F_x\right)
\]
which is surjective on $\Spec (\pi_0(-))$.  
By \Cref{DN_Krull0QQ}, this means that $f$ detects nilpotence.   Composing with the nilpotence detecting map $R\to R'$, we get a nilpotence detecting map $R \to E(\prod_{x }F_x)$.
Since $\prod_x F_x$ is reduced of Krull dimension $0$, we are done.

\end{proof}

\subsection{The case of positive height}\label{sub:genmodoutline}\hfill

For the remainder of this section, we will consider \Cref{thm:modmain} in the case when our fixed height $n$ is at least $1$.  We start by giving an outline of the proof which is completely analogous to the height $0$ case but requires significant additional technical inputs, which are supplied in Subsections \ref{subsec:insep} and \ref{sub:odd-quo}.  

  First, since the map $\one_{T(n)} \to E(k)$ detects nilpotence for any $E(k)$ of height $n$ \cite{DHS}, it suffices to show:

\begin{thm}\label{thm:modmain_red}
  Let $k$ be a perfect field of characteristic $p$, and let $R\in \CAlgw_{E(k)}$.  Then there exists a perfect $k$-algebra $A$ of Krull dimension $0$ and a nilpotence detecting map $R\to E(A)$. 
\end{thm}

The approach is centered around the following detection result for Lubin-Tate theories: 

\begin{prop}\label{prop:vnrE}
Consider the following three conditions on a commutative algebra $R\in \CAlgw_{E(k)}$:
\begin{enumerate}
    \item $R^{\flat}$ is of Krull dimension $0$.
    \item The $E$-colocalization map $E(R^{\flat}) \to R$ is surjective on $\pi_0$.  
    \item $\pi_1(R)=0$.
\end{enumerate}

If $(1)$ is satisfied, then the $E$-colocalization map $E(R^{\flat}) \to R$ is injective on $\pi_0$.  Consequently, the ring $R$ is equivalent to $E(A)$ for some perfect $A$ of Krull dimension $0$ if and only if all three conditions are satisfied.  
\end{prop}
\begin{proof}
For the first claim, by \Cref{cor:cofree_inj}, it suffices to show that the reduction modulo $\m$
\[
R^{\flat} \to \pi_0(R)/\m
\]
is injective.  But since $R^{\flat} \cong (\pi_0(R)/\m)^{\flat}$, the statement follows from \Cref{cor:flatmap_inj}.  

To deduce the second statement, it is clear that $E(A)$, for $A$ perfect of Krull dimension $0$, satisfies (1), (2), and (3).  Conversely, if all three conditions are satisfied, then the first claim implies that the $E$-colocalization map induces an isomorphism in even degrees, and (3) implies it induces an isomorphism in odd degrees.

\end{proof}

Given this characterization of Lubin-Tate theories, we prove \Cref{thm:modmain_red} by combining the following two propositions: 

\begin{dfn}\label{dfn:threemaps}
Define the following three maps in $\CAlgw_{E(k)}$:
\begin{enumerate}
    \item $\mdef{f}\colon  E(k[{\Np}]) \to E(k[{\Npm}])\times E(k) $ as in \Cref{thm:perf_nilp}.
    \item $\mdef{g}\colon E(k)\{z^0\} \to E(A)$ for $A = (\pi_0(E(k)\{z^0\})/I)^{\mathrm{perf}}$, which is to be defined in \Cref{dfn:the_map_g}.
    \item $\mdef{h}\colon E(k)\{z^1\} \xrightarrow{z^1 \mapsto 0} E(k)$. \tqed
\end{enumerate}
\end{dfn}

We have already shown that $f$ detects nilpotence in \Cref{thm:perf_nilp}.  In \Cref{subsec:insep}, we will define the map $g$ and show that it detects nilpotence.  Finally, the goal of \Cref{sub:odd-quo} is to show that $h$ detects nilpotence.   Given these inputs, we can finish the proof of the main theorem by applying the small object argument in a manner analogous to the rational case.

\begin{prop}\label{prop:Slift}
Let $S_0 = \{f, g, h\} \subset (\CAlgw_{E(k)})^{\Delta^1}$ and $R\in \CAlgw_{E(k)}$.  Then $S_0 \perp R$ if and only if $R$ is equivalent to $E(A)$ for some perfect $A$ of Krull dimension $0$.
\end{prop}
\begin{proof}
It suffices to show that $S_0 \perp R$ if and only if $R$ satisfies the conditions of  \Cref{prop:vnrE}.  Indeed, (1) is equivalent to $f\perp R$ by \Cref{cor:vnrlift}, (2) is equivalent to $g\perp R$ by \Cref{prop:insepnilp}, and (3) is equivalent to $h\perp R$ by \Cref{prop:oddnilp}.
\end{proof}

\begin{proof}[Proof of \Cref{thm:modmain_red}]
The maps in the set $S_0$ of \Cref{prop:Slift} detect nilpotence (\Cref{thm:perf_nilp}, \Cref{prop:insepnilp}, \Cref{prop:oddnilp}), so applying \Cref{cor:Small_object_DN}, we learn that any $R\in \CAlgw_{E(k)}$ admits a map of commutative algebras to some $R'$ such that $S_0 \perp R'$, and so we are done by  \Cref{prop:vnrE}.
\end{proof}

\subsection{Making the $E$-colocalization surjective on $\pi_0$}
\label{subsec:insep}\hfill

The goal of this subsection is to construct the map $g$ of \Cref{dfn:threemaps} such that $g\perp R$ if and only if the $E$-colocalization map
\[
E(R^{\flat}) \to R
\]
is surjective on $\pi_0$, and to show that $g$ detects nilpotence.  


\begin{ntn}
We let $k\{z^0 \} := \pi_0(E(k)\{z^0\})/\m$.  Note that by Strickland's theorem (\Cref{thm:strickland}\footnote{Note that $\T(E_0)/\m \cong k\{z^0\}$.}), this is a polynomial ring on infinitely many generators.
\tqed
\end{ntn}

Using the results of \Cref{sec:cofree} on the cofreeness of Lubin--Tate theory, we have the following identification:

\begin{lem}\label{lem:bij_free}
Let $A$ be a perfect $k$-algebra. Then the functor 
\[\left(\pi_0(-)/\m \right)^{\sharp}\colon\CAlgw_{E(k)} \to \mathrm{Perf}_{k}
\] induces an isomorphism
\[
\pi_0\left(\Map_{\CAlgw_{E(k)}}(E(k)\{z^0\}, E(A))\right) \to \Map_{\mathrm{Perf}_{k}}\left(k\{z^0\}^{\sharp}, A \right)
\]
\end{lem}
\begin{proof}

We factor the map 
as
\begin{align*}
\pi_0\left(\Map_{\CAlgw_{E(k)}}(E(k)\{z^0\}, E(A))\right) &\xrightarrow{\pi_0} \Map_{\T}(\pi_0(E(k)\{z^0\}), \pi_0(E(A))) \\
 &\xrightarrow{((-)/\m)^{\sharp}} \Map_{\mathrm{Perf}_{k}}(k\{z^0\}^{\sharp}, A).
\end{align*}
and show that both of these maps are isomorphisms.
For the second map, this follows from the cofreeness of Lubin--Tate theory from \Cref{thm:cofree}.
For the first map, we consider the following diagram
\[
\begin{tikzcd}[column sep=small]
    \pi_0\left(\Map_{\CAlgw_{E(k)}}(E(k)\{z^0\}, E(A))\right) \ar[rr] \ar[d, "\pi_0"] & &
    \pi_0\Omega^\infty E(A) \ar[d, "\cong"] \\
    \Map_{\T}(\pi_0(E(k)\{z^0\}), \pi_0(E(A))) \ar[r] &
    \Map_{\T}(\T(\pi_0(E(k))), \pi_0(E(A))) \ar[r] &
    \pi_0(E(A)), 
\end{tikzcd}
\]
where the top horizontal map is evaluation at $z^0$,
$\T(\pi_0(E(k)))$ is the free $\T$-algebra on the generator $z^0$,
the bottom left horizontal map is induced by precomposition with the map $\T(\pi_0E(k)) \to \pi_0(E(k)\{z^0\})$ and
the bottom right horizontal map is evaluation at $z^0$.
The evaluation at $z^0$ map is an isomorphism since $E(k)\{z^0\}$ is the free $T(n)$-local $E(k)$-algebra on the class $z^0$ (resp. since $\T(\pi_0E(k))$ is the free $\T$-algebra on the class $z^0$).
The bottom left map is an isomorphism since $\pi_0E(A)$ is $\m$-adically complete (\Cref{thm:E_functor}(3)) and 
$\T(\pi_0E(k))_\m^\wedge \cong \pi_0(E(k)\{z^0\})$ (see \cite[Proposition 4.17]{RezkCong}). 
\end{proof}



\begin{dfn}\label{dfn:the_map_g}
Define
\[
\mdef{g}\colon E(k)\{z^0\} \to E(k\{z^0\}^{\sharp})
\]
to be the map that corresponds to the identity of $k\{z^0\}^{\sharp}$ under the bijection of \Cref{lem:bij_free}. 
\tqed
\end{dfn}

\begin{prop}\label{prop:insepnilp}
We have:
\begin{enumerate}
    \item The functor 
    \[
    -\otimes_{E(k)\{z^0\}} E(k\{z^0\}^{\sharp})\colon \Modw_{E(k)\{z^0 \}} \to \Modw_{E(k)\{z^0\}}
    \]
    induced by $g$ is conservative.  In particular, by \Cref{lem:DN_conservative}, $g$ detects nilpotence.
    \item For $R\in \CAlgw_{E(k)}$, we have $g \perp R$ if and only if the map $\pi_0(E(R^{\flat})) \to \pi_0(R)$ is surjective.
\end{enumerate}
\end{prop}
\begin{proof}
For (1), because the functor 
\[ (-)/\m \colon \Modw_{E(k)}(\Sp_{T(n)}) \to \Modw_{E(k)/\m}(\Sp_{T(n)})\]
is conservative, it is enough to observe that the map $g$ has a retract after reducing modulo $\m$. Indeed, this follows immediately from the fact that  $k\{z^0\}$ is a polynomial algebra and thus the module $k\{z^0\}^{\sharp}$ is free as a $k\{z^0\}$-module. 

For (2), we have 
\begin{align*}
\pi_0 \Map_{\CAlgw_{E(k)}}(E(k\{z^0\}^{\sharp}), R) &\cong  \pi_0\Map_{\CAlgw_{E(k)}}(E(k\{z^0\}^{\sharp})  , E(R^{\flat})) \\
&\cong \pi_0\Map_{\mathrm{Perf}_{k}}(k\{z^0\}^{\sharp}, R^{\flat})\\
&\cong \pi_0 \Map_{\CAlgw_{E(k)}}(E(k)\{z^0\}, E(R^{\flat}))\\ 
&\cong \pi_0(E(R^{\flat}))
\end{align*}
The first two isomorphisms are \Cref{thm:E_functor}, the third is \Cref{lem:bij_free}, and the last is by the definition of $E(k)\{z^0 \}$.
\end{proof}

Recall that our goal is to prove that $f$, $g$ and $h$ (as defined in \Cref{dfn:threemaps}) detect nilpotence.  At this point, we have shown that $f$ and $g$ detect nilpotence.  Our proof that $h$ detects nilpotence, which we turn to in the next section, uses the following corollary of the fact that $f$ and $g$ detect nilpotence.

\begin{cor}\label{cor:ecoloc_iso}
Let $R\in \CAlgw_{E(k)}$.  Then there exists a nilpotence detecting commutative ring map $R\to R'$ such that the $E$-colocalization 
\[
E(R'^{\flat}) \to R'
\]
induces an isomorphism on $\pi_0$.
\end{cor}
\begin{proof}
The proof is a variant of the proof of \Cref{thm:modmain_red}.  Namely, we let $S_1 \subset {\CAlgw_{E(k)}}^{\Delta^1}$ be the set of maps
\begin{enumerate}
    \item $f\colon E(k[{\Np}]) \to E(k[{\Npm}])\times E(k) $ as in \Cref{thm:perf_nilp}.
    \item $g\colon E(k)\{z^0\} \to E(A)$ for $A = (\pi_0(E(k)\{z^0\})/I)^{\sharp}$, as in \Cref{dfn:the_map_g}.
\end{enumerate}
Then the maps in $S_1$ detect nilpotence by \Cref{thm:perf_nilp} and \Cref{prop:insepnilp}; hence, \Cref{cor:Small_object_DN} supplies a nilpotence detecting commutative ring map $R\to R'$ such that $S_1\perp R'$.  But \Cref{cor:vnrlift} and \Cref{prop:insepnilp} imply that $S_1\perp R'$ if and only if $R'$ satisfies conditions (1) and (2) of \Cref{prop:vnrE}.  Applying the first conclusion of \Cref{prop:vnrE}, we conclude that the $E$-colocalization of $R'$ induces an isomorphism on $\pi_0$.  
\end{proof}

\subsection{Quotienting by odd classes}\label{sub:odd-quo}\hfill

Now that we have handled $f$ and $g$, we will now turn to the map $h$, which handles the odd elements.  We show:

\begin{prop}\label{prop:oddnilp}
Let
\[
h \colon E(k)\{z^1\} \xrightarrow{z^1 \mapsto 0} E(k)
\]
as in \Cref{dfn:threemaps}.  Then 
\begin{enumerate}
    \item $h$ detects nilpotence.
    \item For $R\in \CAlgw_{E(k)}$, we have $h \perp R$ if and only if $\pi_1(R)=0$.
\end{enumerate}
\end{prop}

Part (2) is immediate, so the content of \Cref{prop:oddnilp} is in statement (1).  For any $R\in \CAlgw_{E(k)}$, an element $\alpha \in \pi_1(R)$ determines a map $E(k)\{z^1\} \to R$ by $z^1\mapsto \alpha$ and we denote by
\[
h_{\alpha}\colon R \to R\mm^{\infty} \alpha
\]
the commutative algebra map obtained by base-change along $h$.  We then deduce (1) from the following more general statement:

\begin{prop}\label{prop:oddnilp_phantom}
Let $R \in \CAlgw_{E(k)}$ such that $\pi_0(R)$ is reduced and $\alpha \in \pi_1(R)$.  Then the fiber of the map
\[
h_{\alpha}\colon R \to R\mm^{\infty} \alpha
\]
is $\otimes^2$-phantom.  In particular, by \Cref{lem:DN_phantom}, $h_{\alpha}$ detects nilpotence.  
\end{prop}

We first deduce \Cref{prop:oddnilp} from \Cref{prop:oddnilp_phantom}.

\begin{proof}[Proof of \Cref{prop:oddnilp} from \Cref{prop:oddnilp_phantom}]
As noted, the essential content is (1).  By \Cref{cor:ecoloc_iso}, there is a map $q\colon E(k)\{z^1\} \to \wt{R}$ such that $\wt{R}$ has the property that the $E$-colocalization map $E(\pi_0(\wt{R})^{\flat}) \to \wt{R}$ induces an isomorphism on $\pi_0$.  In particular, $\pi_0\wt{R}$ is reduced.  
Consider the commutative square
\[
\begin{tikzcd}
 E(k)\{z^1\} \arrow[r,"q"]\arrow[d,"h"] &  \wt{R}\arrow[d,"h_{q(z^1)}"] \\
 E(k) \arrow[r] & \wt{R} \mm^{\infty}q(z^1).
\end{tikzcd}
\]
The top row detects nilpotence by assumption, and the right vertical map detects nilpotence by \Cref{prop:oddnilp_phantom}.  Therefore, by \Cref{lem:DN_composition}(2), $h$ detects nilpotence as well.  
\end{proof}


The proof of \Cref{prop:oddnilp_phantom} occupies the rest of this section. In \Cref{subsec:even-modules}, we give a criterion for checking that a map $M\to R$ in $\Modw_{E(k)}$ is $\otimes^2$-phantom in terms of the existence of a filtration on $M$ whose quotients are odd suspensions of $R$.  Then, we construct such a filtration for the fiber of the map $h_{\alpha}$ in \Cref{subsec:oddfilt}, finishing the proof.  


\subsubsection{Nilpotence of odd maps}
\label{subsec:even-modules}\hfill

\begin{dfn}
Let \deff{$\NN^{\leq}$} denote the poset of nonnegative integers under the natural ordering, with $0$ the initial object, and let \deff{$\NN$} denote the nonnegative integers with trivial poset structure.  The categories $\NN^{\leq}$ and $\NN$ acquire symmetric monoidal structures under addition.

For a symmetric monoidal stable $\infty$-category $\cC$, let $\mdef{\Fil(\cC)} \coloneqq \Fun(\NN^{\leq}, \cC)$ and $\mdef{\Gr(\cC)} \coloneqq \Fun(\NN, \cC)$ be the symmetric monoidal categories of (non-negatively) filtered and graded objects in $\cC$ with the Day convolution.  
\end{dfn}

Then there is a diagram of symmetric monoidal functors
\[
\begin{tikzcd}
\cC \arrow[r, shift right=1 , swap, "\ct"]& \Fil(\cC) \arrow[r, "\gr"]\arrow[l, shift right=1, swap, "\colim"] &\Gr(\cC) \arrow[r,"\bigoplus"] & \cC,
\end{tikzcd}
\]
where the functor \deff{$\gr$} takes the associated graded and \deff{$\ct$} gives an object of $\cC$ the constant filtration.  The functor $\ct$ is then right adjoint to the functor $\colim$, which takes a filtered object to its colimit, and $\bigoplus$ is the functor which takes the sum of all the graded pieces.  Note that $\bigoplus$ and $\gr$ are conservative.    For $M\in \Fil(\cC)$ and $l\in \NN$, we denote by $M\langle l\rangle $ the filtered object such that $M\langle l\rangle_{\bullet} = M_{\bullet+l}$.  Note that there is a canonical map $\tau^l\colon M\langle l\rangle  \to M$ which induces the identity after applying $\colim$.  The reader is referred to \cite[\S 2]{LurRot} for additional background and proofs of these facts.

\begin{dfn}
Let $\cC$ be a symmetric monoidal stable $\infty$-category.  
For $M\in \Gr(\cC)$, we will say that $M$ is \deff{purely even} if for all $i\in \NN$, $M_i$ is a finite direct sum of even suspensions of $\one_{\cC}$.    In this case, the \deff{rank} of $M$, which we denote by $\mathrm{rank}(M)$, is the total number of summands in $\bigoplus_{i\in \NN} M_i$.  If $\Sigma M$ is purely even of rank $r$, then we say that $M$ is \deff{purely odd} of rank $r$.

For $M\in \Fil(\cC)$, we say $M$ is \deff{purely even} (resp. \deff{purely odd}) of rank $r$ if $\gr(M)$ is purely even (resp. purely odd) of rank $r$.  
\end{dfn}

An immediate corollary of these definitions and the fact that $\gr$ is a symmetric monoidal functor is:

\begin{lem}\label{lem:parity_tensor}
For $M,N\in \Fil(\cC)$, we have:
\begin{enumerate}
    \item If $M$ and $N$ are purely even (or both purely odd), then $M\otimes N$ is purely even.  
    \item\label{lem:parity_tensor:exact} If $F\colon \cC \to \DD$ is a unital exact functor, then $F$ sends purely even (resp. purely odd) objects in $\Fil(\cC)$   to purely even (resp. purely odd) filtered objects in $\Fil(\DD)$.  
\end{enumerate}
\end{lem}

In proving \Cref{prop:oddnilp_phantom} we will use the following criterion for producing $\otimes^2$-phantom maps.

\begin{lem}\label{lem:fib_2phant}
Let $R\in \CAlgw_{E(k)}$ such that $\pi_0(R)$ is reduced and let $M_{\bullet} \in \Fil(\Modw_R)$ be purely odd.  Then any map $\colim M_{\bullet} \to R$ is $\otimes^2$-phantom.
\end{lem}

The remainder of \Cref{subsec:even-modules} will be devoted to the proof of \Cref{lem:fib_2phant}, so fix $R$ as in its statement.  We start by considering the rational case.  Here, purely odd filtered objects have the following nilpotence property:

\begin{lem} \label{lem:rational-wedge-finite}
If $T$ is a commutative $\Q$-algebra and $M_{\bullet} \in \Fil(\Mod_T)$ is purely odd of rank $m$, then $(M_{\bullet}^{\otimes m+1})_{h\Sigma_{m+1}} = 0$.  
\end{lem}

\begin{proof}
Since the functor
\[
\bigoplus \circ \gr \colon \Fil(\Mod_T) \to \Mod_T.
\]
is symmetric monoidal, conservative, and commutes with colimits, it is enough to show that $(\bigoplus \gr (M_{\bullet}))^{\otimes m+1}_{h \Sigma_{m+1}} = 0$.  This amounts to showing that the sum of $m$ odd suspensions of $T$ is zero after applying $(-)^{\otimes m+1}_{h\Sigma_{m+1}}$.
  Since any sum of suspensions of $T$ is induced up from $\Q$, it suffices to observe that this is true in the case of $\Q$ where it is easy.
\end{proof}

\begin{lem} \label{lem:rational-odd-nilp}
  Let $T$ be a commutative $\Q$-algebra whose even homotopy groups are reduced.
  Suppose that $N_1 \xrightarrow{f} N_2 \xrightarrow{g} \ct(T)$ is a sequence of maps in $\Fil(\Mod_T)$ such that $N_1$ is purely even of rank 1 and $N_2$ is purely odd of finite rank.  
  Then the composite $g\circ f$ is null. 
\end{lem}

\begin{proof}
    Since the even homotopy groups of $T$ are reduced and $N_1$ is even of rank 1 (i.e., equivalent to some $\Sigma^{2j}\ct(R)\langle l\rangle$), it is enough to show that $g \circ f$ is nilpotent.
    For this we consider the following diagram 
    \[
    \begin{tikzcd}
        N_1^{\otimes m} \ar[r, "f^{\otimes m}"] \ar[d] & N_2^{\otimes m} \ar[r, "g^{\otimes m}"] \ar[d] & \ct(T)^{\otimes m} \ar[d] \\
        (N_1)^{\otimes m}_{h\Sigma_{m}} \ar[r] & (N_2)^{\otimes m}_{h\Sigma_{m}} \ar[r] & \ct(T)^{\otimes m}_{h\Sigma_{m}}
    \end{tikzcd}  
    \]
    where the outside vertical maps are equivalences since $N_1$ and $\ct(T)$ are purely even of rank 1 and $T$ is a $\Q$-algebra.
    To conclude we use \Cref{lem:rational-wedge-finite} and the assumption that $N_2$ is purely odd of finite rank to find that $(N_2)^{\otimes m}_{h\Sigma_{m}} = 0$ for $m \gg 0$.
\end{proof}

The following lemma allows us to bootstrap from the rational case to the $T(n)$-local case.

\begin{lem}\label{lem:May}
Let $R \in \CAlgw_{E(k)}$ such that $\pi_0(R)$ is reduced.  Then the map $R \to R \otimes \Q$ is injective on even homotopy groups.
\end{lem}
\begin{proof}
Since $R$ is an $E(k)$-algebra, it is $2$-periodic and therefore it is enough to prove the statement on $\pi_0$. Now by \cite{kuhn2004tate}, $\Sp_{T(n)}$ is $1$-semi-additive and thus considering $R$ as an object in $\CAlg(\Sp_{T(n)})$, we get by \cite[Corollary 4.3.5]{TeleAmbi} that every torsion element in $\pi_0(R)$ is nilpotent. Since $\pi_0(R)$ is reduced, we get that $\pi_0(R)$ is torsion-free and thus embeds in $\pi_0(R\otimes \QQ)$.\footnote{Alternatively one can use the May nilpotence conjecture, as proven in  \cite[Theorem B]{MNNmaynilp}.} 
\end{proof}

\begin{lem}\label{lem:rk1case}
Let $R \in \CAlgw_{E(k)}$ such that $\pi_0(R)$ is reduced.  Suppose that $N_1 \xrightarrow{f} N_2 \xrightarrow{g} \ct(R)$ is a sequence of maps in $\Fil(\Modw_R)$ such that $N_1$ is purely even of rank 1 and $N_2$ is purely odd of finite rank.  Then the composite $g\circ f$ is null.  
\end{lem}
\begin{proof}
Note that by adjunction, 
\[
\Hom_{\Fil(\Modw_R)}(N_1, \ct(R) ) = \Hom_{\Modw_R}( \colim N_1, R).
\]
Since $N_1$ is purely even of rank $1$, this means that there is an equivalence $N_1 \simeq \Sigma^{2j}\ct(R)\langle l\rangle$ for some $l\in \NN$ and $j\in \ZZ$, and so $g\circ f$ can be identified with a class in $\pi_{2j}(R)$ for some $j$.  Thus, by \Cref{lem:May}, it suffices to show that this element vanishes rationally.


Note that the functor
\[L_{\QQ} \colon \Modw_{R} \to \Mod_{R\otimes \QQ}\]
is exact and unital\footnote{Note that $L_{\QQ}$ does not preserve colimits and is also not symmetric monoidal, and thus is not a map in $\Prig$.},
which by \Cref{lem:parity_tensor}(\ref{lem:parity_tensor:exact}) means that the induced functor 
\[
L_{\QQ} \colon \Fil(\Modw_{R}) \to \Fil(\Mod_{R\otimes \QQ})
\]
preserves purely odd objects of finite rank.  By \Cref{lem:rational-odd-nilp}, this implies that the composite $g\circ f$ is null after applying $L_{\QQ}$, as desired. 
\end{proof}

We now generalize \Cref{lem:rk1case} to the case where $N_1$ is of arbitrary finite rank.

\begin{lem}\label{lem:oddcomp_null}
  Let $R \in \CAlgw_{E(k)}$ such that $\pi_0(R)$ is reduced.  Suppose $N_1 \xrightarrow{f} N_2 \xrightarrow{g} \ct(R)$ is a sequence of maps in $\Fil(\Modw_R)$ such that $N_1$ is purely even of finite rank and $N_2$ is purely odd of finite rank.  Then the composite $g\circ f$ is null.    
\end{lem}
\begin{proof}
We proceed by induction, starting with the case that $N_1$ is of rank $1$, which is \Cref{lem:rk1case}. 
For the general case, recall that the map $\tau^l \colon N_1\langle l\rangle \to N_1$ induces the identity on $\colim$.  Thus, by replacing $N_1$ by $N_1\langle l \rangle $ for $l\gg 0$, we may assume without loss of generality that if $m\in \NN$ is the minimal integer for which $(N_1)_m\neq 0$, then we have that  $\gr(N_2)_r = 0$ for $r\geq m$.  

Then we may choose a map $i\colon \Sigma^{2j}\ct(R)\langle m \rangle \to N_1$ in $\Fil(\Modw_R)$  which is the inclusion of a direct summand on $\gr(-)_m$.  Let $N_1'$ denote the cofiber of $i$.   Note that $N_1'$ is purely even with $\mathrm{rank}(N_1') =\mathrm{rank}(N_1)-1$.  Now consider the diagram
  \begin{center}
    \begin{tikzcd}
      \Sigma^{2j}\ct(R)\langle m \rangle \ar[r,"i"] \ar[d] & N_1 \ar[r ,"f"] \ar[d] & N_2 \ar[r, "g"] \ar[d] & \ct(R) \\
      0 \ar[r] & N_1' \pushout \ar[r,"f'"'] & N_2' \pushout \ar[ur, dashed, "g'"']
    \end{tikzcd}
  \end{center}  
  where $N_2'$ is defined as the displayed pushout.  Observe that because $\gr(N_2)_m=0$, the composite $f\circ i$ is null after applying $\gr$, and therefore  $\gr(N_2') = \gr(N_2) \oplus \gr(\Sigma^{2j+1}\ct(R)\langle m\rangle )$.  Since $N_2$ is purely odd, this implies that $N_2'$ is also purely odd and $\mathrm{rank}(N_2') =\mathrm{rank}(N_2)+1$.  
  
 Now, the composite along the top row is null by \Cref{lem:rk1case}, and therefore we obtain the dashed arrow $g'$.  Moreover, by the inductive hypothesis, $g'\circ f'$ is null, so $g\circ f$ is null as well. 
\end{proof}

We are now ready to prove \Cref{lem:fib_2phant}.

\begin{proof}[Proof of \Cref{lem:fib_2phant}]
First suppose that $M_{\bullet}$ is purely odd of finite rank.  Note that the given map $f_{\infty}\colon \colim M_{\bullet} \to R$ is adjoint to a map $f\colon M_{\bullet} \to \ct(R)$ in $\Fil(\Modw_R)$.  Then, the map $f_{\infty}^{\otimes 2}$ is the colimit of $f^{\otimes 2}$, which can be factored as a composite
\[
M_{\bullet}^{\otimes 2} \xrightarrow{f\otimes \mathrm{id}} M_{\bullet} \xrightarrow{f} \ct(R),
\]
which is necessarily null (and in particular phantom) by combining \Cref{lem:parity_tensor} and \Cref{lem:oddcomp_null}.  The case of a general purely odd $M_{\bullet}$ then follows by \Cref{lem:phantom_colim}.
\end{proof}

\subsubsection{Lifting $h_{\alpha}$ to filtered objects}\label{subsec:oddfilt}\hfill

In this section, we will deduce \Cref{prop:oddnilp_phantom} from \Cref{lem:fib_2phant}.  To do this, for any $\alpha \in \pi_1(R)$, we will construct a filtered commutative $R$-algebra $R_{\bullet}^{\alpha} \in \Fil(\Modw_R)$ such that: 
\begin{enumerate}
    \item There is an equivalence \[
  R\mm^{\infty} \alpha \simeq \colim R^{\alpha}_{\bullet}
  \] in $\CAlgw_R$.
    \item The fiber of the unit map $\ct(R) \to R^{\alpha}_{\bullet}$ is purely odd.
\end{enumerate}

\begin{cnstr}
  Given a $T(n)$-local commutative algebra $R$, let \deff{$\ct(R)\{ z^{1\langle 1\rangle} \}$} denote the free filtered commutative $R$-algebra with an element $z^{1\langle 1\rangle}$ in filtration $1$ and degree $1$.  Then, for a class $\alpha \in \pi_1 R$, the image of $\alpha$ in filtration $1$ of $\ct(R)$ determines the top map in:
  \[
  \begin{tikzcd}
  \ct(R)\{  z^{1\langle 1\rangle} \} \arrow[r,"\alpha \langle 1\rangle"]\arrow[d,"0"]& \ct(R) \\
  \ct(R) &
  \end{tikzcd}
  \]
  and we define \deff{$R^{\alpha}_{\bullet}$} as the pushout in $\CAlg(\Fil(\Modw_R))$.
\end{cnstr}

\begin{lem} \label{lem:colim-right}
  There is an equivalence 
  \[
  R\mm^{\infty} \alpha \simeq \colim R^{\alpha}_{\bullet}
  \]
  in $\CAlgw_R$.
\end{lem}

\begin{proof}
  This follows from the fact that $\colim \colon \Fil(\Modw_R) \to \Modw_R$ is a symmetric monoidal left adjoint.
\end{proof}

\begin{lem} \label{lem:gr-trick}
Let $\alpha \in \pi_1(R)$.  Then there is an equivalence 
  \[
  \gr ( R^{\alpha}_{\bullet}) \simeq \gr(\ct(R)\{\beta^{2\langle 1 \rangle}\})
  \]
  in $\CAlg(\Gr(\Modw_R))$.  
\end{lem}
\begin{proof}
  Since $\gr$ is symmetric monoidal and $\alpha \langle 1\rangle$ has trivial image in the associated graded, $\gr ( R^{\alpha}_{\bullet})$ is given by the pushout
  \[
  \begin{tikzcd}
  \gr(\ct(R)\{ z^{1\langle 1 \rangle} \}) \arrow[r,"0"]\arrow[d,"0"]& \gr (\ct(R)) \arrow[d]\\
  \gr(\ct(R))\arrow[r] & \bullet
  \end{tikzcd}
  \] which is just $\gr (\ct(R)\{\beta^{2\langle 1 \rangle} \})$ because $\gr$ preserves free algebras (note that $\gr (\ct(R)) = \gr (\ct(R))\{ 0\}$) and the free algebra functor $ \gr (\ct(R))\{ -\}$ preserves pushouts.  
\end{proof}

\begin{lem} \label{lem:even-pieces}
Suppose $\alpha \in \pi_1(R)$.  Then:
  \begin{enumerate}
      \item $R^{\alpha}_{\bullet}$ is purely even.
      \item The unit map $\gr (\ct(R) ) \to \gr (R^{\alpha}_{\bullet})$ induces an equivalence in grading $0$.  
  \end{enumerate}
  In particular, this implies that fiber of the unit map  $\ct(R) \to R_\bullet^\alpha$ is purely odd.
\end{lem}
\begin{proof}
By \Cref{lem:gr-trick}, we have
\[
\gr (R^{\alpha}_\bullet)_i \cong \gr ( R\{ \beta^{2\langle 1\rangle} \})_i \cong (\Sigma^{2}R)^{\otimes i}_{h\Sigma_i}.
\]
Since $R$ is an $E(k)$-algebra, this can be identified with 
\[
L_{T(n)} R[ B\Sigma_i] \simeq  R\otimes_{E(k)} L_{T(n)}E(k)[ B\Sigma_i],
\]
which is a finite rank free $R$-module by \cite{Strickland1}.  
(2) is clear by \Cref{lem:gr-trick}.
\end{proof}

\begin{proof}[Proof of \Cref{prop:oddnilp_phantom}]
The fiber of the the map
\[
u_{R^{\alpha}_{\bullet}}\colon \ct(R) \to R^{\alpha}_{\bullet}.
\]
is purely odd  by \Cref{lem:even-pieces}, so its colimit is $\otimes^2$-phantom by \Cref{lem:fib_2phant}.
On the other hand the colimit of $u_{R^{\alpha}_{\bullet}}$ is $h_\alpha$.
\end{proof}


\section{The Nullstellensatz}
\label{sec:null}

Let $L$ be an algebraically closed field.  Then Hilbert's Nullstellensatz asserts that for any \emph{finite} set $\{x_i\}_{i\in I}$ of generators and proper ideal $J\subsetneq
L[\{x_i\}_{i\in I}]$, there exists a $L$-algebra map
\[
L[\{x_i\}_{i\in I}]/J \to L,
\]
corresponding to a common root of the polynomials in $J$.  

In this section, we will prove a chromatic analogue of this statement. In fact we will consider a more general setting: note that the above statement does not necessarily hold if the set $I$ of generators is infinite.  Nevertheless, Lang observed that it does hold true if the cardinality of $L$ is sufficiently large.  More precisely:

\begin{thm}\label{thm:Lang}
The following are equivalent for a nonzero ring $B$ and an infinite cardinal $\alpha$:
\begin{enumerate}
    \item The ring $B$ is an algebraically closed field and $|B| \geq \alpha $.  
    \item For any set $I$ of size strictly less than $\alpha$ and any proper ideal $J\subsetneq B[\{x_i\}_{i\in I}]$, the $B$-algebra $B[\{x_i\}_{i\in I}]/J$ admits a $B$-algebra map to $B$.  
\end{enumerate}
\end{thm}
\begin{proof}
In \cite{Lang}, Lang proves that (1) and (2) are equivalent under the assumption that $B$ is an algebraically closed field. Thus, it will suffice for us to show that (2) implies that $B$ is an algebraically closed field.

Assume that $B$ satisfies (2).
We first claim that $B$ is  field.  Indeed, let $b\in B$ be  a non-invertible  element. Then take $I = \emptyset$ and consider the ideal $J = (b) \subset B$. Since $b$ is non-invertible, $J$ is a proper ideal so we get a retract for the map $B\to B/b$ and thus $b=0$.
To see that $B$ is an algebraically closed field, take $I$ to be a singleton and  let $f(x) \in B[x]$ be a non-constant polynomial. Now $J  = (f(x)) \subsetneq B[x]$ is a proper ideal, so we get a retract of the map $B \to B[x]/(f(x))$ and thus $f(x)$ admits a root in $B$.
\end{proof}

Accordingly, we define:

\begin{dfn}
We say that a presentable $\infty$-category $\CC$ is \deff{$\alpha$-Nullstellensatzian} for a regular cardinal $\alpha$ if every $\alpha$-compact and non-terminal object admits some map to the initial object. 
We say that an object $R$ in a presentable $\infty$-category $\mathcal{C}$ is $\alpha$-Nullstellensatzian if $R$ is non-terminal and $\mathcal{C}_{R/}$ is $\alpha$-Nullstellensatzian.
\tqed
\end{dfn}

Note that if $\beta \leq \alpha$ and $R$ is $\alpha$-Nullstellensatzian, then $R$ is $\beta$-Nullstellensatzian.  If $R$ is $\omega$-Nullstellensatzian, we will say that it is \deff{Nullstellensatzian}.

\Cref{thm:Lang} implies that an ordinary commutative ring $B$ is $\alpha$-Nullstellensatzian if and only if $B$ is an algebraically closed field of cardinality at least $\alpha$.  
Our goal in this section is to deduce from \Cref{cor:mod_algclosed} a full characterization of objects that are $\alpha$-Nullstellensatzian in $\CAlg(\Sp_{T(n)})$.

\begin{thm}
For any  $0 \neq R \in \CAlg(\Sp_{T(n)})$, $R$ is $\alpha$-Nullstellensatzian if and only if there exists some algebraically closed field $L$ and cardinality $|L| \geq \alpha$ such that $R = E(L)$.\footnote{Once again, recall that if $n=0$, then $L$ will be of characteristic $0$ and $E(L) = KU\otimes L$.}
\end{thm}

We start in \Cref{sub:polysplit} by phrasing a sense in which any Nullstellensatzian $R\in \CAlg(\Sp_{T(n)})$ admits roots of ``$T(n)$-local polynomial equations.''  Using \Cref{cor:mod_algclosed}, this is enough to imply that any such $R$ is a Lubin-Tate theory associated to an algebraically closed field.  We deduce this in \Cref{sub:nullchar}, and then add in cardinality considerations, including Lang's result, to deduce our full characterization of $\alpha$-Nullstellensatzian rings.  

Finally, in \Cref{sub:dblnull}, we consider the related question of whether \emph{dualizable} algebras over an algebraically closed Lubin-Tate theory $E(k)$ admit sections.  Using the characterization of $\alpha$-Nullstellensatzian algebras, we show that this holds at least when the field $k$ is uncountable.

\subsection{Splitting polynomials}\label{sub:polysplit}\hfill

Classically, a defining property of an algebraically closed field $L$ is that any polynomial $P \in L[x]$ which has a root in some extension of $L$ also has a root in $L$.  In this section, we phrase a version of this property which holds for a Nullstellensatzian algebra $A$ in any presentably symmetric monoidal stable $\infty$-category. 

\begin{prop}\label{prop:polyroots}
Let $\CC$ be a presentably symmetric monoidal stable $\infty$-category and let $R\in \CAlg(\CC)$ be Nullstellensatzian.  For $W \in \Mod_R(\CC)^{\omega}$ a compact $R$-module and $T \in \CAlg_R(\CC)$, we let $0:R\{W \} \to T$ denote the $R$-algebra map from the free commutative $R$-algebra on $W$ induced by the zero map $W\to T$.  

Then, for any $W_1, W_2 \in \Mod_R(\CC)^{\omega}$ and map $P\colon R\{W_1 \} \to R\{ W_2\}$ in $\CAlg_R(\CC)$, the following are equivalent:
\begin{enumerate}
    \item There exists a nonzero $T\in \CAlg_R(\CC)$ and a factorization in $\CAlg_R(\CC)$:
    \[
    \begin{tikzcd}
    R\{W_1 \}\arrow[r, "P"]\arrow[rd, "0"'] & R\{ W_2\} \arrow[d,dashed]\\
    & T.
    \end{tikzcd}
    \]
    \item There exists a factorization in $\CAlg_A(\CC)$:
    \[
    \begin{tikzcd}
    R\{W_1\}\arrow[r, "P"]\arrow[rd, "0"'] & R\{ W_2\} \arrow[d,dashed]\\
    & R.
    \end{tikzcd}
    \]
\end{enumerate}
\end{prop}

\begin{proof}
It suffices to show that (1) implies (2).  Consider the free forgetful adjunction
\[
R\{-\} \colon \Mod_R(\CC) \rightleftarrows \CAlg_R(\CC) \noloc U.
\]
Since $U$ preserves sifted and therefore filtered colimits (\cite[Corollary 3.2.3.2]{HA})), its left adjoint $R\{ - \}$ preserves compact objects.  Hence, since $W_1, W_2 \in \Mod_R(\CC)^{\omega}$, it follows that the algebras $R\{ W_1\}$ and $R\{W_2\}$ are compact in $\CAlg_R(\CC)$.  By assumption on $T$, there is a commutative diagram
\[
\begin{tikzcd}
R\{ W_1\} \arrow[r,"P"]\arrow[d, "0"] &  R\{ W_2\}\arrow[d,dashed] \\
R \arrow[r] & T.\\
\end{tikzcd}
\]
and thus a map from the pushout to $T$
\[
Q := R\otimes_{R\{ W_1\} } R\{ W_2\} \to T.
\]
Since $Q$ is a pushout of compact objects in $\CAlg_R(\CC)$, $Q$ is itself compact in $\CAlg_R(\CC)$.  Moreover, since $T \neq 0$, the existence of a ring map $Q\to T$ implies that $Q\neq 0$.  Thus, since $R$ is  $\omega$-Nullstellensatzian, there exists a retract $R \to Q \to R$ which gives the desired factorization.  
\end{proof}

One can think of this proposition as saying that $P$ ``has a root'' in some nonzero $R$-algebra $T$ if and only if it does in $R$.  Applied in the case that $P$ is a ``constant polynomial,'' we obtain the following:

\begin{cor}\label{cor:compactinj}
Let $\CC$ be a presentably symmetric monoidal stable $\infty$-category and let $R\in \CAlg(\CC)$ be Nullstellensatzian.  Then for any $0\neq T \in \CAlg_R(\CC)$ and any $W\in \Mod_{R}(\CC)^{\omega}$, the induced map
\[
\Map^{\Sp}_{\CC}(W, R) \to \Map^{\Sp}_{\CC}(W, T) \qind \Sp
\]
on mapping spectra is injective on $\pi_*$.  
\end{cor}
\begin{proof}
Suppose $x\in \ker(\pi_k \Map^{\Sp}_{\CC}(W,R ) \to \pi_k \Map^{\Sp}_{\CC}(W,T))$ and consider the map of commutative algebras
\[
x\colon R\{ \Sigma^k W\} \to R
\]
induced by $x$.  By the choice of $x$, the composite
\[
 R \{ \Sigma^k W\} \xrightarrow{x} R \to T
\]
is homotopic to $0\colon R \{ \Sigma^k W\}\to T$.  Thus, we may apply \Cref{prop:polyroots} with $P = x$ (where $W_1 = \Sigma^k W$ is compact because $W$ is, and $W_2 = 0$).  It follows that $0\colon  R\{ \Sigma^k W\} \to R$ is homotopic to $x\colon  R\{ \Sigma^k W\} \to R$, so $x$ was zero in $\pi_k \Map^{\Sp}_{\CC}(W,R)$. 
\end{proof}

\subsection{$T(n)$-local Nullstellensatzian rings}\label{sub:nullchar}\hfill

In this section, we will prove our main theorem, which characterizes $\alpha$-Nullstellensatzian algebras in $\Sp_{T(n)}$.  First, we show:

\begin{lem} \label{lem:ns-cxo}
  Let $0\neq R \in \CAlg(\Sp_{T(n)})$ be $\omega$-Nullstellensatzian.  
  Then $R$ is even and therefore complex orientable.
\end{lem}

\begin{proof}
By \Cref{cor:mod_algclosed}, there exists a map $R\to E(L)$ for some algebraically closed field $L$.  Let $V_n$ be a type $n$ generalized Smith-Toda complex (note that $V_0 = \Ss$). Since $L_{T(n)}V_n$ is compact in $\Sp_{T(n)}$ by \cite[other Refrence]{HovStrick}, then by \Cref{cor:compactinj}, the map
\[
V_n\otimes R \to V_n\otimes E(L)
\]
is injective on homotopy groups.  Since $V_n\otimes E(L)$ is even, this implies that $V_n \otimes R$ is even.  By \Cref{lem:strong-even}, we conclude.
\end{proof}

\begin{cnstr}\label{cnstr:E1alg_K}
Since $R$ is even and complex orientable, we may choose elements 
\[p, v_1, \cdots ,v_{n-1} \in \pi_*R.\]
By \Cref{cnstr:modm}, the $R$-module
\[
    \mdef{K} \coloneqq R\mm  (p,v_1,\cdots ,v_{n-1})
    \]
    acquires the structure of an $\E_1$-$R$-algebra.\footnote{In particular, at height $n=0$, the sequence of elements is empty and $K \cong R$.}  
    \tqed
\end{cnstr}

\begin{lem}\label{lem:K_compact}
  The $R$-module $K$ of \Cref{cnstr:E1alg_K} is compact.  
\end{lem}
\begin{proof}
Since $K$ is generated under finite colimits from $R$, it is dualizable and thus $K\otimes V_n \in (\Modw_R)^{\omega}$.  But $K$ is a retract of $K\otimes V_n$ and thus is also compact.
\end{proof}

\begin{prop}\label{prop:modm_null}
    Let $0\neq R \in \CAlg(\Sp_{T(n)})$ be $\omega$-Nullstellensatzian.
    Then, for the algebra $K$ of \Cref{cnstr:E1alg_K}, the graded ring $\pi_*K$ is even periodic and $\pi_0(K)$ is an algebraically closed field.
\end{prop}
\begin{proof}
 The proof of this proposition passes through the following more technical statements about $\pi_*(K)$.
 \begin{enumerate}
     \item The graded ring $\pi_*(K)$ is even and commutative.
     \item Suppose $t_1, \dots, t_l$ are polynomial variables of even degrees $2d_1, \cdots , 2d_l$ and $f\in \pi_*(K)[ t_1,\dots ,t_l]$ is a non-constant homogeneous polynomial.  Then there exist $x_1, \dots x_l \in \pi_*(K)$ with $x_i$ of degree $2d_i$ such that $f(x_1,\dots,x_l) = 0$. 
 \end{enumerate}
 
 By \Cref{cor:mod_algclosed}, there exists a map $R\to E(L)$ for some algebraically closed field $L$. Since $K$ is compact (\Cref{lem:K_compact}),
 we can use the dual of the compact $R$-module $K$ as $W$ in \Cref{cor:compactinj} and conclude that the map
\[
\iota \colon K \to K \otimes_R E(L) \simeq E(L)\mm\m 
\]
is injective on homotopy groups.  Since $\pi_*(E(L)\mm\m) = L{u}^{\pm 1}]$ is even and commutative, this implies that $\pi_*K$ is as well.

 To prove (2) we first show that $\pi_*(E(k)\mm\m) = L{u}^{\pm 1}]$ satisfies condition (2).  Indeed, given $f\in L{u}^{\pm 1}][ t_1,\cdots ,t_l]$ as above,
 by using ${u}$, we can reduce the case that $f \in L[ t_1,\cdots ,t_l]$ and $d_i =0$ for $i  \in 1,\dots,l$. In this case, the statement follows from the classical Hilbert's Nullstellensatz. 
 
 Now to prove that $\pi_{*}(K)$ satisfies condition (2), let $f \in \pi_*(K)[t_1, \cdots ,t_l]$ as above.  Choose $\tilde{f} \in \pi_*(K)\langle t_1,\cdots ,t_l\rangle $ a lift to non-commutative polynomials.  
Since $K$ is an $\E_1$-$R$-algebra, the functor 
\[
\pi_*(- \otimes_R K)\colon \CAlgw_R \to \Fun(\Z, \Set)
\]
factors through the category of graded (not necessarily commutative) rings with a map from $\pi_*(K)$.  It follows that $\tilde{f}$ determines a natural transformation
\[
\tilde{f}\colon \prod_{i=1}^{l} \pi_{2d_i}(-\otimes_R K) \to \pi_{2d}(-\otimes_R K)
\]
where $2d  = 2 \sum d_i$ is the degree of $f$.  This is corepresented by some map of free algebras
\[
P_f \colon R\{ \Sigma^{2d} K^{\dual} \} \to R\{  \bigoplus_{i=1}^{l}\Sigma^{2d_i}K^{\dual}\}.
\]

As we showed above, when $f$ is nonconstant, the image of $f$ under the inclusion $\iota$ has a solution in $\pi_*(E(L)\mm\m)$.  Unwinding the definitions, this means that condition (1) of \Cref{prop:polyroots} is satisfied with $\cC = \Modw_R$, $P= P_f$ and $T = E(L)$, and so \Cref{prop:polyroots} implies that $f$ has a solution in $\pi_*(K)$,

Finally, we are ready to use (1) and (2) to finish the proof.
Taking $f(t_1,t_2) = t_1t_2-1$ with $t_1$ in degree $2$ and $t_2$ in degree $-2$, we learn that $\pi_*(K)$ has a unit in degree $2$.   
To see that $\pi_0(K)$ is a field, we note that for any nonzero $a\in \pi_0(K)$, the polynomial $f(t) = at -1$ with $t$ in degree $0$ is non-constant and therefore has a solution in $\pi_0(K)$.  
Finally, taking non-constant polynomials  $f \in \pi_0(K)[x]$ for $x$ in  degree $0$, the condition directly implies that $\pi_0(K)$ is algebraically closed.  
\end{proof}

\begin{rmk}
The property (2) which appears in the proof of \Cref{prop:modm_null} can be interpreted as saying that $\pi_*(K)$ is Nullstellensatzian as a graded commutative ring.
\tqed
\end{rmk}

We have now shown:

\begin{prop}\label{prop:Ealgnull}
Let $0\neq R \in \CAlg(\Sp_{T(n)})$ be $\omega$-Nullstellensatzian.  Then, there exists an algebraically closed field $L$ and an equivalence $R \cong E(L)$.  
\end{prop}

\begin{proof}
This follows by combining \Cref{prop:modm_null} with \Cref{cor:alg-closed-E}.
\end{proof}

Combining this with \Cref{cor:mod_algclosed} and cardinality considerations, we deduce our characterization of $T(n)$-local $\alpha$-Nullstellensatzian rings.

\begin{thm}\label{thm:alpha-null}
Let $\alpha$ be a regular cardinal and $0 \neq R \in \CAlg(\Sp_{T(n)})$.  Then  $R$ is  $\alpha$-Nullstellensatzian if and only if there is an algebraically closed field  $L$ with $|L|\geq \alpha$ such that   $A \cong E(L)$.
\end{thm}

\begin{proof}
First assume that $R$ is $\alpha$-Nullstellensatzian.  Then $R$ is in particular $\omega$-Nullstellensatzian and so by \Cref{prop:Ealgnull}, $R \cong E(L)$ for some algebraically closed $L$.  We wish to show that $|L| \geq \alpha$. If $\alpha = \omega$ this is clear, so it is enough to consider the case $\alpha >\omega$.  Assume for the sake of contradiction that $|L|<\alpha$.
We claim that the perfect field $F \coloneqq L(t^{p^{-\infty}})$  ($F \coloneqq L(t)$ when $n=0$) is
$\alpha$-compact as an object of ${\CRing}_{L/}$.
Indeed, consider the poset 
\[Q = \{I \subset L | |I| < \omega \}\]
ordered by inclusion. For $I \in Q$, let $f_I  \coloneqq \prod_{i \in I} (t-i) \in L[t]$. Consider the functor from $Q$ to  $L$-algebras defined by
\[
H:I \mapsto L[t,f^{-1}_I].
\]
Since $|Q| = |L| <\alpha$ and  $H(I)$ is $\omega$-compact in  ${\CRing}_{L/}$ for all $I \in Q$,  we get by \cite[Proposition 5.3.4.13]{HTT} that $\colim_Q H(I)  = L(t)$ is $\alpha$-compact.
When $n>0$, since \[L(t^{p^{-\infty}}) = \colim [L(t) \to L(t^{1/p}) \to L(t^{1/p^2}) \to  \cdots],\] again using \cite[Proposition 5.3.4.13.]{HTT} we have that 
$L(t^{p^{-\infty}})$ is $\alpha$-compact as well.
 Thus by \Cref{thm:E_functor}(6), $E(F)$ is an $\alpha$-compact  $R \cong E(L)$ algebra. Since $R$ is $\alpha$-Nullstellensatzian, this means there is a retract 
$E(F) \to E(L)$, which is a contradiction.  We thus deduce that $|L| \geq \alpha$.

We now prove the converse: let $L$ be an algebraically closed field of cardinality $|L| \geq \alpha$.  We wish to show that $E(L)$ is $\alpha$-Nullstellensatzian.  Indeed, let $B$ be an $\alpha$-compact non-zero $E(L)$-algebra.  By \Cref{cor:mod_algclosed}, there exists some algebraically closed field  $F$ and a map $B \to E(F)$. 
Now consider the poset 

\[P = \{I \subset F | |I| < \alpha \}\]
ordered by inclusion and the functor from $P$ to perfect $L$-algebras defined by
\[
G:I \mapsto \mathrm{Image}(L[I] \to F)^{\sharp}.
\]
We have that $\colim_P G = F$ and thus by \Cref{thm:E-and-tilt},
\[\colim_{I \in P} E(G(I)) = E(F) \in  \CAlgw_{E(L)}. \]
Since $\alpha$ is regular, the poset $P$ is $\alpha$-filtered and thus since  $B$ is $\alpha$-compact, the map $B \to E(F) $ factors through a map 
\[B \to E(G(I))\] for some $I\in P$.
It is thus enough to show that  there is a retract for the map $L \to G(I)$.  Since $L$ is perfect, it is enough to show that the map 
\[L\to  \mathrm{Image}(L[I] \to F)\] has a retract. 
Denote by $J \lhd L[I]$ the kernel of the map $L[I] \to F$ so that $\mathrm{Image}(L[I] \to F) \cong L[I]/J$. Since $F\neq 0$, we have that $J$ is a proper ideal and so the claim follows from \Cref{thm:Lang}.  Note that the case of height $0$ is included in the above argument, by the convention that being perfect in characteristic $0$ is being reduced, and $(-)^{\sharp}$ means modding out by the nilradical.

\end{proof}

\subsection{Dualizable algebras}\label{sub:dblnull}\hfill

Let $L$ be an algebraically closed field.  A consequence of Hilbert's Nullstellensatz is that any non-zero finite dimensional $L$-algebra $A$ admits a $L$-algebra map $A\to L$.  One can ask whether a similar statement holds for Lubin-Tate theories.

\begin{qst}\label{qst:dual-null}
Let $L$ be an algebraically closed field and let $R \in \CAlgw_{E(L)}$ be a \emph{dualizable} $E(L)$-algebra.  Does there necessarily exist a section $R\to E(L)$?
\tqed
\end{qst}

We remark that dualizable $E(L)$-algebras are generally not compact in $\CAlgw_{E(L)}$ so this does not follow from \Cref{thm:alpha-null}.  Nevertheless, the question is of interest because there are many naturally occurring dualizable $E(L)$-algebras.  While we believe that the answer to \Cref{qst:dual-null} should be affirmative, we have not found a proof in general---however, we give here a proof in the special case that $L$ is  \emph{uncountable} (\Cref{thm:dual-null}).  The key observation is that while dualizable $E(L)$-algebras are not necessarily compact in $\CAlgw_{E(L)}$, they are always \emph{$\omega_1$-compact}.  To justify this, we need some preliminary lemmas.





  


\begin{lem} \label{lem:ladj-preserves-size}
  Let $F \colon \CC \rightleftharpoons \DD \noloc G$ be an adjunction between presentable categories
  and $\kappa \leq \alpha$ a pair of regular cardinals such that $\CC$ is $\kappa$-accessible.
  Then, the following are equivalent:
  \begin{enumerate}
  \item $G$ preserves $\alpha$-filtered colimits.
  \item $F$ preserves $\alpha$-compact objects.
  \item $F$ sends $\kappa$-compact objects to $\alpha$-compact objects.
  \item $F$ preserves $\beta$-compact objects for every regular $\beta \geq \alpha$.
  \end{enumerate}
\end{lem}

\begin{proof}
  This lemma is a simple enhancement of \cite[Prop. 5.5.7.2]{HTT} and
  the equivalence of (1) and (2) is immediate from this proposition.
  In fact, the proof given there demonstrates that (3) implies (1).
  Since it is clear that (4) implies (2) implies (3), it only remains to show that (1) implies (4).
  This follows since if $G$ preserves $\alpha$-filtered colimits then it preserves $\beta$-filtered colimits as well.  
\end{proof}

\begin{lem} \label{lem:monad-size} 
  Suppose that $F \colon \CC \rightleftharpoons \DD \colon G$ is a monadic adjunction between presentable categories such that $F$ and $GF$ preserve $\kappa$-compact objects for some uncountable regular cardinal $\kappa$.
  Then, $G$ detects $\kappa$-compactness.
\end{lem}

\begin{proof}
  Suppose we are given an object $d \in \DD$ such that $G(d)$ is $\kappa$-compact.
  Monadicity provides a presentation 
  \[ d \simeq \colim_{\Delta^{\op}} (FG)^{\circ \bullet} (d) \]
  of $d$ as a geometric realization.
  Our assumptions on $F$, $GF$ and $G(d)$ imply that this is a diagram of $\kappa$-compact objects.
  Then, since $\Delta^{\op}$ is $\omega_1$-small and $\kappa$ is uncountable \cite[Cor. 5.3.4.15]{HTT} implies that the colimits of this diagram is $\kappa$-compact.
\end{proof}

\begin{lem} \label{lem:alg-size}
  Let $\kappa$ be an uncountable regular cardinal and suppose that $\CC$ is a presentably symmetric monoidal category $\CC$ such that tensor products of $\kappa$-compact objects in $\CC$ are $\kappa$-compact.  Then for any operad $\cO$ such that each $\cO(n)$ is a $\kappa$-compact space, the adjunction
  \[ \mathrm{Free}_{\cO} \colon \CC \rightleftharpoons \Alg_\cO(\CC) \noloc \mathrm{U} \]
  has the following properties:  
  \begin{enumerate}
  \item the free $\cO$-algebra functor preserves $\kappa$-compact objects,
  \item the composite $\mathrm{Free}_\cO \circ \mathrm{U}$ preserves $\kappa$-compact objects and
  \item the underlying functor, $\mathrm{U}$, detects $\kappa$-compactness.
  \end{enumerate}  
\end{lem}

\begin{proof}
  The underlying object functor $\mathrm{U}$ preserves sifted colimits \cite[Proposition 3.2.3.1]{HA}, therefore by \Cref{lem:ladj-preserves-size} the free $\cO$-algebra functor preserves $\kappa$-compact objects.
  Evaluating $\mathrm{Free}_\cO \circ \mathrm{U}$ on a $\kappa$-compact object $X$ we have a presentation
  \[ \mathrm{U}(\mathrm{Free}_\cO(X)) \simeq \bigoplus_n ( \cO(n) \otimes X^{\otimes n} )_{h\Sigma_n}. \]
  Each term $\cO(n) \otimes (-)^{\otimes n}$ is $\kappa$-compact as a consequence of our assumptions on $\cO$ and the tensor product on $\CC$.
  Then, since $\coprod_n B\Sigma_n$ is $\omega_1$-small and $\kappa$ is uncountable \cite[Cor. 5.3.4.15]{HTT} implies that $\mathrm{U}(\mathrm{Free}_\cO(X))$ is $\kappa$-compact.
  Finally, \Cref{lem:monad-size} implies that (3) follows from (1) and (2).
\end{proof}

In order to make it easier to use \Cref{lem:alg-size} we prove one more lemma which lets us check the condition on the tensor product only for a single cardinality.
 
 \begin{lem}\label{lem:checkomega}
  Given a $\kappa$-compactly generated presentably symmetric monoidal category $\CC$,
  if tensor products of $\kappa$-compact objects in $\CC$ are $\kappa$-compact, then
  tensor products of $\alpha$-compact objects are $\alpha$-compact for every $\alpha > \kappa$.
\end{lem}

\begin{proof}
  Given a $\kappa$-compact object $c$, the functor $c \otimes -$ preserves colimits and $\kappa$-compact objects by hypothesis. In particular, it has an adjoint and so applying \Cref{lem:ladj-preserves-size}, we learn that $c \otimes -$ preserves $\alpha$-compact objects. Reversing things, this means that for each $\alpha$-compact object $a$ the functor $- \otimes a$ sends $\kappa$-compact objects to $\alpha$-compact objects. Applying \Cref{lem:ladj-preserves-size} again we may conclude.
\end{proof}

\begin{exm}\label{exm:Emodcpttensor}
The hypotheses for $\CC$ in \Cref{lem:alg-size} are satisfied for any $\CC \in \Prig$ and in particular, for $\cC = \Sp_{T(n)}$ and $\cC = \Modw_{E(L)}$.  Indeed, by \Cref{lem:checkomega}, it suffices to check that the tensor product of two compact objects is compact.  But this follows from the fact that if $c$ is compact and $d$ is dualizable, then the tensor product $c\otimes d$ is compact.  
\tqed
\end{exm}

\begin{cor}\label{thm:small-null}
Let $\kappa$ be an uncountable regular cardinal and let $E(L)$ be the Lubin-Tate spectrum associated to an  algebraically closed field $L$ with $|L| \geq \kappa$ .  Then for any $R\in \CAlgw_{E(L)}$ such that the underlying $E(L)$-module of $R$ is $\kappa$-compact, the unit map of $R$ admits a retract $R \to E(L)$.  
\end{cor}
\begin{proof}
This is immediate from \Cref{lem:alg-size}(3) and \Cref{thm:alpha-null}.
\end{proof}

\begin{lem}\label{cor:alg_cpt}
  Let $\kappa$ be a regular cardinal and $\cC \in \Prig$ such that $\one_{\cC}$ is $\kappa$-compact.  Then any dualizable object in $\cC$ is $\kappa$-compact.
\end{lem}
\begin{proof}
This follows from the fact that the tensor product of a dualizable object and a $\kappa$-compact object is $\kappa$-compact.  


\end{proof}

Finally, combining this with \Cref{thm:alpha-null}, we have:
  
\begin{thm}[Dualizable chromatic Nullstellensatz]\label{thm:dual-null}
Let $E(L)$ be the Lubin-Tate spectrum associated to an uncountable algebraically closed field $L$.  Then for any $R\in \CAlgw_{E(L)}$ such that the underlying $E(L)$-module of $R$ is dualizable, the unit map of $A$ admits a retract $R \to E(L)$.  
\end{thm}
\begin{proof}
Combining \Cref{cor:alg_cpt} and \Cref{thm:small-null}, it suffices to show that $E(L)$ is $\omega_1$-compact as an object of $\Modw_{E(L)}$.  To see this, recall first that generalized Moore spectra are $T(n)$-locally compact.   Then, writing $L_{T(n)}\Ss$ as an $\N$-indexed colimit of duals of generalized Moore spectra and tensoring this with $E(L)$, we get the conclusion.  
\end{proof}

\section{The spectrum of a $T(n)$-local commutative algebra}
\label{sec:tn-points}

In the same way that Hilbert's Nullstellensatz and the existence of enough points of the Zariski spectrum imply that commutative rings can be profitably studied through their geometry, the chromatic Nullstellensatz (\Cref{thm:alpha-null}) and the existence of enough chromatic geometric points (\Cref{thm:modmain}) together imply that $T(n)$-local commutative algebras should be studied through their geometry. 
The natural next step in developing such a theory of chromatic algebraic geometry is to collect the geometric points of $R$ together as the points of a topological space $\Spec(R)$ which regulates the geometry of $R$. 
In this section we construct a functor $\Spec_{T(n)}^\cons(-)$ to topological spaces which should be regarded as sending $R$ to its set of geometric points equipped with the \textbf{constructible} topology and explain how our main theorems are reflected in the basic properties of this functor.        

\begin{cnstr} \label{cnstr:con-tn}
  Let $\CC$ denote the category of products of algebraically closed fields equipped with a formal group of height $n$. 
  Using the fully faithful functor $E(-;-)$ of \Cref{thm:GHML}, 
  we define the constructible spectrum functor
\deff{  \[ \Spec_{T(n)}^\cons(-) \colon \CAlg(\Sp_{T(n)})^\op \to \mathrm{CHaus} \]}
  as the left Kan extension depicted below:
  \[ \begin{tikzcd}[sep=huge]
    \CC^\op \ar[dr, hook]\ar[rr, "{(A,\, \mathbb{H}_0)\  \mapsto\  \Spec_{\mathrm{Zar}}(A)}"] & {\ } \ar[d, Rightarrow, dashed] & \mathrm{CHaus} \\
    & \CAlg(\Sp_{T(n)})^\op \ar[ru, dashed, "\Spec_{T(n)}^\cons"']    
  \end{tikzcd}\]
  where we are using the fact that products of fields are reduced and Krull dimension $0$ to ensure that Zariski spectrum functor on $\CC$ lands in compact Hausdorff spaces.
  \tqed
\end{cnstr}


In \Cref{sec:geopoints}, motivated by \Cref{cnstr:con-tn}, we develop an abstract theory of constructible spectra.  Specializing the results of that appendix to $\CAlg(\Sp_{T(n)})$ and the functor $\Spec_{T(n)}^\cons(-)$, we obtain the following theorem.  

\begin{thm} \label{thm:con-tn}
    The constructible spectrum functor
    \[ \Spec^\cons_{T(n)}(-) \colon \CAlg(\Sp_{T(n)})^{\op} \to \mathrm{CHaus} \]
    of \Cref{cnstr:con-tn} enjoys the following properties:
    \begin{enumerate}
    \item $\Spec_{T(n)}^\cons(R)$ is empty if and only if $R=0$. 
    \item If $L$ is an algebraically closed field, then $\Spec^{\cons}_{T(n)}(E(L))$ is a point.
    \item For every point $q\in \Spec^{\cons}_{T(n)}(R)$, there exists an algebraically closed field $L$ and a map $R \to E(L)$ such that the image of the map
    \[\{*\} \cong \Spec^{\cons}_{T(n)}(E(L)) \to \Spec^{\cons}_{T(n)}(R)\]
    is $\{q\}$.
    \item A subset $U \subset \Spec^{\cons}_{T(n)}(R)$ is closed if and only if 
    there exists a map $R \to S$ such that $U$ is the image of the map 
    \[\Spec^{\cons}_{T(n)}(S) \to \Spec^{\cons}_{T(n)}(R).\]
    \item Given a span $S \leftarrow R \to T$ the natural comparison map
    \[ \Spec^{\cons}_{T(n)} \left( S \otimes_R T\right) \to \Spec^{\cons}_{T(n)}(S) \times_{\Spec^{\cons}_{T(n)}(R)} \Spec^{\cons}_{T(n)}(T) \]
    is surjective.
    \end{enumerate}
\end{thm}

\begin{proof}
  Using \Cref{thm:alpha-null} and \Cref{thm:GHML}(1,4) we get that for $n\geq 1$ we can identify the  category of products of Nullstellensatzian objects in $\CAlg(\Sp_{T(n)})$ with the category $\CC$ from \Cref{cnstr:con-tn}.
  In the case of $n=0$, we still get from  \Cref{thm:alpha-null} that $\cC$ is the \textbf{homotopy} category of products of Nullstellensatzian objects in $\CAlg(\Sp_{T(0)})$, and because the target is a $1$-category, this suffices for computing left Kan extensions.
  We checked that $\CAlg_{T(n)}$ is spectral (in the sense of \Cref{dfn:spectralcat}) in \Cref{exm:Tnspectral}, therefore the theorem now follows from \Cref{thm:spectralspec}.
\end{proof}


We have the following more explicit description for the underlying set of $\Spec_{T(n)}^{\cons}(-)$:

\begin{cor}\label{cor:geopoints}
Let $R \in \CAlg(\Sp_{T(n)})$, let $L_1,L_2$ be algebraically closed fields, and let $q_i \colon R \to E(L_i)$, $i=1,2$, be maps in $\CAlg(\Sp_{T(n)})$.  Then the following are equivalent:
\begin{enumerate}
    \item $q_1$ and $q_2$ represent the same point in $\Spec^{\cons}_{T(n)}(R)$.
    \item $E(L_1)\otimes_R E(L_2) \neq 0$.
    \item There exists an algebraically closed field $L_3$ and a commutative diagram \[\begin{tikzcd}
	R & {E(L_1)} \\
	{E(L_2)} & {E(L_3)}.
	\arrow["{q_1}", from=1-1, to=1-2]
	\arrow["{q_2}"', from=1-1, to=2-1]
	\arrow[from=1-2, to=2-2]
	\arrow[from=2-1, to=2-2]
\end{tikzcd}\] 
\item For a map $f\colon c \to a$ in $\Modw_R$ from a compact object $c \in \Modw_R$, $f$ is nilpotent at $E(L_1)$ if and only if it is nilpotent at $E(L_2)$.
\end{enumerate}
\end{cor}
\begin{proof}
First, (1), (2) and (3) are equivalent by \Cref{lem:point-equiv}.  
(3) clearly implies (4). 

To conclude we will prove that (4) implies (2). Indeed, let $V$ be a type $n$ generalized Moore spectrum and consider the map $f\colon R\otimes V^\dual \to E(L_1)$ which is the mate of the composition 
$R \to E(L_1) \to E(L_1)\otimes V.$ 
Assume (4) holds, but $E(L_1)\otimes_R E(L_2) =0$.
We get that $f$ is null at $E(L_2)$ and thus by (4), nilpotent at $E(L_1)$.   This implies the nilpotence  of the unit map
\[E(L_1) \to E(L_1)\otimes V, \]
which is a contradiction.
\end{proof}
In particular, each of the conditions (1)-(4) in \Cref{cor:geopoints} defines an equivalence relation on maps $R \to E(L)$ for $L$ algebraically closed whose equivalence classes give the points of the space $\Spec_{T(n)}^{\cons}(R).$ 

\Cref{cnstr:con-tn} turns out to be closely related to the notion  of nilpotence detecting maps in $\CAlg(\Sp_{T(n)})$. 


\begin{prop} \label{prop:cons-surj-dn}
  A map $R \to S$ of $T(n)$-local commutative algebras detects nilpotence if and only if the associated map on constructible spectra is surjective.
\end{prop}

\begin{proof}
  We begin with the forward direction.
  If $R \to S$ detects nilpotence, then it is nil-conservative by \Cref{lem:DN_nilconservative}. 
  \Cref{lem:surj} then implies that $R \to S$ induces a surjection on constructible spectra.
  
  First we observe that for any $R\in \CAlg(\Sp_{T(n)})$, there exists a nilpotence detecting map $R\to E(A)$ such $A$ is a product of algebraically closed extensions of $k$.  Indeed, by \Cref{thm:modmain}, we have a map $R\to E(A')$ for some $A'\in \Perf_k$ of Krull dimension $0$, and let $A$ be the product of the algebraic closures of the residue fields of $A'$.  By \Cref{thm:DN_Krull0}\footnote{Or \Cref{DN_Krull0QQ} in the case $n=0$.}, $E(A') \to E(A)$ detects nilpotence, and therefore by \Cref{lem:DN_composition}(1), so does $R\to E(A)$.  
  
  Now assume that $R \to S$ induces a surjection on constructible spectra.  By the above observation, we can choose nilpotence detecting maps $R \to E(A)$ and $E(A) \otimes_R S \to E(B)$ where $A$ and $B$ are products of algebraically closed fields.  
  We now consider the diagram
  \[ \begin{tikzcd}
    R \ar[r] \ar[d] & S \ar[d] \\
    E(A) \ar[r] & E(A) \otimes_R S \ar[r] & E(B). 
  \end{tikzcd} \]
  The map $E(A) \to E(A) \otimes_R S$ induces a surjection on the constructible spectrum by \Cref{thm:con-tn}(5).
  The map $E(A) \otimes_R S \to E(B)$ detects nilpotence and therefore induces a surjection on constructible spectra by the first part of the proof.
  Thus, $E(A) \to E(B)$ induces a surjection on constructible spectra.  But $\Spec(E(B)) \to \Spec^{\cons}_{T(n)}(E(A))$ can be identified with the induced map $\Spec_{\mathrm{Zar}}(B) \to \Spec_{\mathrm{Zar}}(A)$, and thus by \Cref{thm:DN_Krull0}\footnote{Or \Cref{DN_Krull0QQ} in the case $n=0$.}, the map $E(A)\to E(B)$ detects nilpotence.
  We may now conclude that $R \to S$ detects nilpotence by \Cref{lem:DN_composition}(2), since $R \to E(B)$  detects nilpotence.
  \qedhere
  
\end{proof}

\begin{rmk}
  \Cref{prop:cons-surj-dn} implies in particular that a map of $T(n)$-local commutative algebras detects nilpotence if and only if it is nil-conservative.
  \tqed
\end{rmk}

\begin{rmk}
\Cref{prop:cons-surj-dn} is quite special to $\Sp_{T(n)}$.  For example, from \Cref{exm:lnf_spec}, we see that for $n>0$ it is far from true for $\LnfSp$.
\tqed
\end{rmk}


\subsection{Algebraic approximations to the constructible spectrum}\label{subsec:algebraic_ approximations}\hfill

In this subsection we shall discuss some methods to compute $\CSpec$. 

\begin{cnv}
  We shall assume throughout this subsection that $n \geq 1$; similar results hold for $n=0$, and are in some sense easier to prove. However, the precise claims and proofs differ enough that we have decided to avoid treating the case $n=0$ here.
  \tqed
\end{cnv}

In the case that $R\in \CAlg(\Sp_{T(n)})$ happens to be an $E(k)$-algebra, there is a natural ``algebraic approximation'' to $\Spec^{\cons}_{T(n)}(R)$.

\begin{dfn}
  Let \deff{$\mathrm{NS}^{\Pi}_{\CAlgw_{E(k)}}$} (resp.  \deff{$\mathrm{NS}^{\Pi}_{\CRing})$} be the full subcategory of products of Nullstellensatzian objects in $\CAlgw_{E(k)}$ (resp. $\CRing$) as in \Cref{cnstr:con-tn}.
  \tqed
\end{dfn}

\begin{lem}\label{lem:NS_NS}
 The functor $E(-)$ restricts to an equivalence of categories
\[E(-)\colon \mathrm{NS}^{\Pi}_{{\CRing}_k} \to \mathrm{NS}^{\Pi}_{\CAlgw_{E(k)}}\]
with inverse given by $\pi_0(-)/\m$.
Furthermore, via this equivalence, the restriction of $\Spec^{\cons}_{\CAlgw_{E(k)}}$ to   $\mathrm{NS}^{\Pi}_{\CAlgw_{E(k)}}$ and the restriction of $\Spec^{\cons}_{{\CRing}_k}$ to   $\mathrm{NS}^{\Pi}_{{\CRing}_k}$ are isomorphic. 
    \end{lem}
\begin{proof}
The equivalence $\mathrm{NS}^{\Pi}_{{\CRing}_k} \to \mathrm{NS}^{\Pi}_{\CAlgw_{E(k)}}$ follows from \Cref{thm:Lang}, \Cref{thm:alpha-null} and \Cref{thm:E_functor}.
The equivalence of the restrictions of the $\Spec^{\cons}$ functor is now a consequence of \Cref{lem:reproduce_Spec}.
\end{proof}

\begin{cnstr}
  As in \Cref{dfn:spd}, we can assign to each $R \in \CAlgw_{E(k)}$ the functor
  \[
  \mdef{\Spd_{\CAlgw_{E(k)}}(R)} \colon \mathrm{NS}^{\Pi}_{{\CRing}_k} \to \mathrm{Set}
  \]
 given by $\pi_0\Map_{\CAlgw_{E(k)}}(R, E(-))$.
  Similarly, using \Cref{dfn:spd}, in commutative $k$-algebras, 
  we also have a functor 
  \[
  \Spd^{\mathrm{alg}}_k(R)\coloneqq\Spd_{{\CRing}_k}(\pi_0(R)/\m) \colon \mathrm{NS}^{\Pi}_{{\CRing}_k} \to \mathrm{Set}.
  \]
  \Cref{lem:NS_NS} then gives a natural transformation 
  \[\mdef{\Psi_{R}}\colon \Spd_{\CAlgw_{E(k)}}(R) \Rightarrow \Spd^{\mathrm{alg}}_k(R) \]
  which we call the algebraic approximation map.
  \tqed 

\end{cnstr}
 
\begin{lem}\label{lem:alg-approx}
Let $R \in \CAlgw_{E(k)}$ and assume that $\Psi_{R}$ is a natural isomorphism.  Then we have an isomorphism  
\[\Spec^{\cons}_{T(n)}(R) \cong \Spec^{\cons}_{{\CRing}}(\pi_0(R)/\m).\]
\end{lem}

\begin{proof}
First, by \Cref{lem:spec-under-same}, we can replace $\Spec^{\cons}_{T(n)}(R)$ with  $\Spec^{\cons}_{\CAlgw_{E(k)}}(R) $  and also $\Spec^{\cons}_{{\CRing}}(\pi_0(R)/\m)$ with $\Spec^{\cons}_{\mathrm{Ring_k}}(\pi_0(R)/\m)$.  Now, by 
\Cref{lem:NS_NS} and \Cref{prop:kan-top}, we are done.
\end{proof}

As an immediate consequence of \Cref{lem:alg-approx}, the following lemmas compute the constructible spectra of Lubin-Tate theories $E(A)$ and of free commutative algebras over $E(k)$.  

\begin{lem}\label{prop:spec-LT}
  For $A \in \Perf_k$, there is a natural isomorphism
  \[ \Spec^\cons_{T(n)}(E(A)) \cong \Spec^\cons_{{\CRing}}(A) \]
  between the constructible spectrum of $E(A)$ and the constructible Zariski spectrum of $A$.
\end{lem}
\begin{proof}
This follows immediately from \Cref{lem:alg-approx} and \Cref{thm:E_functor}.
\end{proof}

\begin{lem} \label{exm:free-algebra-spec}
There is a natural isomorphism
  \[ \Spec^\cons_{T(n)}(E(k)\{z^0\}) \cong \Spec^\cons_{{\CRing}}(E(k)\{z^0\}/\m) \cong \Spec^\cons_{{\CRing}}(k[z_0,z_1,\dots]). \]
\end{lem}
\begin{proof}
This follows by \Cref{lem:alg-approx}, since $\Psi_{E(k)\{z^0\}}$ is a natural isomorphism by \Cref{lem:bij_free}.
\end{proof}

As one can observe from the above discussion, working with the constructible spectrum of $E(k)$-algebras is especially convenient. The following lemma allows us to reduce the computation of
$\Spec^{\cons}_{T(n)}(R)$ for an arbitrary 
 $R \in \CAlg(\Sp_{T(n)})$ to the case of an $E(\overline{\F}_p)$-algebra.
 
 For this, first denote $\mathfrak{G}\coloneqq \mathrm{Aut}_{\CAlg(\Sp_{T(n)})}(E(\overline{\F}_p))$ to be the Morava stabilizer group.
 
 \begin{lem}\label{lem:spec_morava}
 Let  $R \in \CAlg(\Sp_{T(n)})$.  Then the natural map
 \[\Spec^{\cons}_{T(n)}(E(\overline{\F}_p)\otimes R)/\mathfrak{G} \to  \Spec^{\cons}_{T(n)}(R)\]
 is an isomorphism.  
     \end{lem}
 \begin{proof}

The map is defined since $R \to E(\overline{\F}_p)\otimes R$ is $\mathfrak{G}$-equivariant.  Since the source and the target are in $\CHaus$, it is enough to show that the map is bijective. Note that since $\CHaus$ is monadic over $\mathrm{Set}$ (for the ultrafilters monad \cite{linton1966some}), the forgetful functor $U\colon \CHaus \to \mathrm{Set}$ creates coequalizers of $U$-split pairs and thus the underlying set of the quotient is the  quotient of the underlying set.

Surjectivity follows from  \Cref{prop:cons-surj-dn}, \Cref{lem:tnil-po} and the fact that $\one_{T(n)} \to E(\overline{\F}_p)$ detects nilpotence \cite{DHS}. 
We now show injectivity. Let
 $q_1,q_2 \in \Spec^{\cons}_{T(n)}(E(\overline{\F}_p)\otimes R)$ be points that map to the same point $q_0\in \Spec^{\cons}_{T(n)}(R)$. 
 By \Cref{thm:con-tn}(5), there is some $q_3 \in \Spec^{\cons}_{T(n)}(E(\overline{\F}_p)\otimes E(\overline{\F}_p)\otimes R)$ that maps to $q_1$ and $q_2$ in $\Spec^{\cons}_{T(n)}(E(\overline{\F}_p)\otimes \one_{T(n)} \otimes R)$ 
 and $\Spec^{\cons}_{T(n)}(\one_{T(n)} \otimes E(\overline{\F}_p)\otimes R)$, respectively. 
 
 Let $q_4$ be the image of $q_3$ in 
 $\Spec^{\cons}_{T(n)}(E(\overline{\F}_p)\otimes E(\overline{\F}_p))$ and let $A = \overline{\F}_p^{\mathfrak{G}}$ denote the perfect algebra of continuous maps  $\mathfrak{G} \to \overline{\F}_p$.  By \cite[Theorem 4.11]{hovey2004operations2} and \cref{thm:E_functor}, we have that $ E(\overline{\F}_p)\otimes E(\overline{\F}_p)c\cong E(A)$ and thus, by \Cref{prop:spec-LT}, 
\[\Spec^{\cons}_{T(n)}(E(\overline{\F}_p) \otimes E(\overline{\F}_p)) \cong \Spec^\cons_{{\CRing}}(A) \cong \mathfrak{G} \]
with each geometric point represented by a projection map   $E(A) \to E(\overline{\F}_p)$.
By acting with $\mathfrak{G}$, we can assume without loss of generality that $q_4$ is the unit element of $\mathfrak{G}$. 

Writing 
\[
E(\overline{\F}_p)\otimes E(\overline{\F}_p)\otimes R = (E(\overline{\F}_p)\otimes E(\overline{\F}_p))\otimes_{E(\Fpbar)} (E(\Fpbar)\otimes R),
\]
we know that $q_3$ can be represented by some commuting square
 \[\xymatrix{
 E(\overline{\F}_p) \ar[d]^{u_{E(\Fpbar)}\otimes\mathrm{id}}\ar[r] & E(\overline{\F}_p)  \otimes R\ar[d] \\
 E(\overline{\F}_p)\otimes E(\overline{\F}_p) \ar[r]^{q_4} & E(L) .
 }\]
By the assumption that $q_4$ corresponds to the unit element of $\mathfrak{G}$, the map $E(\overline{\F}_p)\otimes E(\overline{\F}_p) \xrightarrow{q_4}  E(L) $ factors through the product map 
$E(\overline{\F}_p)\otimes E(\overline{\F}_p) \to E(\overline{\F}_p)$ (possibly after extending $L$ again, cf. \Cref{cor:geopoints}(3)), which is a retract to $f$. Thus, $q_3$ is in the image of the map
\[ 
\Spec^{\cons}_{T(n)}( E(\overline{\F}_p)\otimes R) \to \Spec^{\cons}_{T(n)}(E(\overline{\F}_p)\otimes E(\overline{\F}_p)\otimes R)
\]
obtained from the product map $E(\overline{\F}_p) \otimes E(\overline{\F}_p) \to E(\overline{\F}_p)$. Taking the appropriate projections, we get that $q_1=q_2$.
 \end{proof}


\section{Orientations}
\label{sec:orient}
The final two sections of this paper are devoted to applications of the chromatic Nullstellensatz. In this section, we investigate the orientability properties of algebraically closed Lubin--Tate theories. 
The core result of this section is the following special property of the category $\Modw_{E(L)}$ for an algebraically closed field $L$.

\begin{thm}\label{thm:pic-split-main} 
Let $L$ be an algebraically closed field and 
$0 \neq R \in \CAlgw_{E(L)}$.  Then the map 
\[\pic(\Modw_{E(L)}) \to \pic(\Modw_R)\]
admits a retract.
\end{thm}

We prove \Cref{thm:pic-split-main} in Sections \ref{sub:orient-prelim} and \ref{sub:Hahnprops}.
The remaining sections are then devoted to consequences of this theorem.
In \Cref{sub:thom} we provide a simple criterion for when an algebraically closed Lubin--Tate theory, $E(L)$, admits an orientation by a Thom spectrum. 
Building on this, in \Cref{sub:stric-pic}, we determine the spectrum of strict units of $E(L)$.
Finally, in \Cref{sub:cons-sp-thom}, we connect the results of this section with \Cref{sec:tn-points} and compute the constructible spectrum of the flat affine line over $E(L)$.




\subsection{Units and biCartesian squares}\label{sub:orient-prelim}\hfill

Here, we give some basic technical results about the functor 
\[
\gl1 E(-) \colon \mathrm{Perf}_{k} \to \Sp_{\geq 0}
\]
which are needed for the proof of \Cref{thm:pic-split-main}.  First, note that there is a natural transformation 
\[\alpha\colon \gl1 E(-) \Rightarrow (-)^{\times}\]
given by reducing $\pi_0 (\gl1 E(-))$ modulo $\m$. Here, $(-)^{\times}$ is considered as a functor 
$ (-)^{\times} \colon \mathrm{Perf}_{k} \to \Sp_{\geq 0} $  via the embedding $\mathrm{Ab} \subset \Sp_{\geq 0}$.  We define the functor
\deff{\[\sl1 \colon \mathrm{Perf}_{\mathbb{F}_p}  \to \Sp_{\geq 0}\]}
as the fiber of $\alpha$.

\begin{lem}\label{lem:sl1-pull-push}
Let 
\[\xymatrix{
A \ar[r]\ar[d] & B\ar[d] \\
C\ar[r] & D
}\]
be a pullback diagram in $\Perf_k$ and assume that 
     the map of abelian groups $B\oplus C \to D$ is surjective. Then 
    
\[\xymatrix{
\sl1 A \ar[r]\ar[d] & \sl1 B\ar[d] \\
\sl1 C\ar[r] & \sl1 D
}\]
is a pushout diagram in $\Sp$.
\end{lem}
\begin{proof}
First, by \Cref{thm:E_functor}(5), we have a pullback diagram 
\[\xymatrix{
E(A) \ar[r]\ar[d] & E(B)\ar[d] \\
E(C) \ar[r] & E(D).
}\]
Now, since the functor $\gl1 \colon \CAlg \to \Sp_{\geq 0}$ preserves limits
 (as it admits a left adjoint, given by $\Ss[-]$),  we get a pullback diagram 
 \[\xymatrix{
\gl1 E(A) \ar[r]\ar[d] & \gl1  E(B)\ar[d] \\
\gl1 E(C) \ar[r] & \gl1  E(D)
}\]
in $\Sp_{\geq 0}$. Furthermore, since $(-)^{\times}$ preserves limits,
\[\xymatrix{
A^{\times} \ar[r]\ar[d] & B^{\times}\ar[d] \\
C^{\times}\ar[r] & D^{\times}
}\]
is also a pullback diagram in $\Sp_{\geq 0}$.  Consequently, by taking fibers, the diagram 
\[\xymatrix{
\sl1 A \ar[r]\ar[d] & \sl1 B\ar[d] \\
\sl1 C\ar[r] & \sl1 D
}\]
is also a pullback in $\Sp_{\geq 0}$.  Thus, to show that it is additionally a pullback in $\Sp$, it is enough to show that the map $\pi_0(\sl1 B) \oplus \pi_0(\sl1 C) \to \pi_0(\sl1 D) $ is surjective.
Indeed, for any perfect $k$-algebra $H$, $\pi_0(\sl1 H)$ admits a complete filtration induced from the powers of the ideal  $\m$ such that
\[
\gr_r \pi_0(\sl1 H) \cong \gr_r \pi_0 (\sl1 k) \otimes_k H \cong H^{\genfrac(){0pt}{2}{r+1}{n-1}}.
\]
  We thus get that the  surjectivity of $\pi_0(\sl1 B) \oplus \pi_0(\sl1 C) \to \pi_0(\sl1 D) $ follows from the surjectivity of the map $B\oplus C \to D$.
\end{proof}


\begin{cor}\label{cor:gl1-ret}
Let \[\xymatrix{
A \ar[r]\ar[d] & B\ar[d] \\
C\ar[r] & D
}\]
be a pullback diagram in $\Perf_k$.
Assume that 
\begin{enumerate}
    \item The map of abelian groups $B\oplus C \to D$ is surjective.
    \item The inclusion of abelian groups
    \[ (B^{\times} \oplus C^{\times})/ A^{\times} \to D^{\times}  \] admits a retract. 
    \item The maps $A \to B$ and $A \to C$ admits retracts in $\Perf_k$.
\end{enumerate}
Then the map 
\[ \gl1 E(A) \to \gl1 E(D) \] admits a retract.
\end{cor}
\begin{proof}
Since $\gl1 E(-)$ is a functor, from (3) we get a retract for the map 
\[
\gl1 E(A) \to  \gl1 E(B) \coprod_{\gl1 E(A)} \gl1 E(C) .
\]
It is thus enough to show that the assembly map 
\[
\alpha\colon \gl1 E(B) \coprod_{\gl1 E(A)} \gl1 E(C) \to \gl1 (D)
\]
admits a retract. 
By (1) and \Cref{lem:sl1-pull-push}, we have a pushout diagram in $\Sp$
\[\xymatrix{
\gl1 E(B) \coprod_{\gl1 E(A)} \gl1 E(C)  \ar[r]\ar[d]^{\alpha} & B^{\times} \coprod_{A^{\times}} C^{\times}\ar[d] \\
\gl1 E(D) \ar[r] & D^{\times}.
}\]
Thus, since by (2)  the map $(B^{\times} \oplus C^{\times})/ A^{\times} \cong B^{\times} \coprod_{A^{\times}} C^{\times} \to D^{\times}$ splits, the same is true for the map $\alpha$.
\end{proof}\

\subsection{Hahn series}
\label{sub:Hahnprops}\hfill

In order to construct the retraction in \Cref{thm:pic-split-main}, we will apply \Cref{cor:gl1-ret} to 
a certain generalization of Laurent series rings known as \emph{Hahn series}.  

\begin{dfn}
Let $\Gamma$ be a totally ordered monoid and $A$ a non-zero domain.
We define the \deff{ring ${\HRing{A}{\Gamma}}$ of Hahn series}
 to be the ring whose elements are functions $f\colon \Gamma \to A$ such that the subset $\{ x \in \Gamma | f(x)\neq 0\} \subset \Gamma$ is well-ordered under the ordering on $\Gamma$.
  Define the addition point-wise and multiplication by the convolution formula
  \[ f\cdot g(x) = \sum_{y+z = x} f(y)g(z). \]
  We will think of the elements of $A[\![t^{\Gamma}]\!]$ as formal power series in $t$ with exponents in $\Gamma$ which are supported on a well-ordered subset of $\Gamma$; here, we think of a function $f:\Gamma \to A$ as the series $\sum_{\gamma\in \Gamma} f(\gamma)t^\gamma$. 
  The multiplicative monoid of non-zero elements ${\HRing{A}{\Gamma}}^{*}$ admits two monoid maps
 \begin{align*}
     \val\colon {\HRing{A}{\Gamma}}^{*} &\to \Gamma & \arc \colon  {\HRing{A}{\Gamma}}^{*} &\to A^*. \\
     f&\mapsto \min_{f(\gamma) \neq 0} \gamma & f &\mapsto f(\val(f))
 \end{align*}
  Note that $\val$ admits a section $t^{(-)}\colon \Gamma \to\HRing{A}{\Gamma}^{*}$  sending $\gamma \in \Gamma$ to $t^{\gamma} \in \HRing{A}{\Gamma}^*$ and $\arc $ admits a section $\iota\colon A^* \to \HRing{A}{\Gamma}^*$
  sending $a\in A$ to the ``constant power series.''
\end{dfn}

The ring of Hahn series was introduced by Hahn in 1908 and
further studied (and generalized) by Krull, Mal'cev, and Neumann, among others.

\begin{dfn}
Let $\Gamma$ be a totally ordered monoid. 
\begin{enumerate}
    \item We denote by \deff{$\Gamma^{\times}$} the group of invertible elements in $\Gamma$.
    \item We denote by \deff{$\Gamma_{\leq 0}$} (resp \deff{$\Gamma_{\geq 0}$}) the submonoid of elements $\gamma\in \Gamma$ with $\gamma\leq 0$ (resp. $\gamma\geq 0$). \tqed
\end{enumerate}
\end{dfn}

The Hahn series rings have the following properties.  

\begin{prop}\label{prop:Hahn-props}
Let $A$ be a non-zero domain and $\Gamma$ a totally ordered monoid. 
\begin{enumerate}
    \item A series $f\in \HRing{A}{\Gamma}$ is in  $\HRing{A}{\Gamma}^{\times}$ if and only if
    $f\neq 0$,  $\val(f) \in \Gamma^{\times}$   and $\arc(f) \in A^{\times}$.
    \item If $A$ is a field and $\Gamma$ is a group, then $\HRing{A}{\Gamma}$ is a field.
    \item  The map
    \[
    t^{(-)} \oplus \mathrm{incl} \colon \Gamma^{\times} \oplus \HRing{A}{\Gamma_{\geq 0}}^{\times} \to \HRing{A}{\Gamma}^{\times}
    \]
    is an isomorphism.  
    \item The inclusion $A^{\times} \to \HRing{A}{\Gamma_{\leq 0}}^{\times}$ is an isomorphism.  
    \item If $\Gamma \xrightarrow{ p} \Gamma$ is an isomorphism and  $A \in \mathrm{Perf}$ is perfect, then $\HRing{A}{\Gamma}$ is perfect.
    \item If $\Gamma$ is a $\Q$-module and $L$ is an algebraically closed field, then $\HRing{L}{\Gamma}$ is an algebraically closed field.
\end{enumerate}
\end{prop}
\begin{proof}
Since $\val$ and $\arc$ are monoid maps, one direction of (1) is obvious. In the other direction,
note that since both $\val$ and  $\arc$ admit sections, it is enough to show that the kernel of 
$\val \times \arc$ consists of invertible objects, which is clear by a standard ``Hensel's lemma'' argument. (2) and (3) are immediate from (1), (4) is a special case of (3),
(5) is clear, and (6) is Theorem 1 from \cite{MacLane}.
\end{proof}

\begin{lem}\label{lem:big-hahn-field2}
Let $L$ be an algebraically closed field and $F$ a non-zero $L$-algebra. Let $\Gamma$ be a totally ordered $\Q$-module such that $|\Gamma| > |F|$. Then there exists  a map of $L$-algebras $F \to \HRing{L}{\Gamma}$.
\end{lem}
\begin{proof}
By taking the quotient by a maximal ideal, we may assume that $F$ is a field. Furthermore, since $L$ is infinite, $F$ is as well and so $|\bar{F}| = |F|$ and we may assume that $F$ is an algebraically closed field.
Now by \Cref{prop:Hahn-props}(6), for each totally ordered $\Q$-module $\Gamma$, we have that $\HRing{L}{\Gamma}$ is an algebraically closed field. Since algebraically closed fields over $F$ are classified by their transcendence degree, we are done as $|\HRing{L}{\Gamma}| \geq |\Gamma| > |F|$.
\end{proof}

\begin{lem}\label{lem:hahn-split}
Let $A\in \mathrm{Perf}_k$ be a non-zero perfect domain and $\Gamma$ a totally ordered monoid with $\Gamma \xrightarrow{p} \Gamma$ an isomorphism. Then the map
\[\gl1 E(A) \to \gl1 E(\HRing{A}{\Gamma})\] admits a retract.
\end{lem}
\begin{proof}
We shall apply \Cref{cor:gl1-ret} to the pullback diagram
\[\xymatrix{
A\ar[r]\ar[d] & \HRing{A}{\Gamma_{\leq 0}}\ar[d]\\
\HRing{A}{\Gamma_{\geq 0}}\ar[r] & \HRing{A}{\Gamma}.
}\]
For this, we verify conditions (1), (2) and (3) of 
\Cref{cor:gl1-ret}.
(1) is clear. (3) is obtained by sending $\sum_{\gamma\in \Gamma_{\geq 0}} f(\gamma) t^{\gamma}$ (resp. $\sum_{\gamma\in \Gamma_{\leq 0}} f(\gamma) t^{\gamma}$) to $f(0)$; note that the fact that these retractions are ring maps uses that restricting to $\Gamma_{\geq 0}$ (resp. $\Gamma_{\leq 0}$) allows only nonnegative (resp. nonpositive) exponents.
For (2), note that \Cref{prop:Hahn-props}(3,4) implies that the inclusion
\[\left( \HRing{A}{\Gamma_{\leq 0}}^{\times} \oplus \HRing{A}{\Gamma_{\geq 0}}^{\times} \right)/ A^{\times} \to \HRing{A}{\Gamma}^{\times}\] splits with complement $\Gamma^{\times}$.
\end{proof}

\begin{rmk}
In \Cref{lem:hahn-split}, as always in this paper, when $n=0$, the field $k$ is assumed to be of characteristic $0$ and we interpret $A\in \mathrm{Perf}_k$ to mean that $A$ is any $k$-algebra. In this case, the condition on $\Gamma$ that $\Gamma \xrightarrow{p} \Gamma$ be an isomorphism is not needed and the proof works verbatim without it.
\tqed
\end{rmk}


We are now ready to prove \Cref{thm:pic-split-main}.

\begin{proof}[Proof of \Cref{thm:pic-split-main}]
By \Cref{cor:mod_algclosed}, there is some algebraically closed field $L\to F$ with a map  
$R \to E(F)$.  Thus, we may assume that $R = E(F)$.
By the upward L{\"o}wenheim--Skolem Theorem\footnote{One can also construct such a totally ordered $\QQ$-module explicitly by choosing a well-ordering on the underlying set of $F$ and taking $\Gamma = \QQ^{F}$ with the lexicographic order.}, there exists a totally ordered $\QQ$-module 
$\Gamma$ with $|\Gamma| > |F|$ and thus by \Cref{lem:big-hahn-field2}, a map of $L$-algebras $F \to \HRing{L}{\Gamma}$. 
So we are reduced to the case that $R = E(\HRing{L}{\Gamma})$. Now, consider the square
\[\xymatrix{ 
\Sigma \gl1 E(L) \ar[r]\ar[d] & \ar[d]\Sigma\gl1 E(\HRing{L}{\Gamma})\\
\pic{E(L)}\ar[r] & \pic(E(\HRing{L}{\Gamma})).
}\]
The induced map on the cofibers of the vertical maps is an isomorphism between two groups of order $2$ by \cite{baker2005invertible}. Thus, the square above is a pushout square in $\Sp$  and to get the desired retract, it is enough to give a retract of the map 
\[\gl1 E(L) \to \gl1 E(\HRing{L}{\Gamma}), \]
which exists by \Cref{lem:hahn-split}.

\end{proof}

\subsection{Thom Spectra}\label{sub:thom}

Using \Cref{thm:pic-split-main} we are able to study the question of when an algebraically closed Lubin--Tate theory admits an orientation by a Thom spectrum quite effectively.

\begin{dfn}
Let $\cC$ be a presentably symmetric monoidal  category and let $X \xrightarrow{f} \pic(\cC)$ be a map in $\Sp_{\geq 0}$. Following \cite{ABGHR2, BAC}, we define the \deff{Thom spectrum} of $f$ to be \[\mdef{Mf} \coloneqq \mathrm{colim}_X f \in \CAlg(\cC),\] where the colimit is taken in $\cC$.
\tqed
\end{dfn}

\begin{prop}\label{prop:Thom} 
Let $\cC$ be a presentably  symmetric monoidal  category and let $X \xrightarrow{f} \pic(\cC)$ be a map in $\Sp_{\geq 0}$. Then the following are equivalent:
\begin{enumerate}
    \item $f$ is null-homotopic.
    \item There is a map $Mf \to \one_{\cC}$ in $\CAlg(\cC)$.
    \item The map \[\pic(\cC) \to \pic(\Mod_{Mf}(\cC))\] admits a retract.
\end{enumerate}
\end{prop}
\begin{proof}
We see that (1) implies (2) and (2) implies (3), so it suffices to show that (3) implies (1). It is enough to show that the composition 
\[X \xrightarrow{f} \pic(\cC) \to \pic(\Mod_{Mf}(\cC))\]
is null. By the universal property of Thom spectra as in \cite[Lemma 3.15]{BAC}, this can be identified with maps $Mf \to Mf$ of commutative algebras in $\cC$\footnote{To see this from the reference, note that the space of lifts of \cite[Definition 3.14]{BAC} can be identified with the space of $\E_{\infty}$-null-homotopies of the composite $X\to \mathrm{Pic}(R) \to \mathrm{Pic}(A)$ by the pullback square in the proof of \cite[Proposition 3.16]{BAC}.}.  Hence, the statement follows from the existence of the map $\mathrm{id}\colon Mf \to Mf$.
\end{proof}

\begin{cor}\label{cor:main-E-picard}
Let $L$ be an algebraically closed field and let 
$X \xrightarrow{f} \pic(\Modw_{E(L)})$ be a map in $\Sp_{\geq 0}$.  Then the following are equivalent:
\begin{enumerate}
    \item $Mf \neq 0$.
    \item There is map $Mf \to E(L)$ in $\CAlgw_{E(L)}$.
    \item $f$ is null-homotopic.
    \item $Mf \cong E(L)[X] \in \CAlgw_{E(L)}$.
\end{enumerate}
\end{cor}
\begin{proof}
By \Cref{thm:pic-split-main}, (1) implies that the map ${\pic(\Modw_{E(L)})} \to \pic(\Modw_{Mf})$ has a retract and thus we get (2) by \Cref{prop:Thom}.  (2) implies (3) as in \Cref{prop:Thom}, and the directions (3) implies (4) and (4) implies (1) are clear.
\end{proof}

\begin{cor}\label{cor:Thom_sp}
  Let $L$ be an algebraically closed field and let $X \xrightarrow{f} \pic(\Sp)$ be a map in $\Sp_{\geq 0}$. Then the following are equivalent:
  \begin{enumerate}
      \item $K(n)\otimes Mf \neq 0$.
      \item There exists a map $Mf \to E(L)$ in $\CAlg(\Sp)$.
      \item The composition of $f$ with the map 
      \[\pic(\Sp) \to \pic(\Modw_{E(L)}) \] is null-homotopic.
      \item There exists an equivalence  $L_{K(n)}(Mf \otimes E(L))\cong E(L)[X] \in \CAlgw_{E(L)}$.
  \end{enumerate}
 In these cases, we have 
 \[\Map_{\CAlg(\Sp)}(Mf,E(L)) \cong \Map_{\Sp_{\geq 0}}(X,\gl1 E(L)) .\]
\end{cor}
\begin{proof}
Letting $\tilde{f}$ denote the composition of $f$ with the map  \[\pic(\Sp) \to \pic(\Modw_{E(L)}),\]
we have that $M\tilde{f} = L_{K(n)}(Mf \otimes E(L)) \in \CAlgw_{E(L)}$. Thus, conditions (1)-(4) above are equivalent to conditions (1)-(4) in \Cref{cor:main-E-picard} for $\tilde{f}$. The final statement follows immediately from (4).
\end{proof}

\subsection{The strict Picard spectrum of $E(L)$}
\label{sub:stric-pic}\hfill

Next, we use \Cref{cor:Thom_sp} to compute the strict Picard spectrum of $\Modw_{E(L)}$.
We begin by handling the torsion part of this strict Picard spectrum.

\begin{prop} \label{prop:Ek-Tor2}
  Let $L$ be an algebraically closed field and let $H$ be a $p$-torsion abelian group\footnote{When the height $n=0$, the condition on $H$ should be replaced with $H$ being torsion.}.
  There is a natural equivalence of connective spectra
  \[ \Hom_{\Sp_{\geq 0}}(H, \pic(\Modw_{E(L)})) \simeq \Sigma^{n+1}H^* \] where $H^*$ stands for Pontryagin dual of $H$. 
 \end{prop}
 \begin{proof}
 The category of  $p$-torsion abelian  groups is equivalent to $\mathrm{Ind}(\mathrm{Ab}^{\mathrm{fin}}_p)$, where $\mathrm{Ab}^{\mathrm{fin}}_p$ is the category of $p$-finite abelian groups. From this fact and the Mittag-Leffler condition, the statement is  reduced to the case where $H$ is finite.
 
  We begin by showing that
  $\pi_m \Map_{\Sp}(H, {\pic(\Modw_{E(L)}})) \cong 0 $ for $m \neq n+1$.
  Suppose we are given a class \[f \in \pi_m \Map_{\Sp}(H, \pic(\Modw_{E(L)})) = \pi_0(\Map_{\Sp}(\Sigma^{m}H, {\pic(\Modw_{E(L)}})))\] with $0 \leq m \leq n$.
  \Cref{cor:main-E-picard} implies that in order to show that $f$ is nullhomotopic, it will suffice to
  show that $Mf \neq 0$. 
 Since $\Omega^{\infty}\Sigma^{m}H\cong B^m H$ is a $\pi$-finite space, we know from \cite[Theorem 0.0.2]{HL} that the underlying object $Mf \in \Modw_{E(L)}$   satisfies 
  \[
    Mf = \colim_{B^m H} f \cong \lim_{B^m H} f.
  \]
  Finally, by \cite[Corollary 5.4.4]{HL} for  $m\leq n$ and \cite[Corollary 5.4.5(2)]{HL}  for $m \geq n+2$, we have  that    
  \[
    \lim_X\colon \Fun(B^m H, \Modw_{E(L)}) \to \Modw_{E(L)}
  \]
  is conservative, and so $Mf \neq 0 $. 

We are now reduced to identifying  the functor \[
G\colon (\mathrm{Ab}^{\mathrm{fin}}_p)^{\op} \to \mathrm{Ab}
\] 
\[
H \mapsto [ \Sigma^{n+1} H, \pic(\Modw_{E(L)}) ]
\]
with the functor $H \mapsto H^{*}$.
Note that both functors are additive, so it is enough to identify them after composing with the forgetful functor $\mathrm{Ab} \to \mathrm{Set}$.
Now 

\begin{align*}
[ \Sigma^{n+1} H, \pic(\Modw_{E(L)}) ]
 &\cong [ \Sigma^{n} H, \gl1(E(L)) ]
 \cong [E(L)[B^nH], E(L)]_{{\CAlgw_{E(L)}}} \\
 &\cong [E(L)^{H^*}, E(L)]_{{\CAlgw_{E(L)}}} 
 \cong H^*
\end{align*}
where, the second isomorphism uses the group ring--$\gl1$ adjunction, the third isomorphism is by \cite[Theorem 5.3.26]{HL}, and the fourth isomorphism follows from the fact that  $E(L)^{H^*}$ is an \'{e}tale $E(L)$-algebra and thus  the space of maps
  \[\Map_{\CAlgw_{E(L)}}(E(L)^{H^*}, E(L)) \]
  can be identified with the set of $\pi_0E(L)$-algebra maps from $\pi_0(E(L)^{H^*}) = \pi_0(E(L))^{H^*}$ to $\pi_0E(L)$.
 \end{proof}

The case $H= C_p$ of \Cref{prop:Ek-Tor2} is used in \cite{fourier} and \cite{HL} to deduce the following results:


\begin{prop}[{\cite{fourier}}]\label{prop:discrep}
 Let $L$ be an algebraically closed field and let $F$ be the fiber of the map 
 \[\gl1 E(L) \to L_n^f\gl1 E(L).\] Then there is an equivalence 
 $\tau_{\geq 0} F \cong \tau_{\geq 0}\Sigma^{n} I_{\QQ_p/\ZZ_p}$\footnote{In the case $n=0$, one can take an arbitrary prime and set $L_0^f = L_{\mathbb{S}[1/p]}$, or alternatively, take $L_0^f = L_{\QQ}$ and replace $I_{\QQ_p/\ZZ_p}$ with   $I_{\QQ/\ZZ}$.}.
 \end{prop}
 
 \begin{prop}[{\cite[Corollary 5.4.10]{HL}}]
 The functor 
 \[(\mathcal{S}_p^{\leq n})^{\op} \to \CAlgw_{E(L)}\]
 \[X \mapsto  E(L)^X\]
 from $n$-truncated $p$-finite spaces is fully faithful. 
 \end{prop}

 From \Cref{prop:Ek-Tor2}, we get a description of the so-called ``strict Picard spectrum'' of $\Modw_{E(L)}$.
 \begin{thm}\label{thm:strict-pic}
   Let $L$ be an algebraically closed field of characteristic $p$ and assume that the height $n\geq 1$\footnote{The case $n=0$ is very different and in it, $\Hom_{\Sp_{\geq 0}}(\ZZ, \pic(\Modw_{E(L)}))$ is not truncated. In fact, it is a straightforward calculation that  $\Hom_{\Sp_{\geq 0}}(\ZZ, \pic(\Modw_{E(L)})) \cong \Sigma \gl1(E(L)) $  in this case. }.  Then we have an equivalence of connective $\ZZ$-modules
  \[\Hom_{\Sp_{\geq 0}}(\ZZ, \pic(\Modw_{E(L)})) \simeq \Sigma^{n+2}\ZZ_p \oplus \Sigma L^\times.\]
 \end{thm}
 \begin{proof}
 First by the same argument as in the beginning of the proof of \Cref{prop:Ek-Tor2}, we get that $\pi_0( \Hom_{\Sp_{\geq 0}}(\ZZ, \pic(\Modw_{E(L)}))) \cong 0$.
 We thus are reduced to proving that 
 \[\Hom_{\Sp_{\geq 0}}(\Sigma \ZZ, \pic(\Modw_{E(L)})) \simeq \Sigma^{n+1}\ZZ_p \oplus L^\times.\]
 Using the pushout of spectra
 \[
 \xymatrix{
 \ZZ[1/p]\ar[d] \ar[r]& \QQ_{p}/\ZZ_p\ar[d] \\
 0 \ar[r]&\Sigma \ZZ,
 }
 \]
 we get obtain a pullback diagram 
 \[
 \xymatrix{
 \Hom_{\Sp_{\geq 0}}(\Sigma \ZZ, \pic(\Modw_{E(L)})) \ar[d] \ar[r]& \Hom_{\Sp_{\geq 0}}(\QQ_{p}/\ZZ_p, \pic(\Modw_{E(L)})) \ar[d]^{g} \\
 0 \ar[r]&\Hom_{\Sp_{\geq 0}}( \ZZ[1/p], \pic(\Modw_{E(L)})).
 }
 \]
 Now by \Cref{prop:Ek-Tor2}, $\Hom_{\Sp_{\geq 0}}(\QQ_{p}/\ZZ_p, \pic(\Modw_{E(L)})) \cong \Sigma^{n+1}\ZZ_p$. Since $\pi_0( \Sigma^{n+1}\ZZ_p)=0$, we can replace the term $\Hom_{\Sp_{\geq 0}}( \ZZ[1/p], \pic(\Modw_{E(L)}))$  in the pullback diagram above with 
 \[
 \tau_{\geq 1} \Hom_{\Sp_{\geq 0}}( \ZZ[1/p], \pic(\Modw_{E(L)})) \cong 
 \Sigma\Hom_{\Sp_{\geq 0}}( \ZZ[1/p], \gl1(E(L))).
 \]
 Now recall the fiber sequence $\sl1L \to \gl1(E(L)) \to L^{\times}$ (where $\sl1L$ is defined as the fiber).
 Since $\sl1L$ has $p$-complete torsion-free homotopy groups, $\sl1L$ is a $p$-complete spectrum and $\Hom_{\Sp}( \ZZ[1/p], \sl1L) =0$.  It follows that 
 \[
 \Hom_{\Sp_{\geq 0}}( \ZZ[1/p], \gl1(E(L))) \cong \Hom_{\Sp_{\geq 0}}( \ZZ[1/p], L^{\times }) \cong L^{\times}.
\]
Finally, from connectivity considerations, the map denoted by $g\colon \Sigma^{n+1}\ZZ_p \to \Sigma L^{\times}$ is null so we are done.
 \end{proof}
 \begin{rmk}\label{rmk:good_times}
Following the proof above, one gets that the resulting map  $\Sigma^{n+2} \ZZ_p \to \pic(\Modw_{E(L)}) $ comes from the ``height $n$'' primitive roots of unity in $E(L)$ in the sense of \cite{carmeli2021chromatic} and that the map $\Sigma L^{\times} \to \pic(\Modw_{E(L)}) $ 
comes from the inclusion of the multiplicative lifts of $L$ in $\pi_0(E(L))$.
In particular, let $\mathfrak{G}_L \cong \mathrm{Aut}_{\CAlg(\Sp_{T(n)})}(E(L)) \cong \mathcal{O}^{\times} \rtimes \mathrm{Gal}(L/\mathrm{\F_p})$ be the $L$-extended Morava stabilizer group. Then the resulting action of $\mathfrak{G}_L$ on $\Sigma^{n+2}\ZZ_p  \oplus \Sigma L^{\times}$  is via the determinant map $\mathcal{O}^\times \to \ZZ_p^{\times} = \mathrm{Aut}(\ZZ_p)$
on the first summand and by $\mathrm{Gal}(L/\mathrm{\F_p})$ on $L^{\times}$.
\tqed
\end{rmk}
 
\subsection{The constructible spectrum and Thom spectra}
\label{sub:cons-sp-thom}\hfill

We conclude the section by using the results we have proved up to this point to analyze the constructible spectrum of the flat affine line.

\begin{prop}\label{prop:break_spec}
  Let $R \in \CAlg(\Sp_{T(n)})$. Then the map 
  \[R[t] \to R[t^{\pm 1}]\times R\]
  from \Cref{dfn:A1_Gm} induces an isomorphism  on constructible spectra.
\end{prop}

\begin{proof}
Since $\Spec^\cons_{T(n)}$ lands in compact Hausdorff spaces, it is enough to check that the resulting map is a bijection, which will follow from the following 4 claims:
\begin{enumerate}
    \item The induced map $\Spec_{T(n)}^{\cons} (R) \to \Spec_{T(n)}^{\cons} (R[t])$ is injective, because $R[t] \to R$ admits a section.    
    \item The induced map $\Spec_{T(n)}^{\cons} ( R[t^{\pm 1}]) \to \Spec_{T(n)}^{\cons} (R[t])$ is injective by \Cref{cor:idempotent}, since $R[t^{\pm 1}]$ is an idempotent algebra over $R[t]$.
    \item The images of the maps from (1) and (2) are disjoint because 
\[R\otimes_{R[t]}R[t^{\pm 1}] = 0 .\] 
    \item The images of the maps from (1) and (2) are jointly surjective, which follows from \Cref{prop:cons-surj-dn} and \Cref{lem:strict_nilp}.
\end{enumerate}
\qedhere


\end{proof}

\begin{prop}\label{prop:spec_affine_E}
  Assume that  $n\geq 1$ and let $A\in \Perf_{k}$.  Then for $r,s \in \NN$, the algebraic approximation map induces an isomorphism
  \[ \Spec^{\cons}_{T(n)}(E(A)[\ZZ^{r}\times \NN^{s}]) \cong \Spec^{\cons}_{{\CRing}}(A[\ZZ^r \times \NN^{s} ]).\]
\end{prop}

\begin{proof}
First, by \Cref{prop:break_spec}, we are reduced to the case that $s=0$.
For this, by \Cref{lem:alg-approx}, it is enough to show that $\Psi_{E(A)[\Z^{r}]}$ (in the notation of the lemma) is a natural isomorphism. Since we have 
\[
E(k)[\Z^{r}]\otimes_{E(k)} E(A) \cong E(A)[\Z^{r}],
\]
we are reduced to the case $A =k$ by \Cref{lem:alg-approx}.

Now for $B = \prod_{i\in I}L_i$ a product of  algebraically closed fields under $k$, we get 
\begin{align*}
    \pi_0(\Map_{\CAlgw_{E(k)}}(E(k)[\Z^{r}],E(B))) &\cong  \pi_0(\Map_{\Sp_{\geq 0}}(\Z^{r},\gl1 E(B))) \\
    &\cong \pi_1(\Map_{\Sp_{\geq 0}}(\Z,\pic( \Modw_{E(B)})))^r \\
     &\cong  (B^\times)^{r} \cong \Map_{{\CRing}_{k}}(k[\Z^r],B)
\end{align*} 
where the third bijection is  \Cref{thm:strict-pic} and the naturality of the fourth bijection follows from \Cref{rmk:good_times}.
\end{proof}

\begin{prop}\label{prop:spec_affine_T}
 Assume that  $n\geq 1$, for $r,s \in \NN$, there is an isomorphism
  \[ \Spec^{\cons}_{T(n)}({\one}_{T(n)}[\ZZ^{r}\times \NN^{s}]) \cong \Spec^{\cons}_{{\CRing}}(\F_p[\ZZ^r \times \NN^{s} ])\]
 
\end{prop}
\begin{proof}
Let $R\coloneqq {\one}_{T(n)}[\ZZ^{r}\times \NN^{s}]$. By \Cref{prop:spec_affine_E}, we have 
\[\Spec^{\cons}_{T(n)}(R\otimes E(\overline{\F}_p)) \cong \Spec^{\cons}_{{\CRing}}(\overline{\F}_p [\ZZ^{r}\times \NN^{s}])\]  By \Cref{rmk:good_times}, the action of the Morava stabilizer group $\mathfrak{G}$ on $\Spec^{\cons}_{{\CRing}}(\overline{\F}_p [\ZZ^{r}\times \NN^{s}]) $ factors through the projection to $\mathrm{Gal}(\F_p)$.
The result now follows from \Cref{lem:spec_morava}.
\end{proof}

\begin{rmk}
Propositions \ref{prop:spec_affine_E} and \ref{prop:spec_affine_T} restrict to the case $n\geq 1$ because they employ the algebraic approximation results of \Cref{subsec:algebraic_ approximations}, which we developed only for $n\neq 0$. The interested reader can verify similar statements for $n=0$ as well.
\tqed
\end{rmk}

\begin{prop}\label{prop:thomimage}
  Let $R \in \CAlg(\Sp_{T(n)})$ and let $f\colon X \to \pic(\Modw_R)$ be a map in $\Sp_{\geq 0}$. 
  Then a geometric point $x\colon R \to E(L)$ of $\Spec_{T(n)}^{\cons}(R)$ is in the image of the map 
  \[ \Spec_{T(n)}^{\cons}(Mf) \to \Spec_{T(n)}^{\cons}(R)\]
  if and only if 
  the composite
  \[X \xrightarrow{f} \pic(\Modw_R) \xrightarrow{\pic(x)} \pic(\Modw_{E(L)})\]
  is null.
\end{prop}
\begin{proof}
By definition, $x$ is in the image of \[ \Spec_{T(n)}^{\cons}(Mf) \to \Spec_{T(n)}^{\cons}(R)\] if and only if $Mf\otimes_R E(L) \neq 0$. But, $Mf\otimes_R E(L)= Mg$ for $g = \pic(x)\circ f$.  Thus, we may conclude by 
\Cref{cor:main-E-picard}, which asserts that $Mg \neq 0$ if and only if $g$ is null.
\end{proof}

\begin{cor}\label{cor:thompoints}
 Let $R \in \CAlg(\Sp_{T(n)})$ and let $f\colon X \to \pic(\Modw_R)$ be a map in $\Sp_{\geq 0}$. 
 \begin{enumerate}
     \item $Mf \neq 0$ if and only if there exists $x\colon R \to E(L)$  in $\Spec_{T(n)}^{\cons}(R)$ such that the composite
  \[X \xrightarrow{f} \pic(\Modw_R) \xrightarrow{\pic(x)} \pic(\Modw_{E(L)})\]
  is null.
  \item The map $R \to Mf$ detects nilpotence if and only if, for every $x:R \to E(L)$  in $\Spec_{T(n)}^{\cons}(R)$, the composite
  \[X \xrightarrow{f} \pic(\Modw_R) \xrightarrow{\pic(x)} \pic(\Modw_{E(L)})\]
  is null.
 \end{enumerate}
\end{cor}

\begin{rmk}
Note that, as a consequence of \Cref{thm:pic-split-main} and \Cref{cor:geopoints}, the conditions of \Cref{prop:thomimage} and \Cref{cor:thompoints} do not depend on the choice of Nullstellensatzian representative $x$ for a given geometric point.
\tqed
\end{rmk}

\section{Chromatic support}
\label{sec:chromsupport}
Let $R$ be a $p$-local ring spectrum.  Two consequences of the Devinatz-Hopkins-Smith nilpotence theorem are that $R$ is $T(n)$-acyclic if and only if $R$ is $K(n)$-acyclic, and that $R$ is zero if and only if $T(n)\otimes R = 0$ for all $0\leq n < \infty$ and $\F_p\otimes R=0$.  It is thus natural to define:

\begin{dfn}
The chromatic support of a $p$-local ring spectrum $R$ is the set
\[
\mdef{\supp}(R) = \{ n\in \N \: |\:  T(n)\otimes R \neq 0\}.
\]
\tqed
\end{dfn}

For an $\E_l$-ring spectrum, $0\leq l <\infty$, the following example shows that there is no restriction on the chromatic support.

\begin{exm}
Fix $0\leq m <\infty$.  By \cite[Theorem 1.4]{BurklundMoore}, we have for each $i$ that the $\MU$-module $\MU/v_i^{m+1}$ admits the structure of an $\E_m$-$\MU$-algebra.  Then, for any subset of $J\subset \N$, the $\E_m$-$MU$ algebra
\[
\bigotimes_{\MU}^{i \in \N \setminus J} \MU/v_i^{m+1}
\]
has chromatic support exactly $J$. 
\tqed
\end{exm}

However, the chromatic support of an \emph{$\E_{\infty}$-ring spectrum} turns out to be quite constrained as a result of the power operations on its homotopy groups.  One basic manifestation of these constraints is the following theorem of Mathew-Naumann-Noel:

\begin{thm}[May nilpotence conjecture, \cite{MNNmaynilp}]
Suppose $R\in \CAlg(\Sp)$ is a commutative algebra such that $R\otimes \Q = 0$.  Then $R$ is $T(n)$-acyclic for all $n\geq 0$.  
\end{thm}

In other words, if $\supp(R)$ is nonempty, then $0\in \supp(R)$.  This phenomenon was greatly generalized by Hahn \cite{Hahnsupport}, affirming a conjecture of Hovey:
\begin{thm}[Hahn]\label{thm:Hahn-later}
If a commutative algebra $R$ is $T(n)$-acyclic, then $R$ is $T(n+1)$-acyclic.
\end{thm}

Therefore, the chromatic support of any $R\in \CAlg(\Sp)$ is either empty or an interval containing $0$.  It is then natural to define:
\[
\mdef{\mathrm{height}(R)} \coloneqq \max\{ n \geq -1 \space  | T(n)\otimes R \neq 0 \}\footnote{Here, we set $T(-1) = \mathbb{S}$.}.
\]

Using the results of our paper, we are able to give an alternate proof of Hahn's theorem:

\begin{proof}
    Suppose $R$ is not $T(n+1)$-acyclic.  Then by \Cref{cor:mod_algclosed} provides a commutative algebra map $R\to E_{n+1}(L)$ to some height $n+1$ Lubin--Tate theory.  But $E_{n+1}(L)$ is not $T(n)$-acyclic, so $R$ cannot be either.  Here, we are using the observation that if $S \to S'$ is a map of commutative algebras and $S'$ is nonzero, then $S$ is nonzero, as the zero algebra admits no nontrivial modules.
\end{proof}

In fact, this proof is in a sense effective, in that it exhibits a commutative algebra map which witnesses the nontriviality of $R$.  This allows us to analyze the chromatic behavior of \emph{functors} applied to a given ring $R$, because if $F$ is any lax monoidal functor, then a ring map $R\to E$ induces a ring map $F(R) \to F(E)$.  

In \Cref{sub:blue}, we apply this to certain geometric fixed point functors to obtain a converse to chromatic blueshift statements for Tate cohomology.   Then in \Cref{sub:red}, we apply this strategy to algebraic $K$-theory and show that algebraic $K$-theory raises the chromatic height of a commutative algebra by exactly $1$.

\subsection{Chromatic blueshift}\label{sub:blue}\hfill

Given any spectrum $X$, one can endow $X$ with the trivial action of the group $C_p$ and extract the Greenlees--May Tate cohomology spectrum, $X^{tC_p}$ \cite{GreenleesMay}. In fact, the resulting endofunctor $(-)^{tC_p}\colon \Sp \to \Sp$ is lax symmetric monoidal \cite[Theorem I.3.1]{NS}, and therefore induces an endofunctor $(-)^{tC_p}\colon \CAlg(\Sp) \to \CAlg(\Sp)$.  Following work by Greenlees--Sadofsky and Hovey--Sadofsky \cite{GS96, HS96}, Kuhn showed:

\begin{thm}[\cite{Kuhn}]
Let $X$ be a $T(n)$-local spectrum.  Then $L_{T(n)}X^{tC_p} = 0.$
\end{thm}

In particular, for a $T(n)$-local commutative algebra $R$, $R^{tC_p}$ has height at most $n-1$.  This lowering of chromatic height has been dubbed \emph{blueshift}.  

The blueshift phenomenon is particularly accessible in the case of Lubin--Tate theory, where it has been understood in much greater generality by work of  Barthel--Hausmann--Naumann--Nikolaus--Noel--Stapleton.  To state their result, we briefly recall the following notions from equivariant homotopy theory:

\begin{dfn}\hfill
\begin{itemize} 
\item For a finite abelian group $B$, let $\mdef{\mathrm{rk}_p(B)} = \mathrm{dim}_{\F_p}(B\otimes_{\Z} \F_p).$ \item For a proper family of subgroups $\mathcal{F}$ of a finite abelian $p$-group $A$, we set
\[
\mdef{\mathrm{cork}_p(\mathcal{F})} = \min \{ \mathrm{rk}_p(A') \: | \: A'\subset A \text{ such that }A'\not\in \mathcal{F} \}.
\]
\item For such a family $\mathcal{F}$ and a genuine $A$-spectrum $X$, we let \deff{$\Phi^{\mathcal{F}}X$} denote the corresponding geometric fixed points.  This recovers the usual geometric fixed points when $\mathcal{F}$ is the family of proper subgroups of $A$ and classical Tate construction $X^{tA}$ when $X$ is Borel-equivariant and $\mathcal{F} = \{0\}$.  \tqed
\end{itemize}
\end{dfn}

Then we have:

\begin{thm}[{\cite[Theorem 3.5]{BHNNNS}}]\label{thm:BHNNNS}
Let $E(k)$ be a Lubin--Tate theory of height $n$ for a perfect field $k$, let $A$ be a finite abelian $p$-group and regard $E(k)$ as a Borel-equivariant genuine $A$-spectrum.  Then for any family $\mathcal{F}$ of subgroups of $A$, 
\[
\height(\Phi^{\mathcal{F}}E(k)) = n - \mathrm{cork}_p(\mathcal{F}).
\]
\end{thm}

Our goal in this section is to observe that by combining \Cref{cor:mod_algclosed} with \Cref{thm:BHNNNS}, we obtain a converse to chromatic blueshift.  We remark that in the case of $G=C_p$, this  statement has been shown to be equivalent to Hahn's theorem by \cite{CMNN}. We show:

\begin{thm}\label{thm:blueshift}
Let $A$ be a finite abelian $p$-group and let $\mathcal{F}$ be a proper family of subgroups of $A$.  Let $R$ be a commutative algebra and regard $A$ as a Borel-equivariant genuine $A$-spectrum with the trivial action.  Then, if $\Phi^{\mathcal{F}}R$ is $T(n-\mathrm{cork}_p(\mathcal{F}))$-acyclic, then $R$ is $T(n)$-acyclic.   
\end{thm}

\begin{proof}


Suppose that $R$ is not $T(n)$-acyclic.  Then by \Cref{cor:mod_algclosed}, $R$ admits a map of commutative algebras
\[
R \to E_n(L)
\]
for some height $n$ Lubin--Tate theory $E_n(L)$.  Applying $\Phi^{\mathcal{F}}$, we obtain a map of commutative algebras 
\[
\Phi^{\mathcal{F}}R  \to \Phi^{\mathcal{F}}E(L).
\]
But $\Phi^{\mathcal{F}} E_n(L)$ is not $T(n-\mathrm{cork}_p(\mathcal{F}))$-acyclic  by \Cref{thm:BHNNNS}, and therefore $\Phi^{\mathcal{F}}R$ cannot be either.  
\end{proof}


While we find it quite plausible that the converse to \Cref{thm:blueshift} is true, we do not know of a proof even in the case $A=C_p$, and so we record it here as a conjecture:

\begin{cnj}
Let $A$, $R$ as in \Cref{thm:blueshift}.  If $R$ is $T(n)$-acyclic, then $\Phi^{\mathcal{F}}R$ is $T(n-\mathrm{cork}_p(\mathcal{F}))$-acyclic.
\end{cnj}

\subsection{Chromatic redshift}\label{sub:red}\hfill

In their work on the algebraic $\mathrm{K}$-theory of ring spectra, Ausoni and Rognes studied the algebraic $\mathrm{K}$-theory of connective topological $\mathrm{K}$-theory and observed, in particular, that the result had nontrivial $T(2)$-localization.  This led to the formulation of the \emph{chromatic redshift conjectures}; the rough philosophy behind this far-reaching family of conjectures was that algebraic $\mathrm{K}$-theory shifts the height of a ring spectrum up by $1$.  These conjectures have since been widely studied \cite{BlumMan,ausoni2010algebraic,rognes2012algebraic,baas2007two,rognes2014chromatic,westerland2017higher,veen2018detecting,angelini2018detecting,angelini2019chromatic,carmeli2021ambidexterity,LMMT,CMNN,hahn2020redshift}.

The recent breakthrough work of Clausen--Mathew--Naumann--Noel \cite{CMNN} and Land--Mathew--Meier--Tamme \cite{LMMT} has significantly advanced the understanding of algebraic $\mathrm{K}$-theory and chromatic support.  Note that if $R$ is a commutative algebra, then $\mathrm{K}(R)$ also naturally admits the structure of an commutative algebra.  In this setting, \cite{CMNN} prove ``half'' of the redshift conjecture: 

\begin{thm}[\cite{CMNN}]
Let $R$ be a non-zero commutative algebra.  Then 
\[
\height(\mathrm{K}(R)) \leq \height(R) + 1.
\]
\end{thm}

The remaining question is whether algebraic $\mathrm{K}$-theory always increases the height by \emph{exactly} one---that is, whether height shifting actually occurs.  This question has since been answered in particular cases by Hahn--Wilson \cite{hahn2020redshift} ($BP\langle n\rangle$) and by the third author \cite{Yuanred} (Lubin--Tate theories and iterated $\mathrm{K}$-theories of fields).   By combining the results of this paper with \cite{Yuanred}, we are able to completely answer the question of height shifting for commutative algebras.

\begin{thm}\label{thm:redshift'}
Let $R$ be a non-zero commutative algebra and assume that $\height(R)\geq 0$.  Then 
\[
\height(\mathrm{K}(R)) = \height(R) + 1.
\]
\end{thm}
\begin{proof}
By \Cref{cor:mod_algclosed}, $R$ admits a map of commutative algebras
\[
R \to E(L)
\]
for some Lubin--Tate theory $E_n(L)$ of height $\height(R)$. This induces a commutative algebra map
\[
\mathrm{K}(R) \to \mathrm{K}(E_n(L))
\]
which has height $\height(R)+1$ by either \cite[Theorem A]{Yuanred} in the case $\height(R)\geq 1$ or the following proposition in the case $\height(R) = 0$.  
\end{proof}

\begin{prop}
Let $L$ be an algebraically closed field of characteristic $0$, let $L[t]$ denote the free commutative $L$-algebra on a generator $t$ in degree $2$, and $L[t^{\pm 1}]$ be obtained by inverting $t$.  Then there is an equivalence
\[
\mathrm{K}(L[t^{\pm 1}])^{\wedge}_p \simeq \ku^{\wedge}_p \oplus \Sigma \ku^{\wedge}_p,
\]
and in particular, the commutative algebra $\mathrm{K}(L[t^{\pm 1}])$ has height $1$.  
\end{prop}
\begin{proof}
By the Dundas-Goodwillie-McCarthy theorem \cite{DGM}, there is a pullback square
\[
\begin{tikzcd}
\mathrm{K}(L[t])\arrow[r]\arrow[d] & \mathrm{K}(L)\arrow[d] \\
\mathrm{TC}(L[t])\arrow[r] & \mathrm{TC}(L) .
\end{tikzcd}
\]
Since $\mathrm{TC}$ of a rational algebra is rational (and thus has trivial $p$-completion), we conclude that the natural map $\mathrm{K}(L[t])^{\wedge}_p \to \mathrm{K}(L)^{\wedge}_p$ is an equivalence.  Moreover, both are equivalent to $\ku^{\wedge}_p$ by Suslin's theorem \cite{Suslin}.  

By the localization (and devissage) theorem of Blumberg-Mandell \cite{BlumMan} as formulated by Barwick--Lawson \cite[Corollary 2.3]{BarwickLawson}, there is a cofiber sequence
\[
\mathrm{K}(L) \xrightarrow{\alpha_*} \mathrm{K}(L[t]) \to \mathrm{K}(L[t^{\pm 1}])
\]
where the first map is induced by the functor given by restriction of scalars along $\alpha\colon L[t] \to L$ by $t\mapsto 0$.  But note that there is a natural cofiber sequence of $L[t]$-modules
\[
\Sigma^2 L[t] \to L[t] \to L.
\]
Hence, since $\Sigma^2$ acts trivially on $\mathrm{K}$-theory, it follows from the additivity theorem (cf. \cite[\S 1.4]{Waldhausen} or \cite[Theorem 7.4]{BarAKT}) that the functor 
\[
(-)\otimes_{L[t]}L \colon \Mod_{L[t]}^{\mathrm{perf}}\to \Mod_{L[t]}^{\mathrm{perf}}
\]
induces the zero map on $\mathrm{K}$-theory.  But $\alpha_*$ factors through this functor, so it also induces the zero map.  It follows that there is an equivalence
\[
\mathrm{K}(L[t^{\pm 1}])^{\wedge}_p \simeq \mathrm{K}(L[t])^{\wedge}_p \oplus \Sigma \mathrm{K}(L)^{\wedge}_p \simeq \ku^{\wedge}_p \oplus \Sigma \ku^{\wedge}_p.
\]
\end{proof}

\begin{rmk}
The hypothesis that $\height(R) \geq 0$ in \Cref{thm:redshift'} is necessary.  To see this, consider $\F_2^{tC_2}$, which is of height $-1$.  Since $\F_2^{tC_2}$ is $1$-periodic, the suspension functor is naturally isomorphic to the identity on $\F_2^{tC_2}$-modules.  But suspension always induces $-1$ on $\mathrm{K}$-theory, so we have that $1=-1$ in $\mathrm{K}(\F_2^{tC_2})$.  This means that $\mathrm{K}(\F_2^{tC_2})$ is trivial after inverting $2$, and therefore of height $-1$.  

We remark, however, that \Cref{thm:redshift'} does extend to $\height(R) = -1$ under the further condition that $R$ is connective.  In this case, using the map $R\to \pi_0(R)$, there exists a map from $R$ to a field, whose $\mathrm{K}_0$ is $\Z$ and will therefore be of height $0$.  
\tqed
\end{rmk}

\appendix

\section{The constructible spectrum}
\label{sec:geopoints}
In \Cref{sec:null}, we introduced Nullstellensatzian objects as a portable abstraction of the properties enjoyed by algebraically closed fields within the category of commutative rings and showed that algebraically closed Lubin--Tate theories are the Nullstellensatzian objects in $\CAlg(\Sp_{T(n)})$. In this appendix, we develop an abstract theory of the \emph{constructible spectrum}, which attaches to each object $R$ of a suitable category $\CC$ a topological space whose points correspond to maps $R \to L$ from $R$ out to Nullstellensatzian objects $L$, up to the equivalence relation of common refinement.

\begin{dfn}
Let $\cC$ be a presentable category.  
\begin{enumerate}
    \item We say that \deff{the terminal object of $\CC$ is strict} if every map 
    $\star_{\cC} \to R$ with source the terminal object of $\CC$ is an isomorphism.
    \item We say that \deff{$\cC$ is weakly spectral} if $\CC$ is compactly generated and its terminal object is both strict and compact. \tqed
\end{enumerate}
\end{dfn}

\begin{rmk}
  In the category of rings, the zero ring is uniquely identified by the single equation $0=1$,
  and the condition that the terminal object of $\CC$ is strict and compact is abstracted from this situation. In \Cref{subsub:exm}, we will see that these conditions are indeed satisfied by almost any category of ring-like objects.
  \tqed
\end{rmk}


\begin{thm} \label{thm:con-spec}
  Let $\cC$ be a weakly spectral category.  Then there is a unique functor
  \[ \Spec^{\cons}_\CC(-)\colon \CC^\op \to \Top^{\mathrm{cpt},T_1,\mathrm{cl}} \]
  to compact, $T_1$ topological spaces and closed, continuous maps which satisfies the following properties:
  \begin{enumerate}
  \item[(A)] $\Spec_\CC^\cons(R)$ is empty if and only if $R \cong \star_{\cC}$.
  \item[(B)] If $R$ is Nullstellensatzian, then $\Spec^{\cons}_\CC(R)$ is a point.
  \item[(C)] $U \subset \Spec^{\cons}_\CC(R)$ is closed if and only if 
  there exists a map $R \to S$ such that $U$ is the image of the map 
  \[\Spec^{\cons}_\CC(S) \to \Spec^{\cons}_\CC(R).\]
  \item[(D)] For every  point $q\in \Spec^{\cons}_\CC(R)$, there exists a Nullstellensatzian object $L\in \cC_{R/}$ such that $\{q\}$  is the image 
  of the map 
    \[\Spec^{\cons}_\CC(L) \to \Spec^{\cons}_\CC(R).\] In this case, we say that the map $R\to L$ \deff{represents $q$}.

  \item[(E)] Given a span $S \leftarrow R \to T$, the natural comparison map
    \[ \Spec^{\cons}_\CC \left( S \coprod_R T\right) \to \Spec^{\cons}_\CC(S) \times_{\Spec^{\cons}_\CC(R)} \Spec^{\cons}_\CC(T) \]
    is surjective.
  \end{enumerate}
\end{thm}

The proof of \Cref{thm:con-spec} will occupy us for much of this appendix.
We begin in \Cref{subsec:null-props} by constructing a sufficient supply of Nullstellensatzian objects in $\CC$.
Next, we construct the functor $\Spec_\CC^\cons(-)$ in \Cref{subsec:construct-spec}.
Then, in \Cref{subsec:spec-properties} we prove properties (A)-(E) and complete the proof of \Cref{thm:con-spec}.
The remaining two subsections are then devoted to examining further properties which are useful for computing the constructible spectrum and examples.

\begin{ntn}
  We refer to the points of $\Spec_\CC^\cons(R)$ as the \emph{geometric points} of $R$.
  \tqed
\end{ntn}

\begin{cnv} \label{cnv:gpoints}
  Throughout this section, we will use $\circ_{\CC}$ and $\star_{\CC}$ to denote the initial and terminal objects of $\CC$, respectively, and drop the subscript when the category is clear from context.
  \tqed
\end{cnv}

\subsubsection{Examples of weakly spectral categories}\label{subsub:exm}\hfill

In order to apply the theory we have constructed in this section, we will need to verify that there is a sufficient supply of compactly generated, presentable categories $\CC$ whose terminal object $\star$ is strict and compact.

\begin{exm} \label{exm:discrete-w-spectral}
  The category $\mathrm{CRing}$ of discrete commutative rings is weakly spectral.
  \tqed
\end{exm}

\begin{lem} \label{lem:prig-strict-0}
  Let $\EE \in \Pr^\rig$. The zero algebra in $\CAlg(\EE)$ is a strict terminal object.
\end{lem}

\begin{proof}
  Let $u$ denote the underlying object functor $\CAlg(\EE) \to \EE$.
  Given an $R \in \CAlg(\EE)$ we can use the multiplication on $R$ to verify that
  $1=0$ in $[{\one}_\EE, u(R)]$ iff $u(R) = 0$ iff $R \cong 0$.
  The condition that $1=0$ in $[{\one}_\EE, u(R)]$ is preserved under maps out, therefore $0$ is a strict terminal object.
\end{proof}

\begin{lem} \label{lem:prig-cpt-0}
  Let $\EE \in \Pr^\rig$ with ${\one}_\EE$ compact.
  The zero algebra in $\CAlg(\EE)$ is compact.
\end{lem}

\begin{proof}
  Let $u$ denote the underlying object functor $\CAlg(\EE) \to \EE$.
  Using that assumption that ${\one}_\EE$ is compact, we see that a filtered colimit of commutative algebras $R_\alpha$ receives a map from the zero algebra iff $1=0$ in $[{\one}_\EE, u(R_\alpha)]$ at some finite stage.
\end{proof}

As a consequence of Lemmas \ref{lem:prig-strict-0} and \ref{lem:prig-cpt-0} we have our first non-trivial example of a weakly spectral category:

\begin{exm} \label{exm:calg-sp-spectral}
  The category $\CAlg(\Sp)$ of commutative algebras in spectra is weakly spectral.
  \tqed
\end{exm}

\begin{rmk} \label{rmk:type-of-alg}
  The arguments in Lemmas \ref{lem:prig-strict-0} and \ref{lem:prig-cpt-0} are not particularly sensitive to the choice of category of algebras. In particular, the same results hold with $\E_m$-algebras in place of commutative algebras.
  \tqed
\end{rmk}



\begin{lem} \label{lem:cpt-gen-spectral}
  Let $\EE \in \Pr^\rig$.
  If $\EE$ contains a compact object $e$ which generates $\EE$ under tensor products, duals and colimits, 
  then $\CAlg(\EE)$ is weakly spectral.
\end{lem}

\begin{proof}
  As $\EE$ is compactly generated, so is $\CAlg(\EE)$ by \cite[Corollary 5.3.1.17]{HA}. 
  By \Cref{lem:prig-cpt-0}, the zero algebra in $\CAlg(\EE)$ is strict it therefore suffices to show that the zero algebra is compact.
  
  Using the assumption that $e$ generates $\EE$ under tensor products, duals and colimits, we can read off that
  a commutative algebra $R$ is equivalent to zero iff 
  $1=0$ in the associative ring $[{\one}_\EE, R \otimes \End(e)]$. Using that $\End(e)$ is compact (and therefore dualizable) we have that
  \[
  [{\one}_\EE, R \otimes \End(e)] =  [ \End(e)^{\dual}, R ]
  \] 
  and can then see that $1=0$ in a filtered colimit iff $1=0$ at some finite stage. 
\end{proof}


\begin{exm} \label{exm:Tn-0}
    The category of $T(n)$-local spectra is generated by any choice of nonzero compact object. Therefore, applying \Cref{lem:cpt-gen-spectral}, we learn that 
    the category $\CAlg(\Sp_{T(n)})$ of $T(n)$-local commutative algebras is weakly spectral.
    \tqed
\end{exm}

\begin{exm}
  Let $X$ be a compact, $T_1$ space.
  The poset $\mathrm{Op}(X)$ of open sets in $X$ under inclusions, considered as a category, is weakly spectral.
  \tqed
\end{exm}

\begin{exm}
  The opposite of the category of profinite sets is weakly spectral. 
  \tqed
\end{exm}






\begin{lem} \label{lem:under-spectral}
If $\CC$ is a weakly spectral category, then for any $R\in \CC$, the undercategory $\CC_{R/-}$ is weakly spectral.
\end{lem}

\begin{proof}
  We have that $\CC_{R/-}$ is compactly generated because the functor $R\coprod - \colon \CC \to \CC_{R/-}$
  has a conservative and filtered colimit preserving right adjoint, and therefore preserves compact generating sets.  The condition that the terminal object is strict and compact is clear.
\end{proof}

\begin{rmk}
  Outside of this appendix we will only work with Examples \ref{exm:discrete-w-spectral} and \ref{exm:Tn-0} and those which can be produced from them using \Cref{lem:under-spectral}.
  \tqed
\end{rmk}

\subsection{Constructing Nullstellensatzian objects}
\label{subsec:null-props}\hfill

In this section, as preparation for proving \Cref{thm:con-spec},
we prove that Nullstellensatzian objects exist in great abundance in $\CC$.

\begin{prop} \label{cor:ns-exist}
  Let $\cC$ be a weakly spectral category and let $\kappa$ be a regular cardinal.  Then an object $ R \in \cC$ is non-terminal if and only if there exists a $\kappa$-Nullstellensatzian object $L$ and a map $R \to L$.
\end{prop}

At its core, our proof of \Cref{cor:ns-exist} is essentially a small-object-type argument where we make successive replacements of $R$ and then argue that after a sufficiently filtered colimit of such replacements one must obtain a Nullstellensatzian object. There is, however, a wrinkle: the quantification over \textbf{non-terminal $\kappa$-compact objects} in the definition of $\kappa$-Nullstellensatzian forces us into an argument which more closely resembles the argument proving the existence of maximal (proper) ideals in commutative rings. Before proceeding we make a digression on semi-lattices which provides a convenient logical super-structure for our later arguments.

\begin{rec} \label{dfn:semi-lattice}
  A \emph{semi-lattice} is a poset $\mathfrak{L}$ in which every pair of objects of $\mathfrak{L}$ has a least upper bound.
  \begin{enumerate}
  \item Given $u,v \in \mathfrak{L}$ we write $u \vee v$ 
    for the least upper bound of the $u$ and $v$ and refer to this as the \deff{join} of $u$ and $v$.
  \item A semi-lattice is \deff{bounded} if it has minimal element $\circ$ and a maximal element $\star$.
  \item A map of bounded semi-lattices is a map of posets which preserves joins, $\circ$ and $\star$. We write \deff{$\BSL$} for the category of bounded semi-lattices.
  \item An \deff{ideal} in $\mathfrak{L}$ is a subset $\mathfrak{I} \subset \mathfrak{L}$ such that
    (i) if $u \leq v$ and $v \in \mathfrak{I}$, then $u \in \mathfrak{I}$ and
    (ii) $\mathfrak{I}$ is closed under joins.
  \item A \deff{maximal ideal} is a proper ideal $\mathfrak{I}$ such that any larger ideal $\mathfrak{I} \subset \mathfrak{J}$ contains $\star$.
  \item As a corollary of the fact that a union of a chain of ideals in a bounded semi-lattice is an ideal, every bounded semi-lattice $\mathfrak{L}$ in which $\circ_\mathfrak{L} \neq \star_\mathfrak{L}$ has at least one maximal ideal.
  \item Under the identification of the category of posets as a full subcategory of $\mathrm{Cat}_\infty$, we can identify bounded semi-lattices as those posets that have all finite colimits and a terminal object. Joins correspond to coproducts. 
  \tqed
  \end{enumerate}
\end{rec}




\begin{cnstr} \label{cnstr:lattice-functor}
  Let \deff{$\mathfrak{L}\colon \mathrm{Cat}_\infty \to \mathrm{Poset}$}
  be the left adjoint to the natural inclusion of posets into categories.
  Concretely,
  $\mathfrak{L}(\AA)$ has an object \deff{$\langle R \rangle$} for each $R \in \AA$
  and $\langle R_1 \rangle \leq \langle R_2 \rangle $ exactly when
  there is a map $R_1 \to R_2$ in $\AA$.
  \tqed
\end{cnstr}

If $\AA$ has finite coproducts and a terminal object, then the same is true of $\mathfrak{L}(\AA)$ and thus $\mathfrak{L}(\AA)$ is a bounded semi-lattice with
\[ \langle \circ_{\AA} \rangle = \circ_{\mathfrak{L}(\AA)} \quad \langle \star_{\AA} \rangle = \star_{\mathfrak{L}(\AA)}, \quad \mathrm{ and } \quad \langle R \rangle \vee \langle S \rangle = \left\langle R \coprod S \right\rangle. \]


\begin{cnstr}
  Given an object $R \in \CC$ and a regular cardinal $\kappa$,
  let \deff{$\mathfrak{L}_\CC^\kappa(R)$} denote the bounded semi-lattice $\mathfrak{L}( (\CC_{R/-})^\kappa )$.

  Given a map $R \to S$, the co-base change functor $ S \coprod_R - \colon \CC_{R/-} \to \CC_{S/-} $
  \begin{enumerate}
  \item[(i)]  preserves $\kappa$-compactness,
  \item[(ii)] preserves coproducts and
  \item[(iii)] preserves terminal objects\footnote{Note that this uses our assumption that the terminal object in $\CC$ is strict.}
  \end{enumerate}
  therefore we can lift $\mathfrak{L}_\CC^\kappa(-)$ to a functor
  $ \mathfrak{L}_\CC^\kappa(-) \colon \CC \to \BSL $.
  \tqed
\end{cnstr}

\begin{exm} \label{exm:empty}
  An object $R \in \CC$ is terminal iff $\mathfrak{L}_\CC^\kappa(R) \cong \{ \star \}$.  
  \tqed
\end{exm}

From the defintion of $\kappa$-Nullstellensatzian, we obtain the following example which served as the motivation for considering bounded semi-lattices:

\begin{exm} \label{lem:lattice-ns}
  Let $\kappa$ be a regular cardinal.
  An object $L \in \CC$ is $\kappa$-Nullstellensatzian
  iff $\mathfrak{L}_\CC^\kappa(L) \cong \{ \circ \to \star \}$.
  \tqed
\end{exm}

\begin{lem} \label{lem:kappa-gen}
   Let $\kappa$ be a regular cardinal such that $\CC$ is $\kappa$-compactly generated.
  The functor 
  $ \CC \to \mathrm{Cat}_\infty $
  which sends $R$ to $(\CC_{R/-})^\kappa$ with functoriality in $R$ via co-base change 
  commutes with $\kappa$-filtered colimits.
\end{lem}

\begin{proof}
  Given a $\kappa$-filtered diagram
  $\colim_{\mathcal{K}} R_\alpha \cong R$,
  we obtain a natural comparison functor
  \[ c \colon \colim_{\mathcal{K}} (\CC_{R_\alpha/-})^\kappa \to ({\cC_{R/-}})^\kappa \]
  which we would like to show is an equivalence.
  We start by showing that $c$ is fully faithful.
  For any vertex $\beta \in \mathcal{K}$ and objects $S,T \in (\CC_{R_\beta/})^\kappa$ we have
  \begin{align*}
    &\Map_{\colim_{\mathcal{K}_{\beta/}} \CC_{R_\alpha/-}}\left( S, T \right)
      \cong \colim_{\mathcal{K}_{\beta/}} \Map_{\cC_{R_{\alpha}/-}}\left(R_\alpha \coprod_{R_\beta} S, R_k \coprod_{R_\beta} T\right) \\
    &\cong \colim_{\mathcal{K}_{\beta/}} \Map_{{\cC_{R_{\beta}/-}}}\left(S, R_\alpha \coprod_{R_\beta} T\right) 
    \cong \Map_{{\cC_{R_{\beta}/-}}}\left(S, \colim_{\mathcal{K}_{\beta/}} R_\alpha \coprod_{R_\beta} T\right) \\
    &\cong \Map_{{\cC_{R_{\beta}/}}}\left(S, R \coprod_{R_\beta} T\right) 
      \cong \Map_{{\cC_{R/-}}}\left(R \coprod_{R_\beta} S, R \coprod_{R_\beta} T\right).
  \end{align*}

  Given an object $W \in (\cC_{R/-})^\kappa$ write $R \to W$ as a $\kappa$-filtered colimit of arrows $S_\alpha \to T_\alpha$ where $S_\alpha, T_\alpha \in \CC^\kappa$. As $W \in {\cC_{R/-}}$ is $\kappa$-compact and expressible as a $\kappa$-filtered colimit of objects $R \coprod_{S_\alpha} T_\alpha$, we learn that $W$ is a retract of some $R \coprod_{S_\alpha} T_\alpha$. Since $S_\alpha$ is $\kappa$-compact we can factor the map $S_\alpha \to R$ through $R_q$ for some $q \in \mathcal{K}$. Thus, we learn that $W$ is a retract of an object in the image of $c$.
  On the other hand, using that $c$ is fully faithful we can lift idempotents on objects in the image of $c$ to a finite stage of the colimit and therefore $c$ is in fact essentially surjective.
\end{proof}

\begin{lem} \label{lem:basic-replace}
  Let $\kappa$ be a regular cardinal.
  Given an object $R \in \CC$ and a maximal ideal $\mathfrak{I}$ in $\mathfrak{L}^\kappa_\CC(R)$,
  there exists an object $S \in \CC$ and map $f \colon R \to S$ such that $\mathfrak{L}^\kappa_\CC(f)$ factors as
  \[ \mathfrak{L}^\kappa_\CC(R) \to \{ \circ \to  \star \} \to \mathfrak{L}_\CC^\kappa(S) \]
  with the preimage of $\circ$ being the maximal ideal $\mathfrak{I}$.
\end{lem}

\begin{proof}
  For each $\alpha \in \mathfrak{I}$ pick a $T_\alpha \in ({\cC_{R/-}})^\kappa$ with $\langle T_\alpha \rangle = \alpha$.
  Let 
  \[S \coloneqq \coprod_{\alpha \in \mathfrak{I}} T_\alpha \] 
  where the coproduct is taken in ${\cC_{R/-}}$.  
  For each $\alpha \in \mathfrak{I}$, the map  $S \to T_{\alpha} \coprod_{R} S$ admits a retract and thus 
  $\langle T_\alpha \rangle$ maps to $\circ$ in $\mathfrak{L}^\kappa_\CC(S)$.
  Similarly, since $\mathfrak{I}$ is a maximal ideal, for any $\langle W \rangle \not\in \mathfrak{I}$, there exists an $\alpha$ such that $ W \coprod_R T_\alpha$ is terminal. 
  Since the terminal object is strict, $W\coprod_R S$ is terminal as well, and so if $\beta \not\in \mathfrak{I}$ then $\beta$ maps to $\star$ in $\mathfrak{L}^\kappa_\CC(S)$.

  To complete the proof we just need to show that $S$ is non-terminal.
  Write $S$ as the filtered colimit
  \[ S \cong \colim_{A \subset \mathfrak{I},\ |A| < \omega} \coprod_{R}^{\alpha \in A} T_\alpha. \]
  Then, as the terminal object is strict and compact, $S$ is terminal iff
  there is some finite subset $A$ of $\mathfrak{I}$ such that $\coprod_R^{\alpha \in A} T_\alpha$ is terminal.
  A finite coproduct of the $T_\alpha$ in $({\cC_{R/-}})^\kappa$ is not terminal because $\mathfrak{I}$ is a proper ideal and joins in $\mathfrak{L}^\kappa_\CC(R)$ can be computed by the coproduct in $({\cC_{R/-}})^\kappa$.
\end{proof}

\begin{prop} \label{prop:ns-exist}
  Let $\kappa$ be a regular cardinal.
  Given an object $R \in \CC$ and a maximal ideal $\mathfrak{I}$ in $\mathfrak{L}^\kappa_\CC(R)$,
  there exists a $\kappa$-Nullstellensatzian object $L$ and a map $f :R \to L$
  such that $(\mathfrak{L}^\kappa_\CC(f))^{-1}(\circ) = \mathfrak{I}$.
\end{prop}

\begin{proof}
  Using \Cref{lem:basic-replace} we can produce a map $g\colon R \to R'$ such that $(\mathfrak{L}^\kappa_\CC(g))^{-1}(\circ) = \mathfrak{I}$. To conclude we now just need to find a map $R' \to L$ to a $\kappa$-Nullstellensatzian object.

  Construct a diagram $F \colon (\kappa+1)\to \CC$ by setting $F(0)=R'$,
  using the replacement procedure from \Cref{lem:basic-replace} at each successor ordinal
  and extending to limit ordinals by colimits.
  
  To conclude we will show that $F(\kappa)$ is $\kappa$-Nullstellensatzian. 
  Using \Cref{lem:kappa-gen} and the fact that $\kappa$ is $\kappa$-filtered we have isomorphisms
  \[   \mathfrak{L}^\kappa_\CC(F(\kappa))
    \cong \mathfrak{L}\left((\CC_{F(\kappa)/-})^\kappa \right) \cong \mathfrak{L} \left( \colim_{\alpha <\kappa} (\CC_{F(\alpha)/-})^\kappa \right). \]
  Then, using the fact that $\mathfrak{L}$ is a left adjoint and the factorization through $\{\circ \to \star\}$ from \Cref{lem:basic-replace},
  we can read off that
   \[ \mathfrak{L} \left( \colim_{\alpha <\kappa} (\CC_{F(\alpha)/-})^\kappa \right) 
    \cong \colim_{\alpha <\kappa} \mathfrak{L} \left( (\CC_{F(\alpha)/-})^\kappa \right) 
      \cong \colim \{ \circ \to \star \}
      \cong \{ \circ \to \star \}. \]
  It follows from \Cref{lem:lattice-ns} that $F(\kappa)$ is $\kappa$-Nullstellensatzian, as desired.
\end{proof}

\Cref{cor:ns-exist} now follows as a corollary of \Cref{prop:ns-exist} and the existence of maximal ideals in bounded semi-lattices.



\subsection{Constructing the spectrum}
\label{subsec:construct-spec}\hfill

We are now ready to define the constructible spectrum by placing a topology on the set of maximal ideals of $\mathfrak{L}_\CC^\omega(R)$. 

\begin{cnstr}
  Given an $\mathfrak{L} \in \BSL$ we construct a topological space \deff{$\Spec(\mathfrak{L})$}
  whose points are the maximal ideals of $\mathfrak{L}$ and whose closed sets are generated by
  \[ \mdef{[\ell]}\coloneqq \{ \mathfrak{I}\ |\ \ell \in \mathfrak{I} \} \]
  for $\ell \in \mathfrak{L}$. We refer to $\Spec(\mathfrak{L})$ as the \deff{spectrum} of $\mathfrak{L}$.
  \tqed
\end{cnstr}

Unfortunately, a map of semi-lattices does not necessarily induce a map between their spectra.
In order to repair this we must restrict attention to only certain special maps which we call \emph{tame}.

\begin{dfn}
  A map of bounded semi-lattices $f\colon \mathfrak{L}_1 \to \mathfrak{L}_2$ is \deff{tame}
  if for every maximal ideal ${\mathfrak{I}}$ of $\mathfrak{L}_2$ the ideal $f^{-1}({\mathfrak{I}})$ is also maximal.
  We write \deff{$\BSL^{\mathrm{tame}}$} for subcategory of tame maps of $\BSL$.
  \tqed
\end{dfn}
  
\begin{lem}
  The spectrum construction $\Spec(-)$ assembles into a functor
  \[ \Spec(-) \colon (\BSL^{\mathrm{tame}})^\op \to \Top. \]
\end{lem}

\begin{proof}
  The restriction to tame maps allows us to define the underlying functor to $\mathrm{Set}$
  by the formula $\Spec(f)(\mathfrak{I}) = f^{-1}({\mathfrak{I}})$.
  Now we just need to check that given a tame map
  $f \colon \mathfrak{L}_1 \to \mathfrak{L}_2$, the induced map of spectra is continuous.
  For this, we observe that for $\ell \in \mathfrak{L}_1$
  \[ \Spec(f)^{-1}([\ell]) = \{ \mathfrak{J}\ |\ \ell \in f^{-1}(\mathfrak{J}) \} =  \{ \mathfrak{J}\ |\ f(\ell) \in \mathfrak{J} \} =  [f(\ell)]. \]
\end{proof}

\begin{lem} \label{lem:prove-tame}
  Given a map $f\colon R \to S$ in $\CC$,
  the associated map $\mathfrak{L}^\omega_\CC(R) \to \mathfrak{L}^\omega_\CC(S)$ is tame.
\end{lem}

\begin{proof}
  \Cref{prop:ns-exist} lets us write every maximal ideal of $\mathfrak{L}^\omega_\CC(S)$ in the form $q^{-1}(\circ)$ for some map $q\colon S \to L$ to a Nullstellensatzian object $L$. Thus it suffices to prove the lemma for maps $q\colon R \to L$ with $L$ Nullstellensatzian.  That is, we must show, for every such map, that $q^{-1}(\circ)$ is maximal.  

  Write $L \in {\cC_{R/-}}$ as a filtered colimit of objects $T_{i} \in ({\cC_{R/-}})^\omega$. 
  Note that $\langle T_{i} \rangle$ is in the ideal $q^{-1}(\circ)$ for each ${i}$.
  Suppose $\langle W \rangle \not\in q^{-1}(\circ) \subset \mathfrak{L}^{\omega}_{\CC}(R)$; then $W \coprod_R L$ is terminal and since the terminal object is strict and compact, there is a $T_{i}$ with $W \coprod_R T_{i}$ terminal. Since $\langle T_{i} \rangle \in q^{-1}(\circ)$, this means $W$ cannot be added to $q^{-1}(\circ)$ and therefore that this ideal is maximal.
\end{proof}

\begin{dfn} \label{cnstr:spec-fun}
  Using \Cref{lem:prove-tame}, we know that $\mathfrak{L}_\CC^\omega(-)$ factors through $\BSL^{\mathrm{tame}}$ and we can therefore define the constructible spectrum functor as the composite:
  \[ \mdef{\Spec_\CC^\cons(-)} \coloneqq \Spec( \mathfrak{L}_\CC^\omega( - )). \]

  When it is clear from context that we are referring to the constructible spectrum (as opposed to another notion of spectrum) we will sometimes drop the adjective $\cons$ from our notation.
  \tqed
\end{dfn}

As a corollary of \Cref{cor:ns-exist}, \Cref{lem:lattice-ns} and \Cref{prop:ns-exist} we obtain the following proposition which verifies claims (A), (B) and (D) of \Cref{thm:con-spec}.

\begin{prop} \label{prop:ns-reps}
  The functor $\Spec_\CC^\cons \colon \CC^\op \to \Top$ has the following properties:
  \begin{enumerate}
  \item $\Spec^\cons_\CC(R) = \emptyset$ iff $R \cong \star$.
  \item If $L$ is Nullstellensatzian, then $\Spec^\cons_\CC(L)$ is point.
  \item Every point of $\Spec^\cons_\CC(R)$ has a Nullstellensatzian representative.
  \end{enumerate}
\end{prop}

\subsection{Properties of the spectrum}
\label{subsec:spec-properties}\hfill

In this section we finish the proof of \Cref{thm:con-spec} using the definition of the constructible spectrum from the previous section and then proceed into a discussion of further properties of $\Spec_\CC^\cons$ that are useful for making computations in practice.

\begin{dfn}
  Given an object $S \in {\cC_{R/-}}$ we write \deff{$[S]_R$} $ \subset \Spec_\CC(R)$ for the image of $\Spec_\CC(S)$ in $\Spec_\CC(R)$. When the map $R \to S$ is clear from context, we shall omit the subscript $R$. 
  \tqed
\end{dfn}

The following lemma follows from the fact that $\Spec^{\cons}_{\CC}(S)$ depends only on the undercategory $\CC_{S/-}$.

\begin{lem}\label{lem:spec-under-same}
Let $R\to S$ be a map.  Then there is an isomorphism
\[
\Spec^{\cons}_{\CC_{R/-}}(S) \cong \Spec^{\cons}_{\CC}(S).
\]
\end{lem}

\begin{lem} \label{lem:simple-intersection}
  For $S \in ({\cC_{R/-}})^\omega$, we have $[ \langle S \rangle ] = [S]_R$.
\end{lem}

\begin{proof}
  Given a point $q \in \Spec_\CC^\cons(R)$, pick a Nullstellensatzian representative $L \in {\cC_{R/-}}$.
  Examining the associated map of lattices
  \[ \mathfrak{L}_{\cC}^{\omega}(R) \to  \mathfrak{L}_{\cC}^{\omega}(L) \cong \{ \circ \to \star \}, \]
  we can read off that
  $q \in [\langle S \rangle]$
  iff $S$ is in the preimage of $\circ$
  iff the map $R \to L$ factors as $R \to S \to L$
  iff $q \in [S]$.
\end{proof}

\begin{lem}\label{lem:point-equiv}
 For $R\in \CC$ and Nullstellensatzian objects $q_1\colon R \to L_1$, $q_2\colon R\to L_2$ in  $\CC_{R/-}$, the following are equivalent:
  \begin{enumerate}
  \item We have $[L_1] = [L_2] \subset \Spec_{\CC}(R)$.
  \item The pushout $L_1 \coprod_{R} L_2$ is non-terminal.  
  \item There exists a Nullstellensatzian object $L_3 \in \CC_{R/-}$ together with maps $L_1 \to L_3$ and $L_2 \to L_3$ in $\CC_{R/-}$.  
  \end{enumerate}
\end{lem}
\begin{proof}
(2) implies (3) by \Cref{prop:ns-reps}(1,3) and (3) implies (1) by \Cref{prop:ns-reps}(1).  

It remains to show that (1) implies (2), so suppose that $[L_1] = [L_2]$.  Write $L_2 = \colim_{\alpha} R_{\alpha}$ as a filtered colimit of objects in $R_{\alpha} \in(\CC_{R/-})^{\omega}$.  Note that $\langle R_{\alpha}\rangle  \in q_2^{-1}(\circ)$ because $[L_2] \subset [\langle R_{\alpha} \rangle ]$. Assume for the sake of contradiction that  
\[\star \cong L_1 \coprod_{R} L_2 \cong L_1 \coprod_R \colim_{\alpha} R_{\alpha} \cong  \colim_{\alpha} L_1 \coprod_R R_{\alpha}. \]
Then by the strictness and compactness of $\star_{\cC}$, there exists some $\alpha$ such that $L_1 \coprod_R R_{\alpha} \cong \star_{\cC}.$
Thus we get that $q_1 (\langle R_{\alpha} \rangle) = \star_{\mathfrak{L}_{\CC}^{\omega}(L_1)
}$ and $\langle R_{\alpha} \rangle \notin  q_1^{-1}(\circ)$ contradicting the equality \[\{q_2^{-1}(\circ)\} = [L_2] = [L_1] = \{q_1^{-1}(\circ)\} .\]
\end{proof}

\begin{cor}\label{cor:point-inside}
If $R\to S$ is a map in $\CC$ and $R\to L_1$ is Nullstellensatzian, then $[L_1] \subset [S]$ if and only if $L_1\coprod_R S$ is non-terminal.
\end{cor}
\begin{proof}
If $L_1\coprod_R S$ is non-terminal, then $[L_1] \subset [S]$ follows immediately from \Cref{prop:ns-reps}(1,3). If $[L_1] \subset [S]$, choose some $S \to L_2$ for $L_2$ Nullstellensatzian that represents a preimage of $L_1$ in $\Spec(S)$, and we are done by \Cref{lem:point-equiv}.
\end{proof}

\begin{lem} \label{lem:filtered-intersection}
  Given a filtered diagram $F \colon \mathcal{K} \to {\cC_{R/-}}$, we have
  \[ \left[ \colim_{\mathcal{K}} F \right] = \bigcap_{\alpha \in \mathcal{K}} [F(\alpha)]. \]      
\end{lem}

\begin{proof}
First, note that we clearly have $\left[ \colim_{\mathcal{K}} F \right] \subset \bigcap_{\alpha \in \mathcal{K}} [F(\alpha)]$, so it suffices to show the other inclusion.  
Now, suppose $q \in \cap_{\alpha \in \mathcal{K}} [F(\alpha)]$ with Nullstellensatzian representative $L \in {\cC_{R/-}}$.   By \Cref{cor:point-inside}, it's enough to show that \[
L\coprod_R \colim_{\alpha\in \mathcal{K}}F(\alpha) = \colim_{\mathcal{K}} L\coprod_R F(\alpha)
\]
is non-terminal.  But since $q$ is in the intersection of the $[F(\alpha)]$, \Cref{cor:point-inside} implies that each $L\coprod_R F(\alpha)$ is non-terminal, and so the conclusion follows from the fact that the terminal object is strict and compact.

\end{proof}

\begin{lem} \label{lem:intersection}
  Given a collection of objects $S_{\alpha} \in {\cC_{R/-}}$, ${\alpha} \in U$ we have
  $ \cap_{{\alpha} \in U}[S_{\alpha}] = [ \coprod_R^{{\alpha} \in U} S_{\alpha}]$ in $\Spec_\CC(R)$.
\end{lem}

\begin{proof}
  We begin with the case of a binary intersection of objects $S$ and $T$.
  Write $S$ and $T$ as filtered colimits of objects $S_{\alpha}$ and $T_\beta$ in $({\cC_{R/-}})^\omega$.
  Then, using Lemmas \ref{lem:simple-intersection} and \ref{lem:filtered-intersection} we have
  \begin{align*}
    \left[ S \coprod_R T \right]
    &= \left[ \colim_{\alpha} S_{{\alpha}} \coprod_R \colim_\beta T_{\beta} \right]
      = \left[ \colim_{{\alpha},\beta} S_{{\alpha}} \coprod_R T_{\beta} \right] 
      = \bigcap_{{\alpha},\beta} \left[ S_{{\alpha}} \coprod_R T_{\beta} \right] \\
    &= \bigcap_{{\alpha},\beta} \left[ \left\langle S_{{\alpha}} \coprod_R T_{\beta} \right\rangle \right]
      = \bigcap_{{\alpha},\beta} [ \langle S_{{\alpha}} \rangle \vee \langle T_{\beta} \rangle]
      = \bigcap_{{\alpha},\beta} [ \langle S_{{\alpha}} \rangle ] \cap [ \langle T_{\beta} \rangle ] \\
    &= \left( \bigcap_{{\alpha}} [ S_{{\alpha}} ] \right) \cap \left( \bigcap_{\beta} [ T_{\beta} ] \right) 
      = [ \colim_{\alpha} S_{{\alpha}} ] \cap [ \colim_\beta T_{\beta} ] 
      = [ S ] \cap [ T ] 
  \end{align*}
  
  Now we return to the general case:
  \begin{align*}
    \left[ \coprod_R^{{\alpha} \in U} S_{\alpha} \right]  
    &= \left[ \colim_{V \subset U,\ |V| < \omega} \coprod_{R}^{{\alpha} \in V} S_{\alpha} \right] 
      = \bigcap_{V \subset U,\ |V| < \omega} \left[ \coprod_{R}^{{\alpha} \in V} S_{\alpha}  \right] 
      = \bigcap_{{\alpha} \in U} [ S_{\alpha} ].
  \end{align*}
\end{proof}



\begin{lem} \label{lem:strong-intersection}
  Given a span $S \leftarrow R \to T$ in $\CC$, the natural comparison map
    \[ \Spec_\CC \left( S \coprod_R T\right) \to \Spec_\CC(S) \times_{\Spec_\CC(R)} \Spec_\CC(T) \]
    is surjective.
\end{lem}

\begin{proof}
  Suppose we are given points $q_1 \in \Spec_\CC(S)$ and $q_2 \in \Spec_\CC(T)$
  which become equal in $\Spec_\CC(R)$.
  Pick Nullstellensatzian representatives $L_1 \in {\cC_{S/}}$ and $L_2 \in {\cC_{T/}}$.
  Examining the map $S \coprod_R T \to L_1 \coprod_R L_2$, \Cref{lem:point-equiv} allows us to conclude.  
  
\end{proof}



\begin{cor} \label{cor:idempotent}
  If $R \to S$ is a map in $\CC$ such that the fold map $S \coprod_R S \to S$ is an equivalence,
  then the map $\Spec_\CC(S) \to \Spec_\CC(R)$ is an inclusion.
\end{cor}

\begin{proof}
  Applying \Cref{lem:strong-intersection} and the assumption on the fold map, we learn that the diagonal map
  \[ \Spec_\CC \left( S \right) \to \Spec_\CC(S) \times_{\Spec_\CC(R)} \Spec_\CC(S) \]
  is surjective. 
  The conclusion follows.
\end{proof}

\begin{lem}\label{lem:top-is-cons}
  $\Spec_\CC(R)$ has the constructible topology.
\end{lem}

\begin{proof}
  To prove the lemma we must show that the closed sets of $\Spec_\CC(R)$ are exactly the sets of the form $[S]$. Recall that $\Spec_\CC(R)$ has a basis of closed sets $[\langle S \rangle]$ as $S$ ranges over $({\cC_{R/-}})^\omega$ and by \Cref{lem:simple-intersection} $[\langle S \rangle] = [S]$.
  Then, using \Cref{lem:intersection} we see that any closed set $\cap_\alpha \langle S_\alpha \rangle$ can be written as $[T]$ for some $T \in {\cC_{R/-}}$:
  \[ \cap_\alpha \langle S_\alpha \rangle = \bigcap_\alpha [ S_\alpha ] \cong \left[ \coprod_R^\alpha S_\alpha \right]. \]

  Given an object $S \in {\cC_{R/-}}$ we can write $S$ as a filtered colimit of $S_\alpha \in ({\cC_{R/-}})^\omega$.
  Then, using Lemmas \ref{lem:simple-intersection} and \ref{lem:filtered-intersection} we find that
  \[ [T] = \bigcap_\alpha [T_\alpha] = \bigcap_\alpha [ \langle T_\alpha \rangle ]\]
  and therefore that $[T]$ is closed.
\end{proof}

\begin{lem} \label{lem:cons-cpt}
  $\Spec_\CC(R)$ is compact.
\end{lem}

\begin{proof}
  Using \Cref{lem:top-is-cons} it will suffice to show that, 
  given a set $\{S_\alpha\}_{\alpha \in U}$ of objects in ${\cC_{R/}}$
  such that the intersection $\cap_{\alpha \in U} [S_\alpha]$ is empty,
  there exists a finite subset $V \subseteq U$ with $\cap_{\alpha \in V} [S_\alpha]$ empty as well.

  Using Lemmas \ref{lem:filtered-intersection} and \ref{lem:intersection}, we can translate into the claim that it suffices to show that
  if $\coprod_R^{\alpha \in U} S_\alpha$ is terminal, then there is some finite subset $V \subseteq U$ such that $\coprod_R^{\alpha \in V} S_\alpha$ is terminal.
  This follows by writing $\coprod_R^{\alpha \in U} S_\alpha$ as a filtered colimit over finite subsets of $U$ and then using the assumption that the terminal object in $\CC$ is compact and strict.
\end{proof}


\begin{proof}[Proof (of \Cref{thm:con-spec}).]
  We constructed the functor $\Spec_\CC(-)\colon \CC^\op \to \Top$ in \Cref{cnstr:spec-fun}
  and proved that $\Spec_\CC(R)$ is compact in \Cref{lem:cons-cpt}.  
  Properties (A), (B) and (D) were proved in \Cref{prop:ns-reps}.
  Properties (C) and (E)  were proved in Lemmas
  \ref{lem:top-is-cons} and \ref{lem:strong-intersection} respectively.
  Given a map $R \to S$ the associated map on constructible spectra is closed as a corollary of (C). Finally, the fact that $\Spec_{\CC}(-)$ is $T_1$ follows because combining (C) and (D), one sees that every point is closed.  
  
  The uniqueness follows from the fact that (A), (B) and (D) determine the underlying set of $\Spec_{\cC}(R)$; indeed, it is the set of Nullstellensatzian objects $L \in \cC_{R/-}$ modulo the equivalence relation that $L_1$ and $L_2$ determine the same point if $L_1\coprod_R L_2$ is non-terminal.
The topology is completely determined by (C).
\end{proof}
\begin{rmk}
  As a consequence of the fact that $\Spec_\CC(-)$ lands in topological spaces and \emph{closed} maps,
  any map $R \to S$ which induces a bijection on geometric points will induce an isomorphism on constructible spectra.
  \tqed
\end{rmk}

As a corollary of \Cref{thm:con-spec}(E) we learn that the collection of maps which induce surjections on the constructible spectrum is stable under co-base change, and therefore we can use these maps as the covers in a topology on $\CC^\op$.

\begin{dfn}
  A \deff{$c$-cover} of $R \in \CC$ is a map $R \to S$ which induces a surjection on constructible spectra. We define the $c$-topology to be the topology on $\CC^\op$ generated the singleton $c$-covers.
  \tqed
\end{dfn}

The following is clear from the definition of maps in $\Top^{\mathrm{cpt},T_1,\mathrm{cl}}$.

\begin{lem}\label{lem:surj_is_quo}
A map $X\to Y$ in $\Top^{\mathrm{cpt},T_1,\mathrm{cl}}$ is surjective on the underlying sets if and only if it is a topological quotient map.  
\end{lem}

This lemma has the following consequence:

\begin{cor} \label{cor:descent}
  The functor $\Spec_\CC^\cons(-)\colon \CC \to (\Top^{\mathrm{cpt},T_1,\mathrm{cl}})^{\op}$ is a sheaf for the $c$-topology.
\end{cor}
\begin{proof}
We would like to show that if $R\to S$ is a $c$-cover, then the diagram 
\[ \left( \Spec_\CC(S\coprod_R S) \rightrightarrows \Spec_\CC( S ) \right) \to \Spec_\CC (R) \] is a coequalizer in $\Top^{\mathrm{cpt},T_1,\mathrm{cl}}$.
Since $\Spec_\CC (S) \to \Spec_\CC (R)$ is surjective by assumption, it is a topological quotient map by \Cref{lem:surj_is_quo}.  We are then done by \Cref{lem:strong-intersection}.
\end{proof}

More generally, this corollary means that in order to compute $\Spec_\CC(R)$, it suffices to find a map $R \to S$ which covers $\Spec_\CC(R)$ and a map $S \coprod_R S \to T$ which covers $\Spec_\CC(S \coprod_R S)$ and then compute the coequalizer
\[ \mathrm{Coeq} \left( \Spec_\CC(T) \rightrightarrows \Spec_\CC( S ) \right). \]

Although it may seem difficult to construct $c$-covers in general, 
in fact, each object $R$ admits a $c$-cover of a particularly simple form:

\begin{exm} \label{exm:prod-points}
  Given an $R \in \CC$, 
  pick a Nullstellensatzian representative $L_q$ of each point $q \in \Spec_\CC(R)$ and consider the map
  \[ R \to \prod_{q \in \Spec_\CC(R)} L_q. \]
  Examining the projections maps to the individual $L_q$'s, one can see that this map is a $c$-cover.
  \tqed
\end{exm}

We end the section with a pair of results which exploit the existence of $c$-covers by products of Nullstellensatzian objects to simplify the process of computing the constructible spectrum.

\begin{dfn}\label{dfn:NS_PROD}
  Let \deff{$\NS_\CC$} $ \subset \CC$ be the full subcategory of Nullstellensatzian objects and
  let $\NSP_\CC \subset \CC$ be the full subcategory of products of Nullstellensatzian objects.
  \tqed
\end{dfn}

\begin{lem} \label{prop:kan-set}
  The composite functor
  \[ \CC^\op \xrightarrow{\Spec_\CC(-)} \Top^{\mathrm{cpt},T_1,\mathrm{cl}} \to \mathrm{Set} \]
  is the left Kan extension of the constant functor $\mathrm{pt} \colon \NS^{\op} \to \mathrm{Set}$ along the inclusion $\NS_{\CC}^{\op} \to \CC^{\op}$.
\end{lem}

\begin{proof}
  This follows from the fact that every point has a Nullstellensatzian representative and any pair of maps $R \to L_1$ and $R \to L_2$ represent the same point iff there is a Nullstellensatzian object $L_3$ and a map $L_1 \coprod_R L_2 \to L_3$.
\end{proof}


  


\begin{lem} \label{lem:matches}
  The restriction of $\Spec_{\CC}(-)\colon \CC^{\op} \to \mathrm{Top}^{\mathrm{cpt}, T_1,\mathrm{cl}}$ to $\NSP_\CC$ is uniquely determined by the full subcategory $\NS_\CC \subseteq \NSP_\CC$.
\end{lem}

\begin{proof}
  The underlying set of the restriction of $\Spec_\CC(-)$ to $\NSP_\CC$ can be recovered as the left Kan extension of the constant functor to $\mathrm{Set}$ from $\mathrm{NS}_\CC$ to $\mathrm{NS}_\CC^\Pi$ by \Cref{prop:kan-set}. 
  Using the description of the constructible topology from \Cref{thm:con-spec}(C) and \Cref{exm:prod-points}, we see that the lift of the functor $\mathrm{NS}_\CC^\Pi \to \mathrm{Set}$ to $\mathrm{Top}^{\mathrm{cpt}, T_1,\mathrm{cl}}$ is uniquely determined.
\end{proof}

\begin{prop} \label{prop:kan-top}
  The constructible spectrum functor
  \[
  \Spec_{\CC}(-) \colon \CC^{\op} \to \Top^{\mathrm{cpt}, T_1,\mathrm{cl}}
  \]
  is the left Kan extension of its restriction to $(\mathrm{NS}^\Pi)^{\op}$.
\end{prop}

\begin{proof}

  Using that the forgetful functor $u \colon \Top^{\mathrm{cpt},T_1,\mathrm{cl}} \to \mathrm{Set}$ is faithful
  and the fact that, as a consequence of \Cref{prop:kan-set}, 
  the functor $u(\Spec_\CC(-))$ is Kan extended from its restriction to $\NSP_\CC$
  we find, after unrolling definitions, that it suffices to show that
  given a functor $G \colon \CC^\op \to \Top^{\mathrm{cpt},T_1,\mathrm{cl}}$,
  any natural transformation
  $\eta \colon u \circ \Spec_\CC \Rightarrow u \circ G $
  whose restriction to $\mathrm{NS}^\Pi$ is continuous and closed is continuous and closed on all $\cC$.

  For this we argue as follows:
  Given an object $R \in \CC$,
  pick a $S \in \mathrm{NS}^\Pi$ and a map $f \colon R \to S$ which induces a surjection on the constructible spectrum (e.g. \Cref{exm:prod-points}), and consider the following diagram:
  \[ \begin{tikzcd}
      \Spec_\CC(S) \arrow[r,"\eta_S"] \ar[d, two heads] & G(S) \ar[d] \\
      \Spec_\CC(R) \arrow[r,"\eta_R"] & G(R)
    \end{tikzcd} \]
  By \Cref{lem:surj_is_quo}, $\eta(R)$ is continuous and closed iff $\Spec_\CC(f) \circ \eta(R)$ is continuous and closed. On the other hand $\Spec_\CC(f) \circ \eta(R) = \eta(S) \circ G(f)$  and the latter is continuous and closed by assumption.
\end{proof}

A corollary of this result is that the constructible spectrum of $R\in\CC$ depends only on the maps from $R$ to products of Nullstellensatzian objects.

\begin{dfn}\label{dfn:spd}
For an object $R\in \CC$, define the functor \deff{
\[ 
\Spd_{\CC} (R)  \colon \NSP_\CC \to \Set
\]}
to be the restriction of $\pi_0\Map_{\CC}(R, -)\colon \CC \to \Set$ to the full subcategory $\mathrm{NS}^\Pi$ of products of Nullstellensatzian objects.  
\tqed
\end{dfn}

As a consequence of \Cref{prop:kan-top}, the functor $\Spec_{\CC}(-)$ factors through $\Spd_{\CC}$ and we have:

\begin{cor}\label{cor:spd}
Suppose that $R_1,R_2\in \CC$.  Then any isomorphism $\Spd_{\CC}(R_1) \cong \Spd_{\CC}(R_2)$ induces an isomorphism $\Spec_{\CC}(R_1)\cong \Spec_{\CC}(R_2)$.
\end{cor}

\subsection{The spectrum of a product}\hfill

\Cref{prop:kan-top} implies that the fundamental step in understanding the functor $\Spec_\CC(-)$ is understanding $\Spec_\CC(R)$ when $R$ is a \emph{product} of Nullstellensatzian objects. The simplest guess  for what happens in this case is that
\begin{enumerate}
\item[(i)] $\Spec_\CC(L_1 \times \cdots \times L_m)$ is a discrete space with $m$ points and more generally
\item[(ii)] $\Spec_\CC\left( \prod_{\alpha \in U} L_\alpha \right)$ is the Stone-\v{C}ech compactification of $U$.\footnote{Recall that $\Spec_\CC(-)$ lands in compact spaces and the discrete space $U$ is not compact if $U$ is not finite.}
\end{enumerate}
In this section we isolate certain natural conditions on $\CC$ 
which ensure that this guess is correct.

\subsubsection{Op-disjunctive categories and finite products}\hfill


\begin{dfn} \label{dfn:op-disjunctive}
  A presentable category $\CC$ is \deff{op-disjunctive} if for any pair of objects $R, S \in \CC$, the product functor induces an equivalence
  \[ {\cC_{R/-}} \times  {\cC_{S/-}} \xrightarrow{\cong} \cC_{R \times S/-}. \]
 By \cite[\S 4]{BarwickMackey}, this is equivalent to the conditions that
 \begin{enumerate}
     \item \deff{Products are disjoint}, that is, for any pair of objects $R, S\in \CC$, there is an equivalence $R \coprod_{R\times S} S \cong \star$.
     \item \deff{Coproducts distribute over products}, that is, $R\coprod (S_1\times S_2) \cong (R\coprod S_1) \times (R\coprod S_2) $.\tqed  
 \end{enumerate}
  
\end{dfn}

\begin{rmk}
  If $\CC$ is op-disjunctive, then the equivalences 
  \[ \CC_{\circ/-} \times \CC_{\star/-} \cong \CC_{\circ \times \star / -} \cong \CC_{\circ/-}\] 
  imply that the terminal object of $\CC$ is strict.
  \tqed
\end{rmk}

\begin{lem} \label{lem:fin-prod}
  If $\CC$ is weakly spectral and op-disjunctive, then
  \[ \Spec_\CC(R \times S) \cong \Spec_\CC(R) \coprod \Spec_\CC(S) .\]
\end{lem}

\begin{proof}
  Using the assumption that $\CC$ is op-disjunctive, we find that the object $R \in  {\cC_{R\times S/-}}$ is $\coprod$-idempotent. \Cref{cor:idempotent} then implies that $\Spec_\CC(R)$ sits inside $\Spec_\CC(R \times S)$ as a closed subspace. Similarly, we learn that $\Spec_\CC(S)$ is a closed subspace of $\Spec_\CC(R \times S)$ as well.
  Using \Cref{thm:con-spec}(E) and the fact that products are disjoint in $\CC$ we find that these closed subspaces are disjoint as $R \coprod_{R \times S} S \cong \star$.

  To conclude, we now argue that $[R] \cup [S]$ is all of $\Spec_\CC(R \times S)$.
  For this, we use the assumption that coproducts distribute over products to conclude that, for $L$ Nullstellensatzian,
  $L \coprod_{R \times S} R$ and $L \coprod_{R \times S} S$ are both terminal
  iff $L \coprod_{R \times S} (R \times S)$ is terminal iff $L$ is terminal.
  \qedhere
  
  
\end{proof}

\Cref{lem:fin-prod} implies that the spectrum of a product $L_1 \times \cdots \times L_m$ of Nullstellensatzian objects is the discrete space with $m$ points, as desired. Moreover, \Cref{lem:fin-prod} also implies that Nullstellensatzian objects are \emph{product-indecomposible} in the sense that they cannot be written as the product of two non-terminal objects. 

\begin{lem} \label{lem:reproduce_Spec}
  Let $\CC$ be weakly spectral and op-disjunctive. 
  The restriction of $\Spec_\CC(-)$ to $\mathrm{NS}_\CC^\Pi$ is uniquely determined by the category $\mathrm{NS}_\CC^\Pi$.
\end{lem}

\begin{proof}
  \Cref{lem:fin-prod} implies that the full subcategory $\mathrm{NS}_\CC \subseteq \mathrm{NS}_\CC^\Pi$ can be singled out as the full subcategory of product-indecomposable objects.
  The lemma now follows from \Cref{lem:matches}.
\end{proof}

\begin{lem} \label{lem:rig-disjunctive}
  If $\DD$ is a semiadditively symmetric monoidal category, then $\CAlg(\DD)$ is op-disjunctive.  In particular, this holds for $\DD \in \Prig$.  
\end{lem}

\begin{proof}
  In order to check that $\CAlg(\DD)$ is op-disjunctive we must show that products are disjoint and coproducts distribute over products.
  Given $S_1,S_2,S_3 \in \CAlg(\DD)_{R/-}$ the natural map
  \[ S_1 \otimes_R ( S_2 \times S_3) \to  ( S_1 \otimes_R S_2) \times (S_1 \otimes_R S_3) \]
  is an equivalence since $S_1 \otimes_R - $ commutes with colimits in $\D$.  In order to show that products are disjoint, we examine the square of discrete commutative algebras
  \[ \begin{tikzcd}
      {[{\one}_\DD, R_1 \times R_2]} \ar[r, "\pi_1"] \ar[d, "\pi_2"] & {[{\one}_\DD, R_1]} \ar[d, "i_1"] \\
      {[{\one}_\DD, R_2]} \ar[r, "i_2"] & {\left[ {\one}_\DD, R_1 \otimes_{R_1 \times R_2} R_2 \right]}
    \end{tikzcd} \]
  and compute in the bottom right corner that
  \[ 1 = i_1\pi_1(1,1) = i_1\pi_1(0,1) + i_2\pi_2(1,0) = 0+0 = 0, \]
  and therefore $R_1 \otimes_{R_1\times R_2} R_2 \simeq 0$.  
\end{proof}

\subsubsection{Ultraproducts and infinite products}\hfill

Let us now consider the constructible spectrum of an infinite product of Nullstellensatzian objects.
Given a set of Nullstellensatzian objects $\{ L_\alpha \}_{\alpha \in U}$ in $\CC$ and a point $\mu \in \beta U$ in the Stone--\v{C}ech compactification of $U$ we are naturally led to consider the \emph{ultraproduct of the $L_\alpha$ over the ultrafilter $\mu$} which is given by
\[ \int_U L_{\alpha} d\mu \coloneqq \colim_{V \in \mu} \prod_{\alpha \in V} L_\alpha. \]

The key observation is then that, because $\mu$ is an ultrafilter, the colimit diagram above is filtered, and therefore the assumption that the terminal object of $\CC$ is strict and compact implies that \emph{$\int_U L_\alpha d\mu$ is always non-terminal}. From this one can see that the constructible spectrum of a product is always at least as large as $\beta U$.
Motivated by the example of discrete commutative rings, where ultraproducts of algebraically closed fields are themselves algebraically closed fields we place the following restriction on $\CC$:

\begin{dfn}
  We say that \deff{ultraproducts in $\CC$ are point-like} if, for any set
  $\{L_\alpha\}_{\alpha \in U}$ of Nullstellensatzian objects in $\CC$
  and ultrafilter $\mu \in \beta U$,
  the constructible spectrum of the ultraproduct
  $ \int_U L_{\alpha} d\mu $
  is a point.
  \tqed
\end{dfn}

In this paper, the most important example of a category with point-like ultraproducts is
the category of $T(n)$-local commutative $E(k)$-algebras.

\begin{exm} \label{exm:E-ultraproducts}
  As a corollary of \Cref{thm:E_functor}
  the functor $E(-) \colon \Perf_k \to \CAlgw_{E(k)}$ commutes with ultraproducts.
  Then, using our characterization of the Nullstellensatzian objects in $\CAlgw_{E(k)}$ from \Cref{thm:alpha-null}, we may conclude that an ultraproduct of Nullstellensatzian objects is Nullstellensatzian and therefore that ultraproducts in $\CAlgw_{E(k)}$ are point-like.
  \tqed
\end{exm}

\begin{lem} \label{lem:stone-cech}
  Suppose $\CC$ is op-disjunctive and ultraproducts in $\CC$ are point-like.
  Then, given a set of Nullstellensatzian objects $\{L_\alpha\}_{\alpha \in U}$
  there is a natural isomorphism
  \[ \Spec_\CC\left( \prod\nolimits_{\alpha \in U} L_\alpha \right) \cong \beta U \]
  between the spectrum of the product and the Stone--\v{C}ech compactification of the set $U$.
\end{lem}

\begin{proof}
  In order to simplify notation, for each $V \subseteq U$ we let 
  $L(V)$ denote the product $\prod_{\alpha \in V} L_\alpha$. 
  Using the assumption that $\CC^\op$ is disjunctive we have an isomorphism
  \[ L(V) \coprod_{L(U)} L(W) \cong L(V \cap W) \]
  for any pair of subsets $U,V$ of $A$. 
  Dually, the diagonal and projection maps allow us to see that
  $ \left[ L(V) \times L(W) \right] = \left[ L(V \cup W) \right] $
  in $\Spec_\CC( L(U) )$.
  In particular this implies that
  $\left[ L(V) \right]$ and $\left[ L(U \setminus V) \right]$
  partition $\Spec_\CC(L(U))$ into a pair of clopen sets.

  Using the clopen sets associated to subsets of $U$
  we can construct a map of Boolean algebras from $\mathcal{P}(U)$ to the boolean algebra of clopen sets of $\Spec_\CC(L(U))$.
  Under Stone duality this corresponds to a map of topological spaces
  \[ r \colon \Spec_\CC\left( L(U) \right) \to \beta U. \]

  We can compute fiber of $r$ over an ultrafilter $\mu \in \beta U$ by using \Cref{lem:filtered-intersection}:
  \[ r^{-1}(\mu) = \cap_{V \in \mu} [L(V)] = \left[ \colim_{V \in \mu} L(V) \right] = \left[ \int_U L_\alpha d\mu \right]. \]
  The assumption that ultraproducts are point-like in $\CC$ therefore implies the $r$ is bijective.
  Finally, as $r$ is a continuous bijection with compact source and Hausdorff target, it is an isomorphism.  
\end{proof}

\begin{prop} \label{prop:ultra-Hausdorff}  
  Suppose $\CC$ is weakly spectral and op-disjunctive, and that ultraproducts in $\CC$ are point-like.
  Then, for every $R\in \CC$, $\Spec_\CC(R)$ is Hausdorff.
\end{prop}

\begin{proof}
  As in \Cref{exm:prod-points} pick a $S \in \mathrm{NS}^\Pi$ and a $c$-cover $R \to S$.  
  We will show that $\Spec_\CC(R)$ is Hausdorff by analyzing the surjective map
  \[ q \colon \Spec_\CC(S) \to \Spec(R). \]
  
  $q$ is closed by \Cref{thm:con-spec}.
  From \Cref{lem:stone-cech} we know that $\Spec_\CC(S)$ is the Stone--\v{C}ech compactification of a set and therefore compact Hausdorff. Compact Hausdorff spaces are normal, so $\Spec_\CC(S)$ is normal. Finally, by \cite[Ex. 31.6]{Munkres}, a closed, continuous surjection with normal source has a normal target (and in particular a Hausdorff target).
  \qedhere
  
\end{proof}


\subsubsection{Spectral categories}
\label{subsub:spectral}\hfill

We have seen above that the constructible spectrum is even more well-behaved under the following additional assumptions:

\begin{dfn}\label{dfn:spectralcat}
We say that a category $\CC$ is \deff{spectral} if it is weakly spectral and additionally:
\begin{enumerate}
    \item $\CC$ is op-disjunctive.
    \item Ultraproducts in $\CC$ are point-like.  \tqed
\end{enumerate}
\end{dfn}

\begin{exm}\label{exm:Tnspectral}
  The category $\CAlg(\Sp_{T(n)})$ is spectral, by virtue of \Cref{exm:calg-sp-spectral}, \Cref{lem:rig-disjunctive}, and \Cref{exm:E-ultraproducts}.
  \tqed
\end{exm}

For spectral categories, we have the following refinement of \Cref{thm:con-spec}:

\begin{thm}\label{thm:spectralspec}
Let $\CC$ be a spectral category.  Then there is a functor
  \[ \Spec^{\cons}_\CC(-)\colon \CC^\op \to \CHaus \]
  to compact Hausdorff spaces which satisfies properties (A) through (E) of \Cref{thm:con-spec} together with:
  \begin{enumerate}
      \item[(F)] The functor $\Spec_{\CC}^{\cons}$ is left Kan extended from the full subcategory $(\mathrm{NS}^{\Pi})^{\op} \subset \CC^{\op}$ of products of Nullstellensatzian objects.  
  \end{enumerate}
  \end{thm}

\begin{proof}
    The fact that the functor lands in compact Hausdorff spaces and satisfies properties (A) through (E) follows from \Cref{thm:con-spec} combined with \Cref{prop:ultra-Hausdorff}.  Finally, (F) follows from the same proof as \Cref{prop:kan-top}, with $\CHaus$ replacing $\mathrm{Top}^{\mathrm{cpt},T_1,\mathrm{cl}}$ everywhere.   
    \end{proof}

\subsection{Examples}
\label{subsec:gspec-tn}\hfill

We conclude this section by analyzing the constructible spectrum in some examples.

\begin{lem} \label{lem:surj}
  Let $\DD \in \Pr^{\rig}$ and assume that the zero algebra is compact in $\CAlg(\DD)$.
  Then, a nil-conservative map $A \to B$ in $\CAlg(\DD)$
  induces a surjective map on constructible spectra.
\end{lem}

\begin{proof}
  The fiber of the map
  \[ \Spec_{\CAlg(\DD)}(B) \to \Spec_{\CAlg(\DD)}(A) \]
  at a Nullstellensatzian commutative $A$-algebra $L$ is given by the geometric points of $B \otimes_A L$.
  Since $A \to B$ is nil-conservative by assumption, $B \otimes_A L$ is non-zero and therefore has at least one geometric point. 
\end{proof}


As one might expect, our abstract definition of the constructible spectrum recovers the usual definition of the constructible topology on the Zariski spectrum when we take $\CC = \CRing$.


\begin{prop} \label{lem:discrete-top}
  Given a discrete commutative algebra $R$, there is an isomorphism
  \[ \Spec^\cons_{\CRing}(R) \cong \Spec^\cons_{\mathrm{Zar}}(R) \]
  between the constructible spectrum of $R$ and the Zariski spectrum of $R$ equipped with its constructible topology.
\end{prop}

\begin{proof}
  In \Cref{thm:Lang} we showed that the Nullstellensatzian objects in $\CRing$ are the algebraically closed fields. Using that fact that two Nullstellensatzian objects $R \to L_1$ and $R \to L_2$ represent the same point of $R$ if and only if their pushout, $L_1 \otimes_R L_2$, is non-zero we see that the points of $\Spec^\cons_{\CRing}(R)$ are the Zariski points. Examining the description of the constructible topology in \Cref{thm:con-spec}(C), we see that it agrees with the usual definition of the constructible topology on the Zariski spectrum.
\end{proof}

\begin{rmk}
  Since ultraproducts of algebraically closed fields are algebraically closed fields,
  we learn from \Cref{prop:ultra-Hausdorff} that for a discrete ring $R$
  its constructible spectrum is Hausdorff.
  In fact, the Zariski constructible spectrum of $R$ is a totally disconnected compact Hausdorff space.
  \tqed
\end{rmk}

This leads us to ask the following question:

\begin{qst}
  In which categories of commutative algebras 
  is the constructible spectrum a totally disconnected compact Hausdorff space?
\end{qst}

For us the most interesting case of this question is that provided by the main example of this paper, $\CAlg(\Sp_{T(n)})$.

\begin{exm} \label{exm:space}
    Let $X$ be a compact $T_1$ space. 
    The poset $\mathrm{Op}(X)$ of open sets of $X$ under inclusion is weakly spectral and it is not too difficult to verify that the Nullstellensatzian objects in $\mathrm{Op}(X)$ are the open complements of the points of $X$.  This means its spectrum $\Spec_{\mathrm{Op}}(\emptyset)$ can be identified with $X$ as a set, and one can then read off for $U \in \mathrm{Op}(X)$ that $[U]_{\emptyset} = X \setminus U$.  It follows that 
    \[ \Spec_{\mathrm{Op}(X)}(\emptyset) \cong X \]
    as topological spaces.  
    \tqed
\end{exm}

An important feature of \Cref{exm:space} is that it implies that every compact, $T_1$ space occurs as a constructible spectrum (and therefore the constructible spectrum does not, in general, have any further special properties).

\begin{exm}
    In the case of (the opposite of) profinite sets, it is not too difficult to show that the constructible spectrum functor recovers the embedding of profinite sets into compact Hausdorff spaces.
    \tqed
\end{exm}

Notably, as the next example shows, algebraically closed Lubin--Tate theories are only Nullstellensatzian in the monochromatic world. In fact, in the seemingly more general $L_n^f$-local setting, the constructible spectrum collapses to the rational case.

\begin{exm}\label{exm:lnf_spec}
    Since $\LnfSp \in \Prig$ is generated by the compact object $L_n^f(\mathbb{S})$, by \Cref{lem:cpt-gen-spectral} we have that $\CAlg(\LnfSp)$ is weakly spectral. Now let $L \in \CAlg(\LnfSp)$ be Nullstellensatzian. By the May nilpotence conjecture \cite{MNNmaynilp}, we have that $L[1/p] \neq 0$; but it is easy to see that $L[1/p]$ is compact under $L$, so we get that $L = L[1/p]$. Now let $R \in \CAlg(\LnfSp)$ be arbitrary.  We deduce from the claim above and \Cref{cor:idempotent} that the natural map 
    \[\Spec_{\CAlg(\LnfSp)}(R[1/p]) \to \Spec_{\CAlg(\LnfSp)}(R) \]
    is an isomorphism. 
    \tqed
\end{exm}


\bibliographystyle{alpha}
\bibliography{bibliography}

\end{document}